\documentclass[11 pt]{amsart}

\usepackage{latexsym}
\usepackage{amssymb, amsmath}
\usepackage{amsthm}
\usepackage{geometry}
\usepackage{hyperref}
\usepackage{changebar}
\usepackage{comment}
\usepackage[T1]{fontenc}
\usepackage[stretch=10]{microtype}
\usepackage{tikz}
\usetikzlibrary{calc}
\usetikzlibrary{intersections}
\usepackage{fixltx2e}
\usepackage{wrapfig}

\numberwithin{equation}{section}

\newtheorem{theorem}{Theorem}[section]
\newtheorem{lemma}[theorem]{Lemma}
\newtheorem{cor}[theorem]{Corollary}
\newtheorem{sublem}[theorem]{Sublemma}
\newtheorem{proposition}[theorem]{Proposition}

\newtheorem{defin}[theorem]{Definition}

\newtheorem{remark}[theorem]{Remark}

\newcommand\cB{{\mathcal B}}
\newcommand{\cC}{\mathcal{C}}
\newcommand\cD{{\mathcal D}}

\newcommand{\cF}{\mathcal{F}}
\newcommand{\cG}{\mathcal{G}}
\newcommand{\cI}{\mathcal{I}}

\newcommand{\cK}{\mathcal{K}}
\newcommand{\cL}{\mathcal{L}}
\newcommand{\cM}{\mathcal{M}}
\newcommand{\cO}{\mathcal{O}}
\newcommand{\cP}{\mathcal{P}}
\newcommand{\cQ}{\mathcal{Q}}
\newcommand\cR{{\mathcal R}}
\newcommand{\cS}{\mathcal{S}}
\newcommand{\cV}{\mathcal{V}}
\newcommand{\cW}{\mathcal{W}}
\newcommand{\cT}{\mathcal{T}}

\newcommand{\bh}{\mathbf{h}}

\newcommand{\bC}{{\mathbb C}}
\newcommand{\bE}{{\mathbb E}}
\newcommand{\bF}{\mathbb{F}}

\newcommand{\bH}{\mathbb{H}}
\newcommand{\bM}{\mathbb{M}}
\newcommand\bN{{\mathbb N}}
\newcommand{\bS}{\mathbb{S}}
\newcommand\bR{{\mathbb R}}
\newcommand\bZ{{\mathbb Z}}
\newcommand{\bW}{\mathbb{W}}

\newcommand{\bx}{\bar{x}}
\newcommand{\bxs}{\bar{x}^s}
\newcommand{\bxu}{\bar{x}^u}
\newcommand{\bys}{\bar{y}^s}
\newcommand{\byu}{\bar{y}^u}

\newcommand{\bgamma}{\bar{\gamma}}

\newcommand{\bp}{\overline{\psi}}
\newcommand{\bpsi}{\overline{\psi}}

\newcommand{\oW}{\overline{W}}
\newcommand{\bvf}{\overline{\vf}}
\newcommand{\barG}{\overline{G}}
\newcommand{\bDelta}{\overline{\Delta}}

\newcommand{\tC}{\tilde C}

\newcommand{\tG}{\widetilde{\cG}}
\newcommand{\tW}{\widetilde{W}}
\newcommand{\tV}{\widetilde{V}}
\newcommand{\tU}{\widetilde{U}}
\newcommand{\tI}{\widetilde{\cI}}
\newcommand{\td}{\tilde{d}}
\newcommand{\tx}{\tilde{x}}
\newcommand{\ty}{\tilde{y}}

\newcommand{\hx}{\hat{x}}
\newcommand{\hy}{\hat{y}}
\newcommand{\heta}{\hat{\eta}}
\newcommand{\hxi}{\hat{\xi}}
\newcommand{\hG}{\widehat{\cG}}

\newcommand{\vf}{\varphi}

\newcommand{\eps}{\epsilon}
\newcommand{\ve}{\varepsilon}

\newcommand\up{\varkappa}
\newcommand\vu{{\tau}_-}       

\newcommand\epsilonr{{\varepsilon}}
\newcommand\vus{{\tau_+}}
\newcommand\ho{{\varpi}}

\newcommand{\Cs}{C_{\#}}
\newcommand\Id{\text{\bf Id}}
\newcommand\intt{\text{int}\, }
\newcommand\spp{\text{sp}\, }
\newcommand\Dom{\text{Dom}\, }
\newcommand{\gap}{\mbox{\scriptsize gap}}

\def\beq{\begin{equation}}
\def\eeq{\end{equation}}
\newcommand*\circled[1]{\textcircled{\raisebox{-.6 pt}{\footnotesize {#1}}}}

\newcommand{\oldh}{f} 
\newcommand{\oldp}{{\alpha}} 
\newcommand{\oldq}{\beta}
\newcommand{\oldbeta}{{1/q}} 
\newcommand{\oldbetazero}{{1/q_0}} 
\newcommand{\oldbetatwo}{{1/(2q)}} 
\newcommand{\oldbetaparam}{q} 
\newcommand{\oldbetazeroparam}{q_0} 
\newcommand{\bal}{\boldsymbol{\alpha}}
\newcommand{\balzero}{\boldsymbol{\alpha_0}}
\DeclareMathOperator{\extd}{\mathbf{d}}

\newcommand{\Rtime}{R}
\newcommand{\bRtime}{\bar\Rtime}
\DeclareMathOperator*{\esssup}{ess-sup}
\newcommand{\bomega}{{\boldsymbol{\omega}}}
\newcommand{\BG}{{\boldsymbol{L}}}
\newcommand{\bdelta}{{\boldsymbol{\delta}}}
\newcommand{\Beta}{{\boldsymbol{\eta}}}

\begin{document}

\title[Exponential decay of correlations for
Sinai billiard flows]{Exponential Decay of Correlations \\for  Finite Horizon Sinai Billiard Flows}
\author{Viviane Baladi 
\and Mark F. Demers
\and Carlangelo Liverani
}
\address{
Sorbonne Universit\'es, UPMC Univ. Paris 6, CNRS, Institut de Math\'ematiques de Jussieu\\ 
\phantom{De}(IMJ-PRG), 
4, Place Jussieu, 75005 Paris, France}
\email{viviane.baladi@imj-prg.fr}
\address{Department of Mathematics, Fairfield University, Fairfield CT 06824, USA}
\email{mdemers@fairfield.edu}
\address{Dipartimento di Matematica
Universit\`a degli Studi di Roma Tor Vergata, Italy}
\email{liverani@mat.uniroma2.it}
\dedicatory{
In memoriam Nikolai Chernov 1956--2014.}
\thanks{CL acknowledges the partial support of the European Advanced Grant  Macroscopic Laws and Dynamical Systems (MALADY) (ERC AdG 246953). 
 MD  was  partially  supported  by  NSF  grants  DMS
1101572 and 1362420.
Part of this work was carried out by VB and MD at the Newton Institute
in 2013 (Mathematics for the Fluid Earth), during visits of VB and MD to Roma 
in 2013, 2014, 2015, and a visit of MD to Paris in 2014,  using
the  Action incitative ENS {\it Spectres en dynamique, g\'eom\'etrie et th\'eorie des
nombres.} VB is grateful to Peter B\'alint for 
his patient explanations on billiards and
thanks S\'ebastien Gou\"ezel for many key ideas. We thank I.P. Toth for several useful remarks.
We are grateful to the two anonymous referees
for many  constructive comments.}

\date{June 6, 2017, amended following referees' reports}

\begin{abstract}
We prove exponential decay of correlations for the billiard flow associated with 
a two-dimensional finite horizon Lorentz Gas (i.e., the  Sinai billiard flow with finite
horizon). Along the way, we describe the spectrum of the generator of the corresponding
semi-group $\cL_t$ of transfer operators, i.e., the resonances of the Sinai billiard flow, on a suitable
Banach space of anisotropic distributions.
\end{abstract}
\maketitle

\tableofcontents

\section{Introduction and statement of results}

This paper completes, on a conceptual level, the study of decay of correlations
of planar dispersing billiard
systems  initiated by Sinai's seminal papers \cite{Sinai63, Sinai70} in which he extended the 
ideas of Hopf and Anosov to the case of piecewise smooth dynamical systems.  
Sinai's breakthrough prompted several works establishing ergodicity for 
more and more general systems. 
Yet, the quantitative study of their mixing properties had to wait almost twenty years until \cite{BSC91} 
established sub-exponential decay of correlations (and the Central Limit Theorem)  for the 
collision map 
associated to certain dispersing billiard systems. The question remained if the discontinuities prevented 
exponential decay of correlations or not. The question was settled in \cite{Li95, Y98, Ch99} where exponential 
decay of correlations was established for a  large class of 
discrete-time dynamical systems with discontinuities (including 
 Poincar\'e maps for the finite horizon Lorentz gas in \cite{Y98}, and more general dispersing
billiards in \cite{Ch99}).

These results for the billiard collision map did not settle the question of the rate of 
decay of correlations for the billiard flow.
It is well known that this seemingly 
uneventful step is highly non trivial: In the case of smooth  systems 
it took 26 years to go from the proof of exponential
decay of correlations for Anosov maps \cite{Sinai72, ruelle} to the first results, by Dolgopyat, on 
exponential decay of correlations for
Anosov flows \cite{Do}.  
The first progress for dispersing billiard flows 
was made by Melbourne \cite{Mel2007} and Chernov \cite{Ch2007}, who proved super-polynomial and 
stretched-exponential decay, respectively. 
The methods of \cite{Y98, Mel2007, Ch2007} employ some kind 
of countable Markov partitions.  As pointed out by one of the
referees, two very recent works have successfully combined Dolgopyat arguments with
countable Markov partitions (in the spirit of \cite{BV, AGY}) to get exponential decorrelations for
(non-billiard) hyperbolic flows with singularities: K. Burns et al. \cite{BMMW}  for some Weil--Petersson flows and Ara\'ujo--Melbourne \cite{AM} for the Lorenz attractor.
However, these works do not constitute compelling evidence for  {\it dispersing billiard flows,}  the difficulties of which come from  the  severe lack of smoothness of the foliations
(which are measurable, but not  continuous). Indeed,  in \cite{AM} the stable foliation is H\"older and consists of long leaves. The situation of \cite{BMMW}  is even simpler, since a very good description of the structure of the singularities is available, and the foliation is  rather regular.
It seems much more difficult (at least to the authors of the present
paper) to implement Dolgopyat-type arguments with
countable Markov partitions to get exponential mixing for dispersive billiard flows. It is thus natural to try a more direct  line
of attack, studying the transfer operators on suitable spaces\footnote{ Unfortunately, the spaces introduced 
e.g. in \cite{BTs, GL, FRS, GLP}
for Anosov diffeomorphisms or flows do not work well
in the presence of discontinuities.} of anisotropic distributions defined 
on the manifold, thus bypassing   non-intrinsic constructions.

The first exponential decay result for piecewise smooth hyperbolic {\it flows} was
obtained by \cite{BaL} who used such a direct functional approach:
They were able to build on Liverani's version \cite{Li04} of the Dolgopyat
argument for contact Anosov flows  and an anisotropic space construction
from \cite{BG}  to prove exponential decay for a large class of piecewise smooth contact
hyperbolic flows.
However, this class   did not
contain billiards, since the blowup of derivatives at the boundary
of the domain (corresponding to grazing orbits in billiards) was not allowed.  
In parallel,  \cite{demers zhang}
succeeded in adapting the anisotropic space previously introduced for
piecewise hyperbolic discrete-time systems by \cite{DL} to allow for the
blowup of derivatives at boundaries, giving a new functional
proof of Young's \cite{Y98} exponential mixing for the billiard {\it map.} 
(This function
space approach  has recently been  applied to a wide variety of billiard  maps and their
perturbations
\cite{demers zhang 2, demers zhang 3}.)

In the present paper, we  exploit the fact that the billiard flow preserves a contact form 
and combine the methods of \cite{BaL} with the function spaces
of \cite{demers zhang}   to establish 
exponential decay of 
correlations (Corollary ~\ref{truemain}) for the billiard flow associated with a finite horizon Lorentz gas.
As a byproduct of our proof, we obtain information on the resonances of the billiard
flow (Theorem ~\ref{resonances}), that is, the spectrum of the generator $X$ of the semi-group
$\cL_t$ given by the flow.  We warn the reader that, just like in \cite{Li04, BaL}, although we do
obtain a spectral gap for the spectrum of the generator $X$ on an anisotropic Banach space $\cB$, our method does {\it not} show that
the time-one transfer operator $\cL_1$ has  a spectral gap on $\cB$. (Note that $X$ is  closed but not bounded on $\cB$.) The spectral gap of $\cL_1$
on a (different) anisotropic Banach space was obtained by Tsujii (\cite[Thm 1.1]{Tsu1}) in the easier case of
smooth Anosov flows.

One key technical hurdle we had to overcome, in order to carry
out the Dolgopyat cancellation argument from \cite{Li04, BaL}
in the dispersive billiard setting,
 is the construction in Section~\ref{Lipschitz}
of a suitable {\it approximate (fake) unstable foliation for the billiard}, see Theorem \ref{thm:foliation} for a detailed description of its properties.
This (a posteriori very natural)
delicate construction is one of the main novelties of the present work (see also Remark ~ \ref{rem:foliation}
below).  Along the way we discovered that using in addition the analogous fake stable foliation
allows a much more  systematic writing of  the Dolgopyat cancellation argument (Section ~\ref{dodo}).
Note also that we  replace the $\cC^{1+\ho}$ estimates on the
holonomies used in \cite{BaL} by ``four-point conditions'' (see condition (vii) and Lemma~ \ref{lem:holder jac} in Section
~\ref{Lipschitz}) which allow one to control the
$\cC^{1+\ho}$ norm of  differences of holonomies (\eqref{eq:amazing}).

\begin{remark}\label{rem:foliation} The construction
of Section~\ref{Lipschitz} and in particular
properties (i)-(viii) there, can be easily used to prove the following interesting property of the stable and unstable foliations of the billiard flow (which are only measurable): 
For each small enough\footnote{It suffices to take $\eta < L_0^5$, where $L_0$ is from
Lemma~\ref{lem:growth}.} $\eta > 0$, there exists a set of Lebesgue measure at most $C
\eta^{\frac 45}$, the complement of which is foliated by leaves of the
stable (or unstable) foliation with length at least $\eta$ such that the
foliation is Whitney--Lipschitz, with
Lipschitz constant not larger than $C\eta^{-\frac 45}$.
(See Remark~\ref{opt} on the improvement of this exponent.)  An analogous
statement holds for the billiard map. 
To our knowledge, this fact was not previously known. We refer to
\cite{KS} for related information on dynamical foliations of billiards.
\end{remark}

The present work  finally
settles the issue of whether billiard 
flows can have exponential decay of correlations. In addition, we  obtain new spectral
information and we develop new tools. 
Many problems remain open (more general dispersing billiards, higher 
dimensional billiards, dynamical zeta functions -- which would give another
interpretation of the resonances --- other Gibbs states
etc.), and we hope that some of the techniques in the present paper may
be useful towards their solution.
Some limit theorems on the billiard flow can be obtained from
information on the Poincar\'e map (see \cite{MT} for the central
limit theorem\footnote{ Speed estimates can be found in \cite{LP}.} and almost sure invariance principle, see
 \cite{M?} for large deviations, see \cite{Pe} for some
Berry--Esseen bounds). Although  the existing proofs \cite{I} of local limit theorems with error terms
are based on a spectral gap for the transfer operator $\cL_1$,
 we hope that a spectral
gap for (suitable perturbations of) the generator $X$ of the  flow  will lead to e.g., a local limit theorem with error terms,
better Berry--Esseen estimates, and rate functions for large deviations.

We remark that the results of the present paper do not apply to the billiard flow
corresponding to an infinite horizon periodic Lorentz gas.   Recall that
although the discrete time collision
map for such billiards enjoys exponential decay of correlations \cite{Ch99},
it is anticipated by physicists  \cite{FM, MM} that correlations for the flow
decay at a polynomial rate.

The paper is as self-contained as possible, with precise references to the
book by Chernov and Markarian \cite{chernov book} or  to the previous
works \cite{BaL} or \cite{demers zhang}, whenever we use  non trivial
facts. Regarding operator-theoretical or
functional-analytic background, very little is expected from
the reader, and we again give precise references to Davies \cite{Davies}.

\subsection{The main results (Theorems ~\ref{main} and ~\ref{resonances}, Corollary \ref{truemain})}
\label{setting}

We shall state our results in the setting of the Sinai billiard flow.
A more general axiomatic setting can probably be considered
in the spirit of \cite{demers zhang 3}.

Let $B_i$, $i = 1, \ldots d$, denote open, convex sets in the two-torus 
$\mathbb{T}^2 = \mathbb{R}^2 / \mathbb{Z}^2$.
We assume that the closures of the sets $B_i$ are pairwise disjoint and that their
boundaries $\Gamma_i$ are $\cC^3$ curves with strictly positive curvature.
 We consider the motion of a point particle in the domain 
$Q := \mathbb{T}^2 \setminus (\cup_{i=1}^d B_i)$ undergoing elastic collisions at the boundaries
and maintaining constant velocity between collisions.  
At collisions, 
the velocity vector changes direction, but not magnitude.
Thus we set the magnitude of the velocity vector equal to one and view the 
phase space of the flow as three-dimensional. 

We adopt the  coordinates  $Z = (x,y, \omega)$ for the flow, where 
$(x,y) \in Q$
is the particle's position on the table, and 
$\omega \in \bS^1$ (we view $\bS^1$ as the quotient of $[0, 2\pi]$, identifying endpoints)
is the angle made by the velocity vector
with the positive $x$-axis. 
The quotient
$$\Omega := \{ (x,y, \omega)\mid (x,y) \in Q\, , \quad \omega \in \bS^1\}  /\sim$$
is often used as the phase space for the billiard flow $\Phi_t:\Omega\to \Omega$, where
$\sim$ identifies   ingoing and outgoing velocity
vectors at the collisions, via the reflection with respect to the boundary
of the scatterer at $Z\in \partial Q$. (It is then  customary
to work with the outgoing -- post-collisional --  velocity vector at collisions.)
The  topological metric space $\Omega$  can be endowed with H\"older
(non differentiable) charts. 
It is thus more convenient\footnote{ In view of defining $\cC^2$ stable curves which can touch the scatterers
on their endpoints,
$\cC^2$ test functions (when considering $\cC^\ell$ functions on $\Omega_0$, we mean
differentiable in the sense of Whitney, viewing $\Omega_0$ as a subset of the
three-torus).} to work
with the $\cC^2$ manifold with boundary 
$$\Omega_0 := \{ (x,y,\omega) \mid (x,y) \in Q\, , \quad \omega \in \bS^1\}  \subset \mathbb{T}^3 =
\mathbb{T}^2\times  \bS^1
\, .$$
For nonzero $t$, we may consider the time-$t$ billiard flow $\Phi_t$ on $\Omega_0$.
The time-zero map $\Phi_0$ is the identity in the interior of $\Omega_0$,
and also at grazing collisions in $\partial \Omega_0$, where the incoming and outgoing
angles coincide. (Note that at $(x,y) \in \partial Q$ there are
exactly two grazing angles, $\omega_{gr}(x,y)\in [0, \pi)$ and $\omega_{gr}(x,y)+\pi$. These two
angles
divide the circle in the two arcs of ingoing, respectively outgoing, angles.). For $(x,y) \in \partial Q$ and $\omega^-$ a 
non-grazing incoming angle, we have that the (reflected) outgoing angle
 $\omega^+$ 
 (i.e., so that  $\Phi_0(x,y,\omega^-)=(x,y,\omega^+)$) is different from  $\omega^-$. In particular,
$\Phi_t(x,y,\omega^-)=\Phi_t(x,y,\omega^+)$ for such $(x,y,\omega^-)$ and
all $t>0$, so that the time-$t$
map is not injective. Similar properties hold
for $\omega$ a non-grazing outgoing angle, reversing time.
The nature of the Banach space norms defined below 
will ensure that these apparent flaws in the flow do
not create problems.\footnote{ In particular strong continuity holds for the transfer operator
$\cL_t$ associated to $\Phi_{t}$ for $t>0$, see Lemma~\ref{lem:strong c}.}

The billiard flow is our primary object of study, but it will sometimes be convenient to use the
billiard map (also called collision map), i.e., the Poincar\'e map $T:\cM \to \cM$
of the flow on the union of scatterers $\cM = \cup_{i=1}^d \Gamma_i \times [-\frac{\pi}{2}, \frac{\pi}{2}]$. 
Natural
coordinates for the collision map $T$ are $(r,\vf)$, where $r$ represents the position on the
boundary of the scatterer, parametrized by arclength and oriented positively with
respect to the domain $Q$, and
$\vf$ is the angle made by the outgoing (post-collision) velocity vector with the outward
pointing normal to the boundary of the scatterer.  
Let $\tau(Z)$ denote the time when the particle at $Z\in  \Omega_0$ first
collides with one of the scatterers.  (It is known that $\tau$ is 
a piecewise $1/2$ H\"older function, see the proof of
Lemma~\ref{lem:smooth}.) We assume that the table has {\em finite horizon},
i.e., there exists $\tau_{\max} < \infty$ such that $\tau(Z) \le \tau_{\max}$ for all $Z \in \Omega_0$.  Since we have assumed the scatterers are a positive distance apart, there
exists a constant $\tau_{\min}>0$ such that $\tau(r,\vf) \ge \tau_{\min}$ for all $(r,\vf) \in \cM$.

\bigskip
The billiard flow preserves the normalised Lebesgue
measure (see e.g. \cite[Def. 2.23]{chernov book})
\begin{equation}\label{Lebb}
m= (2\pi |Q|)^{-1} dx dy d\omega\, ,
\end{equation}
where $|Q|$ denotes the area of the billiard domain.
Our first result gives a precise description of the Sinai billiard flow decorrelations:

\begin{theorem}[Fine correlation structure for the Sinai billiard flow]\label{main}
Let $\Phi_t:\Omega_0\to \Omega_0$ be a finite horizon 
(two-dimensional) billiard flow associated to finitely many  scatterers
$B_i$ with $\cC^3$ boundaries of positive curvature and so that the $\overline B_i$
are pairwise disjoint. Then there exist a constant $\upsilon_{Do}>0$, and for any
$\upsilon_1<\upsilon_{Do}$ a constant
 $C_1>0$, a finite dimensional vector space $F\subset (\cC^1(\Omega_0))^*$, a (nontrivial) bounded operator $\Pi: \cC^2(\Omega_0)\cap\cC^0(\Omega)\to F$ and a matrix $\widehat X\in L(F,F)$ so that for any $\psi\in \cC^1(\Omega_0)$ and $\oldh \in \cC^2(\Omega_0)\cap \cC^0(\Omega)$ we have\footnote{ For integer $\ell$, we set $|f|_{\cC^\ell(\Omega_0)}=\sum_{k=0}^\ell \sup_{x \in \Omega_0}\max_{|\vec k|=k}|\partial^{\vec k} f(x)|$, where $\partial^{\vec k}$ is the partial derivative associated
to $\vec k=(k_1, k_2,k_3)\in \mathbb{ Z}_+^3$ and $|\vec k|=k_1+k_2+k_3$, using the natural chart on the torus.}
$$
\left|\int (\psi \circ \Phi_t) \, \oldh \, dm -
\int \psi e^{\widehat X t} \Pi\oldh \, dm \right|\le C_{1}|f|_{\cC^2(\Omega_0)}|\psi|_{\cC^1(\Omega_0)} \cdot 
e ^{-\upsilon_{1} t} \, , \forall t \ge 0 \,  . 
$$
In addition, the spectrum $\spp(\widehat X)\subset\{z\in\bC\;:\;  -\upsilon_{Do}<\Re(z)< 0\}\cup\{0\}$, and zero is a simple eigenvalue with contribution to $\Pi(f)$
given by $\int f\, dm$.
\end{theorem}

A lower bound for $\upsilon_{Do}$, given by Proposition~\ref{dolgo} below, 
can be  explicitly traced from the proof of that proposition
(it
depends on the constants $\gamma_{Do}$ and $C_{Do}$ from 
Lemma~\ref{dolgolemma}: we show there that $\gamma_{Do}\ge \frac 1 {1222}$ and $C_{Do}$ can also be made explicit, depending in particular on the hyperbolicity
exponents \eqref{Lambda}, if desired) but is rather small
and not  optimal.  We refrain from stating the lower bound for $\upsilon_{Do}$ which would be very unwieldy.

Given Theorem \ref{main} it is immediate, by a standard approximation argument, to obtain
our main result of exponential decay of correlations for the
Lebesgue invariant measure of the  billiard flow \eqref{Lebb} and
H\"older test functions.
\begin{cor}[Exponential mixing of the Sinai billiard flow]\label{truemain} Under the hypotheses of Theorem~\ref{main}, for any
$\kappa \in (0, 1)$ there exists $0 < \upsilon_{corr}(\kappa) < \upsilon_{Do}$ and $C_{corr}(\kappa)>0$ so that for any $\psi, \oldh \in \cC^\kappa(\Omega_0)$ we have\footnote{ If  $\kappa\in (0,1)$, we set $|f|_{\cC^\kappa(\Omega_0)}=\sup_{x\in \Omega_0}|f(x)| + \sup_{x\ne y\in \Omega_0}
\frac{|f(x)-f(y)|}{d(x,y)}$, for  the distance on $\Omega_0$ induced by $\mathbb{T}^3$.}
$$
\left|\int (\psi \circ \Phi_t) \, \oldh \, dm -
\int \psi \, dm \cdot \int \oldh \, dm \right|\le C_{corr}(\kappa)|\psi|_{\cC^{\kappa}(\Omega_0)}|\oldh|_{\cC^{\kappa}(\Omega_0)} \cdot 
e ^{-\upsilon_{corr}(\kappa) t} \, , \forall t \ge 0 \,  . 
$$
\end{cor}

In Corollary~\ref{truemain}, we may take $\psi$ or $\oldh$ (or both)
to be
the velocity  function (which belongs to $C^\kappa(\Omega_0)$), 
or 
the free flight function $\tau$ (the piecewise H\"older function $\tau$ 
can be approached by a function in $C^2(\Omega_0)$, either by
adapting the proof of \cite[Lemma 3.7]{demers zhang} or by noticing
that $\tau$ belongs to a fractional Sobolev space $H^\kappa_p$ with
$\kappa>0$ and $p>2$, and using mollification.).


The proof of Theorem~\ref{main} will be completed in Subsection~\ref{subsec:theend}.
It  is based on a study of the semi-group of transfer operators
defined  by
(a priori just measurably,
since  $\Phi_t$ is not  injective)
$$
\cL_t \oldh = \oldh \circ \Phi_{-t} \,  , \quad t \ge 0\, ,
$$
acting on a suitable Banach space $\cB$ (see Definition~\ref{spcce}) of anisotropic distributions.
Along the way, we shall obtain information on the ``resonances'' of the flow,
that is the spectrum (on $\cB$) of the generator
$X$ of the semi-group, defined by
$
X \oldh =\lim_{t \downarrow 0} \frac{\cL_t \oldh - \oldh }{t}
$,
whenever the limit exists (see Section ~\ref{sec:R}).
To state the corresponding result, letting $\cK(r)$ denote the curvature at a point $r$ on the boundary $\Gamma_i$ of a scatterer, first note that
our assumptions 
 ensure that
there exist constants $0<\cK_{\min} \le \cK_{\max} < \infty$ with
$\cK_{\min} \le \cK(r) \le \cK_{\max}$, for all $r \in \Gamma_i$. 
Then set
\begin{equation}
\label{Lambda}
\Lambda_0 = 1 + 2\cK_{\min} \tau_{\min}\, , \qquad \Lambda = \Lambda_0^{1/\tau_{\max}} \,  .
\end{equation}
$\Lambda_0$ and $\Lambda$ are the hyperbolicity exponents
 (minimum expansion and maximum contraction)
 of the billiard map and flow, respectively (see \cite[eqs. (4.6), (4.17), (4.19)]{chernov book} for $\Lambda_0$). Our value for $\upsilon_{Do}$ satisfies $\upsilon_{Do}<\frac{1}{6} \log \Lambda$.

\begin{theorem}[Resonances of the Sinai billiard flow]\label{resonances}
For any 
$\upsilon_{ess} \in (\upsilon_{Do}, \tfrac{1}{4} \log \Lambda)$
there exists a Banach space $\cB$ of distributions on $\Omega_0$
(with the embeddings $
\cC^1(\Omega_0) \hookrightarrow \cB \hookrightarrow (\cC^1(\Omega_0))^*$) 
so that $X$ is a closed operator
on $\cB$ with a dense domain, and the spectrum $\spp(X)$ 
on $\cB$ satisfies:
\item {a)}
The intersection
$\spp(X) \cap \{  \Re (z) > -\upsilon_{ess}\}$ consists of (at most countably many) isolated eigenvalues $\Sigma=\{ z_j\, , j \ge 0\}$
of finite multiplicities. (``Discrete eigenvalues of the generator.'')
\item {b)} There exists $0<\upsilon_0 \le \upsilon_{Do}<\upsilon_{ess}$ so that
$\spp(X) \cap \{  \Re (z) > -\upsilon_0\}=\{z_0=0\}$,
which is an eigenvalue of algebraic multiplicity equal to one.  (``Spectral gap.'')
\item {c)} $\spp(X)\cap \{z\in\bC\;:\; \Re(z)>-\upsilon_{Do}\}=\spp(\widehat X)$, where multiplicities coincide.
(``Resonances.'')
\end{theorem}

(Note that (c) implies (b), taking $\upsilon_0 = \upsilon_{Do}$, but (c) does not imply (a),  since  $\upsilon_{ess}$
can be taken arbitrarily close to $\tfrac{1}{4} \log \Lambda$. Note also that Theorem~\ref{resonances} does not imply Theorem~\ref{main} immediately --- we use \cite{Butterley} --- since (b) gives
a spectral gap for $X$ and not for any individual $\cL_t$.)
 Claims (a) and (b) encapsulate Lemma~\ref{lem:C1},  Corollaries~ 
 \ref{spX} and   \ref{specX'}, 
 together with
Remark ~\ref{opt}.
It would of course be very interesting to have examples where $\spp(\widehat X)\ne \{0\}$
and examples where $\Sigma \setminus \spp(\widehat X)$ is not empty (even better, an infinite set).

\begin{remark}[On hyperbolicity, transversality, and complexity]\label{cutup}
Billiards  are well known to possess families of invariant
stable and unstable cones, and this is of course  crucial
to get exponential mixing.
For abstract piecewise smooth and hyperbolic (or expanding) systems in dimension
two or higher, it is also essential
in order to get a spectral gap that ``hyperbolicity dominates complexity,'' and
a corresponding {\it complexity} assumption is present
 e.g. in \cite{Y98, BaL}.
This assumption is well known to hold for billiards and is encapsulated
(and precisely quantified) in
the one-step expansion (see e.g. \cite[Lemma 5.56]{chernov book})  used
by Demers--Zhang \cite{demers zhang} for the discrete time billiard, and that we use several times below
(most importantly in Lemma ~ \ref{lem:growth}).
 An additional essential ingredient in the piecewise hyperbolic case is
{\it transversality} between the stable (or unstable) cones and the hyperplanes\footnote{ In the language of \cite{BaL}  the transversality condition 
is relative to the lateral sides of the flow boxes.}
of discontinuities or singularities. In the context of Sinai billiard maps, uniform transversality
between the stable cones and the boundaries of all homogeneity layers
(a sequence of hyperplanes approaching the hyperplane of
grazing singularities, see Definition~\ref{def:defhom}) is again a
well known property.
For the flow, the transversality is more delicate, and in fact we make use of 
a weak transversality property throughout (as a single example,
we mention Lemma~\ref{lem:discarded}), yet
it is neither easy nor necessary to isolate a specific statement or definition embodying
the notion of weak transversality.
\end{remark}

The paper is organised as follows:
Section~\ref{defspace} contains  definitions of our norms (based on
a notion of admissible stable curves $W\in \cW^s$, see Definition ~\ref{admiss}) and Banach spaces 
 $\cB$ 
and
$\cB_w$ of distributions on $\Omega_0$ (Definition ~\ref{spcce}).
In Section ~ \ref{section3}, we prove on the one hand some lemmas on
growth and distortion under the action of the flow $\Phi_t$, in particular giving invariance
of the class $\cW^s$ of stable curves  (modulo the necessary ``cutting up'' causing the complexity
evoked 
in Remark  ~ \ref{cutup},  and which
is controlled by Lemma ~ \ref{lem:growth}), and on the other hand 
the key compact embedding statement  $\cB \subset \cB_w$ (Lemma ~ \ref{lem:compact}).
Just like in \cite{Li04, BaL}, our spectral study of the generator $X$
is based on an analysis of the resolvent
$$
\cR(z) = (z\Id-X)^{-1}\, \qquad z=a+ib \in \bC\, .
$$
In Section~ \ref{sec:R}, we obtain  Lasota--Yorke type estimates (Proposition ~ \ref{prop:LY R}
and its consequences Corollaries ~ \ref{CorLY} and ~ \ref{spX})
on $\cR(z)$ which give (a little more than) claim (a) of  Theorem~\ref{resonances}.
Since the resolvent $\cR(z)$ can  be expressed as the Laplace transform \eqref{eq:R} of the transfer
operator $\cL_t$, these Lasota-Yorke estimates follow from estimates on the
transfer operator proved in Section ~ \ref{section4}.
The approximate unstable foliation is constructed and studied in
Section  ~ \ref{Lipschitz} and Appendix ~ \ref{postponeproof}.
The delicate Dolgopyat-type cancellation Lemma ~ \ref{dolgolemma} which bounds
$$
\int_W \psi\, \cR(a+ib)^m (\oldh)\, dm_W 
$$
is stated and proved
in Section~ \ref{dodo} and Appendix ~\ref{sec:hoihoi}.
Since the right-hand side of this lemma involves the supremum
and Lipschitz norms of the argument $\oldh$, we use mollification operators
$\bM_\epsilon$ 
like\footnote{ Note however that the Sobolev nature
of the norms in \cite{BaL} made the corresponding estimates much
easier than in the present work.} in \cite{BaL} to replace the distribution
$\oldh \in \cB$ by the function $\bM_\epsilon(\oldh)$. The
operators $\bM_\epsilon$ are introduced and studied in Section ~ \ref{molsec}.
Finally, in Section ~ \ref{theend},
putting together Lemma ~ \ref{dolgolemma} and Proposition ~ \ref{prop:LY R},
we first prove the spectral gap (claims (a) and (b)
of Theorem ~ \ref{resonances}) of the generator
$X$ and then, applying recent work of Butterley \cite{Butterley}, we easily
obtain   Theorem~\ref{main} (and thus Corollary ~\ref{truemain} on exponential mixing)
and claim (c) of Theorem ~\ref{resonances}.

\bigskip
{\bf Notation.} We use  $A=C^{\pm 1} B$ for $A, B \in \bR$ and $C\ge 1$ to mean
$
\frac{|B|}{C}\le |A| \le C |B|
$.
Let $0<\oldq\le 1$ 
be a real number and $\oldh$ be a real or complex-valued function on a
metric space $(W, d_W)$, then
$\cC^\oldq(\oldh)$ 
(or $\cC^\oldq_W(\oldh)$ for emphasis) denotes the $\oldq$-H\"older constant of $\oldh$
and $|\oldh|_\infty$ (or $|\oldh|_{L^\infty(W)}$ for emphasis) denotes the supremum of $\oldh$ on $W$.


\section{Definition of the norms}
\label{defspace}
\subsection{Stable and unstable cones for the flow}

We recall the (well known) stable cones for the map $T$ and use them to define stable cones for the flow $\Phi_t$.   We first need to introduce standard notation.
In the coordinates $(x,y,\omega)$ and $(r,\vf)$ from Subsection ~\ref{setting}, 
the flow is given between collisions by
\[
\Phi_t(x_0,y_0,\omega_0) = (x_t, y_t, \omega_t) = (x_0 + t \cos \omega_0, y_0 + t \sin \omega_0, \omega_0)
\]
and at collisions by
$(x^+, y^+, \omega^+) = (x^-, y^-, \omega^- + \pi - 2 \vf)$, where $(r,\vf)$ is the collision
point. 

We will also  work with the following coordinates in the tangent space.  
Setting $\hx$ and $\hy$
to be unit vectors in the $x$ and $y$ directions, respectively, we define
\begin{equation}\label{Jacobi}
\heta = (\cos \omega) \hx + (\sin \omega) \hy , \; \; \; \hxi = (- \sin \omega) \hx + (\cos \omega) \hy
\,  .
\end{equation}
The corresponding coordinates $(d\eta, d\xi, d\omega)$ 
in the tangent space are called the Jacobi coordinates for the flow.
In these coordinates, the linearized flow is given between collisions by (\cite[(3.26)]{chernov book})
\[
D_Z\Phi_t (d\eta, d\xi, d\omega) = (d\eta_t, d\xi_t, d\omega_t) = (d\eta, d\xi + t d\omega, d\omega)
\]
and at collisions by (see \cite[(3.28)]{chernov book})
$(d\eta^+, d\xi^+, d\omega^+) = (d\eta^-, - d\xi^-, - d\omega^- - \frac{2 \cK(r)}{\cos \vf} d\xi^-)$,
where the collision is at the point $(r,\vf)$.

A crucial feature  is that the $d\xi d\omega$-plane perpendicular to the flow direction
is preserved under the flow.  In addition, the flow preserves horizontal 
planes ($\omega = $ const.) between collisions.  We will use these facts to 
introduce after Definition~\ref{admiss} the family $\cW^s$ of stable curves 
with which our Banach norms will be defined.

For $Z \in Q \times \bS^1$, let $P^+(Z)$ denote the point in $\cM=(\cup_i \Gamma_i) \times \bS^1$
that represents the first collision of $Z$ with the one of the scatterers under the flow.
Similarly, denote by $P^-(Z)\in \cM$ the point of first collision of $Z$ with one of the
scatterers under the backwards flow.  Since $P^+(Z)$ and $P^-(Z)$ lie in the phase space
for the map $T$, we will sometimes refer to these points in the map coordinates $z=(r,\vf)$.
In these coordinates, we have the following standard choice
of globally defined unstable and stable cones for the map. (See e.g. \cite[(4.13)]{chernov book}.)
\[
\begin{split}
C^u_z & = \left\{ (dr, d\vf) \in \bR^2 : \cK_{\min} \le \frac{d\vf}{dr} \le \cK_{\max} + \frac{1}{\tau_{\min}} \right\}  \\
C^s_z & = \left\{ (dr, d\vf) \in \bR^2 : - \cK_{\min} \ge \frac{d\vf}{dr} \ge - \cK_{\max} - \frac{1}{\tau_{\min}} \right\} \, .
\end{split}
\]
We have (see \cite[Exercise 4.19]{chernov book}) for any $z=(r,\vf)\in \cM$
\begin{equation}\label{mapinvar}
D_z T (C^u_z) \subset \intt C^u_{T(z)} \cup \{0\}\, \quad
D_z T^{-1} (C^s_z) \subset \intt C^s_{T^{-1}(z)} \cup \{0\}  \, .
\end{equation}

To translate these cones to the Jacobi coordinates for the flow, we use 
\cite[eq.~(2.12)]{chernov book} (note that $\psi$ there denotes
$\psi=\pi/2-\varphi$, with $d\psi=-d\varphi$, see also \cite[eq. (3.19), p.51]{chernov book}), which yields
\begin{equation}
\label{eq:translate}
d\omega = -\cK(r) dr + d\vf \qquad d\xi = \cos \vf dr -\tau(Z) d\omega \, ,
\end{equation}
for a tangent vector $(0, d\xi, d\omega)$, perpendicular to the flow at a point
$Z = (x, y, \omega)$ that collides with the boundary at the point $(r,\vf)$
at time $\tau(Z)$, i.e., $P^+(Z) = (r,\vf)$. 
The slope of this vector in the $d\xi d\omega$-plane is thus
\[
\frac{d\omega}{d\xi} = \frac{1}{\frac{\cos \vf}{-\cK(r) + \frac{d\vf}{dr}} - \tau(Z)}\,  .
\]
Flowing $C^s_z$ backward from the boundary yields the following family of
stable cones for the flow
\begin{equation}
\label{eq:stable cone}
C^s(Z) = \biggl \{ (0,d\xi, d\omega) 
: 
-\frac{1}{\frac{\cos \vf(P^+(Z))}{ \cK(P^+(Z)) + \cK_{\min}} + \tau(Z)} \ge \frac{d\omega}{d\xi} 
\ge -\frac{1}{\frac{\cos \vf(P^+(Z))}{\cK(P^+(Z)) + \cK_{\max} + \frac{1}{\tau_{\min}}} + \tau^(Z)} \biggr \} \, .
\end{equation}
Thus the stable cones flow backwards only until the first collision $P^-(Z)$ 
and then are reset.

\begin{remark}\label{rk1.1} Equivalently, in terms of
$1/\mbox{(wave front curvature)}$, the cone $C^s(Z)=$ 
\[
 \biggl\{ (0,d\xi, d\omega)  : 
-\frac{\cos \vf(P^+(Z))} { \cK(P^+(Z)) + \cK_{\max}+ \frac{1}{\tau_{\min}} } 
- \tau(Z) 
\ge \frac{d\xi}{d\omega} 
\ge -\frac{\cos \vf(P^+(Z))}{\cK(P^+(Z)) + \cK_{\min} } 
- \tau(Z) \biggr\}\,  .
\]
\end{remark}

  Analogously, using \cite[eq.~(2.11)]{chernov book}
the unstable cones for the flow are defined by flowing $C^u_z$ forwards,
\[
C^u(Z) = \biggl\{ (0,d\xi, d\omega)
: \frac{1}{\frac{\cos \vf(P^-(Z))}{\cK(P^-(Z)) + \cK_{\min}}
+\tau^-(Z)} 
\le \frac{d\omega}{d\xi} 
\le \frac{1}{\frac{\cos \vf(P^-(Z))}{\cK(P^-(Z)) + \cK_{\max} + \frac{1}{\tau_{\min}}} + \tau^-(Z)} \biggr\} 
\, ,
\]
where $\tau^-(Z)$ denotes the time of first collision starting at $Z$ under the backwards flow.
Just like $\tau$, the function $\tau^-$ is (piecewise) $1/2$ H\"older continuous.

Note that both families of cones are bounded away from the horizontal in the $\xi\omega$-plane
due to the finite horizon condition.  Near tangential collisions ($\tau=\cos \varphi=0$), both families of cones
squeeze to a line (approaching the vertical direction), but never at the same point:  $C^u(Z)$ is arbitrarily
close to vertical just after nearly tangential collisions, while $C^s(Z)$ is bounded
away from the vertical at these points due to the fact that $\tau_{\min}>0$.  The roles
of $C^s(Z)$ and $C^u(Z)$ are reversed just before nearly tangential collisions.
Thus $C^s(Z)$ and $C^u(Z)$ are uniformly
transverse in the phase space for the flow.

Our cones are planar and therefore trivially have empty interior
in $\bR^3$, contrary to the flow cones used in \cite{BaL}. See also Remark~\ref{kerct}. 
Like in \cite{BaL} (see also
\cite{BuL}), we get {\it strict} contraction only for large enough times:

\begin{lemma}\label{flowinvar}
For any $Z\in \Omega_0$, and any $t\ge 0$, we have
$$
D_Z \Phi_t (C^u(Z)) \subseteq  C^u(\Phi_t(Z)) \, , \quad
D_Z \Phi_{-t} (C^s(Z)) \subseteq  C^s(\Phi_{-t}(Z))  \,  .
$$
In addition, if $t > \tau_{\max}$ then
(slightly abusing notation, ``int'' refers to the interior of the cone in the
plane $d\eta=0$) for any $Z \in \Omega_0$
$$
D_Z \Phi_t (C^u(Z)) \subset  \{0\} \cup \intt C^u(\Phi_t(Z)) \, ,\quad
D_Z \Phi_{-t} (C^s(Z)) \subset  \{0\} \cup \intt C^s(\Phi_{-t}(Z))  \,  .
$$
\end{lemma}

\begin{proof}
First note that the planes $\cT^\perp_Z \Omega_0 \subset \cT_Z\Omega_0$ defined by $d\eta =0$ are perpendicular to the flow and
preserved by the flow, see  \cite[Cor. 3.12, (3.14)]{chernov book}.
So we may restrict $D_Z\Phi_t$ to these planes.

We consider the statements for unstable cones; the others are similar.

If $0 < t < \tau(Z)$, then the claim immediately follows from the definition of the cones
since 
$D_Z \Phi_t (C^u(Z)) = C^u(\Phi_t(Z))$. This applies to any factor 
$ \left[ \begin{array}{cc}
1 & t  \\
0 & 1  \\ \end{array} \right]$ between
collisions in the decomposition  \cite[(3.29)]{chernov book} for
$D_Z\Phi_t$ restricted to $\cT^\perp_Z \Omega_0$.

For the action at a collision,  using \cite[(3.29)]{chernov book} again,
\eqref{mapinvar} and the definition of the unstable cone give
(with $Z^-$ and $Z^+$ representing the moments just before and just after collision, respectively)
$$
 -\left[ \begin{array}{cc}
1 & 0  \\
2 \cK/\cos \vf & 1  \\ \end{array} \right](C^u(Z^-))
\subset \{0\}\cup \intt (C^u(Z^+) ) \, .
$$
where we view the cones (and take interior) as subsets of $\cT^\perp_{Z^\pm} \Omega_0$.
There is at least one collision factor  in the decomposition if $t > \tau_{\max}$, ending the proof.
\end{proof}


\subsection{Admissible stable and unstable curves for the flow}
\label{stable curves}

Following \cite{demers zhang}, a $\cC^1$ curve $V$ in $\cM$ is called
a {\em stable curve for the map $T$} if the tangent vector at each point
$z$ in $V$ lies in $C^s_z$.
We call a $\cC^1$ curve\footnote{ Our curves  $W$ are
diffeomorphic to a bounded open interval, in particular relatively open at endpoints.
However, we require the $\cC^r$ ($r=\oldp,\oldq,1,2$)  norms of functions supported on $W$ to be bounded in the sense of Whitney
on the closed interval.} $W\subset \Omega_0$ a {\em stable curve for the flow} if at every point $Z \in W$, the tangent vector
$\cT_ZW$  to $W$ lies in $C^s(Z)$.  
 An essential property of stable curves $W$ is 
 \begin{equation}\label{esscurve}
\# ( \overline {W} \cap \partial \Omega_0 ) \le 2
 \end{equation}
(where $\overline W$ denotes the closure of $W$)
since either both the scatterer and
the wavefront corresponding to a stable curve are convex in opposition,
or the curvature of a stable curve is bounded away (for  this claim, see e.g.,  the paragraph containing
 \eqref{cardinality}) from the curvature of the scatterer.
Our definition also implies that a stable curve $W$ is perpendicular to
the flow, which is essential in the following remark:

\begin{remark}[Checking that stable curves lie in the kernel of the contact form]\label{kerct}
In the Dolgopyat estimate we will   use that
stable curves $W$ belong to the kernel of the invariant contact form $\bal$.

In\footnote{ We use a bold upright symbol $\extd$ to denote the
exterior derivative. This should not be confused with the $d$ 
used for coordinates in a vector ($dr$,$d\varphi$), ($d\eta,d\xi,d\omega$),
adopted from \cite{chernov book}.} the present billiard case, $\bal= p \, \extd q$
(see e.g. \cite[App. 4]{Ar}), and since the velocity
$p=v=(v_1,v_2)=(\cos \omega,\sin \omega)$, we have
$$
\bal=\cos \omega \, \extd x + \sin\omega \, \extd y \, , 
$$
so that $\extd \bal=-\sin \omega \extd \omega \wedge \extd x +
 \cos\omega \extd \omega \wedge \extd y$ and 
$$\bal \wedge \extd\bal =\extd x \wedge \extd\omega\wedge \extd y\, .$$
In particular, $\bal$ coincides with the Jacobi coordinate
$(0,0,d\eta)$ corresponding to  the velocity (i.e., flow) 
direction (referring -- again -- to \cite[Cor. 3.12, (3.14)]{chernov book}), 
and the stable and unstable cones lie in the
kernel of the contact form $\bal$, as desired.
Note that by definition of contact, there is no surface tangent to
the kernel of $\bal$ on an open set (in fact, the maximal dimension
of a manifold everywhere tangent to this kernel is one). 
\end{remark}

\begin{remark}[Putting the contact form in standard form]\label{Rk1.4}
If we replace the coordinates $(x,y,\omega)\in Q\times \bS^1$ by $(\omega,\tilde \xi,\tilde \eta)
\in  \bS^1 \times \widetilde Q$ where
\begin{equation}\label{Jacobidirect}
\omega=\omega\, ,\quad 
\tilde \xi=-x \sin \omega + y \cos \omega\, ,\quad
\tilde \eta = x \cos \omega + y \sin \omega \, ,
\end{equation}
(defining $\widetilde Q$  by the above),
then inverting gives
$
\omega=\omega$, $x=\tilde \eta \cos \omega - \tilde \xi \sin \omega$,
$y = \tilde \eta \sin \omega + \tilde \xi \cos \omega
$,
so that
\begin{align*}
\bal&= \cos \omega \extd x + \sin \omega \extd y\\
&=\cos \omega [\cos \omega \extd \tilde \eta - \tilde \eta \sin \omega 
\extd \omega - \sin \omega \extd \tilde \xi 
-\tilde \xi \cos \omega \extd \omega ]
\\&\qquad 
+ \sin \omega[\sin \omega \extd \tilde \eta +
\tilde \eta \cos \omega \extd \omega + \cos \omega \extd \tilde \xi
- \tilde \xi \sin \omega \extd \omega ]
= \extd \tilde \eta - \tilde \xi \extd \omega \, .
\end {align*}
Therefore,   the contact form $\bal$ takes the
standard form $\balzero=\extd \tilde \eta - \tilde \xi \extd \omega$ in the
coordinates $(\omega, \tilde \xi,\tilde \eta)$. 
Note that the change of coordinates \eqref{Jacobidirect} in the
phase space of the flow
is the same as the one used to construct the Jacobi coordinates
$(d\eta,d\xi,d\omega)$, except that \eqref{Jacobi} corresponded to changing
variables in the tangent space. So there is no contradiction
between the fact that $\bal=\extd \tilde \eta - \tilde \xi \extd \omega $
and the fact that $\bal$
vanishes  on those vectors so that $d\eta=0$.
(Indeed, vectors which are perpendicular
to $(1,0,0)$ in $(d\eta, d\xi, d\omega)$ coordinates are perpendicular to 
$(\cos \omega, \sin \omega, 0)$ in $(dx, dy, d\omega)$ coordinates by definition of $d\eta$.  
Now given the 
definition of $\bal$, we have (in $(dx, dy, d\omega)$ coordinates),
$
\bal(v) = \cos \omega \extd x (v) + \sin \omega \extd y (v)
= (\cos \omega, \sin \omega, 0) \cdot v
$, 
so that $\bal(v) = 0$ if and only if $v$ is perpendicular to $(d\eta, 0 ,0)$.)

In Sections ~\ref{Lipschitz} and ~\ref{dodo}  (analogously to what was done in
\cite{BaL}), we will need coordinate charts and local coordinates $(x^u,x^s,x^0)$ such that at the origin of such charts,
$(1,0,0)$ lies in the unstable cone, $(0,1,0)$ lies in the stable cone,
and $(0,0,1)$ is the flow direction, while
$$
\bal= \extd x^0 - x^s \extd x^u \, .
$$
(We call such charts cone-compatible Darboux charts.)
The above  does not seem to hold for the
coordinates $(\omega,\tilde \xi,\tilde \eta)$ (recall the conditions
on the slope $d\omega/d\xi$ defining $C^u$ and $C^s$), but
it can be achieved by applying a ``symplectic'' change of coordinates,
as in \cite[Lemma A.4]{BaL},
\begin{align*}&\omega= A x^s+B x^u\, , \quad \tilde \xi= C x^s+ D x^u\, ,
\quad \tilde \eta= x^0 + \frac{AC}{2} (x^s)^2
+ AD x^s x^u + \frac{BD}{2} (x^d)^2)\, , 
\end{align*}
 with $AD-BC=1$, for $A,B,C,D$ real-valued functions. 
This is possible, since inverting the above gives
\begin{align*}
&x^s= D \omega - B \tilde\xi \, , \quad x^u= -C \omega + A  \tilde \xi\, ,\quad
x^0=\tilde \eta - [ \frac{AC}{2} (x^s)^2
+ AD x^s x^u + \frac{BD}{2} (x^u)^2)]\, ,
\end{align*}
and we may  find $A>0$, $B>0$, $C<0$, $D>0$, with $AD-BC=1$
so that $(D,-B)$ lies in the stable cone
while $(-C,A)$ lies in the unstable cone. 
In order to
be able to choose the functions $A$, $B$, $C$, and
$D$  (and thus the charts) in a $\cC^2$ way, we shall  discard small
neighborhoods where the cone directions do not behave
in a nice way. This  will be performed in Remarks ~\ref{suitcharts0}
and ~\ref{suitcharts}.
\end{remark}

The homogeneity layer decay rate $k^{-\chi}$ with $\chi=2$  in Definition~\ref{def:defhom} below
is standard in the literature.
Letting $\chi>1$ tend to $1$ 
gives better estimates,
as we explain in Remark~\ref{opt}.

\begin{defin}[Homogeneity strips]\label{def:defhom}
To control distortion,
we define the usual homogeneity strips for the map by 
\begin{equation}
\label{defhom}
\bH_k = \{ (r,\vf) \in \cM : \pi/2 - 1/k^2 \le \vf \le \pi/2 - 1/(k+1)^2 \}
\end{equation}
for $k \ge k_0$, and similarly for $\bH_{-k}$ near $\vf = -\pi/2$.
We also put
\begin{equation*}
\bH_0 = \{ (r,\vf) \in \cM : -\pi/2 + 1/k_0^2 \le \vf \le \pi/2 - 1/k_0^2 \}\, .
\end{equation*}
Finally, we set for $|k|\ge k_0$
\begin{equation*}
\bS_k=\{(r,\vf) : |\vf|=\pi/2-k^{-2}\}\, ,
\mbox{ and }
\bS=\cup_{|k|\ge k_0} \bS_k\, .
\end{equation*}
A subset $Z$ of $\cM$ is called {\it homogeneous} if
$Z \subset \bH_{k_1}$
for some $k_1 \in \{ 0 \} \cup \{ k \ge k_0 \}$. 
\end{defin}
  
  \begin{defin}[Singularity sets]\label{singul}
 We define $\cS_0=\partial \cM=\{ |\varphi|=\pi/2\}$, and, inductively, for $n\ge 1$,
 the iterated singularity 
 sets
 $$
 \cS_n=\cS_{n-1} \cup T^{-1}(\cS_{n-1})
 , \quad
 \cS_{-n}=\cS_{-(n-1)}\cup T(\cS_{-(n-1)}) \, .
 $$
 The extended singularity sets are defined for $n\ge 0$ by
  $$
  \cS^{\bH}_n= \cS_n \cup (\cup_{m=0}^n T^{-m} \bS) 
  , \quad \cS^{\bH}_{-n}=\cS_{-n} \cup (\cup_{m=0}^n T^{m} \bS) 
 \,  .
  $$
\end{defin}

Recall that $\cS_{-n}\setminus \cS_0$
consists of finitely many smooth unstable curves
while  $\cS_{n}\setminus \cS_0$ consists of finitely many smooth stable curves, see
\cite[(2.19), (2.27), Proposition 4.45, Exercise 4.46,(5.14), (5.15) \S5.4]{chernov book}.

\begin{defin}\label{admiss} An admissible stable curve for $L>0$ and $B>0$ is a subset
$W$ of $\Omega_0$, satisfying the following conditions:
\begin{itemize}
  \item[(W1)] $W$ is a stable curve
  {\bf for the flow}. (This implies that  $P^+(W)$ is a stable curve
for the map, by definition of the stable cones.
If $P^-(W)$  intersects $n\ge 1$ scatterers,
then  $P^-(W)$ is  a union of $n$
stable curves for the map,  using $P^-(W) = T^{-1}(P^+(W))$.)
 \item[(W2)]  $W$ has length at most $L$ and is $\cC^2$, with curvature bounded by $B>0$.  
for the map $T$ (in particular, $\cos \varphi\ne 0$ for any $(r,\varphi)\in P^+(W)$).
\end{itemize}
If in addition $P^+(W)$ is {\bf homogeneous,} then $W$ is called homogeneous.
\end{defin}

We denote by $\cW^s$ the set of all admissible stable curves for fixed 
$L = L_0$ and $B = B_0$, where $L_0$ is chosen
in Lemma~\ref{lem:growth} below (determined in particular
by the complexity condition for the map, see also Definition~\ref{cG_t}) and $B_0$ is chosen after Lemma~\ref{lem:preserved}.
The set $\cW^s$ is
not empty. In fact   by 
simply flowing
a homogeneous stable curve for the map backwards under the flow, we may construct
homogeneous flow stable curves $W$.  
Unstable curves and (homogeneous) admissible unstable curves are defined similarly.

\begin{defin}[Distances $d_{\cW^s}(W_1,W_2)$ between stable curves]
Let $W_1, W_2 \in \cW^s$ be two stable curves. 
If there does not exist any unstable curve 
$W^u_{1,2}$ (for the flow) so that  $W^u_{1,2} \cap W_1$ and $W^u_{1,2} \cap W_2$ 
are both  nonempty, we set  $d_{\cW^s}(W_1,W_2)=\infty$.
If there exists \footnote{ This holds for example if $W_1\cap W_2\ne \emptyset$. It does {\it not} hold if $W_2=\Phi_t(W_1)$.} such a curve  $W^u_{1,2}$,
writing
the projections $P^+(W_i)$ 
as graphs $P^+(W_i) = \{ (r, \vf_i(r)) : r \in I_i \}$ over some
interval $I_i$, we define
$d_{\cW^s}(W_1,W_2) = | I_1 \bigtriangleup I_2| + |\vf_1 - \vf_2|_{\cC^1(I_1 \cap I_2)}$,
where $\bigtriangleup$ denotes the symmetric difference of two intervals, and
$|\vf|_{\cC^\ell(I)}=\sum_{k=0}^\ell\sup_I |D^k \vf |$.
\end{defin}

Note that this distance stated in terms of the projected curves
is 
the distance used in \cite{demers zhang} between 
stable curves
for the map (with the small difference that the curves there were
assumed to be homogeneous).  For future reference, we denote this distance between projected curves
by $\td_{\cW^s}(P^+(W_1), P^+(W_2))$.

\begin{remark}
Our distance between $W_1$ and $W_2$ is based
on $P^+(W_1)$ and $P^+(W_2)$, so if $W_1$ and $W_2$ are two halves of the same
stable curve split at a single point making a collision, then the distance
between them is infinite in our norms since $P^+(W_1)$ and $P^+(W_2)$ lie on
different scatterers (or the opposite side of the same scatterer,
but then the distance is still of order one).\end{remark}

\begin{defin}[Distance $d(\psi_1, \psi_2)$ between test functions]
\label{def:dist}
Let  $\psi_1$ and $\psi_2$ be supported
on stable curves $W_1$ and $W_2$, respectively. Since $P^+(W_i)$ is represented by the
graph of $\vf_i$ over $I_i$, let $G_{W_i}(r) = (r, \vf_i(r))$ denote this graph and
$S_{W_i}(r) = \Phi_{-t(r)} \circ G_{W_i}(r)$
 denote the natural map 
from $I_i$ to $W_i$ defined by the flow.
We use this map to define the distance between test functions by
$d(\psi_1, \psi_2) = |\psi_1 \circ S_{W_1} - \psi_2 \circ S_{W_2}|_{\cC^0(I_1 \cap I_2)}$
when $I_1 \cap I_2 \neq \emptyset$.  If $I_1 \cap I_2 = \emptyset$, we set $d(\psi_1, \psi_2) = \infty$,
i.e., we simply do not compare these functions.
\end{defin} 

(The notation for $S_{W_i}$  
and in particular the map $\Phi_{-t(r)}$ used in the definition above
hides the change of coordinates from $(r,\vf)$ to $(x,y,\omega)$.  The 
corresponding Jacobian is a product of
the Jacobian of the flow between collisions and that of the
change of variables.  See Sections~\ref{section3} and~\ref{section4}.)

\subsection{The norms $|\cdot|_w$, $\|\cdot\|_s$, $\|\cdot\|_u$, $\|\cdot\|_0$,
and the spaces $\cB$ and $\cB_w$}

Given $W \in \cW^s$ and $0<\oldq <1$, we define the H\"older norm of a test function $\psi$
by $|\psi|_{\cC^\oldq(W)} = |\psi|_{\cC^0(W)} + \cC^\oldq_W(\psi)$.  Letting $\cC^1(W)$ denote the
set of continuously differentiable functions on $W$, we define the spaces $\cC^\oldq(W)$
to be the closure of $\cC^1(W)$ in the $| \cdot |_{\cC^\oldq(W)}$ norm\footnote{ This
space $\cC^\oldq(W)$ coincides with the ``little H\"older space'' noted $b_{\infty,\infty}^\oldq$
in \cite[Prop. 2.1.2, Def. 2.1.3.1]{RS} obtained by taking the closure of $\cC^\infty$
in the $| \cdot |_{\cC^\oldq(W)}$ norm, it is
strictly smaller than the set of all H\"older
continuous functions with exponent $\oldq$, but contains all H\"older continuous
functions with exponent $\oldq' > \oldq$.}.

Now\footnote{ \label{foot10}The value $1/3$ is determined by the choice \eqref{defhom}
of homogeneity layers, see Lemma~\ref{lem:distortion}. Replacing $k^2$ by $k^\chi$
for $\chi >1$ replaces the bound $\oldp\le 1/3$ by $\oldp \le 1/(\chi+1)$.} fix $0<\oldp \le 1/3$.  For $\oldh \in \cC^1(\Omega_0)$, define the
weak norm of $\oldh$ by
\[
|\oldh|_w = \sup_{W \in \cW^s} \sup_{\substack{\psi \in \cC^\oldp(W) \\ |\psi|_{\cC^\oldp(W)} \le 1}}
\int_W \oldh \psi \, dm_W \, ,
\]
where $dm_W$ denotes arclength along $W$.

Now choose $1 < \oldbetaparam < \infty$ and\footnote{ In view of \eqref{neater}, it will be natural to further restrict to those $\oldq$ which satisfy in addition $\oldq \le   1- \oldbeta$.}
$0<\oldq <  \min \{ \oldp, 1/\oldbetaparam \}$. 
For $\oldh \in \cC^1(\Omega_0)$,
define the strong stable norm of $\oldh$ by
\[
\| \oldh \|_s = \sup_{W \in \cW^s} \sup_{\substack{\psi \in \cC^\oldq(W) \\
|W|^{\oldbeta} |\psi|_{\cC^\oldq(W)} \le 1}}
\int_W \oldh \psi \, dm_W \, .
\]
Choose $0<\gamma \le \min \{ \oldbeta , \oldp-\oldq \} < 1/3$.  Define the unstable norm
of $\oldh$ by\footnote{ In previous works, $d(\psi_1,\psi_2) =0$ was replaced by
$d_\oldq(\psi_1,\psi_2) \le \ve$ where $\oldq > 0$
and the distance used the $\cC^\oldq$ instead of the $\cC^0$ norm. The two formulations are equivalent by the triangle
inequality, using the strong stable norm, and since
$d(\psi_1,\psi_2)=0$ implies $d_\oldq(\psi_1,\psi_2)=0$.}
\[
\| \oldh \|_u = \sup_{\ve >0} \sup_{\substack{W_1, W_2 \in \cW^s \\ d_{\cW^s}(W_1,W_2) \le \ve}} 
\sup_{\substack{\psi_i \in \cC^\oldp(W_i) \\ |\psi_i|_{\cC^\oldp(W_i)} \le 1 \\ d(\psi_1,\psi_2) =0}}
\ve^{-\gamma} \left| \int_{W_1} \oldh \psi_1 \, dm_{W_1} - \int_{W_2} \oldh \psi_2 \, dm_{W_2} \right|  \, .
\]
We define the  neutral norm of $\oldh$ by
\[
\| \oldh \|_0 = 
\sup_{W \in \cW^s}
 \sup_{\substack{\psi \in \cC^\oldp(W) \\ |\psi|_{\cC^\oldp(W)} \le 1}}
\int_W  \partial_t (\oldh \circ \Phi_t)|_{t=0} \, \psi  \, dm_W   \,  .
\]
Finally, define the strong norm of $\oldh$ by 
$$\| \oldh \|_\cB = \|\oldh\|_s + c_u \| \oldh \|_u + \|\oldh \|_0 \, , $$
where
$c_u>0$ is a constant to be chosen  in Section~\ref{sec:R}.

We will prove in Lemma~\ref{lem:embedding} that there exists
a constant $C$ so that for any
 $\oldh\in \cC^1(\Omega_0)$
  $$\| \oldh \|_\cB \le C |\oldh|_{\cC^1(\Omega_0)}\, .
  $$
In the other direction, we relate our norms to the dual of H\"older spaces.
For $0<r<1$ and any $\oldh\in L^\infty(\Omega_0)$, set
\begin{equation}\label{duallittle}
| \oldh |_{(\cC^{r})^*} = \sup \left\{ \int \oldh \psi\ dm :
\psi \in \cC^r(\Omega_0),  \; |\psi|_{\cC^r(\Omega_0)} \le 1 \right\}<\infty \, .
\end{equation}
where $\cC^r(\Omega_0)$ is the closure of $\cC^1$ for the $\cC^r$ norm (little H\"older space).
The following estimate   follows from the role of $0<\oldq<\oldp$ in our definition
(the proof is given in Section~\ref{molsec}).

\begin{lemma}[Relation with the dual of $\cC^r(\Omega_0)$]\label{relating}
There exists $C\ge 1$ so that 
$$
| \oldh |_{(\cC^{\oldp}(\Omega_0))^*}\le C |\oldh|_{w} \,, \quad  | \oldh |_{(\cC^{\oldq}(\Omega_0))^*}\le C \|\oldh\|_s\qquad \forall \oldh\in \cC^0(\Omega_0)\, .
$$
\end{lemma}

We introduce the subspace $\cC^0_\sim$ of those continuous functions on $\Omega_0$ which
can be viewed as continuous functions on the quotient space $\Omega=\Omega_0/\sim$:
\begin{equation}\label{contit}
\cC^0_\sim=\cC^0(\Omega)=\{  \oldh  \in \cC^0(\Omega_0) \mid \oldh(Z)=\oldh(Y) \mbox{ if } Z \sim Y \} \, .
\end{equation}
It is well known that $\Phi_t$ is continuous on the topological metric quotient\footnote{ The requirement $\|\oldh\|_0<\infty$ essentially implies that
$\oldh$ is in some kind of Sobolev space with positive
exponent in the flow direction,  thus continuous almost everywhere on the quotient
$\Omega$. We shall neither prove nor use this.}
 space $\Omega$
 see
e.g. \cite[Exercise 2.27]{chernov book}.  This implies 
 that if $\oldh\in \cC^0_\sim$ then $\oldh\circ \Phi_{-t} \in \cC^0_\sim$ for all
$t\in \bR$.
We can now introduce our main Banach spaces:

\begin{defin}[Banach spaces $\cB_w$, $\cB$]\label{spcce}
Let $\cB_w$  be the completion with respect to $| \cdot |_w$ of 
the set $\{  \oldh  \in \cC^0_\sim \mid   | \oldh |_w< \infty \}$.
Let $\cB\subset \cB_w$ be  the completion for the norm $\|\cdot \|_\cB$ of
\begin{equation}\label{smallspace}
\{ \cL_t \oldh \mid t\ge 0, \quad \oldh \in \cC^0_\sim \cap  \cC^2(\Omega_0) \} \, .
\end{equation}
\end{defin}
The apparently contrived definition of $\cB$ will ensure joint continuity of $(t,\oldh)\mapsto \cL_t(\oldh)$
in Lemma~\ref{lem:strong c}, adapting \cite[(3.9)--(3.12)]{BaL}.

Lemma~\ref{relating}  implies that $\cB \subset (\cC^\oldq(\Omega_0))^*$ and 
$\cB_w \subset (\cC^\oldp(\Omega_0))^*$ (defining $(\cC^r)^*$ to be the closure of 
$\cC^1$ functions $\oldh$ for the
$(\cC^r)^*$ norm), while Lemma ~\ref{lem:C1} implies that $\cC^1(\Omega_0)\subset \cB$.
Clearly, $\cB\subset \cB_w$ (we shall see in Lemma ~ \ref{lem:compact}
that the embedding is compact).  The injectivity of the various embeddings is summarized in
Lemma~\ref{lem:embedding}.

\smallskip

Note for further use that, by the group property of the flow,
$\oldh \circ \Phi_{r+s+\delta} - \oldh \circ \Phi_{r+s} = 
(\oldh \circ \Phi_{s+\delta} - \oldh\circ \Phi_s) \circ \Phi_r$ for any
$\oldh \in \cC^0_\sim$, so that, whenever
both sides of the identity below are well-defined, 
\begin{equation}
\label{eq:flow derivative}
\partial_t (\oldh \circ \Phi_t)|_{t=r+s} (Z)= ( \partial_t \oldh \circ \Phi_t)|_{t=s} \circ \Phi_r (Z)\,  .
\end{equation}


\section{Preliminary lemmas}
\label{section3}

In this section, we establish some basic facts about expansion and control of complexity and prove
several key properties of our Banach spaces.


\subsection{Growth and distortion}

When we flow a stable curve $W$ backwards under the flow, $\Phi_{-t}(W)$ may be
cut by singularity curves or undergo large expansion at collisions.  
For fixed $t>0$ and $Z$, the $t$-trace of $Z$ 
(on the collision space $\cM$, up to time $t$) is the set
$\{\Phi_{s}(Z)\cap \cM\, , s \in (0,t]\}$.
(Note that $\Phi_{-t}(W)$ has a cusp at $Z$ if the $t$-trace
of $Z$ contains a tangential collision.) 
Recall 
\cite[p.76]{chernov book}
that the trace of $E\subset \Omega_0$ is simply $P^+(E)$.

\begin{defin}[The partition $\cG_t(W)$]\label{cG_t}
For $t\ge 0$ and $W\in \cW^s$, we partition $\Phi_{-t}(W)$ as 
follows.\footnote{ Whenever we partition a curve into finitely or countably
many subcurves, we drop the (at most countably many) division points.}
First, we partition $\Phi_{-t}(W)$ at any  points
whose $t$-trace intersects either the boundary of one of the homogeneity strips $\bH_k$
or a tangential collision point $\{ \vf =\pm \pi/2\}$.
Second, if any component defined thus far was subdivided at a previous time $s$,
for $0 < s< t$ due to growing to length greater than $L_0$, we continue
to consider $\Phi_{-t}(W)$ to be subdivided at the image of this point
under $\Phi_{s-t}$.  
Finally, if a component reaches length $=L_0$, we subdivide
it
into two curves of length  $L_0/2$.
The countable collection of components of
$\Phi_{-t}(W)$ defined in this way is denoted by $\hG_t(W)$.
If one of the elements of $\hG_t(W)$ is 
in the midst of a collision at time $t$, i.e., if this component
intersects the
boundary of one of the scatterers (such an intersection 
can contain at most two points, since both the scatterer and
the stable curve (a diverging wavefront under the backwards flow) are convex in opposition, recall \eqref{esscurve}), then we split this component into 
two or three pieces
 temporarily:
one or two curves consisting of points that have just completed a collision and one curve
consisting of points that are about to make a collision.  
This refined set of components is denoted by $\cG_t(W)$.
\end{defin}  

Note that for $s>0$, if $W_i \in \hG_t(W)$, then $\Phi_{-s}(W_i)$ is a union of elements
in $\hG_{t+s}(W)$.  However, if $W_i \in \cG_t(W)$, then $\Phi_{-s}(W_i)$ may not be
a union of elements of $\cG_{t+s}(W)$ if $W_i$ was part of an curve undergoing a collision
at time $t$ which has since completed its collision without crossing any boundaries
of homogeneity strips.  In this case, $\Phi_{-s}(W_i)$ is contained in an element
of $\hG_{t+s}(W)$.

The following lemma proves the invariance of the family $\cW^s$ under $\Phi_{-t}$, $t > 0$.
While the ideas in the proof are by now well-known, we include the lemma 
since we are not aware of a published proof  that includes the family of
flow-stable curves $\cW^s$ we have introduced here.  Our proof follows the approach of
\cite[\S~4.6]{chernov book}, which
contains many of the ideas we will need.

\begin{lemma}[Invariance of stable curves]
\label{lem:preserved}
For $W \in \cW^s$ and $t \ge 0$, let $\cG_t(W) = \{ W_i \}_i$.  Then $W_i \in \cW^s$ for each $i$.
\end{lemma}

\begin{proof}
Due to the inductive definition of $\cG_t(W)$ and the invariance 
(Lemma~\ref{flowinvar}) of the stable cones, 
each element $W_i$ satisfies condition (W1) 
required for $\cW^s$.  The
one point to prove is (W2): The curvature of such curves remains uniformly bounded for all times.
At collisions,
\[
\frac{d\omega^+}{d\xi^+} = \frac{d\omega^-}{d\xi^-} + \frac{2\cK(r)}{\cos \vf} \, ,
\]
(cf.~\cite[eq. (3.33)]{chernov book}),
and this quantity represents the curvature of the wave front (a stable curve 
$W$ projected onto the $xy$-plane).  Yet
the expansion in the $\omega$ direction is of the same order, so as we will
show below, the curvature of $W$ in
$\bR^3$ remains uniformly bounded, even near tangential collisions. 

We begin by parametrizing a stable curve $W$ approaching a collision 
(under the backwards flow) 
by points $Z_s = (x_s,y_s,\omega_s)$
and define $Z_{st} = \Phi_{-t}(Z_s)$.  The points $Z_{st}$ fill a $2D$ surface in $\Omega_0$
and we choose the interval of $t$ to be large enough that we follow $W$ through precisely one
collision.  We assume that $\Phi_{-t}(W)$ is still smooth after completing the collision.

We will denote derivatives with respect to $t$ by dots and those with respect to $s$ by
primes.   Define $u = -x' \sin \omega + y' \cos \omega$.  Then since $(x',y',\omega')$ remains
perpendicular to the velocity vector, $(\cos \omega, \sin \omega,0)$, we have
\begin{equation}
\label{eq:u}
u^2 = (x')^2 + (y')^2 \; \;\; 
\mbox{and} \; \; \; u' = -x''\sin \omega + y'' \cos \omega \, .
\end{equation}  
We will need the following relations, which follow from
\cite[eqs. (4.25),(4,26)]{chernov book},
\begin{equation}
\label{eq:omega}
\omega' = u B \; \; \; \mbox{and} \; \; \;
\omega'' = u'B + u^2 \frac{dB}{d\xi}\, ,
\end{equation}
where $B = d\omega/d\xi$ represents the curvature of $W$ projected onto the $xy$-plane.

We will denote the curvature of $W$ in $\mathbb{R}^3$ by $\kappa$. By definition,
\[
\kappa^2 = \frac{(\omega'y''-y'\omega'')^2 + (\omega'x'' - x'\omega'')^2 + (x'y''-y'x'')^2}
{((x')^2 + (y')^2 + (\omega')^2)^3}\,  .
\]
Notice that the arclength factor  satisfies 
$(x')^2 + (y')^2 + (\omega')^2 = u^2(1+B^2)$, by \eqref{eq:u} and \eqref{eq:omega}.
We consider each term in the numerator separately.  For the first term, using 
\eqref{eq:u} and again
\eqref{eq:omega},
we let $B_\xi$ denote $dB/d\xi$ and write
\[
\begin{split}
\omega'y''-y'\omega'' & = B(-x'\sin \omega + y' \cos \omega)y'' - y'(u'B + u^2 B_\xi) \\
& = B(-x'y''\sin \omega + y'y'' \cos \omega) + (y'x'' \sin \omega - y'y'' \cos \omega)B - 
y' u^2 B_\xi \\
& = B \sin \omega (y'x'' - x'y'') - y' u^2 B_\xi \, .
\end{split}
\]
A similar calculation for the second term yields,
\[
\omega'x'' - x'\omega'' = B \cos \omega (y'x''-x'y'') - x'u^2 B_\xi \,  .
\]
Substituting these relations into the expression for $\kappa^2$ and using 
again \eqref{eq:u} and that $x'\cos \omega + y'\sin \omega =0$, we obtain $\kappa^2=$
\[
\begin{split}
& 
\frac{B^2 (\sin^2 \omega \, (y'x''-x'y'')^2 + (y')^2 u^4 B_\xi^2+ \cos^2 \omega \, (y'x''-x'y'')^2) + 
(x')^2 u^4 B_\xi^2 + (x'y''-y'x'')^2}
{((x')^2 + (y')^2 + (\omega')^2)^3} \\
& = \frac{(1+B^2)(x'y''-y'x'')^2 + u^6 B_\xi^2}{u^6(1+B^2)^3}
= \frac{B^2}{(1+B^2)^2} + \frac{B_\xi^2}{(1+B^2)^3} \,  ,
\end{split}
\]
where in the last equality, we have used the fact that the curvature of the planar
wavefront satisfies $B^2 = (x'y'' - y'x'')^2/u^6$ by definition.
The first term in the expression for $\kappa^2$ is bounded by 1/4 for all $B \in \bR$.  
The second term is bounded above by $B_\xi^2/B^6$ and
we use \cite[eqs. (4.36), (4.38)]{chernov book} to conclude
that $B_\xi/B^3$ remains uniformly bounded for all time, even after undergoing
collisions arbitrarily close to tangential.
\end{proof}

A corollary of the above proof (in particular \cite[eq. (4.38)]{chernov book})
is that there exists a constant $B_0>0$ such that if $W$ is a stable curve with curvature
less than $B_0$, then each smooth component of $\Phi_{-t}(W)$
also has curvature less than $B_0$.  We fix this choice of $B_0$ in the definition of $\cW^s$.

\smallskip

If $W$ is a flow-stable curve, $Z\in W$, and $t\in  \bR$, we denote
by $J_W \Phi_t (Z)$ the Jacobian of $\Phi_t$ from $W$ to $\Phi_t(W)$, at
$Z$, with respect to the arclength measure on $W$. 
Similarly, if $n$ is an integer then $J_U T^n (z)$ denotes the Jacobian of $T^n$ from 
the map-stable curve $U$ to $T^n(U)$, at
$z\in U$, with respect to the arclength measure on $U$.

\begin{lemma}
\label{lem:expansion}
There exists $C_0 >0$, such that for all $W \in \cW^s$, all $t>0$, all $W_i \in \cG_t(W)$, 
and all $Z\in W_i$; the Jacobian
$J_{W_i}\Phi_t(Z) = C_0^{\pm 1} J_{P^+(W_i)}T^n(P^+(Z))$, where $n$ is the number of collisions of
$P^+(Z)$ before time $t$.
\end{lemma}

\begin{proof}
Set $V_i = \Phi_tW_i$. Since $J_{V_i}\Phi_{-t}(\Phi_t Z) = (J_{W_i}\Phi_t(Z))^{-1}$,
it is equivalent to estimate
$J_{V_i}\Phi_{-t}$.  

First we show that for $V \in \cW^s$, the expansion from $P^+(V)$ to $V$ given by 
the inverse of the flow from $Z\in V$ to $P^+(V)$
is of order $1$ even though $P^+(V)$ may not be homogeneous.  
(The proof of this is a refinement of \cite[Exercise 3.15]{chernov book}, taking
advantage of the fact that we only consider stable curves.)

For $Z \in V$, setting $P^+(Z) = z = (r,\vf)$, we let $J_{P^+(V)}\Phi_{-\tau(Z)}$ denote the Jacobian of the 
map from $P^+(V)$ (in $(r,\vf)$ coordinates) to $V$ (in $(x,y,\omega)$ coordinates) under the 
backwards flow.
Let $dz = (dr, d\vf)$ denote the image of a vector
$dZ = (0, d\xi, d\omega) \in C^s(Z)$. 
Recalling \eqref{eq:translate}, and
since $\| dZ \|^2 = (d\xi)^2 + (d\omega)^2$, we obtain
\begin{equation}
\label{eq:equiv}
\| dZ \|^2 = (dr)^2 \left[ \left(\cos \vf + \tau(Z) \cK(r) - \tau(Z) \frac{d\vf}{dr}\right)^2 
+ \left(\cK(r) - \frac{d\vf}{dr}\right)^2 \right] \, .
\end{equation}
On the one hand, \eqref{eq:equiv} implies  $\| dZ \|^2 \le C \| dz\|^2$, where
$C$ is a constant  depending on $\cK_{\max}$, $\tau_{\max}$ and $1/\tau_{\min}$,
since $(dr, d\vf) \in C^s_z$.  On the other hand, since $\frac{d\vf}{dr} < 0$, we get that $\| dZ \|^2 \ge$
\[
\begin{split}
 (dr)^2 \left( \cK(r) - \frac{d\vf}{dr} \right)^2 \ge (dr)^2 [ \cK(r)^2 + \left(\frac{d\vf}{dr}\right)^2]
= \cK(r)^2 (dr)^2 + (d\vf)^2 \ge \min \{ \cK^2_{\min}, 1 \} \| dz \|^2  \, ,
\end{split}
\]
proving the claim.

Next, we fix $Z \in V_i$ as above and decompose its past orbit as follows,
let $z_0 = P^+(Z)$ and $z_{-i} = T^{-1}z_{-i+1}$ for
$i = 1, \ldots n$, where $n$ is the number of collisions between $Z$ and $\Phi_{-t}(Z)$.
Then, letting $dZ_0$ denote a vector tangent to $V$ at $Z$ perpendicular to the flow
and $dZ_{-t}$ denote its image under $D\Phi_{-t}$, we have 
\begin{equation}
\label{eq:to map}
J_{V_i}\Phi_{-t}(Z) = \frac{\| dZ_{-t} \|}{\| dZ_0 \|} = \frac{\| dZ_{-t} \|}{\| dz_{-n} \|} \frac{\|dz_{-n}\|}{\|dz_0\|}
\frac{\| dz_0 \|}{\| dZ_0 \|} \,  ,
\end{equation}
where $dz_i$ represents the image of $dZ_0$ at $z_i$.  The first and third factors above
are of order 1 by the previous claim.  The middle factor is precisely
$J_{P^+(V_i)}T^{-n}(P^+(Z))$.  Now since
$J_{V_i}\Phi_{-t}(\Phi_tZ) = (J_{W_i}\Phi_t(Z))^{-1}$
and $J_{P^+(V_i)}T^{-n}(P^+(\Phi_t Z)) = (J_{P^+(W_i)}T^n(P^+(Z)))^{-1}$,
the lemma is proved.
\end{proof}

The following lemma will be the key to the bounded distortion Lemma~\ref{lem:distortion}.
(Note that the exponent $1/2$ is intrinsic to the billiard and cannot be improved.)
\begin{lemma}
\label{lem:smooth}
For any $W \in \cW^s$, let $S_W(r) = \Phi_{-t(r)} \circ G_{W}(r)$ denote the map 
from $I_W$ to $W$ defined at the end of Subsection~\ref{stable curves}.  There exists $C>0$, depending
only on the table and choice of $L_0$ and $B_0$, such that 
$|\ln JS_W|_{\cC^{1/2}(I_W)} \le C$ and $|\ln JS_W^{-1}|_{\cC^{1/2}(W)} \le C$,
where $JS_W^{\pm 1}$ denotes the Jacobian of $S_W^{\pm 1}$.
\end{lemma}

\begin{proof}
As in the proof of Lemma~\ref{lem:expansion}, 
for $Z \in W$ and $z=P^+(Z)$, let $J_{P^+(W)}\Phi_{-\tau(Z)}(z)$
denote the Jacobian of the map
from $P^+(W)$ to $W$ 
under the backwards flow. 
Then for $r \in I_W$ with $z = G_W(r)$, we have
$JS_W(r) = J_{P^+(W)}\Phi_{-\tau(Z)}(z) JG_W(r)$.

First note that (as in \cite[eq. (3.18)]{demers zhang}),
\begin{equation}
\label{eq:dG_W}
| JG_W | = \sqrt{1 + \left( \frac{d\vf_W}{dr}   \right)^2} \le
C_g := \sqrt{ 1 + (\cK_{\max} + \tau_{\min}^{-1})^2} \, . 
\end{equation}
Also $JG_W \ge 1$ and since by \eqref{eq:equiv} and the estimates following,  
$J_{P^+(W)}\Phi_{-\tau(Z)}(z)$ is uniformly bounded above and below away from zero,
the $\cC^0$ bound is proved.  We proceed to estimate the H\"older constant.
Since $JG_W$ is $\cC^1$ using \eqref{eq:dG_W} and the fact that $W$ has bounded
curvature, it remains to estimate $\cC^{1/2}(\ln J_{P^+(W)}\Phi_{-\tau})$.

Let $Z_1, Z_2 \in W$ and set $z_j = (r_j,\vf_j) = P^+(Z_j)$, $j=1,2$.
Using again \eqref{eq:equiv}, we have
\begin{align}
\nonumber
&\ln \frac{J_{P^+(W)}(\Phi_{-\tau(Z_1)}(z_1))}{J_{P^+(W)}(\Phi_{-\tau(Z_2)}(z_2))}
= \ln \frac{|dz_2|}{|dz_1|}
\\
\label{eq:log J} 
&\qquad\qquad\qquad+
\frac 12 \ln  \frac{(dr_1)^2 \left[ \left( \cos \vf_1 + \tau(Z_1)\cK(r_1) - \tau(Z_1)\frac{d\vf_1}{dr_1}
\right)^2
+ \left(\cK(r_1) - \frac{d\vf_1}{dr_1}\right)^2 \right]}
{(dr_2)^2 \left[ \left( \cos \vf_2 + \tau(Z_2)\cK(r_2) - \tau(Z_2)\frac{d\vf_2}{dr_2}\right)^2
+ \left(\cK(r_2) - \frac{d\vf_2}{dr_2} \right)^2 \right]} 
\, ,
\end{align}
where $dz_j = (dr_j, d\vf_j)$ denotes the tangent vector to $P^+(W)$ at $z_j$. 
Without loss of generality, we may set $dr_1=dr_2=1$.
For the first term on the right-hand side, we have
\[
\ln \frac{|dz_2|}{|dz_1|} = \frac 12 \ln \left( \frac{1 + (d\vf_2)^2}{1 + (d\vf_1)^2} \right)
\le \frac 12 (d\vf_2 - d\vf_1)(d\vf_2 + d\vf_1)
\le Cd(z_1,z_2)  \, ,
\]
where in the last step, we used that $d\vf_1, d\vf_2$ are uniformly bounded by
definition of $C^s_z$ as well as the fact that
$|\frac{d\vf_1}{dr_1} - \frac{d\vf_2}{dr_2}| \le C d(z_1,z_2)$ since $P^+(W)$ has 
uniformly bounded curvature.
Using again \eqref{eq:equiv} and the fact that $W$ has uniformly bounded
curvature according to (W2), we get 
$C d_W(Z_1,Z_2) \le d(z_1,z_2) \le C^{-1} d_W(Z_1,Z_2)$,
for some $C>0$.  

Now we turn to the second term on the right-hand side of \eqref{eq:log J}.
Since $J_{P^+(W)}\Phi_{-\tau(Z)}(z)$ is bounded away from $0$ and all the functions
appearing in the numerator and denominator of the right-hand side
are uniformly bounded, it suffices to consider
the differences in each function of the right-hand side separately.

It is easy to verify by direct inspection
(using that our billiard has finite horizon, while
$P^+(Z_1)$ and $P^+(Z_2)$ lie in a connected component of
the complement of the singularity set for $T$ in $\cM$, and 
the curvatures of the scatterers are bounded above)
that\footnote{ The factor $1/2$ here is intrinsic and not related to the
quadratic decay choice in \eqref{defhom}.}
\begin{equation}
\label{ceiling}
|\tau(Z_1) - \tau(Z_2)| \le C d_W(Z_1,Z_2)^{1/2} \, .
\end{equation} 
 Since $\cK$ and $\cos \vf$ are
smooth functions of
$r$ and $\vf$, the differences in $\cK$ and $\cos \vf$
are Lipschitz in $d_W(Z_1,Z_2)$ as well.  
Putting these estimates together and using again that $d_W(Z_1, Z_2) \le C^{-1}d(z_1, z_2)$, 
we have that the H\"older constant
$\cC^{1/2}(\ln J_{P^+(W)} \Phi_{-\tau}) \le C$ for some $C>0$, independent of $W$.

The estimate for the inverse follows similarly, using the fact that
$JS_W^{-1} = (JS_W)^{-1} \circ S_W^{-1}$ and exploiting $d(z_1,z_2) \le C^{-1} d_W(Z_1,Z_2)$.
\end{proof}

In the following distortion bound, the exponent $1/3$ is 
a consequence of our choice of
decay $1/k^2$ for the homogeneity layers.

\begin{lemma}[Bounded distortion]
\label{lem:distortion}
There exists $C_d >0$, 
such that for all $W \in \cW^s$, $t \ge 0$, and $Z_1, Z_2 \in W_i \in \cG_t(W)$,
\[
\left| \frac{J_{W_i}\Phi_t(Z_1)}{J_{W_i}\Phi_t(Z_2)} - 1 \right| \le C_d d_{W_i}(Z_1,Z_2)^{1/3} \, .
\]
\end{lemma}

\begin{proof}
Fixing $W \in \cW^s$ and $t>0$, if $\Phi_{-s}(W)$, $0 \le s \le t$, has undergone no collisions then
the required bound holds trivially due to the linearity of the flow between collisions.  On the
other hand, if $\Phi_{-t}(W)$ has undergone a collision, then
 $W_i$ is a homogeneous stable curve by construction of $\cG_t(W)$.

Let $n$ denote the number of collisions from $W_i$ to $\Phi_t(W_i)$.
We will use \eqref{eq:to map} in order to leverage the bounded distortion enjoyed by the
collision map $T$.  We begin with the first factor, which gives the Jacobian of the map
from $P^+(W_i)$ to $W_i$ and is $J_{P^+(W_i)}\Phi_{-\tau(Z)}(z)
$ in the notation
of the proof of Lemma~\ref{lem:smooth}.  This is log-H\"older with exponent $1/2$
by that lemma.

Comparable estimates hold for the last factor in \eqref{eq:to map},
which represents the Jacobian of the map from $P^+(\Phi_t(W_i))$ to $\Phi_t(W_i)$.
Finally, since $P^+(\Phi_t(W_i)) = T^n(P^+(W_i))$, the middle factor in
\eqref{eq:to map} enjoys bounded distortion along stable curves with the
H\"older exponent of $1/3$ \cite[Lemma 5.27]{chernov book}.
Notice that for the stable Jacobian of the map to enjoy bounded distortion at the last iterate
(from $T^{n-1}(P^+(W_i))$ to $T^n(P^+(W_i))$),
it is essential that $T^{n-1}(P^+(W_i))$ be a homogeneous stable curve, but not that
$T^n(P^+(W_i)) \subset P^+(W)$ be homogeneous.  This is because $J_{T^{n-1}(P^+(W_i))}T(z_1)$
is proportional to $\cos \vf(z_1)$, not $\cos \vf(T(z_1))$.
Combining the estimates for these three factors completes the proof of the lemma.
\end{proof}

Recall the hyperbolicity constant $\Lambda$ for the flow from \eqref{Lambda}.
The following elementary lemma is one of the keys to exponential mixing:

\begin{lemma}[Exponential decay of stable-H\"older constants under the flow]
\label{lem:Holder}
For any $0 \le \kappa \le 1$
there exists  $C_1 >0$ so that for any $W$,
 for each $\psi \in \cC^\kappa(W)$, all $t>0$, and  all $W_i \in \cG_t(W)$,
$$
\cC^\kappa_{W_i}(\psi \circ \Phi_t) \le C_1 \Lambda^{-t \kappa}\cC^\kappa_{W}(\psi) \, . $$
\end{lemma}

\begin{proof}
This will be a consequence of Lemma~\ref{lem:expansion}, since
for $Z_1, Z_2 \in W_i$, 
\[
\frac{|\psi \circ \Phi_t(Z_1) - \psi \circ \Phi_t(Z_2)|}{d_{W_i}(Z_1,Z_2)^\kappa}
= \frac{|\psi \circ \Phi_t(Z_1) - \psi \circ \Phi_t(Z_2)|}{d_W(\Phi_t(Z_1),\Phi_t(Z_2))^\kappa} 
\frac{d_W(\Phi_t(Z_1), \Phi_t(Z_2))^\kappa}{d_{W_i}(Z_1,Z_2)^\kappa}
\le \cC^\kappa_W(\psi) C \Lambda_0^{-n\kappa} \, ,
\]
where $n$ represents the number of collisions undergone by $W_i$ by time $t$
and we have used  that $\Lambda_0$ is the minimum expansion
factor for the map \cite[(4.19)]{chernov book}.  The lemma follows since 
$\lfloor \frac{t}{\tau_{\max}} \rfloor - 1\le n \le \lfloor \frac{t}{\tau_{\min}} \rfloor + 1$.
\end{proof}

We next present a growth lemma\footnote{ In Step 1 of Section ~\ref{Lipschitz}, we shall apply a growth
lemma \cite[Theorem 5.52]{chernov book}
directly.} adapted 
from \cite[\S 3.2]{demers zhang} which is a direct consequence of the one-step expansion
\cite[Lemma 5.56]{chernov book}.
Recall $L_0$ from the definition of $\cW^s$.

\begin{defin}[$\cI_t(W)$]
For $W \in \cW^s$ and $t \ge 0$,
let $\cI_t(W)$ be those elements $W_i \in \cG_t(W)$ such that $\Phi_s(W_i)$
is contained in an element $V^s_i \in \hG_{t-s}(W)$ with 
$|V^s_i| < L_0/3$ for all $0 \le s \le t$.  (The curves $\Phi_t(W_i)$ 
corresponding to $\cI_t(W)$
are repeatedly cut by the singularities of $\Phi_{-t}$ without ever growing to size $L_0/3$
before time $t$.)
\end{defin}

\begin{lemma}[Modified growth lemma]
\label{lem:growth}
 For any  $\lambda \in (\Lambda_0^{-1/\tau_{\max}},1)$, if $L_0$ is small enough, then there exists $C_2 \ge 1$ 
such that for all $1 \le \oldbetazeroparam \le \infty$, $W \in \cW^s$, and $t \ge 0$,
\begin{enumerate}
  \item[(a)] $\displaystyle 
  \sum_{W_i \in \cI_t(W)} \frac{|W_i|^{\oldbetazero}}{|W|^{\oldbetazero}} |J_{W_i}\Phi_t|_{\cC^0(W_i)} \le
  C_2 \lambda^{t(1-\oldbetazero)}$;  
  \item[(b)]  $\displaystyle 
  \sum_{W_i \in \cG_t(W)} \frac{|W_i|^{\oldbetazero}}{|W|^{\oldbetazero}} |J_{W_i}\Phi_t|_{\cC^0(W_i)} \le
  C_2 $.
  \end{enumerate}
\end{lemma}
\begin{proof}
By Lemma~\ref{lem:smooth}, there exists $C_3 \ge 1$ such that $|W_i| = C_3^{\pm 1} |P^+(W_i)|$.
We next invoke the one-step expansion for the map from
\cite[Lemma 5.56]{chernov book}.  Using \cite[(3.31), (5.39), (5.36)]{chernov book},
(since $\cB_1^{-} = 1/(\tau_0 + 1/\cB_0^{+})$, we have  $\cB_1^{-} \tau_{\min}\le 1$), we 
set $\tilde L_0 = C_3^{-1} L_0$
and for $\lambda_1 \in (\Lambda_0^{-1},1)$, choose 
$L_0\leq \tau_{\min}/4$ sufficiently small
that the one-step expansion for the map satisfies
\begin{equation}
\label{theta1}
 \sup_{|V| \le \tilde{L}_0} \sum_{i} \frac{|TV_i|_*}{|V_i|_*} \le \lambda_1 \, ,
\end{equation}
where $V$ is a (not necessarily homogeneous)
map stable curve, $V_i$ are the homogeneous components of $T^{-1}V$, 
and $| \cdot |_*$
denotes length in the adapted metric defined in \cite[Section 5.10]{chernov book}:
For a tangent vector $(dr, d\vf)$, define 
$\| (dr, d\vf) \|_* = \frac{\cK + |\frac{d\vf}{dr}|}{\sqrt{1 + (\frac{d\vf}{dr})^2}} \| (dr, d\vf) \|$, where
$\cK$ denotes the curvature at the given point.

Fix $n$ and consider curves $W_i \in \cG_t(W)$ undergoing $n$ collisions
between $\Phi_t(W_i)$ and $W$. 
Then $P^+(W_i)$ is contained in an element
of $\tG_n(P^+(W))$, where $\tG_n(P^+(W))$ are the connected homogeneous
components of $T^{-n}(P^+(W))$, defined inductively as in \cite[Sect 3.2]{demers zhang}.
If $W_i$ does not have an endpoint on a scatterer (i.e., is not in the midst
of a collision), then the correspondence is one to one.  If $W_i$ is part of a
longer element making a collision, then there are at most two such $W_i$ for
each element of $\tG_n(P^+(W)$), due to the fact that at most two points of
a diverging wave front can be in contact with a scatterer at one time
(this divides the curve into at most three pieces, but note that in this case,
two of the three will have made $n$ collisions, while one of the three will have made
$n+1$ collisions, and so will be grouped in the $(n+1)$th generation,
$\tG_{n+1}(P^+(W)$).

Since the expansion rates for the flow are comparable to those for the map
by Lemma~\ref{lem:expansion}, and we have just seen that there are at most two components
of $\cG_t(W)$ for each component of the corresponding partition $\tG_n(P^+(W))$ into
homogeneous components of $P^+(W)$ for the map, we will use the analogous estimates for the map
from \cite{demers zhang} to prove the growth lemma.
The added complication here is that the components of $\cG_t(W)$ may have undergone
different numbers of collisions and so they belong to an assortment of generations
for the discrete time map.

Fix $W \in \cW^s$, and $W_i \in \cG_t(W)$ for  $t >0$.  Let $V = P^+(W)$ and $V_i = P^+(W_i)$.
For (a), by \cite[(3.2), (3.3), Lemma 3.1]{demers zhang}
(choosing $\delta_1=\tilde L_0/3$ in the notation used there),
 there exists $C>0$  such that
for any $n \in \bN$,
\begin{equation}
\label{eq:contract map}
\sum_{V_i \in \tI_n(V)} |J_{V_i}T^n|_{\cC^0(V_i)} \le C\lambda_1^n \, ,
\end{equation}
where  $\tG_k(V)$ are the smooth components of $T^{-k}(V)$ defined in
the beginning of the proof and 
$\tI_n(V)$ denotes the set of smooth components $V_i$ of $T^{-n}(V)$
such that $T^j(V_i)$ never belongs to a curve in $\tG_{n-j}(V)$ that has length greater 
than $\tilde{L}_0/3$ for $0 \le j \le n$.
Since the number of collisions undergone by each $W_i$ before time $t$ may vary, we 
denote by $\cI_{t,n}(W)$ those components of $\cI_t(W)$ which experience $n$ collisions
by time $t$.  Note that 
$\lfloor \frac{t}{\tau_{\max}} \rfloor - 1\le n \le \lfloor \frac{t}{\tau_{\min}} \rfloor + 1$.
Thus by \eqref{eq:contract map} and Lemma~\ref{lem:expansion},
\[
\begin{split}
\sum_{W_i \in \cI_t(W)} |J_{W_i}\Phi_t|_{\cC^0(W_i)}
&= \sum_{n = \lfloor \frac{t}{\tau_{\max}} \rfloor - 1}^{\lfloor \frac{t}{\tau_{\min}} \rfloor + 1}
\sum_{W_i \in \cI_{t,n}(W)} |J_{W_i}\Phi_t|_{\cC^0(W_i)} \\
&\le 2C_0 \sum_{n = \lfloor \frac{t}{\tau_{\max}} \rfloor - 1}^{\lfloor \frac{t}{\tau_{\min}} \rfloor + 1}
C\lambda_1^n \le C' \lambda_1^{t/\tau_{\max}} \, ,
\end{split}
\]
proving part (a) of the  lemma for $\oldbetazero=0$ with $\lambda = \lambda_1^{1/\tau_{\max}}$.  Part (a) of the lemma
for $\oldbetazero \in (0,1] $ follows by an application of the H\"older inequality, as in 
\cite[Lemma 3.3]{demers zhang}:
\[
\begin{split}
\sum_{W_i \in \cI_t(W)} \frac{|W_i|^{\frac{1}{q_0}}}{|W|^{\frac{1}{q_0}}} |J_{W_i}\Phi_t|_{\cC^0(W_i)}
& \le \left(\sum_{W_i \in cI_t(W)} \frac{|W_i|}{|W|} |J_{W_i}\Phi_t|_{\cC^0(W_i)} \right)^{\! \frac{1}{q_0}}
\! \left(\sum_{W_i \in \cI_t(W)} |J_{W_i}\Phi_t|_{\cC^0(W_i)} \right)^{\! 1-\frac{1}{q_0}} \\
& \le (1 + C_d)^{\oldbetazero} (C' \lambda_1^{t/\tau_{\max}})^{1-\oldbetazero},
\end{split}
\]
where we have used  that $|J_{W_i}\Phi_t|_{\cC^0(W_i)} \le (1+C_d) \frac{|\Phi_t(W_i)|}{|W_i|}$
by Lemma~\ref{lem:distortion}, and $\sum_i \frac{|\Phi_t(W_i)|}{|W|} \le 1$.

To prove part (b), we shall adapt the argument of
\cite[Lemma 3.2]{demers zhang}. First, we 
subdivide the time interval $[0,t]$ into times $t_k = t - k \tau_{\min}$,
$k = 0, \ldots, \lfloor t/\tau_{\min} \rfloor$.  Using these subintervals, we group 
the components  $W_i\in \cG_t(W)$ into ``most recent long ancestors,''
defining sets $L_k(W,t)$, $k = 0, \ldots, \lfloor t/\tau_{\min} \rfloor$,
and grouping in  $A_t(U)$ those $W_i\in \cG_t(W)$ so that $U\in L_k(W,t)$, as follows.
For each $W^t_i \in \cG_t(W)$ and each $k$, 
we have $\Phi_{k \tau_{\min}}(W^t_i) \subset W^{t_k}_j$
for some $W^{t_k}_j \in \hG_{t_k}(W)$.  If $W^{t_k}_j$ is undergoing a collision, 
we can adjust
$t_k$ (for that $W^{t_k}_j$ only) by adding a small time   $|\tilde \delta(k,W_i^t)| <L_0$ so that 
$W^{t_k+\tilde\delta}_j \in \cG_{t_k+\tilde \delta}(W)$ is an admissible stable curve.
  Note that even with this small correction,
the times\footnote{ We write $t_k$ for $t_k+\tilde \delta(k,W_i^t)$, for simplicity.} $t_k$ 
are still 
at least $\tau_{\min}/2$ apart since $L_0$ has been chosen $< \tau_{\min}/4$.
We say that such a curve $W^{t_k}_j$ is the most recent long ancestor of $W^t_i$ if
$|P^+(W^{t_k}_j)| \ge L_0/3$ and $k\ge 0$ is the least such $k$ with this property for $W^t_i$.
In this case, we put $i \in A_t(W^{t_k}_j)$ and $j \in L_k(W,t)$.  If no such 
$k$ exists, then by definition, $W^t_i \in \cI_t(W)$
(that we denote by $i\in  \cI_t(W)$ for simplicity).  Using this grouping, we estimate,
\[
\sum_{W^t_i \in \cG_t(W)} |J_{W^t_i}\Phi_t|_{\cC^0(W^t_i)}
= \sum_{k=0}^{\lfloor t/\tau_{\min} \rfloor} \sum_{j \in L_k(W,t)} 
\sum_{i \in A_t(W^{t_k}_j)} |J_{W^t_i}\Phi_t|_{\cC^0(W^t_i)}
+ \sum_{i \in \cI_t(W)} |J_{W^t_i}\Phi_t|_{\cC^0(W^t_i)} \,  .
\]
The  sum over $\cI_t(W)$ is exponentially small in $t$ by (a), so we focus on the sum over $k$.
For each $W^{t_k}_j$, we have $|J_{W^t_i} \Phi_t|_{\cC^0(W^t_i)}
\le |J_{W^t_i} \Phi_{k \tau_{\min}}|_{\cC^0(W^t_i)} |J_{W^{t_k}_j} \Phi_{t_k} 
|_{\cC^0(W^{t_k}_j)}$.
Thus, using part (a) of the lemma from time $t_k$ to time $t$,
\begin{align*}
\sum_{i \in A_t(W^{t_k}_j)} |J_{W^t_i}\Phi_t|_{\cC^0(W^t_i)}
&\le |J_{W^{t_k}_j} \Phi_{t_k} |_{\cC^0(W^{t_k}_j)} 
\sum_{i \in A_t(W^{t_k}_j)} |J_{W^t_i} \Phi_{k \tau_{\min}}|_{\cC^0(W^t_i)}\\
&\le |J_{W^{t_k}_j} \Phi_{t_k} |_{\cC^0(W^{t_k}_j)} C \lambda^{k \tau_{\min}}\,  .
\end{align*}
Using this estimate, plus the fact that 
$|J_{W^{t_k}_j} \Phi_{t_k} |_{\cC^0(W^{t_k}_j)} \le C_d \frac{|\Phi_{t_k}(W^{t_k}_j)|}{|W^{t_k}_j|} \le C L_0^{-1}
|\Phi_{t_k}(W^{t_k}_j)|$ by bounded distortion (Lemma~\ref{lem:distortion}), we have
\[
\sum_{W^t_i \in \cG_t(W)} |J_{W^t_i}\Phi_t|_{\cC^0(W^t_i)}
\le \sum_{k=0}^{\lfloor t/\tau_{\min} \rfloor} \sum_{j \in L_k(W,t)} C L_0^{-1} 
|\Phi_{t_k}(W^{t_k}_j)| \lambda^{k \tau_{\min}} + C \lambda^t \,  .
\]
For each $k$, the sum over $j$ is at most $|W|$ since $\cup_j \Phi_{t_k}(W^{t_k}_j)$ is a disjoint union,
and the sum over $k$ is uniformly bounded in $t$, proving part (b) for $\oldbetazero = 0$.
Part (b) for $\oldbetazero >0$ now follows from a H\"older inequality as in part (a).
\end{proof}


\subsection{Embeddings of smooth functions in $\cB$. Compactness of $\cB$ in $\cB_w$}

Recalling $\cC^0_\sim$ from \eqref{contit}, it is convenient to introduce two auxiliary Banach spaces:

\noindent The space $\cB^0$  is the completion with respect to $\| \cdot \|_\cB$ of 
\begin{equation}
\label{defB0}
\{ \oldh \in \cC^0(\Omega_0) \mid  \| \oldh \|_\cB< \infty \} \, .
\end{equation}
The space $\cB^0_\sim$  is the completion with respect to $\| \cdot \|_\cB$ of 
\begin{equation}\label{defB0sim}
\{ \oldh \in \cC^0_\sim \mid  \| \oldh \|_\cB< \infty \} \, . 
\end{equation}

These definitions ensure that 
$\cB\subset \cB^0_\sim$ and $\cC^0_\sim \cap \cC^1 \subset \cB^0_\sim$.
 Also,   $\cB^0_\sim \cap \cC^0_\sim$ is dense in $\cB^0_\sim$ and  
similarly for $\cB_w$. 
 This will allow us to define  transfer
operators $\cL_t \oldh=\oldh \circ \Phi_{-t}$ on $\cB^0_\sim$ by density and the Lasota-Yorke estimate
(Proposition~\ref{prop:L_t}), bypassing the
use of analogues of the spaces $\cC^\beta(T^{-n}\cW^s)$ in \cite{demers zhang}.
 (In view of \cite[Lemma 3.7]{demers zhang} for the discrete
time billiard, one can expect that $\cB=\cB^0_\sim$, but this will not be needed.)

We establish  relations between our spaces $\cB$ and $\cB_w$ and
smooth  functions on $\Omega_0$.

\begin{lemma}[Embedding smooth functions]
\label{lem:embedding}
The following continuous injective embeddings hold
\[
\begin{split} 
& \cC^1(\Omega_0) \hookrightarrow \cB^0 ,
\quad 
\cC^0_\sim\cap \cC^1(\Omega_0) \hookrightarrow \cB^0_\sim \hookrightarrow \cB_w ,
\quad
\cB \hookrightarrow (\cC^\beta(\Omega_0))^*, \\
& \cB \hookrightarrow (\cC^1(\Omega_0))^*, \quad
\mbox{and} \quad
\cC^0_\sim\cap \cC^2(\Omega_0) 
\hookrightarrow \cC^1(\Omega_0)\hookrightarrow \cB \hookrightarrow \cB_w \, .
\end{split}
\]
\end{lemma}

(The proof uses $\oldbeta \le 1$, $\beta <1/q$, and  $\gamma \le 1/2$.)

\begin{proof}
First we show that if $\oldh \in \cC^1(\Omega_0)$, then $\oldh \in \cB^0$ and 
in particular,  
\begin{equation}
\label{laborne}
\| \oldh \|_{\cB} \le C |\oldh|_{\cC^1(\Omega_0)}
\end{equation}
for a uniform constant $C$. 
This immediately implies that if $\oldh \in \cC^1(\Omega_0)\cap \cC^0_\sim$, then $\oldh \in \cB^0_\sim$.

To estimate the neutral component $\| \oldh \|_0$, fix $W \in \cW^s$ and $\psi \in \cC^\oldp(W)$ with 
$|\psi|_{\cC^\oldp} \le 1$.  Then, recalling \eqref{esscurve} and the notation in \eqref{Jacobi},
we have 
\[
\int_W \partial_t (\oldh \circ \Phi_t)|_{t=0} \, \psi \, dm_W 
=\int_W  \nabla \oldh \cdot \heta\,  \psi \, dm_W\le
|\nabla \oldh \cdot \heta|_\infty |\psi|_\infty |W|
 \le L_0 |\oldh|_{\cC^1(\Omega_0)}\,  .
\]

To estimate $\| \oldh \|_s$, fix $W \in \cW^s$ and $\psi \in \cC^\oldq(W)$ with 
$|W|^{\oldbeta} |\psi|_{\cC^\oldq(W)} \le 1$.  Then
\[
 \int_W \oldh \psi \, dm_W \le |\oldh|_\infty |\psi|_{\infty} |W| \le |\oldh|_{\infty} L_0^{1-\oldbeta}\,  .
\]

Finally, we estimate $\| \oldh \|_u$.  Fix $\ve \in( 0,1)$ (if $\ve \ge 1$
we may bound the relevant
quotient by a constant multiple of the weak norm) and $W_1, W_2 \in \cW^s$ with 
$d_{\cW^s}(W_1,W_2) \le \ve$.
For $i=1,2$, let $\psi_i \in \cC^\oldp(W_i)$ with $|\psi_i|_{\cC^\oldp(W_i)} \le 1$ and $d(\psi_1, \psi_2)=0$.

By definition of $d_{\cW^s}$, the trace curves $P^+(W_1)$ and $P^+(W_2)$ are defined over a common $r$-interval $I$, apart from at most
two endpieces which have length no more than $\ve$.  Let $U_1 \subset W_1, U_2 \subset W_2$
denote the curves for which $P^+(U_1)$ and $P^+(U_2)$ are defined as graphs over $I$.  Denote by $V_1$ and $V_2$
the (at most two) pieces which are not defined over $I$.  Without loss of generality, we may assume $V_1 \subset W_1$
and $V_2 \subset W_2$.  Now,
\begin{align*}
 \int_{W_1} \oldh \psi_1 \, dm_{W_1} - \int_{W_2} \oldh \psi_2 \, dm_{W_2} &= 
\int_{U_1} \oldh \psi_1 dm_{W_1} - \int_{U_2} \oldh \psi_2 dm_{W_2}\\
&\qquad 
 + \int_{V_1} \oldh \psi_1 dm_{W_1} - \int_{V_2} \oldh \psi_2 dm_{W_2} \,  .
\end{align*}
We use the fact that $|P^+(V_i)| \le \ve$ to bound the estimate over the unmatched pieces,
\[
 \left| \int_{V_i} \oldh \psi_i dm_{W_i} \right| \le |V_i| |\oldh|_{\infty} \le C \ve |\oldh|_{\infty}\, ,
\]
where the constant $C$ depends only on the Jacobian of the map from $P^+(V_i)$ to $V_i$,
which is uniformly bounded by the proof of Lemma~\ref{lem:smooth}.

Next, we estimate the contribution from the matched curves $U_i$,
\begin{align*}
 \int_{U_1} \oldh \psi_1 dm_{W_1} - \int_{U_2} \oldh \psi_2 dm_{W_2}
 &= \int_{U_1} (\oldh \psi_1 - (\oldh\cdot \psi_2 )\circ \Theta \cdot J\Theta) dm_{W_1}
\\ &
\le |U_1| |\oldh \psi_1 -( \oldh \cdot \psi_2)\circ \Theta \cdot J\Theta|_{C^0(U_1)}\, ,
\end{align*}
where $\Theta$ is the map from $U_1$ to $U_2$ defined via the trace curves:
$\Theta = S_{U_2} \circ S_{U_1}^{-1}$, where $S_{U_i}(r) = \Phi_{-t(r)} \circ G_{U_i}(r) : I \to U_i$ 
is defined as in Lemma~\ref{lem:smooth}.
Note that since $d(\psi_1,\psi_2)=0$, we have $\psi_2 \circ \Theta = \psi_1$ on $U_1$.
Thus to estimate the sup norm of the difference above, it suffices to estimate
$|\oldh - \oldh \circ \Theta|_{\cC^0(U_1)}$ and $|1- J\Theta|_{\cC^0(U_1)}$.  Note also that
$J\Theta$ is bounded by Lemma~\ref{lem:smooth}, so this factor contributes only a bounded
constant to our estimate.  

Now since the distance between $P^+(U_1)$ and $P^+(U_2)$ along vertical segments is at most
$\ve$ and since the Jacobians of $S_{U_2}$ and $S_{U_2}^{-1}$ are uniformly bounded by
Lemma~\ref{lem:smooth}, we have $d(Z, \Theta(Z)) \le C \ve$, for $Z \in U_1$.  Since 
$\oldh \in \cC^1(\Omega_0)$,
we have $|\oldh - \oldh \circ \Theta|_{\cC^0(U_1)} \le |\oldh|_{\cC^1} \ve$.

It remains to estimate $|1-J\Theta|_{\cC^0(U_1)}$.  Using the notation of 
Lemma~\ref{lem:smooth} and denoting by $\tau_1$ and $\tau_2$ the first collision times starting on $U_1$ and $U_2$,  
we have
\begin{equation}
 \label{eq:JTheta}
\begin{split}
J\Theta & = (J_{P^+(U_2)} \Phi_{-\tau_2}) \circ (G_{U_2} \circ S_{U_1}^{-1}) \cdot (JG_{U_2}\circ S_{U_1}^{-1})
 \cdot (JG_{U_1}^{-1} \circ \Phi_{\tau_1}) \cdot (J_{U_1}\Phi_{\tau_1}) \\
& = \frac{(J_{P^+(U_2)} \Phi_{-\tau_2}) \circ G_{U_2} \circ S_{U_1}^{-1} }
{(J_{P^+(U_1)} \Phi_{-\tau_1}) \circ \Phi_{\tau_1}} \cdot \frac{JG_{U_2} \circ S_{U_1}^{-1}}
{JG_{U_1} \circ S_{U_1}^{-1}} \, .
\end{split}
\end{equation}
For $Z \in U_1$, the points $G_{U_2} \circ S_{U_1}^{-1}(Z)$ and $\Phi_{\tau_1}(Z)$ lie on the same vertical line in $(r,\vf)$ coordinates; thus the contribution of $\cK$ in the analogue of
\eqref{eq:log J} vanishes. Hence,
adapting \eqref{eq:log J} and the following lines
(using $|\tau_1(Z)-\tau_2(S_{U_2}S_{U_1}^{-1}(Z))|\le C \sqrt \ve$), we have 
\[
\sup_{U_1} \left| \ln \frac{J_{P^+(U_2)} \Phi_{-\tau_2}(G_{U_2} \circ S_{U_1}^{-1}) }{J_{P^+(U_1)} \Phi_{-\tau_1} (\Phi_{\tau_1})}  \right|
\le C\ve^{1/2}\,  ,
\]
for some uniform constant $C>0$.  Also, by \eqref{eq:dG_W} and using $\vf_1(r)$ and $\vf_2(r)$ to denote the functions defining $P^+(U_1)$ and
$P^+(U_2)$, we obtain the following bound on $I$:
\[
\begin{split}
\left| \ln \frac{JG_{U_2}(S_{U_1})}{JG_{U_1}(S_{U_1})} \right| 
& = \frac 12 \left| \ln \left( 1 + \Big(\frac{d\vf_2}{dr}\Big)^2 \right) - \ln \left( 1 + \Big(\frac{d\vf_1}{dr}\Big)^2\right) \right| \\
& \le \frac{1}{2 + 2\cK_{\min}} \left| \Big(\frac{d\vf_2}{dr}\Big)^2 -  \Big(\frac{d\vf_1}{dr}\Big)^2\right|
\le \frac{\cK_{\max} + \tau_{\min}^{-1}}{1 + \cK_{\min}} \left| \frac{d\vf_2}{dr} - \frac{d\vf_1}{dr} \right| \le C\ve \, ,
\end{split}
\]
by definition of $d_{\cW^s}(W_1, W_2)$.  By \eqref{eq:JTheta}, these two estimates imply  $|1-J\Theta|_{\cC^0(U_1)} \le C \ve^{1/2}$.

Putting together the estimates for matched and unmatched pieces and dividing by $\ve^\gamma$ yields
\[
\ve^{-\gamma} \left|\int_{W_1} \oldh \psi_1 \, dm_W - \int_{W_2} \oldh \psi_2 \, dm_W\right|  
\le C\ve^{1-\gamma} |\oldh|_{\infty} + C\ve^{1/2-\gamma}|\oldh|_{\cC^1(\Omega_0)}  \, . 
\]
(Recall that $\gamma < \oldp \le 1/3$.)

The above estimates imply $\| \oldh\|_{\cB} \le C |\oldh|_{\cC^1(\Omega)}$ for a uniform constant $C$, as claimed.
This implies  continuity of the embeddings $\cC^1(\Omega_0) \hookrightarrow \cB^0$
and  $\cC^0_\sim
\cap \cC^1(\Omega_0) \hookrightarrow \cB^0_\sim$.  Injectivity
is obvious.

As for the embedding $\cB^0_\sim \hookrightarrow \cB_w$, continuity follows from the 
fact that $| \cdot |_w \le \| \cdot \|_s$, while injectivity follows from the definition
of the spaces $\cC^\oldp(W)$ and $\cC^\oldq(W)$ as the closures of $\cC^1(W)$ in the respective H\"older norms.  Similarly, the embedding $\cB \hookrightarrow \cB_w$ is continuous and injective
by the same observations.  The  embedding 
$\cC^1(\Omega_0) \hookrightarrow \cB$ 
 is the content of  Lemma~ \ref{lem:C1}, while injectivity is again obvious.
 Finally, continuity of the embeddings $\cB \hookrightarrow (\cC^{\beta}(\Omega_0))^*$
 and $\cB \hookrightarrow (\cC^1(\Omega_0))^*$ follow from
 Lemma~\ref{relating}, while injectivity is proved in Lemma~\ref{lem:injective}, using $\beta <1/q$.
\end{proof}

We close this section with the following compactness result.

\begin{lemma}[Compact embedding]
\label{lem:compact}
The unit ball of $\cB$  is  compactly embedded in $\cB_w$.
\end{lemma}
 (Recall that $\oldq <  \oldp$. Like in \cite{demers zhang}, the proof also uses  $c_u>0$.)
 \begin{proof}
Recalling \eqref{defB0sim}, it suffices to
show that the unit ball of $\cB^0_\sim$  is compactly embedded in $\cB_w$. 
The general strategy of our proof will be to create, for each $\ve > 0$, an $\ve$-covering of 
$\Omega$ by finitely many curves $W_i$ in $\cW^s$ and an $\ve$-covering of $C^\alpha(W_i)$
by finitely many functions $\psi_j$ so that the weak norm of any $f \in \cB$ with $\| f \|_{\cB} \le 1$ can be 
uniformly approximated by 
measuring it against the finitely many functionals defined by $\int_{W_i} f \, \psi_j \, dm_W$.

We may assume without restricting generality that there exists $\ell$
so that $P^+(W) \in  \Gamma_\ell \times [-\pi/2,\pi/2]$, i.e, we
argue one scatterer-component $\Omega^{(\ell)}:= (P^+)^{-1}(\Gamma_\ell \times [-\pi/2,\pi/2])$
at a time. 
Let $0<\ve \le 1$ be fixed. 
We split $\Omega^{(\ell)}$ into
two parts, the good set 
$$A(\ell,\ve)=
 (P ^+)^{-1}(\{-\pi/2+\ve
 \le \varphi\le \pi/2-\ve
 \})\cap \Omega^{(\ell)}\, ,
$$
and the bad set $B(\ell,\ve)=\Omega^{(\ell)}\setminus A(\ell,\ve)$
(that is, $\ve$-close to tangential collisions). 
The image under $P^+$ of stable curves are graphs of
decreasing functions $\varphi_W$ of absolute value of the slope  larger than $\cK_{\min}>0$ 
and with uniformly bounded second derivatives.
If $W\subset B(\ell,\ve)$, then since  $P^+(W)$
is transversal to the boundary of the scatterer, it has length $O(\ve)$.
Since the
expansion from $P^+(W)$ to $W$ is of order $1$ for a stable curve
by Lemma \ref{lem:smooth}, there exists
$C=C(Q)$ so that any admissible curve $W \subset B(\ell,\ve)$
has length at most $C\ve$.

Let $\oldh\in \cC^0_\sim$ with $\|\oldh\|_\cB\le 1$. First, we estimate the weak norm of $\oldh$
on curves $W\in B(\ell,\ve)$. If 
$|\psi|_{\cC^\oldp(W)}\le 1$, then, using the bound $C\ve$ on the length of $W$,
\begin{equation}
\label{smallcurve}
\int_W \oldh \psi \, dm_W \le \|\oldh\|_s |W|^\oldbeta |\psi|_{\cC^\oldq(W)} \le C^\oldbeta \ve^\oldbeta \|\oldh\|_s\, .
\end{equation}

In order to study curves $W\in A(\ell,\ve)$, we need some preparations.

Letting $\delta_1(L_0)$ and $B_1(B_0)$  (uniform
in $\ve$)  
be as
 in the proof of Lemma~\ref{lem:growth},   for any map-stable 
curve
$V\in \cW^s(T)\cap P^+(A(\ell,\ve))$
of length $\le \delta_1$ and curvature bounded by $B_1$,
the surface 
$$
V^0=\{ \Phi_{-t}(r,\varphi_V(r))\mid r \in I_V\, , 0\le  t < \tau_-(r,\varphi_V(r))\}
$$
(where $V$ is the graph of $\varphi_V$ defined on $I_V$,
and $\tau_-(z)$, for $z\in \cM$, is the smallest $t>0$ so that
$\Phi_{-t}(z)\in \cM$)
is foliated by admissible flow-stable curves $W$ so that 
 $P^+(W)\subset V$.
Indeed, if $Z\in V^0$, there is a unique curve $W=W(Z, V)\subset V^0$ containing
$Z$ and so that $W$ is everywhere perpendicular to the flow. We take
 $W$ maximal with these properties, noting that  the endpoints
of  $W$ either (1) project to endpoints of $V$; (2) lie on a scatterer, i.e., are of the form 
\begin{equation}\label{point2}
\Phi_{-\tau_-(r,\varphi_V(r))}(r,\varphi_V(r))\, ,
\end{equation} 
or: (3)
 undergo a grazing collision under the flow in forward time. Such a curve
$W$ is $\cC^2$ and flow-stable by construction, it satisfies the admissibility
requirements for $L_0$ and $B_0$,  if $\delta_1$ and $B_1$ are small
 enough. (Note that for any flow stable curves $W_1\ne W_2\subset V^0$ we have
$d_{\cW^s}(W_1,W_2)=\infty$.)  All such curves $W$  are 
constant-time flow translates of one another, except near
collisions, where some shortening of the curve,  due to the variable collision times,
may occur.
Conversely, if $W\subset A(\ell,\ve)$ is an admissible flow curve for
$L_0$ and $B$, then $W$ belongs to the surface
$(P^+(W))^0$, which is foliated by  flow-stable curves, admissible for
$\delta_2(L_0)$ and $B_2(B_0)$, uniformly in $\ve$.

On a fixed $r$-interval $I$, the set of functions $\{\varphi_{V}\}$ 
for map-stable curves $V\in \cW^s(T)$, defined on $I$
is uniformly bounded in $\cC^2$ norm and therefore compact in the $\cC^1$ norm.
There exists a finite set of 
admissible map-stable curves $\{V'_i\}_{i=1}^{I'_\ve}\subset P^+(A(\ell,\ve))$ 
so that for any flow-stable admissible curve $W \subset A(\ell,\ve)$ 
there exists $i$ with $\td_{\cW^s}(V'_i,P^+(W))<\ve$. 
Therefore, since the stable and unstable cones $C^s$ and $C^u$ are uniformly
transverse, and are both orthogonal to the flow direction,
we may choose a finite set of 
admissible map-stable curves $\{V_i\}_{i=1}^{I_\ve}\subset P^+(A(\ell,\ve))$ 
so that for any flow-stable admissible curve $W \subset A(\ell,\ve)$,
there exist:
\begin{itemize}
\item[(i)] an index $i_W$ 
so that $\tilde{d}_{\cW^s}(P^+(W), V_{i_W}) < \ve$, and a flow-unstable curve $W^u$ with
$W^u\cap W\ne \emptyset$ and $W^u \cap V_{i_W}^0\ne \emptyset$;
\item[(ii)] a flow-stable curve $W' \subset V_{i_W}^0$ with $W^u \cap V_{i_W}^0 \in W'$ 
and $d_{\cW^s}(W, W')< C\sqrt{\ve}$, for some uniform $C>0$.
\end{itemize}
Item (i) above is obvious, as is the existence of $W'$ with 
$W^u \cap V_{i_W}^0 \in W' \subset V_{i_W}^0$ in item (ii).  To prove the bound
on the distance between $W$ and $W'$ in (ii), first note that the tangent vectors to $P^+(W^u)$
lie  in the map-unstable cone by forward invariance of the unstable cones, and
that $P^+(W^u)$ intersects 
both $V_{i_W}$ and $P^+(W)$.  By choice of the index $i_W$, the length
of the segment in $P^+(W^u)$  connecting $P^+(W)$ to $V_{i_W}$
is at most $C\ve$.  Since 
$W^u$ intersects both $W$ and $W'$ and the length 
of unstable curves is expanded going forward, it follows that the   length
of the segment in $W^u$ 
connecting $W$ and $W'$ is at most $C\ve$ as well.  Finally, the fact that $P^+(W') \subset V_{i_W}$ 
and $\tilde{d}_{\cW^s}(P^+(W), V_{i_W}) < \ve$, together with the $1/2$-H\"older continuity of 
$\tau$, implies that the distance $\tilde{d}_{\cW^s}(P^+(W), P^+(W'))$
(which is due only to the endpoints-discrepancy)
satisfies $\tilde{d}_{\cW^s}(P^+(W), P^+(W')) < C\sqrt{\ve}$
(recalling \eqref{point2}) as required.\footnote{ The bound $<\sqrt{\epsilon}$ can perhaps be improved to
$<C \epsilon$ by making a special choice of $W^u$.}

\smallskip
We next decompose each $V_i$ corresponding
to collision times  to handle
the shortening of the curve  due to the variable collision times mentioned
above. We let $M_\ve=[\epsilon^{-1}]$, 
$$
\tau_{i,\min}=\min \tau_-(r,\varphi_V(r))\, ,\, 
\,
\tau_{i,m}= \tau_{i,\min} + m M_\ve^ {-1} (\max \tau_-(r,\varphi_V(r))-\tau_{i,\min})\, , \, 
m=0, \ldots, M_\ve \, .
$$
Note that $\tau_{i,\min}$ is the  first collision time when
flowing $V_i$ backwards. It is not enough to flow $V_i$ back until this first collision:  
If $W$ is a stable curve with $P^+(W)$ $\ve$-close to $V_i$, the curve $W$ may be in fact very far 
from any of  the curves spanned by $V_i$ up to time $-\tau_{i,\min}$
(for example, if a little piece of $V_i$ hits a
nearby scatterer under the backwards flow, but most of the rest of $V_i$ continues without
 collision for some time).  Still, there is a subcurve $V_i'$ of $V_i$, differing in 
 length\footnote{ By construction of the family $V_i$.} by no 
more  than $\ve$ from $V_i$,
 such that  $V_i'$ does not hit this close scatterer, and the backward translates of  $V_i'$
 contain a stable curve which is close to $W$ and connected with $W$ by an unstable curve.  
By the triangle inequality, $V_i'$ and $P^+(W)$ differ by no more than $2\ve$.) 

Let us now formalise the above discussion: We construct nested curves by setting
$V_{i,0}=V_i=\{V_i\cap \tau_{-}^{-1}[\tau_{i,0}, \infty)\}$ and,
for $m=0, \ldots, M_\ve$,
$$
V_{i,m}=\{V_i\cap \tau_{-}^{-1}[\tau_{i,m}, \infty)\} \subset V_{i,m-1}\, .
$$
(If any of the $V_{i,m}$ is disconnected, we  subdivide it  into
its finitely many connected components, without making this explicit in the
notation for the sake of conciseness.)
Finally, in each $V_{i,m}$ we choose a point $v_{i,m}$ so that
$\tau_-(v_{i,m})=\tau_{i,m}$.

Then, there exists a nonnegative
real number $\Theta_\ve=O(\ve)$, and a finite integer $N_\ve=O(\Theta_\ve^{-1})$, 
such that\footnote{ We do not claim
that $I_\ve$, $N_\ve$, or $\Theta_\ve^{-1}$ are bounded
uniformly in $\ve$.}
the set of flow-stable curves
\[
\begin{split}
&\{ W(\Phi_{-\theta}(\Phi_{-\tau_-( v_{i,m})\frac{n}{N_\ve}}(v_{i,m})), V_{i,m}) , \, \\
&\qquad\qquad i=1,\ldots ,I_\ve, \, m=0, \ldots, M_\ve, \, 
n=0,\ldots, N_\ve-1,\, 
\theta \in[0,\Theta_\ve]
\}
\end{split}
\] 
forms an $\sqrt{\ve}$-covering
of the admissible flow-stable curves $\cW^s|_{(A(\ell,\ve))}$ in the distance $d_{\cW^s}$,
which follows by applying items (i) and (ii) explained above.

By definition, the
uniquely defined $\cC^1$ functions  $\tilde t_{i,m, n,\theta}$  on $V_{i,m}$
so that
$$
\widetilde W_{i,m, n,\theta}: =W(\Phi_{-\theta}\Phi_{-\tau_-( v_{i,m})\frac{n}{N_\ve}}(v_{i,m}), 
V_{i,m})
\mbox{ is 
given by }
\Phi_{-\tilde t_{i,m, n,\theta}(r)}(r,\varphi_{V_i}(r))=:
\tilde S_{i,m, n,\theta}(r)
$$ 
satisfy
$
\tilde t_{i,m,n,\theta}-\tilde t_{i,m,n,0}\equiv \theta
$.

Set $t_{i,m, n}=\tilde t_{i,m,n,0}$,
so that
$W_{i,m, n}:=W(\Phi_{-\tau_-( v_{i,m})\frac{n}{N_\ve}}(v_{i,m}), V_{i,m})$
is given by 
$$\Phi_{-t_{i,m,n}(r)}(r,\varphi_{V_{i,m}}(r))=:
S_{i,m, n}(r) \, .
$$

Let $|\Gamma_\ell|$ denote the arclength of $\Gamma_\ell$, and define
$\bS^1_\ell$ to be the circle of length $|\Gamma_\ell|$. Since any ball of
finite radius in the $\cC^\oldp$ norm is compactly embedded in $\cC^\oldq$,
we may choose finitely many functions $\bar \psi_j\in \cC^\oldp$
such that $\{\bar \psi_j\}_{j=1}^{J_\ve}$ forms an $\ve$-covering 
in the $\cC^\oldq(\bS^1_\ell)$-norm of the ball of radius $CC_0$
in $\cC^\oldp(\bS^1_\ell)$,
with
$
C_0
$ the constant depending on $\cK_{\max}$, $\tau_{\max}$, and $1/\tau_{\min}$ from
Lemma~\ref{lem:smooth}.

This ends the announced preparations for the case $W \subset A(\ell,\ve)$.

\smallskip
From now on, fix $W=S_W(I_W)\in \cW^s|_{A(\ell,\ve)}$, and $\psi \in \cC^\oldp(W)$ with
$|\psi|_{\cC^\oldp(W)}\le 1$.  We view $I_W$ as a subset of
$\bS^1_\ell$. 
Let $\bar \psi=\psi \circ S_W$ be the push down
of $\psi$ to $I_W$. By definition, $S_W(r)=\Phi_{-t_W(r)}G_W(r)$ with
$G_W$ the graph of the function $\varphi_W$.
By the proof of Lemma ~\ref{lem:smooth}, the Jacobian 
of $S_W$ is 
bounded by $C_0$, and we have
$$
|\bar \psi|_{\cC^\oldp(I_W)}\le C C_0 \, .
$$
Choose $\bar \psi_j \in \cC^\oldp(\bS^1_\ell)$ so that
$|\bar \psi -\bar\psi_j|_{\cC^\oldq(I_W)}\le \ve$. 

Take $i$, $m$,  $n$, and $\theta$ so that
$d_{\cW^s}(W,\widetilde W_{i,m, n,\theta})<\sqrt{\ve}$.
Define $\psi_{j,i,m, n}=\bar \psi_j \circ
S_{i,m, n}^{-1}$ to be the lift of $\bar \psi_j$ to 
$W_{i,m, n}$.
Note 
that $|\psi_{j,i,m,n}|_{\cC^\alpha(W_{i,m,n})}\le 2 CC_0$, 
again by  Lemma \ref{lem:expansion}.

Then, normalizing $\psi$ and $\psi_{j,i,m, n}$
by $2CC_0$, and letting $J_{W,i,m, n,\theta}=J_{W_{i,m, n}}\Phi_{-\theta}$ be the Jacobian of the map
$\tilde S_{i,m, n,\theta} \circ S_{i,m, n}^{-1}=\Phi_{-\theta}$ from $W_{i,m,n}$ 
to $\widetilde W_{i,m, n,\theta}$,
we decompose
\begin{equation}
\begin{split}
&\left | \int_W \oldh\psi dm_W -\int_{W_{i,m, n}} \oldh\psi_{j,i,m, n} dm_{W_{i,m, n}}\right |\\
&\quad\le  \left | \int_W \oldh\psi dm_W -
\int_{\widetilde W_{i,m, n,\theta}}  \oldh(\bar \psi_{j}  \circ \tilde S_{i,m, n,\theta}^{-1})
dm_{\widetilde W_{i,m, n,\theta}}\right |\\
\label{llast} &\quad\quad+ \left | \int_{ W_{i,m, n}} 
[(\oldh \circ \tilde S_{i,m, n,\theta} \circ S_{i,m, n}^{-1}) \psi_{j,i,m,n} J_{W,i,m, n,\theta} 
- \oldh\psi_{j,i,m, n} ]\, dm_{W_{i,m, n}}
\right | \, ,
\end{split}
\end{equation}
where 
$|J_{W,i,m, n,\theta}-1|_{C^\oldq(W_{i,m, n})}\le C\ve 
$, so that we can use the triangle inequality and $\|\oldh\|_s$
to eliminate $J_{W,i,m, n,\theta}$ in the last term of \eqref{llast}. Next, recalling the group property \eqref{eq:flow derivative}, we estimate the remaining
difference in the last term of \eqref{llast},
\[
\begin{split}
\int_{W_{i,m,n}} &(f \circ \Phi_{-\theta} - f) \psi_{j,i,m,m} \, dm_{W_{i,m,n}}
= \int_{W_{i,m,n}} \int_0^\theta \partial_s (f \circ \Phi_{-s}) \psi_{j,i,m,n} \, ds dm_{W_{i,m,n}} \\
&= \int_0^\theta  \int_{W_{i,m,n}} \partial_r (f \circ \Phi_r)|_{r=0} \circ \Phi_{-s} \psi_{j,i,m,n} \, dm_{W_{i,m,n}} ds \\
& = \int_0^\theta \int_{\Phi_{-s}(W_{i,m,n})} \partial_r (f \circ \Phi_r)|_{r=0} \, \psi_{j,i,m,n} \circ \Phi_s \,
J_{\Phi_{-s}(W_{i,m,n})} \Phi_s \\
& \le \theta \|f \|_0 |\psi_{j,i,m,n} \circ \Phi_s |_{\cC^\alpha(\Phi_s(W_{i,m,n}))} 
|J_{\Phi_{-s}(W_{i,m,n})} \Phi_s |_{\cC^\alpha(\Phi_s(W_{i,m,n}))} 
\le C \ve \| f \|_0,
\end{split}
\]
where in the last line we have used the smoothness of the flow between collisions to bound
both the H\"older norm of the test function and the Jacobian of the change of variables.

Finally, for the first term on the right-hand side of \eqref{llast}, we again use 
the triangle inequality and $\|\oldh\|_s$  to replace
$\bar \psi_j \circ \tilde S_{i,m, n,\theta}^{-1}$ by a test function
$\psi_{i,j,m, n,\theta}$ with $d(\psi,\psi_{i,j,m, n,\theta})=0$ and use the unstable norm of $f$ to 
bound the remaining difference of integrals.
Putting these estimates together yields, 
$$
\left | \int_{W} \oldh\psi\, dm_W 
-\int_{W_{i,m, n}}  \oldh\psi_{j,i,m, n} \, dm_{W_{i,m, n}}\right |
\le (\ve^{\gamma/2} \|\oldh\|_u +\ve\|\oldh\|_s+\ve \|\oldh\|_0   )2 C C_0 \, .
$$
Recalling that $2\gamma<1$,
we have proved that for each $0<\ve\le 1$, there exist
finitely many bounded linear functionals $\ell_{i,j,m, n}$ with
$\ell_{i,j,m, n}(\oldh)=\int_{W_{i,m, n}} \oldh\psi_{j,i,m, n} dm_{W_{i,m, n}}$ such that
\begin{align*}
|\oldh|_w
&\le \max_{i\le I_\ve, m\le M_\ve,j \le J_\ve, n\le N_\ve-1}
\ell_{i,j,m, n}(\oldh)+
2CC_0(( \ve^\oldbeta+\ve) \|\oldh\|_s +\ve^{\gamma/2} \|\oldh\|_u + \ve \|\oldh\|_0)\\
&\le  \max
\ell_{i,j,m, n}(\oldh) + C' c_u^{-1} \ve^{\gamma/2} \|\oldh\|_\cB\, .
\end{align*}
Since $c_u>0$, this implies the required compactness.
\end{proof}


\section{Lasota-Yorke-type bounds and strong continuity for the semi-group $\cL_t$}
\label{section4}

The
transfer operator for the flow is defined on\footnote{ Defining the transfer operator on $L^\infty$ leads
to problems since an element
of $L^\infty$ is not in general well-defined on a stable curve $W$.} $\cC^0_\sim$ 
(recall \eqref{contit}) by 
$\cL_t \oldh = \oldh \circ \Phi_{-t}
$ for $t\ge 0$.
(In particular, $\cL_0$ is the identity on $\cC^0_\sim$.)
Letting $m$ denote Riemannian volume on $\Omega_0$,
we also have  $\cL_t \oldh = \oldh \circ \Phi_{-t}$ for any $\oldh \in L^1(\Omega_0, m)$,
since $\Phi_t$ preserves $m$.
The Banach space  $\cB$ is not contained in 
$\cC^0_\sim$ or  $L^1(m)$.
However, we have that
 $\cC^0_\sim \cap \cB^0_\sim$ is dense in the auxiliary space $\cB^0_\sim$ by definition
 \eqref{defB0sim}.
 In this section, we will prove the following proposition
 (recall $\cB^0$  from  \eqref{defB0},
and the constants $\Lambda=\Lambda_0^{1/\tau_{\max}}>1$ from (\ref{Lambda}) and 
 $\lambda \in(\Lambda_0^{-1/\tau_{\max}},1)$
from the growth Lemma~\ref{lem:growth}).

\begin{proposition}
\label{prop:L_t} Recall that $\oldq< \oldp$ and $\gamma \le \min\{\oldp-\oldq, \oldbeta\}$.
There exists $C>0$ such that for all $\oldh \in \cC^0 \cap \cB^0$ and $t > 0$,
\begin{eqnarray}
|\cL_t \oldh|_w & \le & C |\oldh|_w ÷, , \label{eq:weak L} \\
\| \cL_t \oldh \|_s & \le & C ( \Lambda^{-\oldq t} + \lambda^{(1-\oldbeta)t}) \| \oldh \|_s 
+ C L_0^{-\oldbeta} |\oldh|_w \, ,  \label{eq:strong stable L}  \\
\| \cL_t \oldh \|_0 & \le & C \| \oldh \|_0 \, ,   \label{eq:neutral L}\\
\| \cL_t \oldh \|_u & \le & C t^\gamma \Lambda^{-\gamma t} \| \oldh \|_u + C  \| \oldh \|_0 
+ C  \| \oldh \|_s \, .
\label{eq:strong unstable L} 
\end{eqnarray}
If we assume in addition $\oldq \le 1-\oldbeta$, and if we allow
$C$ to depend on $L_0$, then \eqref{eq:strong stable L} implies
\begin{equation}\label{neater}
\| \cL_t \oldh \|_s  \le  C  \Lambda^{-\oldq t} \| \oldh \|_s 
+ C  |\oldh|_w \, .
\end{equation}
\end{proposition}
Recall that $\|\cdot\|_u$ will appear with a (small) factor $c_u$ to be introduced later
in the norm $\|\cdot\|_\cB$. 
Even with this weighting the above
bounds   (just as in \cite{Li04,BaL}) are not honest Lasota-Yorke bounds
because of \eqref{eq:neutral L}, which is neither a contracted term
nor compact. Integration with respect to
time in the resolvent $\cR(z)$ will yield the true Lasota-Yorke bounds
of Proposition~\ref{prop:LY R}.

It follows from the above proposition that
for any $\oldh\in \cC^0_\sim \cap \cB^0_\sim\subset \cC^0 \cap \cB^0$ and any $t > 0$, the image
 $\cL_t \oldh$, which is defined as an element of $\cC^0_\sim$, still belongs to $\cB^0_\sim$
and satisfies $\|\cL_t \oldh\|_\cB\leq C\|\oldh\|_\cB$. Since $\cC^0_\sim \cap \cB^0_\sim$ is
dense in $\cB^0_\sim$ the operator
$\cL_t$ can be extended to a continuous operator on
$\cB^0_\sim$ for any $t\ge 0$, and indeed Proposition~\ref{prop:L_t} holds for all $\oldh \in \cB^0_\sim$.
Since $\cB\subset \cB^0_\sim$, we  get, noticing that the
set in \eqref{smallspace} is $\cL_t$-invariant for all $t$:

\begin{cor}[Continuity and bounds  on $\cB$]
The operator
$\cL_t$ is  continuous  on
$\cB$ for any $t\ge 0$,  and the  bounds
in Proposition~\ref{prop:L_t} hold for all $\oldh \in \cB$.
\end{cor}

\smallskip
The family of bounded operators $\cL_t$ on $\cB$ satisfy:
\begin{enumerate}
\item[(i)]
$\cL_0$ is the identity on $\cB^0_\sim$ and thus on $\cB$;
\item[(ii)]
For any $0\le t, s < \infty$ we have $\cL_t \circ \cL_s= \cL_{s+t}$.
\end{enumerate}

In Lemma~\ref{lem:strong c} (stated and proved Subsection~\ref{lastt}), we will show that for any $\oldh\in \cB$
$$
\lim_{t\downarrow 0} \cL_t(\oldh)=\oldh \, ,
$$
with convergence in $\cB$. This implies
the third condition required for a one-parameter semi-group of
bounded operators (see \cite[\S 6, p. 152]{Davies}), that
\begin{enumerate}
\item[(iii)]
the map $(t,\oldh)\mapsto
\cL_t(\oldh)$ from $[0,\infty)\times \cB$ to
$\cB$ is jointly continuous.
\end{enumerate}

Note also that $\cL_t$ is bounded on
$\cB_w$ (same argument as above).  
In fact, Lemma~\ref{lem:lip} in Subsection~\ref{lastt} will show that 
$\cL_t$ is Lipschitz when viewed as an operator from $\cB$ to $\cB_w$.

In Subsections~\ref{easy norms} and \ref{unstable norm}, we prove  Proposition~\ref{prop:L_t}.
Subsection~\ref{lastt} is devoted to Lemmas ~\ref{lem:strong c} and ~\ref{lem:lip}.
We shall use several times without mention the key observation (recall \eqref{esscurve}) that
a stable curve $W \in \cW^s$, and more generally any $W_i \in \cG_t(W)$ for
$t>0$,  may intersect $\partial \Omega_0$ in at most two
points.


\subsection{Weak stable, strong stable, and neutral norm estimates for $\cL_t$}
\label{easy norms}

We start with  \eqref{eq:weak L}.
Let $\oldh \in \cC^0 \cap \cB^0$, $W \in \cW^s$ and $\psi \in \cC^\oldp(W)$ such that
$|\psi|_{\cC^\oldp(W)} \leq 1$.	 For $t  >0$, we write,
\begin{equation}
\label{eq:start}
\int_W\cL_t \oldh \, \psi \, dm_W
=\sum_{W_i \in\cG_t(W)}\int_{W_i}\oldh \, J_{W_i}\Phi_t \, \psi \circ \Phi_t dm_W
\le \sum_{W_i \in\cG_t(W)} |\oldh|_w
| J_{W_i}\Phi_t|_{\cC^\oldp(W_i)} |\psi\circ \Phi_t|_{\cC^\oldp(W_i)} \,  ,
\end{equation}
where we have used  the definition of the weak norm on each $W_i$.
Since $\oldp\le 1/3$, the distortion bounds given by Lemma~\ref{lem:distortion} imply that
\begin{equation}
\label{eq:dist holder}
| J_{W_i}\Phi_t|_{\cC^\oldp(W_i)} \leq C_d | J_{W_i}\Phi_t|_{L^\infty(W_i)} \,  .
\end{equation}
By Lemma~\ref{lem:Holder}, we have
$|\psi \circ \Phi_t|_{\cC^\oldp(W_i)} \leq C_1 |\psi|_{\cC^\oldp(W)}$.
Using these estimates in equation \eqref{eq:start}, we obtain
\begin{equation}
\label{eq:weak}
\int_W \cL_t \oldh \, \psi \, dm_W \; \leq \; C |\oldh|_w \sum_{W_i \in\cG_t(W)}
|J_{W_i}\Phi_t|_{L^\infty(W_i)}  \le C  | \oldh |_w \, ,
\end{equation}
where in the last inequality we have used Lemma~\ref{lem:growth}(b) with $\oldbetazero = 0$.
Taking the supremum over all $W \in \cW^s$ and $\psi \in \cC^\oldp(W)$ with
$|\psi|_{\cC^\oldp(W)} \leq 1$
yields \eqref{eq:weak L}.


We next prove \eqref{eq:strong stable L}.
Let $\oldh \in \cC^0 \cap \cB^0$, $t>0$, $W \in \cW^s$, and let $\{W^t_i\}$ denote the elements of $\cG_t(W)$.
For $\psi \in \cC^\oldq(W)$,  with $|W|^{\oldbeta} |\psi|_{\cC^ \oldq(W)} \le 1$, define
$\bp_i = |W^t_i|^{-1} \int_{W^t_i} \psi \circ \Phi_t \, dm_W$.
Using equation \eqref{eq:start}, we write
\begin{equation}
\label{eq:stable split}
\int_W\cL_t \oldh\, \psi \, dm_W	 =
\sum_{i} \int_{W^t_i}\oldh \, J_{W^t_i}\Phi_t \,(\psi \circ \Phi_t - \bp_i) \, dm_W
 + \bp_i  \int_{W^t_i}\oldh \, J_{W^t_i}\Phi_t \, dm_W \,  .
\end{equation}

To bound the first term of \eqref{eq:stable split},
we first estimate $|\psi \circ \Phi_t - \bp_i|_{\cC^\oldq(W^t_i)}$.
Since $\bp_i$ is constant on $W^t_i$, we have
$\cC^\oldq_{W^t_i}(\psi \circ \Phi_t - \bp_i) \leq C_1 \Lambda^{-\oldq t} \cC^\oldq_W(\psi)$
by Lemma~\ref{lem:Holder}.
To estimate the $L^\infty$ norm, note that $\bp_i = \psi \circ \Phi_t(Z_i)$ for some
$Z_i \in W^t_i$.  Thus for each $Z \in W^t_i$,
\[
| \psi \circ \Phi_t(Z) - \bp_i|
    = |\psi \circ \Phi_t(Z) - \psi \circ \Phi_t(Z_i)|
    \leq \cC^\oldq_{W^t_i}(\psi \circ \Phi_t) |W^t_i|^\oldq \leq C \cC^\oldq_W(\psi) \Lambda^{-\oldq t}  \,  .
\]
These estimates together with the fact that $|W|^{\oldbeta}|\psi|_{\cC^\oldq(W)} \leq 1$, imply
\begin{equation}
\label{eq:C^q small}
|\psi \circ \Phi_t - \bp_i|_{\cC^\oldq(W^t_i)} \leq C \Lambda^{-\oldq t} |\psi|_{\cC^\oldq(W)}
\leq C \Lambda^{-\oldq t} |W|^{-\oldbeta} \, .
\end{equation}

We apply \eqref{eq:C^q small}, the distortion estimate \eqref{eq:dist holder}
and the definition of the strong stable norm to the first term on the right-hand side
of \eqref{eq:stable split},
\begin{equation}
\label{eq:first stable}
\begin{split}
\sum_i \int_{W^t_i} \oldh \, J_{W^t_i}\Phi_t  \, (\psi \circ \Phi_t - \bp_i) \, dm_W & \leq
C \sum_i \|\oldh\|_s \frac{|W^t_i|^\oldbeta}{|W|^\oldbeta} 
| J_{W^t_i}\Phi_t |_{L^\infty(W^t_i)} \Lambda^{-\oldq t}\\
& \leq \; C'  \Lambda^{-\oldq t} \|\oldh\|_s \,  ,
\end{split}
\end{equation}
where in the second line we have used 
Lemma~\ref{lem:growth}(b) with $\oldbetazeroparam = \oldbetaparam$.

For the second term on the right-hand side of \eqref{eq:stable split}, we use the fact that
$|\bp_i| \leq |W|^{-\oldbeta} $ since $|W|^{\oldbeta}|\psi|_{\cC^ \oldq(W)} \le 1$.
We group the curves $W^t_i \in \cG_t(W)$
according to most recent long ancestors, as in the proof of Lemma~\ref{lem:growth}.
As before, $t_k = t - k \tau_{\min}$, $k = 0, \ldots, \lfloor t/\tau_{\min} \rfloor$,
and each $W^t_i$ belongs either to $A_t(W^{t_k}_j)$ for some $W^{t_k}_j \in \cG_{t_k}(W)$
or $W^t_i \in \cI_t(W)$. 

Using this grouping, we estimate the second term in \eqref{eq:stable split} by
\[
\begin{split}
\sum_{i}  |W|^{-\oldbeta} \int_{W^t_i}\oldh \, J_{W^t_i}\Phi_t \, dm_W
= & \sum_{k=0}^{\lfloor t/\tau_{\min} \rfloor} \sum_{j\in L_k(W,t)}\sum_{ i\in A_t(W^{t_k}_j)}
     |W|^{-\oldbeta} \int_{W^t_i}\oldh J_{W^t_i}\Phi_t \, dm_W \\
&  + \sum_{ i\in \cI_t(W)}
     |W|^{-\oldbeta} \int_{W^t_i} \oldh J_{W^t_i}\Phi_t \, dm_W \,  .
\end{split}
\]
We estimate the terms in the sum over $k$
using the weak norm and the terms corresponding to $\cI_t(W)$ using the strong stable norm,
\begin{equation}
\begin{split}
\label{eq:weak term split}
\sum_{i}  |W|^{-\oldbeta} \int_{W^t_i} \oldh \, J_{W^t_i}\Phi_t \, dm_W
 \leq & 
C\sum_{k=0}^{\lfloor t/\tau_{\min} \rfloor} \sum_{j\in L_k(W,t)}
\sum_{i\in A_t(W^{t_k}_j)}  |\oldh|_w
\frac{|J_{W^t_i}\Phi_t |_{L^\infty(W^t_i)}}{|W|^{\oldbeta}}\\
&  \; \; +
C \sum_{\ i\in \cI_t(W)} \frac{|W^t_i|^\oldbeta}{|W|^\oldbeta} \|\oldh\|_s | J_{W^t_i}\Phi_t|_{L^\infty(W^t_i)} 
\, .
\end{split}
\end{equation}

In the first sum above corresponding to $k\geq 0$, we have as 
in the proof of Lemma~\ref{lem:growth}
$$
|J_{W^t_i} \Phi_t|_{L^\infty(W^t_i)}
\le |J_{W^t_i} \Phi_{k \tau_{\min}}|_{L^\infty(W^t_i)} |J_{W^{t_k}_j} \Phi_{t_k} |_{L^\infty(W^{t_k}_j)}
\,  .
$$
Thus using Lemma~\ref{lem:growth}(a) for $\oldbetazero=0$ from time $t_k$ to time $t$,
\begin{equation}
\label{eq:first strong}
\begin{split}
\sum_{k=0}^{\lfloor t/\tau_{\min} \rfloor} &\sum_{j \in L_k(W,t)} 
\sum_{i \in A_t(W^{t_k}_j)}  |W|^{-\oldbeta} |J_{W^t_i}\Phi_t|_{L^\infty(W^t_i)} \\
& \leq \sum_{k=0}^{\lfloor t/\tau_{\min} \rfloor} \sum_{j \in L_k(W,t)} |J_{W^{t_k}_j}\Phi_{t_k}|_{L^\infty(W^{t_k}_j)} |W|^{-\oldbeta}
\sum_{i \in A_t(W^{t_k}_j)} |J_{W^t_i}\Phi_{k\tau_{\min}}|_{L^\infty(W^t_i)} \\
& \le C L_0^{-\oldbeta} \sum_{k=0}^{\lfloor t/\tau_{\min} \rfloor} \sum_{j \in L_k(W,t)} 
\frac{|W^{t_k}_j|^\oldbeta}{|W|^\oldbeta}  |J_{W^{t_k}_j}\Phi_{t_k}|_{L^\infty(W^{t_k}_j)} \lambda^{k \tau_{\min}}\, ,
\end{split}
\end{equation}
since $|W^{t_k}_j| \ge L_0/3$.  Since each $L_k(W,t) \subset \cG_{t_k}(W)$ by definition,
the  sum over $j$ is bounded independently of $k$ and $t$ by Lemma~\ref{lem:growth}(b).
The sum over $k$ is also bounded independently of $t$.
Finally, for the sum corresponding to $\cI_t(W)$ in \eqref{eq:weak term split}, we 
use Lemma~\ref{lem:growth}(a) with $\oldbetazeroparam = \oldbetaparam$ to get
\[
\sum_{i \in \cI_t(W)} \frac{|W^t_i|^\oldbeta}{|W|^\oldbeta} |J_{W^t_i}\Phi_t|_{L^\infty(W^t_i)}
\le C \lambda^{t(1-\oldbeta)} \, .
\]

Putting this estimate together with \eqref{eq:first strong} in \eqref{eq:weak term split}, we obtain
\begin{equation}
\label{eq:second stable}
\sum_{i}  |W|^{-\oldbeta}\left| \int_{W^t_i}\oldh \, J_{W^t_i}\Phi_t \, dm_W\right|
\; \leq \; C L_0^{-\oldbeta} |\oldh|_w + C\|\oldh\|_s\lambda^{t(1-\oldbeta)}\,  .
\end{equation}

Finally, combining \eqref{eq:first stable} and \eqref{eq:second stable} with \eqref{eq:stable split}
yields
\begin{equation}
\label{eq:strong stable}
\int_W \cL_t \oldh\psi dm_W \leq C\left(\Lambda^{-\oldq t}+\lambda^{t(1-\oldbeta)}\right)\|\oldh\|_s
         + C L_0^{-\oldbeta}  |\oldh|_w \,  .
\end{equation}
Taking the supremum over $W \in \cW^s$ and $\psi \in \cC^\oldq(W)$ with $|W|^{\oldbeta}|\psi|_{\cC^ \oldq(W)}\le 1$
proves \eqref{eq:strong stable L}.


We prove \eqref{eq:neutral L}. Let $\oldh$ and $t>0$ be as before.
Fix $W \in \cW^s$ and $\psi \in \cC^\oldp(W)$ with $|\psi|_{\cC^\oldp(W)} \le 1$.  Then for $\oldh \in \cC^0(\Omega_0) \cap \cB^0$,
\[
\int_W \partial_s \left( ( \cL_t \oldh ) \circ \Phi_s \right)|_{s=0} \, \psi \, dm_W
= \int_W ( \partial_s \left( \oldh \circ \Phi_s \right)|_{s=0} \circ \Phi_{-t}) \, \psi \, dm_W \, ,
\]
where we  used the group property  \eqref{eq:flow derivative}.
Now, 
\[
\begin{split}
\int_W  (\partial_s \left( \oldh \circ \Phi_s \right)|_{s=0} \circ \Phi_{-t} )\, \psi \, dm_W
&  =
\sum_{W_i \in \cG_t(W)} \int_{W_i} (\partial_s \left( \oldh \circ \Phi_s \right)|_{s=0}) \cdot
J_{W_i}\Phi_t \cdot  \psi \circ \Phi_t \, dm_W  \\
& \le \sum_i \| \oldh \|_0 |J_{W_i} \Phi_t|_{\cC^\oldp(W_i)} |\psi\circ \Phi_t |_{\cC^\oldp(W_i)}
\le C \| \oldh \|_0  \, ,
\end{split}
\]
where we estimated the H\"older norms of the functions as for~\eqref{eq:weak L}
and used Lemma~\ref{lem:growth}(b) to bound the sum.  Taking the appropriate
suprema over $W$ and $\psi$ proves \eqref{eq:neutral L}.


\subsection{Unstable norm estimate for $\cL_t$}
\label{unstable norm}

Fix $\oldh$, $t$, and $W$ as in the previous subsection, let $\ve \le 1$, and
consider two curves $W^1, W^2 \in\cW^s$ with $d_{\cW^s}(W^1,W^2) \leq \ve$.
For $t > 0$, we will use the collision map to partition $\Phi_{-t}(W^\ell)$ into
``matched'' pieces $U^\ell_j$ and
``unmatched'' pieces $V^\ell_i$, $\ell=1,2$.
Define $\tW^\ell_i = P^+(W^\ell_i)$ for each $W^\ell_i \in \cG_t(W^\ell)$ and 
$\tW^\ell = P^+(W^\ell)$.
Note that $\tW^\ell_i \in \tG_n(\tW^\ell)$ 
(defined in the beginning of the
proof of Lemma~\ref{lem:growth}) for some $t/\tau_{\max}\le n
\le t/\tau_{\min}$.  We recall the matching of stable
curves for the map used in \cite[\S 4.3]{demers zhang}.  Note that
$\td_{\cW^s}(\tW^1, \tW^2) \le \ve$ by definition of $d_{\cW^s}$.

Recall from Definition~\ref{singul}, $\cS_0 = \{ (r,\vf) \in \cM : \vf = \pm \pi/2 \}$ and 
$\cS_{-n}^{\bH} = \cup_{i=0}^n T^i(\cS_0 \cup \cup_{k \ge k_0} \partial \bH_k)$
denotes the extended singularity set for $T^{-n}$, where we include the boundaries of the
homogeneity strips to ensure applicability of Lemma~\ref{lem:distortion}.
Let $\omega$ be a connected component of $\tW^1 \setminus \cS_{-n}^{\bH}$.
To each
point $z = (r,\vf) \in T^{-n}\omega$, we associate a vertical line segment $\gamma_z$ of length
at most $C\Lambda_0^{-n}\ve$ such that its image $T^n\gamma_z$,
if not cut by a singularity or the boundary of a homogeneity strip,
will have length $C\ve$.  By \cite[\S 4.4]{chernov book}, all the tangent vectors to $T^i\gamma_z$
lie in the unstable cone
$C^u(T^iz)$ for each $i \ge 1$ so that they remain uniformly transverse to the stable cone
and enjoy minimum expansion given by the factor $C\Lambda_0^i$.

Doing this for each connected component of $\tW^1 \setminus \cS_{-n}^{\bH}$,
we subdivide $\tW^1 \setminus \cS_{-n}^{\bH}$ into a countable
collection of subintervals of points for which $T^n\gamma_z$ intersects
$\tW^2 \setminus \cS_{-n}^{\bH}$ and subintervals for which this is not the case.
This 
induces a matching partition on $\tW^2 \setminus \cS_{-n}^{\bH}$.

We denote by $\tV^\ell_i$ the pieces in $T^{-n}\tW^\ell$ which are not matched up by this process
and note that
the images $T^n\tV^\ell_i$ occur either at the endpoints of $\tW^\ell$ or because the vertical segment
$\gamma_z$ has been cut by a singularity.  In both cases, the length of the
curves $T^n\tV^\ell_i$ can be at most $C\ve$ due to the uniform transversality of
$\cS_{-n}^{\bH}$ 
(see \cite[Prop. 4.41]{chernov book}) with the map-stable cone $C^s_z$ and of $C^s_z$ with $C^u_z$.

In the remaining
pieces the foliation $\{ T^n\gamma_z \}_{z \in T^{-n}\tW^1}$ provides a one-to-one correspondence
between points in $\tW^1$ and $\tW^2$.
We further subdivide these pieces in
such a way that the lengths of their images under $T^{-i}$
are less than $L_0$ for each $0 \le i \le n$ and
the pieces are pairwise matched by
the foliation $\{\gamma_z\}$. We call these matched pieces $\tU^\ell_j$.
Possibly changing the constant $L_0/2$ to $L_0/C$ for some uniform constant
$C>1$ (depending only on the distortion constant and the angle between stable and unstable cones) in the definition of $\tG_n(\tW^\ell)$, we may arrange it so that
$\tU^\ell_j \subset \tW^{\ell,n}_i$ for some $\tW^{\ell,n}_i \in \tG_n(\tW^\ell)$ and
$\tV^\ell_k \subset \tW^{\ell,n}_{i'}$ for some $\tW^{\ell,n}_{i'} \in \tG_n(\tW^\ell)$
for all $j, k \ge 1$ and $\ell =1,2$.  There is at most one $\tU^\ell_j$ and two $\tV^\ell_j$
per $\tW^{\ell, n}_i \in \tG_n(\tW^\ell)$.

In this way, we write $\tW^\ell = (\cup_j T^n\tU^\ell_j) \cup (\cup_i T^n\tV^\ell_i)$.
The images $T^n\tV^\ell_i$ of the unmatched pieces must be short
while the images of the matched pieces
$\tU^\ell_j$
may be long or short.

Returning to the components of $\Phi_{-t}(W^\ell)$ for the flow, $\cG_t(W^\ell)$, note that
if $W^\ell_i$ is not part of a curve in $\hG_t(W^\ell)$ undergoing a collision at time $-t$,
then $P^+(W^\ell_i)$ is a union of at-most  one matched and two unmatched curves
$\tU^\ell_j$ and $\tV^\ell_k$ as described above.  We may thus define
$U^\ell_j$ as the subset of $W^\ell_i$ such that $P^+(U^\ell_j) = \tU^\ell_j$ and
$V^\ell_k$ as a connected subset of $W^\ell_i$ such that $P^+(V^\ell_k) = \tV^\ell_k$.
If $W^\ell_i$ is part of a curve in $\hG_t(W^\ell)$ undergoing a collision at time $-t$, then
$P^+(W^\ell_i)$ may correspond to only part of a curve $\tU^\ell_j$ and 
at most two $\tV^\ell_k$ as  described 
above.  However, we will 
still consider those pieces
as matched or unmatched as defined.

With the above matching, $d_{\cW^s}(U^1_j, U^2_j)$ may be infinite since
there may not be an unstable curve connecting the two.  However, there is an unstable
curve $\tW^u_j$ for the map connecting $\tU^1_j$ and $\tU^2_j$.  Note that 
$\tW^u_j$ is the trace of an unstable curve for the flow $W^u_j$ which has nonempty intersection with
$U^1_j$.  Although $W^u_j$ may not intersect $U^2_j$, the weak unstable manifold containing
$W^u_j$ does intersect $U^2_j$, since $P^+(W^u_j) \cap P^+(U^2_j)$ intersect.
Thus there is a time $s_j$ so that 
\begin{equation}
\label{sj}
\Phi_{s_j}(U^2_j) \cap W^u_j \neq \emptyset \, .
\end{equation}  

For $\psi_\ell$  on $W^\ell$ 
with $|\psi_\ell|_{C^\oldp(W^\ell)} \leq 1$ and
$d(\psi_1, \psi_2) = 0$,
with the above construction 
we must bound
\begin{equation}
\label{eq:unstable split}
\begin{split}
\left| \int_{W^1} \cL_th \, \psi_1 \, dm_W \right. & \left. - \int_{W^2} \cL_th \, \psi_2 \, dm_W \right|
  \; \leq \; \sum_{\ell,k} \left|\int_{V^\ell_k} \oldh \, J_{V^\ell_k}\Phi_t \, \psi_\ell\circ \Phi_t \, dm_W \right|\\
  & + \sum_j \left| \int_{U^1_j} \oldh \, J_{U^1_j}\Phi_t \, \psi_1\circ \Phi_t \, dm_W
    - \int_{U^2_j} \oldh \, J_{U^2_j}\Phi_t \, \psi_2\circ \Phi_t \, dm_W \right| \, .
\end{split}
\end{equation}

We do the estimate over the unmatched pieces $V^\ell_k$ first
using the strong stable norm.  To do this, we group the $V^\ell_k$ according to when 
the associated vertical segments were most recently cut as follows (notice that we may just
as well define the vertical segments $\gamma_z$ on $T^{-n}\tW^2$ as on $T^{-n}\tW^1$).
For $z \in T^{-n}\tW^\ell$, let $\gamma^j_z$ denote the component of 
$T^j\gamma_z$ still connected to $T^{j-n}\tW^\ell$.
Define 
\[
\begin{split}
A^\ell(i) = \{ k :  (\cup_{z \in \tV^\ell_k} & \gamma^{n-i-1}_z) \cap (\cS_1^{\bH} \setminus \cS_0^{\bH}) \neq \emptyset \mbox{ and $i \in [0, n-1]$ is  minimal for this property } \} \, .
\end{split}
\]
Let $S$ be the singularity curve which intersects $(\cup_{z \in \tV^\ell_k} \gamma^{n-i-1}_z)$
for some $k \in A^\ell(i)$.  
Then $TS$ is an element of $\cS_{-1}^{\bH}$ and since
the curves $\gamma^{n-i}_z$ have length at most $C \Lambda_0^{-i} \ve$, by the uniform
transversality of stable curves with curves in $\cS_{-1}^{\bH}$, we have
$|T^{n-i} \tV^\ell_k| \le C \ve \Lambda_0^{-i}$.

Next we subdivide the interval $[0,t]$ into times $t_j = t - j \tau_{\min}/2$, 
$j = 0, \ldots, 2\lfloor \frac{t}{\tau_{\min}} \rfloor$.  For each $k \in A^\ell(i)$, 
there exists at least
one $j_k$ such that $\Phi_{j_k \tau_{\min}/2}(V^\ell_k)$ is not undergoing a collision 
(due to the choice $L_0 \le \tau_{\min}/4$ from Lemma~\ref{lem:growth}), and 
$P^+(\Phi_{j_k \tau_{\min}/2}(V^\ell_k)) = T^{n-i}\tV^\ell_k$.  It follows from
Lemma~\ref{lem:smooth} that $|\Phi_{j_k \tau_{\min}/2}(V^\ell_k)| \le C \ve \Lambda_0^{-i}$.

In addition, the $i$th collision must occur at a time $s \in [ \lfloor i \tau_{\min} \rfloor, \lfloor i \tau_{\max} \rfloor + 1]$ so that the number of $t_j$ corresponding to a fixed $i$ is 
bounded by a uniform constant times $i$.  Now for $k \in A^\ell(i)$,
\begin{equation}
\label{eq:one term}
\begin{split}
\int_{V^\ell_k} \oldh \, J_{V^\ell_k}\Phi_t \, \psi_\ell \circ \Phi_t \, dm_W
& = \int_{\Phi_{j_k\tau_{\min}/2}(V^\ell_k)} \!\!\!\!\!\!\cL_{j_k \tau_{\min}/2} \oldh \, J_{\Phi_{j_k \tau_{\min}/2}(V^\ell_k)}\Phi_{t_{j_k}} \, \psi_\ell \circ \Phi_{t_{j_k}} \, dm_W \\
& \le C \| \cL_{j_k \tau_{\min}/2} \oldh \|_s |\Phi_{j_k \tau_{\min}/2}(V^\ell_k)|^\oldbeta
|J_{\Phi_{j_k \tau_{\min}/2}(V^\ell_k)}\Phi_{t_{j_k}}|_{\cC^\oldq} \\
& \le C \| \oldh \|_s \ve^\oldbeta \Lambda_0^{-i/q} |J_{\Phi_{j_k \tau_{\min}/2}(V^\ell_k)}\Phi_{t_{j_k}}|_{L^\infty(\Phi_{j_k \tau_{\min}/2}(V^\ell_k))} \,   ,
\end{split}
\end{equation}
where we have used \eqref{eq:strong stable L} to bound $\| \cL_{j_k \tau_{\min}/2} \oldh \|_s$,
 Lemma~\ref{lem:Holder} to estimate $|\psi_\ell \circ \Phi_{t_{j_k}}|_{\cC^\oldq}
\leq C_1 |\psi_\ell|_{\cC^\oldp(W^\ell)}$, and \eqref{eq:dist holder}.  

In order to sum over the relevant $\ell$ and $k$, we fix $i$, define $J(i)$ to be the set
of $j$ possible for the $i$th collision and $K(j)$ to be the set of $k$ for which $j_k = j$.
Thus using \eqref{eq:one term}, we estimate
\begin{equation}
\label{eq:first unstable}
\begin{split}
\sum_{\ell,k} & \left|\int_{V^\ell_k}\oldh \, J_{V^\ell_k}\Phi_t \, \psi_\ell\circ \Phi_t \, dm_W \right| \\
&\quad \leq
C \| \oldh \|_s \ve^\oldbeta \sum_{\ell,i} \sum_{j \in J(i)} \sum_{k \in K(j)} \Lambda_0^{-i/q} |J_{\Phi_{j \tau_{min}/2}(V^\ell_k)}\Phi_{t_j}|_{L^\infty(\Phi_{j \tau_{min}/2}(V^\ell_k))} \\
&\quad \leq C \| \oldh \|_s \ve^\oldbeta \sum_{\ell,i} \sum_{j \in J(i)} \Lambda_0^{-i/q}
\leq C  \|\oldh\|_s \ve^\oldbeta	 \sum_{\ell,i} i \Lambda_0^{-i /q} 
\le C \ve^\oldbeta \|\oldh\|_s \, ,
\end{split}
\end{equation}
where for the sum over $k$, we have noted that 
$\Phi_{j \tau_{\min}/2}(V^\ell_k) \subset \Phi_{-t_j}(W^\ell)$, and we have applied 
Lemma~\ref{lem:growth}(b) with $\oldbetazero = 0$,
since there are at most two such curves corresponding to
each element of $\cG_{t_j}(W^\ell)$; for the sum over $j$, we have used the fact that 
$\#J(i) \le Ci$ for some $C>0$ depending on $\tau_{\max}/\tau_{\min}$.

Next, we must estimate the second term on the right-hand side of \eqref{eq:unstable split}.
Recall that $\tU^\ell_j = P^+(U^\ell_j)$ can be represented as the graph 
$G^\ell_j(r) = (r, \vf_{\tU^\ell_j}(r))$ of a function
over some $r$-interval $I_j$. By the definition of matched pieces,
$\tU^1_j$ and $\tU^2_j$ are defined over the same interval $I_j$.
 Let $S^1_j(r) = \Phi_{-t(r)} \circ G^1_j(r)$ denote the
(invertible) map from $I_j$ to $U^1_j$ and $S^2_j = \Phi_{-t(r)} \circ G^\ell_j(r)$ 
denote the map from
$I_j$ to $\Phi_{s_j}(U^2_j)$.  
For each $j$, define
\[
\phi_j = [ J_{U^1_j}\Phi_t \cdot \psi_1 \circ \Phi_t ] \circ S^1_j \circ (S^2_j)^{-1}  \,  .
\]
The function $\phi_j$ is well-defined on $U^2_j$, and recalling $s_j$ from \eqref{sj}, we estimate,
\begin{equation}
\label{eq:stepone}
\begin{split}
\left|\int_{U^1_j}\oldh J_{U^1_j}\Phi_t \, \psi_1\circ \Phi_t - \right. & \left.
\int_{U^2_j}\oldh J_{U^2_j}\Phi_t \, \psi_2\circ \Phi_t \right|
\leq \left|\int_{U^1_j}\oldh J_{U^1_j}\Phi_t \, \psi_1\circ \Phi_t -
 \int_{\Phi_{s_j}(U^2_j)} \oldh \,\phi_j \,  \right| \\
& + \left|  \int_{\Phi_{s_j}(U^2_j)} \oldh \,(\phi_j - J_{\Phi_{s_j}(U^2_j)}\Phi_{t-s_j} 
\, \psi_2 \circ \Phi_{t-s_j}) \right|  \\
& + \left| \int_{\Phi_{s_j}(U^2_j)}\oldh J_{\Phi_{s_j}U^2_j}\Phi_{t-s_j} \, \psi_2\circ \Phi_{t-s_j} 
- \int_{U^2_j}\oldh J_{U^2_j}\Phi_t  \, \psi_2\circ \Phi_t \right| \, .
\end{split}
\end{equation}

We estimate the first term in equation~\eqref{eq:stepone} using the unstable norm.
Lemma~\ref{lem:Holder} and the distortion bounds given by \eqref{eq:dist holder} imply that
\begin{equation}
\label{eq:c1-unst 1}
|J_{U^1_j}\Phi_t \cdot \psi_1 \circ
\Phi_t|_{\cC^\alpha(U^1_j)} 
\le C |J_{U^1_j}\Phi_t|_{L^\infty(U^1_j)} \, .
\end{equation}
Similarly, since $|S^1_j \circ (S^2_j)^{-1}|_{\cC^1} \le C$ by Lemma~\ref{lem:smooth},
we have
\[
|\phi_j|_{\cC^\oldp(U^2_j)}
\le |J_{U^1_j}\Phi_t \cdot \psi_1 \circ
\Phi_t |_{\cC^\oldp(U^1_j)} |S^1_j \circ (S^2_j)^{-1}|_{\cC^1}
\le C |J_{U^1_j}T^n|_{L^\infty(U^1_j)} \, .
\]
By the definition of $\phi_j$ and $d(\cdot, \cdot)$,
\[
d(J_{U^1_j}\Phi_t \psi_1\circ \Phi_t, \phi_j )
= \left| \left[ J_{U^1_j}\Phi_t \psi_1\circ \Phi_t \right] \circ S^1_j
  -  \phi_j  \circ S^2_j  \right|_{\cC^0(I_j)} \; = \; 0 \, .
\]
To complete the estimate on the first term of \eqref{eq:stepone},
we need the following lemma:

\begin{lemma}[Lemma 4.2 of \cite{demers zhang}]
\label{lem:graph contract}
There exists $C>0$, independent of $n$, $W^1$, and $W^2$, such that for each $j$,
\[
\td_{\cW^s}(\tU^1_j,\tU^2_j)\leq C\Lambda_0^{-n}n \ve =: \ve_1 \, .
\]
\end{lemma}
Since $P^+(\Phi_{s_j}(U^2_j)) = \tU^2_j$, it follows from the definition of
$d_{\cW^s}$ that $d_{\cW^s}(U^1_j, \Phi_{s_j}(U^2_j)) \le \ve_1$ as well. 
In view of \eqref{eq:c1-unst 1}, we renormalize the test functions by
$R_j = C|J_{U^1_j}\Phi_t|_{\cC^0(U^1_j)}$.
Then we apply the definition of the unstable norm with
$\ve_1$ in place of $\ve$.	Thus,
\begin{equation}
\label{eq:second unstable}
\begin{split}
&\sum_j \left|\int_{U^1_j}\oldh  J_{U^1_j}\Phi_t \, \psi_1\circ \Phi_t -
\int_{\Phi_{s_j}(U^2_j)} \oldh  \, \phi_j \, \right|	 \\
& \leq C \ve_1^\oldbeta \|\oldh\|_u \sum_j |J_{U^1_j}\Phi_t|_{L^\infty(U^1_j)}
   \leq C \|\oldh\|_u t^\gamma \Lambda^{-t\gamma} \ve^\gamma \, ,
   \end{split}
\end{equation}
using the fact that 
$\lfloor \frac{t}{\tau_{\max}} \rfloor \le n \le \lfloor \frac{t}{\tau_{\min}}\rfloor + 1$ and
the sum is bounded by Lemma~\ref{lem:growth}(b)
for $\oldbetazero=0$ since there is at most
one matched piece $U^1_j$ corresponding to each curve
$W^1_i \in \cG_t(W^1)$.

Next we estimate the second term in \eqref{eq:stepone} using the
strong stable norm.
\begin{equation}
\label{eq:unstable strong}
\begin{split}
& \left|  \int_{\Phi_{s_j}(U^2_j)} \oldh \,(\phi_j - J_{\Phi_{s_j}(U^2_j)}\Phi_{t-s_j} 
\, \psi_2 \circ \Phi_{t-s_j} ) \right| \\
& \leq  \|\oldh\|_s |\Phi_{s_j}(U^2_j)|^\oldbeta 
       \left|\phi_j - J_{\Phi_{s_j}(U^2_j)}\Phi_{t-s_j} \cdot \psi_2 \circ \Phi_{t-s_j} \right|_{\cC^\oldq(\Phi_{s_j}(U^2_j))} \,  .
\end{split}
\end{equation}
Using the fact that $|S^1_j|_{\cC^1} |(S^2_j)^{-1}|_{\cC^1} \le C$ and 
$|\psi_1 \circ \Phi_t|_{\cC^\oldq(U^1_j)} \le C |\psi|_{\cC^\oldq(W^1)} \le C$, 
we split the estimate on the norm of the test function,
\begin{equation}
\label{eq:split test}
\begin{split}
\Big| \phi_j - J_{\Phi_{s_j}(U^2_j)}\Phi_{t-s_j} & \cdot \psi_2 \circ \Phi_{t-s_j} \Big|_{\cC^\oldq(\Phi_{s_j}(U^2_j))}
\le C \left| J_{U^1_j}\Phi_t \circ S^1_j - J_{\Phi_{s_j}(U^2_j)}\Phi_{t-s_j} \circ S^2_j \right|_{\cC^\oldq(I_j)}
 \\
& + C \left| J_{\Phi_{s_j}(U^2_j)}\Phi_{t-s_j} \right|_{\cC^\oldq(\Phi_{s_j}(U^2_j))}
\left| \psi_1 \circ \Phi_t \circ S^1_j - \psi_2 \circ \Phi_{t-s_j} \circ S^2_j \right|_{\cC^\oldq(I_j)}\,  .
\end{split}
\end{equation}

We will use the following lemma (which is the analogue of
\cite[Lemmas 4.3 and 4.4]{demers zhang}) to estimate the two differences above.

\begin{lemma}
\label{lem:dist match}
There exists $C>0$ such that for each $j \ge 1$, 
\begin{itemize}
\item[a)] $\displaystyle
|J_{U^1_j}\Phi_t\circ S^1_j
-J_{\Phi_{s_j}(U^2_j)}\Phi_{t-s_j} \circ S^2_j|_{\cC^\oldq(I_j)}
\leq C | J_{U^2_j}\Phi_t|_{C^0(U^2_j)} \ve^{1/3-\oldq}$ ; 
\item[b)] $\displaystyle
\left| \psi_1 \circ \Phi_t \circ S^1_j - \psi_2 \circ \Phi_{t-s_j} \circ S^2_j \right|_{\cC^\oldq(I_j)}
\le C\ve^{\oldp-\oldq}$ .
\end{itemize}
\end{lemma}
\begin{proof}
(a) Let $G_{\tW^\ell}$ denote the graph representing $\tW^\ell = P^+(W^\ell)$ and 
$S^\ell$ denote the map from $I_{\tW^\ell}$ to $W^\ell$, $\ell=1,2$.  Then
$\Phi_t \circ S^1_j = S^1 \circ G_{\tW^1}^{-1} \circ T^n \circ G_{\tU^1_j}$, so that
\begin{equation}
\label{eq:jac expansion}
J\Phi_t \circ S^1_j = J(S^1 \circ G_{\tW^1}^{-1}) \circ (T^n \circ G_{\tU^1_j})
\cdot J_{\tU^1_j}T^n \circ G_{\tU^1_j} \cdot JG_{\tU^1_j} \cdot (JS^1_j)^{-1}\,  ,
\end{equation}
and a similar expression holds for $J_{\Phi_{s_j}(U^2_j)}\Phi_{t-s_j}$.
All Jacobians except for $J_{\tU^1_j}T^n$ are uniformly bounded by
Lemma~\ref{lem:smooth} (note that $S^1 \circ G_{\tW^1}^{-1}$ is the natural
flow map from $P^+(W^1)$ to $W^1$).  

Fixing $r \in I_j$, by \cite[eq. (4.16)]{demers zhang}, we have
\[
|J_{\tU^1_j}T^n \circ G_{\tU^1_j}(r) - J_{\tU^2_j}T^n \circ G_{\tU^2_j}(r)|
\le C |J_{\tU^2_j}T^n|_{L^\infty(\tU^2_j)} \ve^{1/3} \, .
\]
Also, $|JG_{\tU^1_j}(r) - JG_{\tU^2_j}(r)| \le C \ve$
using \eqref{eq:dG_W} and the fact that $\td_{\cW^s}(\tU^1_j, \tU^2_j) \le \ve$.
Now,
\[
|(JS^1_j)^{-1}(r) - (JS^2_j)^{-1}(r)| = \frac{1}{JS^1_j(r)JS^2_j(r)}|JS^1_j(r) - JS^2_j(r)| \, ,
\]
and since by Lemma~\ref{lem:smooth}, $JS^\ell_j$ is bounded away from 0, we can focus on
the difference.  Now in the notation of Lemma~\ref{lem:smooth},
$JS^\ell_j(r) = J_{\tU^\ell_j}\Phi_{-t(r)}(G_{\tU^\ell_j}(r)) JG_{\tU^\ell_j}(r)$, and as already noted,
the difference of $JG_{\tU^\ell_j}$ is bounded by $\ve$.  Thus we focus on the difference
involving $J_{\tU^\ell_j}\Phi_{-t(r)}(G_{\tU^\ell_j}(r))$.   Letting $z_\ell = G_{\tU^\ell_j}(r) = (r_\ell, \vf_\ell)$
and $(dr_\ell, d\vf_\ell)$ denote the tangent vector to $\tU^\ell_j$ at $z_\ell$, we have by 
\eqref{eq:log J}
\[
\ln \frac{J_{\tU^1_j}\Phi_{-\tau(Z_1)}(z_1)}{J_{\tU^2_j}\Phi_{-\tau(Z_2)}(z_2)}
= \frac 12 \ln  \frac{(dr_1)^2 [(\cos \vf_1 + \tau(Z_1)\cK(r_1) - \tau(Z_1)\frac{d\vf_1}{dr_1})^2
+ (\cK(r_1) - \frac{d\vf_1}{dr_1})^2]}
{(dr_2)^2 [(\cos \vf_2 + \tau(Z_2)\cK(r_2) - \tau(Z_2)\frac{d\vf_2}{dr_2})^2
+ (\cK(r_2) - \frac{d\vf_2}{dr_2})^2]} \,  ,
\]
where $z_\ell = P^+(Z_\ell)$, $Z_1 \in U^1_j$, $Z_2 \in \Phi_{s_j}(Z_2)$.  Once again, we
pair corresponding terms and estimate the differences.  We have $d(z_1,z_2) \le \ve$
and $|\frac{d\vf_1}{dr_1} - \frac{d\vf_2}{dr_2}| \le C\ve$ since $\td_{\cW^s}(\tU^1_j, \tU^2_j) \le \ve$.
This leaves only $|\tau(Z_1) - \tau(Z_2)|$ to estimate.  Since by definition of 
$d_{\cW^s}$, there is an unstable curve connecting $U^1_j$ and $\Phi_{s_j}(U^2_j)$,
we have $d(Z_1,Z_2) \le C\ve$ and so $|\tau(Z_1) - \tau(Z_2)| \le C\ve^{1/2}$.

Finally, note that
$J(S^1 \circ G_{\tW^1}^{-1}) = J_{\tW^1}\Phi_{-\tau}$ and so satisfies the same estimate
as above.  Since $d(T^n \circ G_{\tU^1_j}(r), T^n \circ G_{\tU^2_j}(r)) \le C\ve$ by the
uniform transversality of the foliation $\{ T^n \gamma_z \}$, the triangle inequality
together with Lemma~\ref{lem:smooth} yields that this difference is also bounded
by $C\ve^{1/2}$.  

These estimates together with Lemma~\ref{lem:expansion} yield by \eqref{eq:jac expansion},
\[
|J_{U^1_j}\Phi_t \circ S^1_j(r) - J_{\Phi_{s_j}(U^2_j)}\Phi_{t-s_j} \circ S^2_j(r)|
\le C|J_{U^2_j}\Phi_t|_{L^\infty(U^2_j)} \ve^{1/3} \,  .
\]
Now since both Jacobians satisfy bounded distortion along stable curves
with exponent 1/3 by Lemma~\ref{lem:distortion},
we use the H\"older interpolation from
\cite[Lemma 4.3]{demers zhang 2} to conclude the proof of part (a) the lemma.

\smallskip
\noindent
(b)  Fix $r \in I_j$ and set
$z_\ell = v^\ell_j(r)$, $Z_\ell = S^\ell_j(r)$, $\ell = 1,2$.  Also, let 
$Y_1 = \Phi_t(Z_1)$ and $Y_2 = \Phi_{t-s_j}(Z_2)$ denote the images of $Z_1$ and 
$Z_2$ in $W^1$ and $W^2$, respectively.  Now,
\[
|\psi_1 \circ \Phi_t \circ S^1_j(r) - \psi_2 \circ \Phi_{t-s_j} \circ S^2_j(r)|
= |\psi_1(Y_1) - \psi_2(Y_2)|  \, . 
\]
Let $r'$ denote the arclength coordinate of $w_2 = P^+(Y_2)$ and set $w_1 = P^+(Y_1)$.  
Note that $w_1$ does not have the same arclength coordinate as $w_2$ since 
the billiard map does not preserve vertical lines.
Using the
same notation as in part (a), let $w_1' = G_{\widetilde{W}^1}(r')$ and $Y_1' = S^1(r')$
denote the lifts of $r'$ to $P^+(W^1)$ and $W^1$, respectively.  
Since $w_1'$ and $w_2$ lie on the same vertical segment and 
$\tilde d_{\cW^s}(P^+(W^1), P^+(W^2)) \le \ve$, we have $|w_1' - w_2 | \le \ve$.
Also, since $z_1$ and $z_2$
line on the same vertical line segment, the segment connecting $w_1$ and $w_2$ lies in
the map unstable cone and so is uniformly transverse to the stable cone.  Thus
$|w_1 - w_2| \le C \ve$ and so by the triangle inequality, $|w_1 - w_1'| \le C \ve$.  It then follows
from Lemma~\ref{lem:smooth} that $d(Y_1, Y_1') \le C\ve$.  Putting these estimates together,
and using the fact that $|\psi_1 \circ S^1 - \psi_2 \circ S^2|_{\infty} = 0$, we estimate,
\[
\begin{split}
|\psi_1 \circ \Phi_t \circ S^1_j(r) - \psi_2 \circ \Phi_{t-s_j} \circ S^2_j(r)|
& \le |\psi_1(Y_1) - \psi_1(Y_1')| + |\psi_1(Y_1') - \psi_2(Y_2)| \\
& \le C d(Y_1, Y_1')^\oldp \le C \ve^\oldp  \, .
\end{split}
\]
Finally, since $\psi_1$ and $\psi_2$ are H\"older continuous with exponent $\oldp$,
we again use the H\"older interpolation from \cite[Lemma 4.3]{demers zhang 2},
to conclude part (b) of the lemma.
\end{proof}

With Lemma~\ref{lem:dist match} proved, \eqref{eq:unstable strong}, and \eqref{eq:split test}
complete the estimate
on the second term in \eqref{eq:stepone}. Indeed, we have
\begin{equation}
\label{eq:unstable three}
\begin{split}
\sum_j  \left|  \int_{\Phi_{s_j}(U^2_j)} \oldh \,(\phi_j - J_{\Phi_{s_j}(U^2_j)}\Phi_{t-s_j}) 
\, \psi_2 \circ \Phi_{t-s_j} \right| 
& \leq  \sum_j C \|\oldh\|_s |\Phi_{s_j}(U^2_j)|^\oldbeta 
|J_{U^2_j}\Phi_t|_{L^\infty(U^2_j)} \, \ve^{\alpha -\oldq} \\
& \le C \| \oldh\|_s  \ve^{\oldp-\oldq}\, ,
\end{split}
\end{equation}
where again the sum is bounded by Lemma~\ref{lem:growth}(b)
(using that $|\Phi_{s_j}(U^2_j)|$ is comparable to $|U^2_j|$).

It remains to estimate the third term in \eqref{eq:stepone}, which we do using the  neutral norm. For this, we first state and prove a needed bound on the times $s_j$ from
 \eqref{sj},
\begin{lemma}
\label{lem:close}
There exists $C>0$ such that
 $s_j \le C\ve^{1/2}$ for each $j$.
\end{lemma}
\begin{proof}
We first show that the collision times of $U^1_j$ and $U^2_j$ remain close throughout 
their orbits until time $t$.  Let $n$ denote the number of collisions from time 0 to time $t$.

Let $Z \in \Phi_t(U^1_j)$, $Y \in \Phi_t(U^2_j)$ be two points such that $P^+(Z)$ and
$P^+(Y)$ are connected by one of the
curves $T^n(\gamma_z)$ defined during the matching process at the beginning of this 
section.  Since $d_{\cW^s}(W^1,W^2) \le \ve$, we have $d(P^+(Z), P^+(Y)) \le C\ve$ and
$d(Z,Y) \le C\ve$, for some uniform constant $C>0$. Now
\[
\begin{split}
|\tau^{-n}(P^+(Z)) - \tau^{-n}(P^+(Y))|
& \le \sum_{i = 0}^{n-1} |\tau^-(T^{-i}P^+(Z)) - \tau^-(T^{-i}P^+(Y))| \\
& \le \sum_{i=0}^{n-1} C \Lambda_0^{-i/2} d(P^+(Z), P^+(Y))^{1/2} \le C' \ve^{1/2} \, ,
\end{split}
\]
due to the $1/2$-H\"older continuity of $\tau^-$
(analogous to \eqref{ceiling}) and the uniform expansion along unstable curves.
Since $|\tau(Z)- \tau(Y)| \le C d(Z,Y)^{1/2} \le C \ve^{1/2}$ and since
$t = \tau(\Phi_{-t}(Z)) + \tau^{-n}(P^+(Z)) - \tau(Z)$ with an analogous expression
for $t$ in terms of $Y$, we conclude that
$|\tau(\Phi_{-t}Z) - \tau(\Phi_{-t}Y)| \le C \ve^{1/2}$.  Thus it must be that
$d(\Phi_{-t}Z, \Phi_{-t}Y) \le C \ve^{1/2}$.  

Since $\Phi_{s_j}(U^2_j)$ and $U^1_j$ are connected by an unstable curve of length
$C\ve$, if we choose $Y$ in the above analysis to be the point in $U^2_j$ whose
translate by $s_j$ belongs to this unstable curve, then the triangle inequality yields
$s_j \le C\ve^{1/2}$, proving Lemma~\ref{lem:close}.
\end{proof}

We may now estimate the second term in \eqref{eq:stepone}:
\begin{equation}
\label{eq:unstable third one}
\begin{split}
\int_{\Phi_{s_j}(U^2_j)}\oldh J_{\Phi_{s_j}(U^2_j)}\Phi_{t-s_j} \, \psi_2\circ \Phi_{t-s_j} 
&-  \int_{U^2_j}\oldh J_{U^2_j}\Phi_t  \, \psi_2\circ \Phi_t \\
&= \int_0^{s_j} {\partial_s} \int_{\Phi_s(U^2_j)} \oldh J_{\Phi_s(U^2_j)} \Phi_{t-s} \,
\psi_2 \circ \Phi_{t-s} \, ds \\
& = \int_0^{s_j} \partial_s \int_{U^2_j} \oldh \circ \Phi_s \, J_{U^2_j} \Phi_t \, \psi_2 \circ \Phi_t \, ds\, .
\end{split}
\end{equation}
Thus, recalling \eqref{eq:flow derivative}, 
(with $s_j \le C \sqrt \epsilon$ by Lemma~\ref{lem:close},)
\begin{equation}
\begin{split}
\label{eq:unstable third two}
 \int_0^{s_j}\int_{U^2_j} \partial_s (\oldh \circ \Phi_s) & J_{U^2_j}\Phi_t \, \psi \circ \Phi_t \, dm_W\, ds
=  \int_0^{s_j}\int_{U^2_j} ( \partial_r \oldh \circ \Phi_r)|_{r=s} \circ \Phi_s \,
J_{U^2_j}\Phi_t \, (\psi \circ \Phi_t) \, dm_W \, ds\\
& \le s_j \| \oldh \|_0  |J_{
(U^2_j)}\Phi_{t
} \, (\psi \circ \Phi_{t
})|_{\cC^\oldp(
U^2_j)}
\le C s_j  \| \oldh \|_0 |J_{
(U^2_j)}\Phi_{t
}|_{L^\infty(
U^2_j)} \,  , 
\end{split}
\end{equation}
where once again, we have used Lemma~\ref{lem:Holder} and 
\eqref{eq:dist holder}.

We use Lemma~\ref{lem:close}  together
with \eqref{eq:unstable third one} and \eqref{eq:unstable third two}
to estimate the third term in \eqref{eq:stepone},
\[
\begin{split}
\sum_j \left| \int_{\Phi_{s_j}(U^2_j)}\oldh J_{\Phi_{s_j}(U^2_j)}\Phi_{t-s_j} \, \psi_2\circ \Phi_{t-s_j} 
- \int_{U^2_j}\oldh J_{U^2_j}\Phi_t  \, \psi_2\circ \Phi_t \right| & 
\le \sum_j C \| \oldh \|_0 s_j |J_{U^2_j} \Phi_t|_{L^\infty(U^2_j)} \\
& \le C \ve^{1/2} \| \oldh\|_0 \,  ,
\end{split}
\]
where again we have used Lemma~\ref{lem:growth}(b).

Now we use this bound, together with \eqref{eq:first unstable},
\eqref{eq:second unstable} and \eqref{eq:unstable three} to estimate \eqref{eq:unstable split}
\begin{equation}
 \label{eq:unstable curve}
\begin{split}
 \left|\int_{W^1} \cL_t \oldh \, \psi_1 \, dm_W - \int_{W^2} \cL_t \oldh \, \psi_2 \, dm_W \right|
  \; &\leq \; C \|\oldh\|_s \ve^\oldbeta + C \|\oldh\|_u t^\gamma \Lambda^{-t \gamma} \ve^\gamma\\
  &\qquad + C \|\oldh\|_s  \ve^{\oldp-\oldq} + C \| \oldh \|_0 \ve^{1/2} \, .
\end{split}
\end{equation}
Since $\gamma \le \min \{ \oldbeta, \oldp-\oldq \} < 1/3$, we divide through by $\ve^\gamma$ and take
the appropriate suprema to complete the proof of \eqref{eq:strong unstable L}.


\subsection{Strong continuity of $\cL_t$ on $\cB$}
\label{lastt}

\noindent To complete this section, we state and prove the announced lemma
on the regularity of $t\mapsto \cL_t$. 
(See also the variant given in Lemma \ref{lem:lip} below.)

\begin{lemma}[Strong continuity]
 \label{lem:strong c}
 For each $\oldh \in \cB$ we have $\lim_{t \downarrow 0} \| \cL_t \oldh - \oldh \|_\cB = 0$.
\end{lemma}

\begin{proof}
 First we prove the statement for $\oldh \in \cC^0_\sim \cap \cC^2(\Omega_0)$, then we extend it to general $\oldh \in \cB$.
 We fix $\oldh \in \cC^0_\sim \cap \cC^2(\Omega_0)$ and consider each component of the norm separately.
 
 To estimate the strong stable norm, fix $W \in \cW^s$ and $\psi \in \cC^\oldq(W)$ with 
 $|W|^{\oldbeta}|\psi|_{\cC^\oldq(W)} \le 1$.  Then,\footnote{ Note for any
curve $W$ and any $x$ in W, the point $\Phi_s(x)$ only hits
$\partial \Omega_0$ at finitely many times $s$, 
recall also that $\oldh\in \cC^0_\sim$.}
\[
 \int_W (\cL_t \oldh - \oldh) \psi \, dm_W = \int_W \int_0^t \partial_s (\cL_s \oldh) ds \, \psi \, dm_W \, .
\]
As usual, by the group property (see \eqref{eq:flow derivative}), we have
\begin{equation}
\label{eq:grad}
 \partial_s (\cL_s \oldh) = \partial_r (\oldh \circ \Phi_r) |_{r=0} \circ \Phi_{-s}
 = (\nabla \oldh \cdot \heta) \circ \Phi_{-s} = \cL_s(\nabla \oldh \cdot \heta)\, ,
\end{equation}
where $\heta$ denotes the unit vector in the direction of the flow.  Since
$\nabla \oldh \cdot \heta \in \cC^1(\Omega_0)\subset \cB^0$ (it is not an element
of $\cC^0_\sim$ or $\cB^0_\sim$ in general, but this is immaterial)
we can apply \eqref{eq:strong stable} to estimate,
\[
 \int_W (\cL_t \oldh - \oldh) \psi \, dm_W = \int_0^t \int_W \cL_s(\nabla \oldh \cdot \heta) \psi \, dm_W ds
 \le Ct (\|\nabla \oldh \cdot \heta\|_s + |\nabla \oldh \cdot \heta |_w) \,  ,
\]
where  \eqref{laborne}  gives $\|\nabla \oldh \cdot \heta\|_s + |\nabla \oldh \cdot \heta |_w
\le |\nabla \oldh \cdot \heta |_{\cC^0(\Omega_0)}$.
Taking the supremum over $W$ and $\psi$ yields the estimate on the strong stable norm,
\[
 \| \cL_t \oldh - \oldh \|_s \le C t | \nabla \oldh \cdot \heta |_{\cC^0(\Omega_0)}  \, .
\]

To estimate the unstable norm, fix $\ve >0$ and let $W_1, W_2 \in \cW^s$ satisfy $d_{\cW^s}(W_1,W_2) < \ve$.
Take $\psi_i \in \cC^\oldp(W_i)$ with $|\psi_i|_{\cC^\oldp(W)} \le 1$ and $d(\psi_1,\psi_2)=0$.  
Then using again \eqref{eq:grad} and \eqref{eq:unstable curve}, we have
\[
\begin{split}
 \int_{W_1} (\cL_t \oldh - \oldh) \psi_1 \, dm_W & - \int_{W_2} (\cL_t \oldh - \oldh) \psi_2 \, dm_W \\
 & = \int_{W_1} \int_0^t \partial_s (\cL_s \oldh) ds \, \psi_1 \, dm_W
 - \int_{W_2} \int_0^t \partial_s (\cL_s \oldh) ds \, \psi_2 \, dm_W \\
 & = \int_0^t \Big(\int_{W_1} \cL_s(\nabla \oldh \cdot \heta) \psi_1 \, dm_W - \int_{W_2} \cL_s(\nabla \oldh \cdot \heta) \psi_2 \, dm_W \Big) ds \\
 & \le t C (\| \nabla \oldh \cdot \heta \|_s \ve^\oldbeta + \| \nabla \oldh \cdot \heta \|_u  \ve^\gamma +  \| \nabla \oldh \cdot \heta \|_s  \ve^{\oldp-\oldq} + \| \nabla \oldh \cdot \heta \|_0 \ve^{1/2}) \,  .
 \end{split}
\]
Dividing through by $\ve^{\gamma}$ (recall that $\gamma < 1/3$), recalling \eqref{laborne}, and taking the appropriate suprema yields the bound
$\| \cL_t \oldh - \oldh \|_u \le Ct |\nabla \oldh \cdot \heta |_{\cC^1(\Omega_0)}$.

Finally, to estimate the neutral norm, we  fix 
$0<|r_0|\le 1$,
$W \in \cW^s$ and $\psi \in \cC^\oldp(W)$ with $|\psi|_{\cC^\oldp(W)}\le 1$.  Then using
\eqref{eq:grad} and the neutral estimate \eqref{eq:neutral L}, we have,
\[
\begin{split}
\int_W \partial_r \big((\cL_t f - &\oldh)\circ  \Phi_r\big)|_{r=0} \, \psi \, dm_W 
 =\int_W \partial_r \left( \Big( \int_0^t \cL_s(\nabla \oldh \cdot \heta) ds \Big) \circ \Phi_r \right)\Big|_{r=0} \, \psi \, dm_W  \\ 
 & = \int_0^t  \int_W \partial_r \big( \cL_s(\nabla \oldh \cdot \heta) \circ \Phi_r \big)|_{r=0} \, \psi \, dm_W \,  ds  
\;  \le  \; Ct   \| \nabla \oldh \cdot \heta \|_0 \, . 
 \end{split}
\]

Taken together, the estimates on the three components of the norm imply the bound
$\| \cL_t \oldh - \oldh \|_\cB \le C t | \nabla \oldh \cdot \heta |_{\cC^1(\Omega_0)}$, which proves the lemma
for $\oldh \in \cC^2(\Omega_0)\cap \cC^0_\sim$.

For more general $\oldh \in \cB$, we proceed by approximation.
By the definition of $\cB$, in particular \eqref{smallspace}, we can
approach any $\oldh$ in $\cB$ by a sequence,
$\cL_{t_n}(h_n)$ with $h_n\in \cC^2(\Omega_0)\cap \cC^0_\sim$ and $t_n\ge 0$. Then we write
$$
\cL_t (\oldh) -\oldh =\cL_t(\oldh-\cL_{t_n}h_n) + \cL_{t_n}(\cL_t(h_n)-h_n)
+(\cL_{t_n}h_n-\oldh)\, .
$$
By Proposition~ \ref{prop:L_t}, there exists $\hat C$
so that  $\|\cL_t \oldh\|_\cB \le \hat C  \|\oldh\|_\cB$
for all $\oldh\in \cB\subset \cB^0$.
Thus the first and last term on the right-hand side
above tend to zero in $\cB$ as $n\to \infty$,
uniformly in $t$.
Finally,
$$\|\cL_{t_n}(\cL_t(h_n)-h_n)\|\le \hat C  \|\cL_t(h_n)-h_n\|_{\cB}\, ,
$$
so that
it suffices to see that
$\lim_{t \downarrow 0}\|\cL_t g - g \|_{\cB}=0$ for any $\cC^2$ function
$g\in \cC^0_\sim$, which is precisely what we proved above.
\end{proof}

To apply the results of Butterley \cite{Butterley} to prove Theorem~\ref{main}, we will  use
the following lemma.  

\begin{lemma}
 \label{lem:lip}
 There exists $C>0$ such that $| \cL_t \oldh- \oldh |_w \le Ct\|\oldh\|_{\cB}$ for all $t \ge 0$
 and all $f\in \cB$.
\end{lemma}

\begin{proof}
 Let $\oldh \in \cC^2(\Omega_0)$ and fix $W \in \cW^s$ and $\psi \in \cC^\oldp(W)$ with 
 $|\psi|_{\cC^\oldp(W)} \le 1$.  Then using the group property as in \eqref{eq:flow derivative}
 and changing variables,
 \[
  \begin{split}
   \int_W (\cL_t \oldh - \oldh) \psi \, dm_W & = \int_W \int_0^t \partial_s
   ( \oldh \circ \Phi_{-s}) \psi \, dm_W
   =  \int_0^t \int_W \partial_r (\oldh \circ \Phi_r)|_{r=0} \circ \Phi_{-s} \, \psi \, dm_W ds \\
   & = \int_0^t \sum_{W_i \in \cG_s(W)} \int_{W_i} \partial_r (\oldh \circ \Phi_r)|_{r=0}
      \, \psi \circ \Phi_s \, J_{W_i}\Phi_s \, dm_W ds \\
   & \le t \| \oldh\|_0 \sum_{W_i \in \cG_s(W)} |\psi \circ \Phi_s|_{\cC^\oldp(W_i)} |J_{W_i}\Phi_s |_{\cC^\oldp(W_i)}
   \le C t \| \oldh \|_0 \, ,
  \end{split}
 \]
where we have used Lemma~\ref{lem:Holder} and Lemma~\ref{lem:growth}(b) in the last line
to bound the sum independently of $s$.  Now taking the supremum over $W$ and $\psi$
yields $| \cL_t \oldh - \oldh |_w \le Ct \|\oldh\|_0$, and taking the supremum over $\oldh$ with
$\|\oldh\|_\cB \le 1$ proves the lemma.
\end{proof}


\section{Quasi-compactness of the resolvent $\cR(z)$, first results on the spectrum of $X$}
\label{sec:R}

In this section, we use the control over $\cL_t$ established in the previous section to
deduce the quasi-compactness of the resolvent acting on $\cB$ 
and obtain the first spectral properties
of the generator $X$.  We begin by defining the
generator $X$ of the semi-group acting on elements of $\cB$,
\[
 X\oldh = \lim_{t \downarrow 0} \frac{\cL_t \oldh - \oldh}{t}  \, .
\]
The domain of $X$ is defined as the set of $\oldh \in \cB$ for which the above limit
converges in the norm of $\cB$. 
By Lemma~\ref{lem:strong c}, the operator $\cL_t$ is
strongly continuous on $\cB$ which implies \cite[Lemmas 6.1.11, 6.1.14]{Davies} that  $X$ is a closed operator, with
domain dense
in $\cB$. 
(Indeed by Lemma~\ref{lem:C1}, the domain of $X$ contains $\cC^2(\Omega_0)$.)
This allows us to
define the resolvent $\cR(z) = (z Id-X)^{-1}$.  
For $z = a + ib$, $a>0$, the resolvent
$\cR(z)$ has the representation,
\begin{equation}
\label{eq:R}
\cR(z)\oldh = \int_0^\infty e^{-zt} \cL_t \oldh \,  dt 
\end{equation}
and is a bounded operator on $\cB$.
By induction, the iterates of $\cR(z)$ thus satisfy 
\begin{equation}
\label{eq:Rn}
\cR(z)^n \oldh = \int_0^\infty \frac{t^{n-1}}{(n-1)!} e^{-zt} \cL_t \oldh \, dt \, .
\end{equation}
We will use repeatedly the fact that for any $z \in \mathbb{C}$ with Re$(z)>0$,
\begin{equation}
 \label{eq:int}
 \left| \int_0^\infty \frac{t^{n-1}}{(n-1)!} e^{-zt} dt \right| \le (\mbox{Re}(z))^{-n}\,  .
\end{equation}

Recall the hyperbolicity exponents $\Lambda>1$ from \eqref{Lambda} and $\lambda$
from the growth Lemma~\ref{lem:growth}
In Subsection~\ref{S1},
we will prove the following proposition.

\begin{proposition}
[Lasota-Yorke inequalities for $\cR(z)$]
\label{prop:LY R}
Recall that $\gamma\le \min\{\oldp-\oldq, \oldbeta\}$. For any 
$$1 > {\tilde \lambda} > \max \{ \Lambda^{-\oldq}, \Lambda^{-\gamma}, \lambda^{1-\oldbeta} \}\, $$
there exists $C\ge 1$ such that for all $z \in \mathbb{C}$ with Re$(z) = a >0$, all
$\oldh \in \cB$ and all $n \ge 0$,
\begin{eqnarray}
|\cR(z)^n \oldh|_w & \le & C a^{-n} |\oldh|_w   \label{eq:weak R} \\
\| \cR(z)^n \oldh \|_s & \le & C(a - \ln {\tilde \lambda})^{-n} \| \oldh \|_s + C a^{-n} |\oldh|_w   \label{eq:strong stable R} \\
\| \cR(z)^n \oldh \|_u & \le & C(a - \ln {\tilde \lambda})^{-n} \| \oldh \|_u + C a^{-n} \| \oldh \|_s + C a^{-n}\| \oldh \|_0
\label{eq:strong unstable R} \\
\| \cR(z)^n \oldh \|_0 & \le & C a^{1-n}(1 + a^{-1} |z|) |\oldh|_w \, . \label{eq:neutral R}
\end{eqnarray}
\end{proposition}

From this proposition, we  deduce below a bound on the essential
spectral radius of $\cR(z)$:

\begin{cor}[Spectral and essential spectral radii of $\cR(z)$]\label{CorLY}
Let ${\tilde \lambda}<1$ and $C$ be as in Proposition~\ref{prop:LY R}.
If the constant $c_u$ defining the norm in $\cB$ is small enough (depending only on
$C$ from Proposition~\ref{prop:LY R})
then, for any $z=a+ib$ with $a\ge 1$, the essential
spectral radius of $\cR(z)$ on $\cB$ is bounded by $(a - \ln {\tilde \lambda})^{-1}$, while its spectral radius
is bounded by $a^{-1}$.
\end{cor}

\begin{remark}[Optimal strip]\label{opt} Since our choice of homogeneity layers
gives $\oldp\le 1/3$, it is easy to see that the optimal choice for our parameters
in view of Corollary ~\ref{CorLY}  is
$$
\frac{1}{\oldbetaparam} =\frac{1}{2}, \quad \oldp=\frac{1}{3} \, , \quad \oldq= \frac{1}{6}\, , \quad
\gamma = \frac{1}{6}\, ,
$$
giving a Banach space $\cB$ for 
which we can take $\tilde \lambda > \Lambda^{-1/6}$ arbitrarily close to $ \Lambda^{-1/6}$.
By choosing homogeneity layers $k^\chi$ with  $\chi>1$ arbitrarily close to $1$
(see footnote
 \ref{foot10}), we can consider the space $\cB$ associated to
$\oldbeta=1/2$,  $\oldq=1/4$, $\oldp<1/2$ arbitrarily close to $1/2$, and $\gamma<1/4$ arbitrarily close 
to $1/4$, so that we can take $\tilde \lambda > \Lambda^{-1/4}$ arbitrarily close to $ \Lambda^{-1/4}$. (This seems to be the optimal\footnote{ In particular,
taking other values than the ``Hilbert-space'' choice $\oldbetaparam=2$
does not help here. However, for the proof of Proposition~\ref{dolgo},
it is crucial to let $\oldbetaparam$ tend to $1$.} strip, since for a  piecewise hyperbolic 
symplectic map or hamiltonian
flow
we expect $\Lambda^{-1/2}$ in view of \cite[Remark 5.9]{DL}, and the further square root appears natural in view
of the billiard singularity type, see Lemma~\ref{lem:smooth}.)
To get exponential decay of correlations  (Corollary~\ref{truemain}), we shall need to
consider  in Section~\ref{theend} a Banach space $\widehat \cB\subset \cB$ 
corresponding to $\oldq>0$ very small,
and $\oldbeta <1$ very close to $1$ (we may take   $\oldp\le 1/3$
and take smaller $\gamma$ if needed). However, the non-essential spectra of $\cR$ on  the
Banach spaces $\cB$ and $\widehat \cB$ corresponding to 
these different parameters coincide 
(see e.g.  \cite[App. A]{BT}, using that the spaces
$\widehat \cB\subset \cB$
are continuously embedded in $(C^{1/4}(\Omega_0))^*$, while their intersection
$\widehat \cB$
is a Banach space dense in $\cB$, recalling \eqref{smallspace}), so the final spectral result also holds on the space
$\cB$
for $\frac{1}{\oldbetaparam} =\frac{1}{2}$, $\oldp=\frac{1}{3}$,  $\oldq= \frac{1}{6}$,  $\gamma = \frac{1}{6}$.
Note finally that the choice of homogeneity layer decay also affects the constants
in the approximate foliation constructed in Section~\ref{Lipschitz} and used to prove
Lemma~\ref{dolgolemma} and thus Theorem~\ref{main} and its corollary. The choice of homogeneity
layers made there is independent from the one used to get Corollary~\ref{CorLY}, and in any case
 taking  $1<\chi <2$ close
to $1$ will
in fact give a better exponent for the foliations in 
Section~\ref{Lipschitz} and Remark ~\ref{rem:foliation} . 
\end{remark}

\begin{proof}[Proof of Corollary~\ref{CorLY}]
Fixing $z = a + ib$ with $a >0$, we choose $n\ge 1$ minimal 
so that $2C/(1 -  a^{-1} \ln {\tilde \lambda})^n =:  \nu_a^n < 1$, 
where
$C$ is from Proposition~\ref{prop:LY R}.
If $a\ge 1$, we  have $0<c_u  < (1- a^{-1} \ln {\tilde \lambda})^{-n}$,
if $c_u>0$ satisfies
$
c_u < (1- a^{-1} \ln \tilde \lambda)^{-n}
$, 
with $n$ fixed as above.
Now,
\begin{equation}\label{babyLY} 
\begin{split}
 a^n \| \cR(z)^n \oldh \|_{\cB} & = a^{n}\| \cR(z)^n \oldh \|_s + c_ua^{n} \| \cR(z)^n \oldh \|_u +a^{n} \| \cR(z)^n \oldh \|_0 \\
 & \le C(1- a^{-1} \ln {\tilde \lambda})^{-n} \|\oldh\|_s + C  |\oldh|_w
  + c_u C(1 -  a^{-1} \ln {\tilde \lambda})^{-n} \| \oldh \|_u  \\ 
  & \qquad + c_u C  \| \oldh \|_s+ c_u C \| \oldh \|_0
  + C a (1 +  a^{-1} |z|) |\oldh|_w \\
  & \le 2C(1- a^{-1} \ln \tilde \lambda)^{-n} ( \|\oldh\|_s + c_u \|\oldh\|_u + \|\oldh\|_0) + C(1 + |z|
   +a) |\oldh|_w \\
 & \le \nu_a^n \|\oldh\|_\cB + C(1 + |z| +a)|\oldh|_w\,  .
 \end{split}
 \end{equation}
By a classical result of Hennion \cite{hennion},
this, together with the compactness of the unit ball of $\cB$ in $\cB_w$
(Lemma~ \ref{lem:compact}), implies that the essential
spectral radius of $\cR(z)$ on $\cB$ is bounded by $(a - \ln \tilde \lambda)^{-1}$, while its spectral radius
is bounded by $a^{-1}$.
\end{proof}

As a consequence of Corollary ~\ref{CorLY},  we get our
first result on the spectrum of the generator $X$ (see also Corollary ~\ref{specX'}):

\begin{cor}[Spectrum of $X$]\label{spX}
Under the assumptions  of  Corollary~\ref{CorLY}, the following holds:
The spectrum of $X$ on  $\cB$
is contained in the left half-plane $\Re z \le 0$, and
its intersection with the half-plane
$
\{ z \in \cC \mid \Re z > \ln \tilde \lambda \}
$
consists of at most countably many isolated eigenvalues of finite
multiplicity. 
In addition, the spectrum on the imaginary axis 
consists in a finite union of discrete additive subgroups of $\bR$, and 
if $b \in \bR$ then $X\psi =ib \psi$
for $\psi \in \cB$ implies $\psi \in L^\infty$, while $X\psi=ib\psi$ for $\psi \in L^1$
implies $\psi \in \cB$.
Finally, $X$ has an eigenvalue of algebraic multiplicity one at $0$.
\end{cor}

\begin{proof}
The arguments are standard, see e.g.  \cite[Lemma 3.6, Cor. 3.7]{BaL}.
A nonzero
$\rho\in \bC$ lies in the spectrum of $\cR(z)$ on $\cB$ 
if and only if $\rho=(z-\tilde \rho)^{-1}$,
where $\tilde \rho$ lies in the
spectrum of $X$ as a closed operator on $\cB$ 
(see e.g. \cite[Lemma 8.1.9]{Davies}). For such 
a pair $(\rho,\tilde \rho)$,  is easy to check that, for any $k\ge 1$
and any $\psi \in \cB$, we have
$(\cR(z)-\rho)^k (\psi)=0$ if and only if $(X-\tilde \rho)^k (\psi)=0$, so
that $\rho$ is an eigenvalue of $\cR(z)$ of algebraic multiplicity $m_0$,
$1\le m_0\le \infty$, if and only if $\tilde \rho$ is an eigenvalue 
of $X$ of algebraic multiplicity $m_0$.
The first two claims then follow  from Corollary~ \ref{CorLY}.

It remains to study the spectrum on the imaginary axis.
First note that if $X(\psi)=ib\psi$
for $b\neq 0$, then  $\cR(a+ib)(\psi)=a^{-1}\psi$. 
Another simple computation, using
\cite[Theorem 6.1.16]{Davies}
(noting that $\tilde \psi_t = e^{ib t} \psi$ satisfies $\partial_t \tilde \psi_t|_{t=s}= X (\tilde \psi_s)$
so that $\cL_t (\tilde \psi_0)=\tilde \psi_t$) gives 
\begin{equation}\label{eq:flow-inv}
\psi \circ \Phi_t=e^{-ib t} \psi\, .
\end{equation}

In the following, we will use two properties of $\cB$: First, there exists a  set $\cD\subset L^\infty(\Omega_0)$ dense in $\cB$ (see \eqref{smallspace} and use that $\cL_t$ is a contraction in $L^\infty$, 
i.e., $|\cL_t \psi|_\infty \le |\psi|_\infty$). 
Second, if $\psi\in \cB$ and $\int \psi \varphi\, dm=0$ for all $\varphi\in C^2$, then $\psi=0$ (due to the embedding properties  in Lemma \ref{relating}).

We next show  that $\psi \in L^\infty$ if $\psi \circ \Phi_t=e^{-ib t} \psi$ for $b \in \bR$.
We already observed that this implies $\cR(a+ib)(\psi)=a^{-1}\psi$. Next, let $\{\zeta_j\}_{i=1}^K$, $\zeta_1=a^{-1}$, be the finitely many eigenvalues of modulus $a^{-1}$ of $R(z)$. By the  spectral decomposition, \cite{Ka}, we can write
\begin{equation}\label{eq:spec-bound}
\cR(z)=\sum_{i=1}^K\left[ \zeta_i\Pi_i(z)+N_i(z)\right]+Q(z)\, , \quad z=a+ib\, ,
\end{equation}
where  the
operators  $\Pi_i(z)$ and $N_i(z)$ are finite rank, the $\Pi_i(z)$, $N_i(z)$, and $Q(z)$ commute,
with $\Pi_i(z)^2=\Pi_i(z)$, $\Pi_i(z) Q(z)= N_j(z)\Pi_i(z)=0$, for $i\neq j$,
 $\Pi_i(z) N_i(z)=N_i(z)$, and, if $N_i(z)\neq 0$, there exists
$d_i(z)\in \bN$ such that $N_i(z)^{d_i(z)+1}=0$ while $N_i(z)^{d_i(z)}\ne 0$.
In addition there exist $C(z)>0$ and $\rho_0(z)<a^{-1}$ such that $\|Q^n(z)\|\leq C(z)\rho_0(z)^n$. 
(We suppress the $z$ dependence when no confusion can arise.)
We next show that  $N_i(z)\equiv 0$ for all $i\in\{1,\dots,K\}$. Assume otherwise, then, for any $n\ge 1$, we would have
\[
\begin{split}
\cR(z)^n
&=\sum_{i=1}^K\sum_{j=0}^{d_i(z)}\binom nj \zeta_i^{j-n}\Pi_i(z) N_i(z)^j+\cO(\rho_0(z)^n)\, ,
\end{split}
\]
which, setting $d=\min_i d_i$, would immediately imply that there is  $\tilde \psi\in \cB$ for which 
$\|\cR(z)^n\tilde \psi\|\geq C(z) a^{-n}n^{d(z)}$, contradicting  $\eqref{eq:weak R}$.
Accordingly,  \eqref{eq:spec-bound} implies 
$\lim_{n\to\infty} \frac 1n \sum_{k=0}^{n-1} a^{k}\cR(z)^k=\Pi_1(z)
$.
Let $\psi_0\in \cD$, then for $k\geq 1$, using that $\cL_t$ preserves volume,
\begin{equation}\label{crucialvol}
\left|a^{k}\cR(z)^k(\psi_0)\right|_\infty
\leq a^{k}\int_0^\infty \frac{t^{k-1}e^{-at}}{(k-1)!}  |\psi_0\circ \Phi_{-t}|_\infty \, dt\leq |\psi_0|_{\infty} \, .
\end{equation}
Hence, for each $\psi_0\in\cD$ and $\vf\in C^\infty$, 
\begin{equation}\label{eq:Pi-bound}
\left|\int \Pi_1(z) (\psi_0)\cdot \vf\, dm\right| \leq |\psi_0|_{\infty} |\vf|_{L^1}\, .
\end{equation}
This implies that $\Pi_1(\cD)\subset L^\infty$. Since the range of $\Pi_1$ is finite-dimensional, it follows that  $\Pi_1(z)$ is bounded  from $\cB$ to $L^\infty$,
proving our claim that $\psi \in L^\infty$.

We next check that if $\psi\in L^1\setminus\{0\}$ and 
$\cR(a+ib)(\psi)=a^{-1}\psi$ for $b\in \bR$, then $\psi \in \cB$. Consider a sequence $\{\psi_\ve\}\subset C^\infty$ converging to $\psi$ in $L^1$. 
Then the limit
\[
\lim_{n\to\infty} \frac 1n\sum_{k=0}^{n-1} a^k \cR(z)^k (\psi_\epsilon)
\, , \qquad
z=a+ib\, ,
\]
always exists, and 
it is zero if $a^{-1}\not\in\spp (\cR(z))$, while it equals $\Pi_1(z)(\psi_\ve)$ if $a^{-1}\in\spp (\cR(z))$, where $\Pi_1(z)$ is the eigenprojector associated to $a^{-1}$.
In the first case, for each $\varphi\in C^2$,
\[
\begin{split}
\left|\int\psi \varphi\, dm\right|
&\leq \lim_{\ve\to 0}\lim_{n\to\infty}\frac 1n\sum_{k=0}^{n-1}a^k\int_0^\infty
e^{-at}\frac{t^{k-1}}{(k-1)!}|\psi_\ve-\psi|_{L^1}\cdot|\varphi\circ \Phi(t)|_{\infty}\, dt\\
&\leq \lim_{\ve\to 0}|\psi_\ve-\psi|_{L^1}\cdot|\varphi|_{\infty}=0\, .
\end{split}
\]
We would then have $\psi\equiv 0$, a contradiction. Hence $a^{-1}$ is an eigenvalue, and by the same computation as above
\begin{equation}\label{eq:close-psi}
\left|\int(\psi-\Pi_1(z)(\psi_\ve)) \varphi\, dx\right|\leq |\psi_\ve-\psi|_{L^1}\cdot|\varphi|_{\infty}\, .
\end{equation}
Since $\Pi_1(z)$ is a projector,  $\Pi_1(z)=\sum_k\tilde \psi_k(z) [\ell_k(z)](\cdot)$, where $\tilde \psi_k(z)\in  \cB\cap L^\infty$, the $\ell_k(z)$ belong to the dual of $\cB$, and $\ell_k(\tilde\psi_j)=\delta_{kj}$. Then \eqref{eq:close-psi} shows that  $\ell_k(\psi_\ve)$ is bounded
uniformly in $\epsilon$. We can then extract a  subsequence 
$\ve_j$ 
such that $\Pi_1(z)(\psi_{\ve_j})$ is convergent. Thus, $\psi$ is a linear combination of the $\tilde \psi_k$,  concluding the proof that $\psi\in \cB$.

We next use the facts about $L^1$ and $L^\infty$ to show the discrete subgroup claim.
If $X(\psi_k)=ib_k\psi_k$, with $\psi_k\ne 0$ 
for $k\in\{1,2\}$, we have $\psi_1,\psi_2\in L^\infty$, and
\begin{align}\label{cutandpaste}
\cR(z)(\psi_1\psi_2)&=\int_0^\infty e^{-zt}(\psi_1\circ \Phi_{-t})(\psi_2\circ \Phi_{-t})\, dt
=\psi_1\psi_2 \int_0^\infty e^{-zt +i(b_1+b_2) t}dt\\
\nonumber &=(z-ib_1-ib_2)^{-1}\psi_1\psi_2\, .
\end{align}
Clearly,  $\psi:=\psi_1\psi_2\in L^1\setminus\{0\}$ and 
$\cR(a+ib_1+ib_2)(\psi)=a^{-1}\psi$, so that    $\psi\in \cB$.
Then, \eqref{cutandpaste} implies that either $\psi_1\psi_2=0$ or $ib_1+ib_2\in\spp(X)$. A similar argument applied to $\bar\psi_k$ shows that $-ib_k\in\spp(X)$. Thus $|\bar \psi_k|^2$ belongs to the finite dimensional eigenspace of the eigenvalue zero and $\{im b_k\}_{m\in\bZ}\subset \spp(X)$.  This ends the proof of the subgroup claim.

Finally, if $A$ is a positive measure invariant set, then $\Id_A\in L^\infty$ 
(which belongs to $\cB$ by the argument above) is an eigenvector associated to zero, and if $\psi$ is an eigenvector associated to zero (we can assume without
loss of generality that $\psi$ is real) then $\{\psi\geq \lambda\}$ are invariant sets,
that is, $\psi$ must be piecewise constant (otherwise zero would have infinite multiplicity). In other words the eigenspace of zero is spanned by the characteristic functions of the ergodic decomposition of Lebesgue.
By ergodicity of the billiard flow, zero is a simple eigenvalue.
\end{proof}

\smallskip


\subsection{
Norm estimates  for $\cR(z)$}
\label{S1}
In this 
subsection, we prove Proposition~\ref{prop:LY R}.
Fix $1 > {\tilde \lambda} > \max \{ \Lambda^{-\oldq}, \Lambda^{-\gamma}, \lambda^{(1-\oldbeta)} \}$.

Now taking $W \in \cW^s$ and $\psi \in \cC^\oldp(W)$ with $|\psi|_{\cC^\oldp(W)} \le 1$, we have
by \eqref{eq:weak} and \eqref{eq:int},
\[
\left| \int_W \cR(z)^n \oldh \, \psi \, dm_W \right|
= \left| \int_0^\infty \int_W \cL_t \oldh \, \psi \, dm_W \frac{t^{n-1}}{(n-1)!} e^{-zt} \, dt \right|
\le C |\oldh|_w a^{-n} \, ,
\]
which proves \eqref{eq:weak R}.

Similarly, taking $\psi \in \cC^\oldq(W)$ with $|W|^{\oldbeta}|\psi|_{\cC^\oldq(W)} \le 1$, we have
by \eqref{eq:strong stable},
\[
\begin{split}
\left| \int_W \cR(z)^n \oldh \, \psi \, dm_W \right|
& = \left| \int_0^\infty \int_W \cL_t \oldh \, \psi \, dm_W \frac{t^{n-1}}{(n-1)!} e^{-zt} \, dt  \right| \\
& \le \int_0^\infty [C (\Lambda^{-tq} + \lambda^{t(1-\oldbeta)}) \| \oldh \|_s + C |\oldh|_w] \frac{t^{n-1}}{(n-1)!}
e^{-at} \, dt \\
& \le C(a - \ln {\tilde \lambda})^{-n} \|\oldh\|_s + C a^{-n} |\oldh|_w \, , 
\end{split}
\]
where we have again used \eqref{eq:int}, proving \eqref{eq:strong stable R}.


\smallskip

To prove \eqref{eq:strong unstable R},  
fixing $\ve>0$ and taking $W_1, W_2$ with $d_{\cW^s}(W_1,W_2)<\ve$, 
and $\psi_i \in \cC^\oldp(W_i)$
with $|\psi_i|_{\cC^\oldp(W_i)} \le 1$ and $d(\psi_1,\psi_2)=0$, we estimate
\[
\begin{split}
 \ve^{-\gamma} &\left| \int_{W_1} \cR(z)^n \oldh \, \psi_1 \, dm_W \right.  \left. - \int_{W_2} \cR(z)^n \oldh \, \psi_2 \, dm_W \right| \\
 &\qquad = \left| \int_0^\infty \frac{t^{n-1}}{(n-1)!} e^{-zt} \ve^{-\gamma} \left(
 \int_{W_1} \cL_t \oldh \, \psi_1 \, dm_W - \int_{W_2} \cL_t \oldh \, \psi_2 \, dm_W \right) dt \right| \\
 &\qquad \le \int_0^\infty \frac{t^{n-1}}{(n-1)!} e^{-at} (C {\tilde \lambda}^t \| \oldh \|_u + C (\| \oldh \|_0 + \| \oldh \|_s)) dt \\
 &\qquad \le C (a-\ln {\tilde \lambda})^{-n} \|\oldh\|_u + C a^{-n} (\| \oldh \|_s + \|\oldh\|_0) \, ,
\end{split}
\]
where we have used \eqref{eq:int} and \eqref{eq:strong unstable L}.
Taking the appropriate suprema
proves \eqref{eq:strong unstable R}.


\smallskip
We prove the neutral norm estimate \eqref{eq:neutral R}.
Let $W \in \cW^s$ and $\psi \in \cC^\oldp(W)$ with $|\psi|_{\cC^\oldp(W)} \le 1$.
Using the definition of the neutral norm, we write for 
$\oldh \in \cB\cap \cC^1(\Omega_0)$,
\[
\begin{split}
\int_W  \partial_s ((\cR(z)^n \oldh) \circ \Phi_s)|_{s=0} \, \psi \, dm_W 
& = 
\int_W \int_0^\infty \frac{t^{n-1}}{(n-1)!} e^{-zt} \partial_s \left( (\cL_t \oldh) \circ \Phi_s \right) |_{s=0} dt \, 
\psi \, dm_W \\
& =  \int_W \int_0^\infty \frac{t^{n-1}}
{(n-1)!} e^{-zt} \partial_s (\oldh \circ \Phi_{-s})|_{s=t} \, dt  \,
\psi \, dm_W \\
& = - \int_0^\infty \left( \frac{t^{n-2}}{(n-2)!} - \frac{zt^{n-1}}{(n-1)!} \right) e^{-zt}  
\int_W \cL_t \oldh \, \psi \, dm_W \, dt \, ,
\end{split}
\]
where in the last line we have integrated by parts and used the fact that $\cL_t \oldh = \oldh \circ \Phi_{-t}$.
Finally, 
recalling the weak norm estimate
\eqref{eq:weak} for the integral on $W$ and using \eqref{eq:int} to integrate
with respect to $dt$  yields \eqref{eq:neutral R}.
This ends the proof of Proposition~\ref{prop:LY R}.


\section{Construction of approximate unstable Lipschitz foliations}
\label{Lipschitz}

Fix a length scale $0<\rho<L_0$. For any homogeneous flow stable curve $W \in \cW^s$ such that 
$$ (a) \quad d(W, \partial \Omega_0) \ge 2C_d \rho \quad \mbox { and } \quad
(b) \quad P^+(W) \subset \bH_{k_W} \, ,
$$ 
where $C_d > 1$ is defined by \eqref{eq:step length} and
$k_W$ satisfies\footnote{ The exponent $-1/5$ comes from our choice of homogeneity 
layer (the length of the fiber, $k_W^2 \rho$ must be less than the width
of the strip $k_W^{-3}$); this choice will impact the exponents throughout this section, without
further notice.}
$k_W \le C \rho^{-1/5}$, and for any
$\up>0$ 
and any $x^0$ such that $\Phi_t(W)$ remains at least $C_d\rho$ away from a collision for all
$|t| \le |x^0|$, 
we shall construct in Theorem~\ref{thm:foliation}
below a surface $\cF_{\up,x^0}$ containing $\Phi_{x^0}(W)$
and transversal to the flow, and a
transverse  foliation of a two--dimensional neighborhood of $W$ in
$\cF_{\up, x^0}$  by
(one-dimensional) curves $\gamma$ so that each $\gamma$
is a flow-unstable curve,
and so that $\Phi_{-t}(\gamma)$ is an unstable curve for $0\le t <\up$
and ``most'' curves $\gamma$.

But first, we need to define good local coordinates (making precise
Remark ~\ref{Rk1.4}):
\begin{remark}[$\cC^2$ cone-compatible Darboux charts]\label{suitcharts0}
Let $W$ satisfy (a) and (b) above.
For any $Z \in W$,  by (b) we have
$\cos \vf(P^+(Z)) \ge C\rho^{2/5}$ for some $C>0$.
Similarly, if $Z' \in \Omega_0$ satisfies $d(Z, Z') < C_d\rho$, we have $\tau(Z') \ge C_d\rho$
and $\cos \vf(P^+(Z')) \ge C\rho^{2/5}$ since $P^+(Z')$ lies either in $\bH_{k_W}$ or
$\bH_{k_W \pm 1}$.
Thus it follows from \eqref{eq:stable cone} that the width of the stable cones at
$Z$ and $Z'$ has angle at least of order $\rho^{2/5}$, and the maximum and minimum
slopes in $C^s(Z)$ and $C^s(Z')$ are uniformly bounded multiples of one another.
Similar considerations hold for the unstable cones.
Thus, fixing $Z\in W$, we may adopt local coordinates $(x^u, x^s, x^0)$
in a $c_1 \rho$ neighborhood
of $Z$,  as introduced in Remark~\ref{Rk1.4}, for which the contact
form is in standard form and $Z$ is the origin.  By the above discussion, the charts
effecting this change of coordinates are uniformly $\cC^2$.  The constant $c_1 \ge 1$ is chosen 
large enough that the box $\{ (x^u, x^s, x^0) : x^s \in W, |x^u| \le 2\rho, |x^0| \le 2\rho \}$
is well-defined in these new coordinates, and if necessary, $C_d$ is increased accordingly.
\end{remark}

Adopting such local coordinates $x^u$, $x^s$, and $x^0$ 
such that $W = \{ (0, x^s,0) : x^s \in [0,|W|] \}$, for each
fixed $0<\varpi \le 1/35$ and $\up >0$,  and each
$x^0$ so that $d(\Phi_t(W), \partial \Omega_0) \ge C_d \rho$ for all $|t| \le |x^0|$, 
Theorem~ \ref{thm:foliation} will provide a
surface 
$$
\cF 
= \{ \bF(x^u,x^s)=(x^u, G(x^u, x^s), H(x^u, x^s)+x^0) : x^s \in [0,|W|], x^u \in [-\rho, \rho] \}
$$
transversal to the flow $\Phi_s$ and
so that the curves $x^u \mapsto \gamma_{x^s}(x^u)=(x^u, G(x^u, x^s), H(x^u, x^s)+x^0)$
are pairwise disjoint, satisfying additional
requirements. (We write $\gamma \in \cF$ and view $\cF$ as foliated
by the curves $\gamma=\gamma_{x^s}$; in particular, 
$\forall Z=(0,x^s,0) \in W$, 
 the curve $\gamma_Z:=\gamma_{x^s}$ is the unique $\gamma \in \cF$ such that $\gamma \cap W = \{Z\}$.)
 
 \begin{theorem}\label{thm:foliation} 
Assuming conditions (a) and (b) above, for each
fixed $0<\varpi \le 1/35$ and $\up >0$,  and each
$x^0$ so that $d(\Phi_t(W), \partial \Omega_0) \ge C_d \rho$ for all $|t| \le |x^0|$, 
it is possible to construct  $\cF=\cF_{\up,x^0}$ so that it
satisfies the following conditions, for a constant $C$ depending on
$\varpi$, but independent
  of $\rho$, $\up$, $x^0$, and $W$:
\begin{itemize}
  \item[(i)]  $\partial_{x^u} H=G$, which implies $\bal(\partial_{x^u}\bF)=0$, i.e.,
  each $\gamma \in \cF$   lies in the kernel of the contact form
  $\bal$.
  In addition,  each $\gamma\in \cF$ is an unstable curve 
  (in particular, $(1, \partial_{x^u}G, \partial_{x^u} H)
\in C^u$ and $|\gamma| \le L_0$) with $|\gamma| \ge \rho$.
  \item[(ii)] For all $x^s\in [0, |W|]$ we have $G(0, x^s)= x^s$ and $H(0, x^s)=0$.
  \item[(iii)]   Let $\Delta_\up$ denote the set of $Z \in W$ such that the $\up$-trace of $\Phi_{-\up}(\gamma_Z)$
  contains no intersections with the boundaries of homogeneity strips or tangential collisions.
  Then for any $Z\in \Delta_\up$, we have 
  $\Phi_{-s}(\gamma_Z) \in \cW^u$ for all $0 \le s \le \up$.
  \item[(iv)]    $m_W(W \setminus \Delta_\up) \le C \rho$.
  \item[(v)]  $C^{-1} \le \| \partial_{x^s}G \|  \le C$,
which implies  $|  \partial_{x^s}\partial_{x^u}H|_{\infty} \le C$ by (i).
  \item[(vi)]  $\partial_{x^u}\partial_{x^s}G\in \cC^0$ and $|  \partial_{x^u}\partial_{x^s}G|_{\infty} \le C\rho^{-4/5}$.
  \item[(vii)]     The following four-point estimate holds:
  \begin{align*}
|\partial_{x^s} G(x^u, x^s)- \partial_{x^s}G(x^u, y^s)
&- \partial_{x^s} G(y^u, x^s)+ \partial_{x^s} G(y^u, y^s) | \\
&\le C \rho^{-4/5 - 11\varpi/15}  
|x^u-y^u|^{1-7\varpi}|x^s- y^s|^\varpi
\, .
\end{align*} 
   \item[(viii)]   $| \partial_{x^s} H |_{\cC^0}\le C\rho$  and
  $| \partial_{x^s} H |_{\cC^\varpi} \le C\rho^{6/5-116\varpi/15}$.  
\end{itemize}
\end{theorem} 

The rest of the section is devoted to the proof of the above Theorem.
We will first (Subsections~ \ref{sec:step1}--\ref{sec:step4}) construct a foliation with the (analogue
of the) above properties in a neighborhood of
$P^+(W)$ in the phase space $\cM$ of the billiard map and then (Subsection~ \ref{sec:step5}) describe how to lift it to $W$
to get a foliation for the billiard flow.  We denote analogous objects for the map-foliation by
$\overline\bF$, $\barG$, $\bDelta_\up$, etc.

Before carrying out this construction, we mention a remark which will be useful below:

\begin{remark}[Taking the exponential of a four-point estimate]\label{rklog4pt}
We claim that if (ii), (v),  and (vi) hold, then to show (vii) it is enough to prove
(vii) for the function $\log \partial_s G$ (up to changing the constant factor).
Indeed, using (ii), first observe that (vii)  for $\log \partial_{x^s} G$ at $x^u=0$ gives
\begin{align*}
| \log \partial_{x^s} G(y^u, x^s)- \log \partial_{x^s}  G(y^u, y^s) | 
&\le C^2 \rho^{-4/5 - 11\varpi/15}  
|y^u|^{1-7\varpi}|x^s- y^s|^\varpi\\
&\le 2 C^2 \rho^{1/5 - 116\varpi/15)}  
|x^s- y^s|^\varpi 
\, ,
\end{align*}
since $|y^u| \le \rho$. 
Then note that for any $C_\infty >1$, there is a constant $C'$ depending
only on $C_\infty$ so that,  if a function $g(x^u, x^s)$ satisfies $|g|_{\infty} \le C_\infty$, and
$$
\sup_{x^u} |g(x^u, x^s)- g(x^u, y^s)|\le C_s|x^s- y^s|^\varpi\, ,
\quad \sup_{x^s} |g(x^u, x^s)- g(y^u, x^s)|\le C_u|x^u- y^u| \, , 
$$
and
$$
|g(x^u, x^s)- g(x^u, y^s)- g(y^u, x^s)+ g(y^u, y^s) |  \le C_{us} 
|x^u-y^u|^{1-7\varpi}|x^s- y^s|^\varpi \, , 
$$
with $C_u C_s \le C_\infty C_{us}$, then
$$
|\exp g(x^u, x^s)- \exp g(x^u, y^s)- \exp g(y^u, x^s)+ \exp g (y^u, y^s) |  \le C'
C_{us} 
|x^u-y^u|^{1-7\varpi}|x^s- y^s|^\varpi \, .
$$
Indeed, first write $\exp g=\sum_{k=0}^\infty g^k/(k!)$. Then note that an  easy induction  gives
  $|g(x^u, x^s)^k- g(x^u, y^s)^k|\le 2^{k-1} C_\infty^{k-1} C_s |x^s-y^s|^\varpi$ and 
  $|g(x^u, x^s)^k- g(y^u, x^s)^k|\le 2^{k-1} C_\infty^{k-1} C_u |x^u-y^u|$ for all $k\ge 1$.
  Finally,  the trivial formula
$$
  aa'-bb'-cc'+dd'=c(a'-b'-c'+d')+(a-b-c+d) d'+(a-c)(a'-b') + (a-b)(b'-d'),
  $$
  implies
  $$
| g(x^u, x^s)^k- g(x^u, y^s)^k-  g(y^u, x^s)^k+ g (y^u, y^s)^k |  \le 2^{k-1} C_\infty^{k-1}
C_{us} 
|x^u-y^u|^{1-7\varpi}|x^s- y^s|^\varpi \, ,
$$
for all $k\ge 1$. To conclude, use $\sum_{k \ge 0} 2^{k-1} C_\infty^{k-1} / k! = (2C_\infty)^{-1} e^{2C_\infty}$.

Now apply this bound to $g = \log \partial_{x^s} G$ with $C_\infty$ derived from (v),
$C_s = C^{\pm 1} \rho^{1/5- 116\varpi/15}$, $C_u = C^{\pm 1} \rho^{-4/5}$ from (vi),
and $C_{us} = C^{\pm 1} \rho^{-4/5 - 11\varpi/15}$.  Then  $C_uC_s \le C_\infty C_{us}$
whenever
$\varpi \le 1/35$. 
\end{remark}

\subsection{Constructing the regular part of the map-foliation,  (iii--iv)}\label{sec:step1}

We construct the analogue $\bDelta_\up$
for $P^+(W)$ of the domain $\Delta_\up$ by putting a transverse foliation on 
curves $P^+(W_i)$ for $W_i \in \cG_\up(W)$ (recall Definition~\ref{cG_t}) and pushing it forward to $P^+(W)$.
Let $W_i \in \cG_\up(W)$.  This fixes $n = n(W_i)$
such that $T^n(P^+(W_i)) \subset P^+(W)$.

Using (a), 
we choose a uniformly $\cC^2$ foliation of curves $\{ \ell \}$ lying in the unstable
cone for the map and transverse to $P^+(W_i)$ of length $|\ell| = k_W^2 \rho/|J^u_\ell T^n(z_\ell)|$,
where $z_\ell = P^+(W_i) \cap \ell$, $J^u_\ell T^n$ is the (unstable)
Jacobian along $\ell$, and
$k_W$ is the index of the homogeneity strip containing $P^+(W)$.  We call this foliation
$\{ \ell_z \}_{z \in P^+(W_i)}$ a {\em seeding foliation}.
Those curves $\ell$ on which $T^n$ is smooth (in particular, 
avoiding intersections with the boundaries of
homogeneity strips) create a transverse foliation of map-unstable curves $\bar \gamma = T^n(\ell)$ 
over the subset $\bDelta_\up$ of $P^+(W)$ such that $T^j(\ell)$ is a homogeneous unstable
curve for each $j =0, 1, \ldots, n$.
We now estimate the size of 
$\bDelta_\up$ in $P^+(W)$.

If $T^j(\ell)$ is a single homogeneous curve, then it has length 
\begin{equation}
\label{eq:step length}
|T^j(\ell)| 
= \tC_d^{\pm 1} k_W^2 \rho / J^u_{T^j(\ell)}T^{n-j}(T^j(z_\ell))\,  ,
\end{equation}
where
$$
\tC_d(j)=1+C_d \sup_{x,y} d(T^j(x), T^j(y))^{1/3}\, ,
$$
with $C_d$ is the distortion constant from Lemma~\ref{lem:distortion}.
Slightly abusing notation, we replace $C_d$ by $\max\{C_d,\max_j \tC_d(j)\}$, observing
that $\sup_{x,y,j}d(T^j(x), T^j(y))^{1/3}$ is bounded uniformly in $\rho$, $W$ and $\up$.
Consequently, $|\bar \gamma| = |T^n(\ell)| = C_d^{\pm 1} k_W^2 \rho$.

We note that by property (b) above regarding the homogeneity strip into which $W$ is allowed
to project, we have $k_W^2 \rho \le C k_W^{-3}$, so that there is no need to cut any curves
$T^n\ell$ that have not already been cut previously at time $1, \ldots n-1$, since they will automatically
lie in a bounded number of homogeneity strips and so enjoy bounded distortion with a 
(properly adjusted) distortion constant $C_d$.  
Since the
expansion in one step at $x$ is proportional to the reciprocal of $\cos \vf(Tx)$, and this in
turn is proportional to the square of the index of the
homogeneity strip containing $T(x)$ (see \cite[eq. (4.20)]{chernov book}), it follows
 that the unstable Jacobian at the last
step is proportional to $k_W^2$, in other words 
\begin{equation}\label{ridk}
k_W^2 / J^u_{T^{n-1}\ell}T(T^{n-1}z_\ell) \le C_d \, .
\end{equation}

The set of cuts appearing in  forward iterates of $\ell$ is the extended
singularity set
$\cS^{\bH}_n$ from Definition~\ref{singul}.  The growth
lemma  from \cite{chernov book} we shall apply below 
estimates the measure of the subset of $W_i$
for which the corresponding $\ell$ are cut by considering the iterate $j$ at which
the cut is made, and estimating the length of
$T^j(W_i)$ that has had its $T^j(\ell)$ cut.
For this, we claim that we need only look at the sets $\cS_1$, $\cS_{-1}$, and the
boundaries of the homogeneity strips (and not any of their iterates):
Notice that if a curve $S$ is in $\cS_1 \setminus \cS_0$, it has a
negative slope (is a stable curve).  Its image under $T$, if it does not
belong to $\cS_0$, belongs to $\cS_{-1} \setminus \cS_0$ and so has a
positive slope (is an unstable curve).  There are exactly two ``images'' of
curves in $\cS_1 \setminus \cS_0$, one coming
from each ``side'' of the curve which goes to a different scatterer.

We make cuts at time $j$ when either  (1) $T^j(\ell)$
crosses the boundary of a homogeneity strip;  or (2) $T^{j-1}(\ell)$ crosses a
curve in $\cS_1$ whose image belongs to $\cS_{-1} \setminus \cS_0$.
(These are the only two possible cases: If $T^{j-1}(\ell)$ crosses
a curve in $\cS_1$ whose image belongs to $\cS_0$, then in fact
$T^j(\ell)$ will cross the boundary of a homogeneity strip; and this is case
(1).)

For (1), the boundary of the homogeneity strip is horizontal
(in the global $(r,\varphi)$ coordinates), so the stable
curve $T^j(W_i)$ is uniformly transverse
to it. Thus the length of $T^j(W_i)$
that has its $T^j(\ell)$ cut is at most a constant times the length of
$T^j(\ell)$.   Here it is important that $T^j(W_i)$ is
transverse to the boundary of the homogeneity strips and not whether
$T^j(\ell)$ is, although in fact, they both are.

For (2), although $\cS_1$ is made of stable
curves (with negative slope), we do not cut
$T^{j-1}(\ell)$ when it lands across this curve (note that at time $j-1$ the
derivatives are still comparable along $T^{j-1}(\ell)$ --- there is no problem
with smoothness or distortion at this step).  We cut at the next step,
when $T^{j-1}(\ell)$ maps to $T^j(\ell)$, which now has two (or more)
connected components.  The subset of
$T^j(W_i)$ which is missing part of its foliation $\{ T^j(\ell) \}$ is the subset of
$T^j(W_i)$ whose $T^j(\ell)$ intersects this singularity curve in $\cS_{-1}$
(the image of the curve in $\cS_1$ at time $j-1$).  Again, it is the
uniform transversality of $\cS_{-1} \setminus \cS_0$ with the stable
curve $T^j(W_i)$ that guarantees that the length of $T^j(W_i)$ that has had
its $T^j(\ell)$ cut is at most a constant times the length of $T^j(\ell)$.

Thus, there exists a constant $C_0$, depending only on $C_d$
and the uniform transversality of stable curves and extended singularities (as explained above),
such that if the intersection of $T^j(\ell)$ and $T^j(P^+(W_i))$ lies at a distance at least
$C_0 \rho/ J^u_{T^j\ell} T^{n-j-1}(T^jz_\ell)$ from the endpoints of $T^j(P^+(W_i))$, then
$T^j(\ell)$ is not cut at time $j$, $j = 0, 1, \ldots n-1$.  (Recall that we do not cut at the last time step
$n$ when we arrive back at $P^+(W)$ because we assumed that
$P^+(W)$ is already a homogeneous stable curve.)

Setting $i=n-j$ and
letting $r_{-i}(z)$ denote the distance from $T^{-i}(z)$ to the nearest endpoint of the connected
homogeneous components of $T^{-i}(W)$ belonging to $\cG_i(W)$, we conclude
that if $T^{n-i}(\ell)$ is cut at time $i$, then $r_{-i}(z) \le C_0 \rho \Lambda_0^{-i+1}$, where
$z$ is the point of intersection of $T^{n-i}(\ell)$ and $T^{-i}(W)$, and $\Lambda_0 > 1$
is defined above \eqref{Lambda}.
Thus 
\[
P^+(W) \setminus \bDelta_\up \; \subseteq \; \bigcup_{i=0}^{n-1} \left\{ z \in P^+(W) : r_{-i}(z) \le C_0 \rho \Lambda^{-i+1} \right\} \, .
\]

Using the growth lemma \cite[Theorem 5.52 and Exercise 5.49]{chernov book}, there exist
constants $C>0$, $\lambda_0 <1$, such that for all $\ve > 0$,
\[
\begin{split}\label{lambda0def}
m_{P^+(W)}(r_{-i}(z) < \ve) & \le C (\lambda_0 \Lambda_0)^i m_{P^+(W)}( r_0(z) < \ve \Lambda_0^{-i}) 
+ C \ve |P^+(W)| \\
& \le 2C \ve \lambda_0^i + C \ve |P^+(W)| \,  .
\end{split}
\]
Applying this bound with $\ve = C_0 \rho \Lambda_0^{-i+1}$, and summing over $i$, we estimate,
\[
m_{P^+(W)} (P^+(W) \setminus \bDelta_\up) \le \sum_{i=0}^{n-1} C' \rho \Lambda_0^{-i+1} \lambda_0^{i} + C' \rho \Lambda_0^{-i+1}
|P^+(W)| \le C'' \rho \,   .
\]
This implies the analogue  for $\bDelta_\up$
of item (iv) of the foliation.

\subsection
{Smoothness of the regular part of the map-foliation,  (v--vi--vii)}  
\label{sec:step2}

We proceed to check the regularity of the parts of the foliation
$\{\ell_z \}_{z \in T^{-n}\bDelta_\up}$ which survive uncut until time $n$.

Let $\oW_i \subset P^+(W_i)$ denote the maximal subcurve of $P^+(W_i)$
so that $\{ T^n \ell_z \}_{z \in \oW_i} = \{ \bar\gamma_x \}_{x \in T^n\oW_i}$ 
is not cut.
Let $k_j$ denote the index of the homogeneity strip containing
$T^{n-j}(\oW_i)$.  Since the foliation is not cut at step $j$ and has length given by
\eqref{eq:step length}, there exists $C>0$ (depending only on distortion and uniform
transversality of the unstable cone with the boundaries of homogeneity strips) such that
for any $\ell$ so that $z_\ell\in \oW_i$, 
\begin{equation}
\label{eq:strip bound}
\frac{k_W^2 \rho}{J^u_{T^{n-j}\ell}T^{j}(T^{n-j}z_\ell)} \le C k_j^{-3}
\implies k_j^2 \le \left( \frac{C J^u_{T^{n-j}\ell}T^{j}(T^{n-j}z_\ell)}{k_W^2 \rho} \right)^{2/3} \,  .
\end{equation}

Now let $V_0 = T^n(\oW_i)\subset P^+(W)$ and let $V$ denote a  map-stable curve 
such that the surviving foliation $\{ \bar \gamma \}$  associated to
$\oW_i$ via $V_0$,
intersects every point of $V$.
Let $\bh_V:V_0 \to V$ denote the holonomy map defined by the foliation
$\{ \bar \gamma \}$.

Set $V^n = T^{-n}V$ and note that since the surviving foliation intersects every point of $V$, it follows
that $V^n$ lies in the same homogeneity strip as $\oW_i$ and the seeding foliation of 
unstable curves $\{ \ell \}$ intersects every point of $V^n$.  Let 
$\bh_{V^n}$ denote the holonomy map from $\oW_i$ to $V^n$.

The lemma below is a refinement\footnote{ The analogous time reversed counterpart
from \cite{chernov book} would be 
$|\ln J\bh_V(x)| \le C (d(x,\tx)^{1/3} + \phi_V(x,\tx))$.  Here, we leverage the fact that we
require the transverse foliation to survive for only $n$ steps (not for all time) in addition to the
fact that the surviving curves must be long in the length scale $\rho$, to improve the
exponent of $d(x,\tx)$ to 1, at the cost of having the constant depend on $\rho$.}
of \cite[Theorem~5.42]{chernov book}
for the Jacobian $J\bh_V$.  (We shall use it below to obtain claim (vi), after
relating the Jacobian with $\partial_{x^s}G$ in appropriate
coordinates.)

\begin{lemma}[Bounds on the Jacobian outside of the gaps]
\label{lem:jac holonomy}
 There exists $C>0$, independent
 of $W$, $\rho$, and $\up$, such that  for any $x \in V_0$,
\[
|\ln J\bh_V(x)| = \left| \ln \left(\frac{J^s_{V_0}T^{-n}(x)}{J^s_VT^{-n}(\tx)} J\bh_{V^n}(T^{-n}x) \right)\right| 
\le C \big(d(x,\tx) \rho^{-2/3} k_W^{-2} + \phi_V(x,\tx) \big) \, .
\]
where $\tx = \bh_V(x)$, $J^s_VT^{-n}$ denotes the (stable) Jacobian of $T^{-n}$ along $V$,
 and $\phi_V(x,\tx)$ is the angle between the tangent vector to $V_0$ at $x$ and 
the tangent vector to $V$ at $\tx$.
\end{lemma}

\begin{proof}
We begin by writing,
\[
\ln J\bh_V(x) = \ln J\bh_{V^n}(x_n) + \sum_{j=0}^{n-1} \ln J^s_{T^{-j}V_0}T^{-1}(x_j) - \ln J^s_{T^{-j}V}T^{-1}(\bx_j) \, ,
\]
where $x_j = T^{-j}(x)$ and $\tx_j = T^{-j}(\tx)$, $j = 0, 1, \ldots n$. 

Letting
$x_j=(r_j,\varphi_j)$,  $\cK_j=\cK(r_j)$, $\tau_j=\tau(x_j)$,
and similarly for $\tilde{\vf}_j$, $\widetilde{\cK}_j$ and $\widetilde{\tau}_j$,
we have the time reversal counterpart to \cite[eq. (5.24)]{chernov book},
\begin{equation}
\label{eq:jac factors}
\ln J^s_{T^{-j}V_0}T^{-1}(x_j) = - \ln \cos \vf_{j+1} + \ln (\cos \vf_j + \tau_{j+1}(\cK_j + |\cV_j|))
+ \tfrac12 \ln (1+ \cV_{j+1}^2) - \tfrac12 \ln (1+\cV_j^2) \, ,
\end{equation}
where $\cV_j < 0$ represents the slope of the tangent vector to $T^{-j}(V_0)$ at $x_j$.
Using the analogous expression for $\ln J^s_{T^{-j}V}T^{-1}(\tx_j)$ and letting $\ell_n$ denote the
curve of the initial seeding foliation containing $x_n$,
we first compare,
\[
\begin{split}
|\ln \cos \vf_{j+1} - \ln \cos \tilde{\vf}_{j+1}| & \le C \frac{|\vf_{j+1} - \tilde{\vf}_{j+1}|}{\cos \vf_{j+1}}
\le C d(x_{j+1}, \tx_{j+1}) \, k_{j+1}^2 \\
& \le C \frac{d(x, \tx)}{J^u_{T^{n-j-1}\ell_n}T^{j+1}(x_{j+1})} 
\frac{(J^u_{T^{n-j-1}\ell_n}T^{j+1}(x_{j+1}))^{2/3}}{k_W^{4/3} \rho^{2/3}} \\
& \le C \frac{d(x,\tx)}{k_W^{4/3} \rho^{2/3} (J^u_{T^{n-j-1}\ell_n}T^{j+1}(x_{j+1}))^{1/3}}
\le C  \frac{d(x,\tx)}{\rho^{2/3} k_W^2 \Lambda_0^{j/3}} \, ,
\end{split}
\]
where we have used \eqref{eq:strip bound} in the second line, and $\Lambda_0$
was defined above (\ref{Lambda}). 
(We refer to the explanations above \eqref{ridk} regarding the factor $k_W^2$.) 
The other terms appearing
in \eqref{eq:jac factors} are bounded and differentiable functions of their arguments, except
possibly $\tau_j$, but for this we have
$$|\tau_{j+1} - \widetilde{\tau}_{j+1}| \le C( d(x_j, \tx_j) + d(x_{j+1}, \tx_{j+1})) \le C' d(x_j, \tx_j)$$
(see \cite[eq. (5.28)]{chernov book}).
This combined with the previous estimate yields,
\[
|\ln J^s_{T^{-j}V_0}T^{-1}(x_j) - \ln J^s_{T^{-j}V}T^{-1}(\tx_j)| \le C(d(x, \tx) \Lambda_0^{-j/3}(\rho^{-2/3} k_W^{-2} + 1) +
\phi_j + \phi_{j+1}) \, ,
\] 
where $\phi_j$ is the angle formed by the tangent vectors to $T^{-j}V_0$ 
at $x_j$ and $T^{-j}V$
 at $\tx_j$ (in particular $\phi_0=\phi(x,\tx)$.

Now it follows from \cite[eq. (5.29)]{chernov book} that 
\[
\phi_j \le C(d(x, \tx) j \Lambda_0^{-j} + \phi_0 \Lambda_0^{-j}) \, ,
\]
and summing over $j$ completes the required estimate on the difference of Jacobians.

Finally, we estimate $\ln J\bh_{V^n}(x_n)$.  Since the holonomy $\bh_{V^n} : \oW_i \to V^n$
corresponding to the seeding foliation is uniformly $\cC^2$, 
its Jacobian is a $\cC^1$ function of the distance and angle between the two curves.
Thus using again \cite[eq. (5.29)]{chernov book},
\[
|\ln J\bh_{V^n}(x_n)| \le C (\phi(x_n, \tx_n) + d(x_n,\tx_n)) \le C ( d(x, \tx) n \Lambda_0^{-n} + \phi(x, \tx) \Lambda_0^{-n}) \,  . 
\]
Combining this with our previous estimate completes the proof of Lemma~\ref{lem:jac holonomy}.
\end{proof}

We adopt smooth local coordinates $(\bxu,\bxs) \in \cM$ to write the foliation 
$\bar \gamma = (\bxu, \barG(\bxu, \bxs))$, where $\bxu$ and $\bxs$ range over 
the relevant intervals.  The curve $V_0$ corresponds to $\{ (0, \bxs) : \bxs \in [0,|V_0|] \}$,
and (this is the first condition of (ii) for the map-foliation)
\begin{equation}\label{cond(1)BLu}
\barG(0,\bxs) = \bxs\, .
\end{equation}
This is the canonical
straightening map of the foliation $\barG$, with respect to the point
$(0,0)$, i.e., the map
sending the horizontal lines of $\bR^2$ on the leaves of the foliation,
preserving vertical lines and equal to the identity on $\{0\} \times \bR$. In other words,
\[\bar \bF(\bxu,\bxs)=(\bxu, \barG(\bxu,\bxs))\]
where $\barG(\bxu,\bxs)$ is the first component of the point on the vertical
through $\bxu$
on the leaf going through $(0,\bxs)$, 
with $\barG(0,\bxs)=\bxs$. Geometrically,
$\bxs \mapsto \barG(\bxu, \bxs)$ is the holonomy
$\bh_{0 \to \bxu}$   of the foliation
between the vertical transversals at  $0$ and $\bxu$, so that
$\partial_{\bxs} \barG(\bxu,\bxs)$
is the Jacobian of this holonomy between verticals. 
The
image $U$ under $D\bar \bF$ of the horizontal vector field  is tangent
to the foliation, and has first component equal to one, i.e.,
$U(\bxu,\bxs)=(1, u(\bxu,\bxs))$, with $u(\bar \bF(\bxu, \bxs))=
\partial_{\bxu} \barG(\bxu,\bxs)$.

  We next use
Lemma~\ref{lem:jac holonomy}  to derive the
smoothness properties (v)--(vi) of this foliation.
Write $V = \{ ( v(\bxs),\bxs) \}$ for some smooth function $v$. 
Taking $V = V_c$ in  Lemma~\ref{lem:jac holonomy}  to be the (purely vertical) curve defined by 
$v(\bxs) \equiv c$ for some constant $c$
(which we can do since the map-unstable cones are globally defined, i.e.,~they have a uniform
width), we have by the above discussion,
\begin{equation}
\label{magicc}
J\bh_{V_c}(0,\bxs) = \partial_{\bxs} \barG(c, \bxs )\, .
\end{equation}
Since $\tx=\bh_V(x)$ implies $d(x,\tx) \le C k_W^2 \rho$, we have 
\begin{equation}
\label{eq:uniform}
d(x,\tx) \rho^{-2/3} k_W^{-2} \le C \rho^{1/3} \, .  
\end{equation}
Thus 
Lemma~\ref{lem:jac holonomy} implies that
$\partial_{\bxs} \barG$ is bounded independently of $\rho$ 
(both above away from infinity and below away
from $0$), proving the analogue of (v) for $P^+(W)$.  

The next lemma establishes the analogue of (vi) for $\barG$.

\begin{lemma}
\label{lem:c}
$\partial_{\bxu} \partial_{\bxs} \barG \in \cC^0$ and
there exists $C>0$ such that
$|\partial_{\bxu} \partial_{\bxs} \barG|_\infty  \le C\rho^{-2/3} k_W^{-2}$.
\end{lemma}

\begin{proof}
Since we have chosen a smooth change of coordinates, angles in $(r,\vf)$-coordinates are transformed smoothly to 
$(\bxu, \bxs)$-coordinates.  Thus for the family of vertical stable curves $V_c$,
we have $\phi_{V_c}(x,\tx)$ varying as a smooth function of $d(x,\tx)$.

We have
$$\partial_{\bxu} \partial_{\bxs} \barG(c,\bxs) = \partial_c J\bh_{V_c}(0,\bxs) \, .
$$  

First we prove the bound on $|\partial_{\bxu} \partial_{\bxs} \barG|_\infty$.
As before, let $V^n_c = T^{-n}V_c$, and let $\bh_{V^n_c}$ denote the holonomy along
$\{ \bar \gamma \}$ from $V^n_0$ to $V^n_c$.  Note that 
\begin{equation}\label{use4}
\frac{J\bh_{V^n_{c+\delta}}}{J\bh_{V^n_c}} = J\bh_{V^n_c, V^n_{c+\delta}}\, , 
\end{equation} where $\bh_{V^n_c, V^n_{c+\delta}}$ is the holonomy map from
$V^n_c$ to $V^n_{c+\delta}$.
For each $\bxs$,
\[
\begin{split}
|J\bh_{V_{c+\delta}}(0,\bxs) & - J\bh_{V_c}(0,\bxs)|
 = \left| \frac{J^s_{V_0}T^{-n}(0,\bxs)}{J^s_{V_{c+\delta}}T^{-n}(\bh_{V_{c+\delta}}(0,\bxs))}
J\bh_{V^n_{c+\delta}}(T^{-n}(0,\bxs)) \right.  \\
& \qquad \qquad \qquad \qquad \left. - \frac{J^s_{V_0}T^{-n}(0,\bxs)}{J^s_{V_c}T^{-n}(\bh_{V_c}(0,\bxs))} J\bh_{V^n_c}(T^{-n}(0,\bxs)) \right| \\
& = \frac{J^s_{V_0}T^{-n}(0,\bxs)}{J^s_{V_c}T^{-n}(\bh_{V_c}(0,\bxs))} J\bh_{V^n_c}(T^{-n}(0, \bxs)) \\
& \qquad \times \left| \frac{J^s_{V_c}T^{-n}(\bh_{V_c}(0,\bxs))}{J^s_{V_{c+\delta}}T^{-n}(\bh_{V_{c+\delta}}(0,\bxs))}
J\bh_{V^n_c, V^n_{c+\delta}}(T^{-n}(\bh_{V_c}(0,\bxs))) -1 \right| \\
& \le C \delta \rho^{-2/3} k_W^{-2} \,  ,
\end{split}
\]
where in the last line, we have used the smooth dependence of $\phi_{V_c}(x,\bx)$
on $d(x,\bx)$ and replaced the factor $d(x,\bx)$ from 
Lemma~\ref{lem:jac holonomy}
by $\delta$ since the estimates of that lemma hold for the holonomy between 
any two stable curves connected by 
the foliation $\{ \bar \gamma \}$. (To bound
the first factor via Lemma~\ref{lem:jac holonomy},
we also used that $|c|\le k^2_W \rho$ so that
$1+c\rho^{-2/3}k_W^{-2}\le C$.)
Thus $|\partial_{\bxu} \partial_{\bxs} \barG|_\infty \le C \rho^{-2/3} k_W^{-2}$,
proving the required bound.

To see that in fact $\partial_{\bxu} \partial_{\bxs} \barG$ is continuous, write
\begin{equation}
\label{eq:second partial}
\partial_{\bxu} \partial_{\bxs} \barG(c,\bxs) 
= \partial_c J\bh_{V_c}(0,\bxs) =  \partial_c 
\left(\frac{J^s_{V_0}T^{-n}(0,\bxs)}{J^s_{V_c}T^{-n}(\bh_{V_c}(0,\bxs))} J\bh_{V_c^n}(T^{-n}(0,\bxs)) \right) \,  .
 \end{equation}
Since the seeding foliation is uniformly $\cC^2$, the Jacobian $J\bh_{V^n_c}$ is a $\cC^1$
function of its base point and the angle $\phi(x_n, \tx_n^c)$ between the stable curves $V_0^n$ and 
$V_c^n$ at the points $x_n = T^{-n}(0,\bxs)$ and $\tx_n^c = T^{-n}(\bh_{V_c}(0,\bxs))$, respectively.  
Thus we need only consider the dependence of $\phi(x_n, \tx_n^c)$ on $c$.  

Since we have adopted a smooth change of coordinates, the angle 
$\phi(x_n, \tx_n^c)$ is a smooth function of the difference in slopes
$\cV(x_n)- \cV(\tx_n^c)$ of the stable curves $V^n_0$ and $V^n_c$ in 
$(r,\vf)$ coordinates, at the points $x_n$ and $\tx_n^c$, respectively.
By the time reversal of \cite[eq.~(3.39)]{chernov book} we have $\cV = \cB^- \cos \vf - \cK$, 
where $\cB^-$ is the curvature of the wavefront just after impact (under the backwards flow).
Thus,
\begin{equation}
\label{eq:c deriv}
\partial_c (\cV(x_n) - \cV(\tx_n^c)) = - \partial_c \cK(\tx_n^c) + \partial_c \cos \vf(\tx_n^c) \cB^-(\tx_n^c)
+ \cos \vf(\tx_n^c) \partial_c \cB^-(\tx_n^c)\,  .
\end{equation}
Moreover, the curvature of the wave front $\cB^-(\tx_n^c)$ can be expressed 
recursively as 
\begin{equation}
\label{eq:recursive}
\cB^-(\tx^c_n) = - \frac{2 \cK(\tx^c_n)}{\cos \vf(\tx^c_n)} 
+ \frac{\cB^-(\tx^c_{n-1})}{1 - \tau(\tx^c_n) \cB^-(\tx^c_{n-1})}
\end{equation}
where $\tx_j^c = T^{-j}(\bh_{V_c}(0,\bxs))$.
Since our scatterers are $\cC^3$, $\cK$ is $\cC^1$ so 
every term appearing above is a locally $\cC^1$ function of its argument 
composed with
$T^{-j}$ for some $j = 1, \ldots, n$, except possibly $\cB^-(\tx^c_{n-1})$.  
Note that this is also true of $\tau$, even though the derivative of $\tau(x)$ has
a term involving $\cK(Tx)/\cos(Tx)$.  To deal with $\cB^-(\tx^c_{n-1})$, simply apply
relation \eqref{eq:recursive} recursively with the chain rule to deduce that $\cB^-(\tx^c_n)$
is continuously differentiable with respect to $c$ if $\cB^-(\tx^c_0)$ is.  But this last statement
obviously holds since the curves $V_c$ are vertical segments in $(\bxu, \bxs)$ coordinates
and the change of coordinates is smooth.

Differentiating with respect to $c$ in \eqref{eq:c deriv}
and using the chain rule together with the above observations
yields a continuous function
(which we need not compute) times the unstable Jacobians $J^u_{\bgamma_{x_0}}T^{-j}(\bh_{V_c}(x_0))$, where $\bgamma_{x_0}$ is the unstable curve of the surviving foliation
passing through $x_0$.  This last factor is also continuous in $c$ since $T$ is locally $\cC^2$.

Next we turn to the factor $J^s_{V_c}T^{-n}(\bh_{V_c}(0,\bxs))$ in \eqref{eq:second partial}.
This is simply the product of factors of the form,
\[
J^s_{T^{-j}V_c}T^{-1}(\tx_j^c) = \frac{\cos \vf(\tx_j^c) + \tau(\tx_j^c)(\cK(\tx_j^c) - \cV(\tx_j^c))}{\cos \vf(\tx_{j+1}^c)} \sqrt{\frac{1 + \cV(\tx_{j+1}^c)^2}{1+\cV(\tx_j^c)^2}}, \qquad j = 0, 1, \ldots n-1\, ,
\] 
using \eqref{eq:jac factors}.  All functions
appearing here are again (locally) smooth functions of their arguments so that
as above, differentiating with respect to $c$ and using \eqref{eq:c deriv} once more for 
$\partial_c \cV(\tx_j^c)$, yields a sum of continuous functions times unstable Jacobians of the form
$J^u_{\bgamma_{x_0}}T^{-j}(\bh_{V_c}(x_0))$, which are again continuous.
This completes the verification of the continuity of $\partial_{\bxu} \partial_{\bxs} \barG(c,\bxs)$
and thus the proof of Lemma~\ref{lem:c}.
\end{proof}

Although the Jacobians of holonomy maps corresponding to billiards are 
not necessarily H\"older continuous (see \cite[p.~124]{chernov book}), 
we shall next see
that the holonomy we have
constructed here does have H\"older continuous Jacobian,  leveraging the
fact that the surviving foliation must have length $\rho$, as in the proof
of Lemma~\ref{lem:jac holonomy}.

\begin{lemma}[H\"older continuity of the Jacobian of surviving foliation]
\label{lem:holder jac}
Let $\{ \bar\gamma \}$ be the surviving foliation transversal to $V_0$ introduced
just before Lemma~\ref{lem:jac holonomy}.
Fix $C'>0$ and consider the set of stable curves 
$V_1$, $V_2 \in \cW^s$ which are connected by the
transverse foliation $\{ \bar\gamma \}$ and such that the angles between $V_1$
and $V_2$ satisfy, $\phi(x, \bh_{12}(x)) \le C' d(x, \bh_{12}(x))$
for all $x \in V_1$, where $\bh_{12}$ is the holonomy map from $V_1$ to $V_2$.
We require that $V_0$ be in this family.
For any $0<\varpi \le 1/15$, there exists $C>0$ such that given two such  $V_1$, $V_2 \in \cW^s$, 
we have
\[
\begin{split}
&\left| \ln \frac{J\bh_{12}( v_1(\bxs),\bxs)}{J\bh_{12}( v_1(\bys),\bys)} \right|\\
& \le 
C d(\bxs, \bys)^\varpi d(V_1, V_2)^{1-\varpi} \max\{ k_W^{-2(2+5\varpi)/3} \rho^{-2(1+\varpi)/3}, 
k_W^{-82\varpi/3} \rho^{-38\varpi/3} \} \, ,
\end{split}\]
where
$V_1$ is the graph of the function $v_1$, and $\bxs, \bys$ are 
in the domain $[0,|V_0|]$ of $v_1$.
\end{lemma}

The  proof of the above lemma is  in Appendix~\ref{postponeproof}.
In view of \eqref{magicc}, \eqref{use4}, and Remark~\ref{rklog4pt}, Lemma~\ref{lem:holder jac} gives the analogue of the
four-point condition (vii) for the
billiard map. 

\subsection
{Interpolating the map-foliation across the gaps,  (vi)}  
\label{sec:step3}

We will fill in the gaps in the surviving foliation by interpolation, to obtain a 
full foliation in a $k_W^2 \rho$ neighborhood of $P^+(W)$. 
More precisely, we will first obtain bounds on the Lipschitz norm of the tangent vector $\partial_u \bar F$
to the foliation across the gaps. Interpolating
linearly between the Jacobians, following
\cite[App. D]{BaL} (see \eqref{formint})  these bounds will then give the analogue
of (vi)  for the map,  while the analogue of (vii) for the map will be discussed
in Section~\ref{sec:step4}.

First,
if there are gaps at the endpoints of $P^+(W)$, we may trivially extend the surviving foliation 
smoothly across these gaps since no interpolation is necessary.

Gaps are created by the intersection of $T^{n-j}(\ell)$ with a singularity curve 
or\footnote{ We use implicitly here that $T^{-j}(P^+(W))$ crosses an extended
singularity line  only if there exists $\ell$ as constructed in
Step 1 so that  $T^{n-j}(\ell)$ crosses
a singularity line.}
with boundaries of homogeneity strips,
since to not be cut under the original scheme means to lie in a
single homogeneity strip.  
Some gaps are created by an intersection with the boundary
of a single homogeneity strip.  If we choose not to cut those curves $T^{n-j}(\ell)$ which cross these two homogeneity strips, we will just be considering
unstable curves that lie in two adjacent strips.  The
expansion and contraction factors along these curves still enjoy bounded distortion,
with a slightly larger distortion constant.
We fill these gaps by 
not cutting the transverse foliation at time $n-j$.

The parts of the foliation that cross more than two strips will belong to
the gap near $\cS_0$ for which all curves are cut. For that gap, we linearly
interpolate between the surviving curves at each end of the gap at time $0$, and for
this we shall estimate the Lipschitz constant of the interpolation.

If $T^{-j}(P^+(W))$ lands in a neighborhood of $\cS_0$
it may cross
countably many homogeneity strips.  Since we enforce a definite length
scale for the transverse foliation, this means that the foliation is cut
in all homogeneity strips above a certain index (depending on the length
of $T^{n-j}(\ell)$).  So we can consider this to be a single gap in our
construction (even though $T^{-j}(P^+(W))$ is subdivided into countably
many homogeneous curves).

With this construction, since $\up>0$ is fixed, there are finitely many gaps in the foliation of
$P^+(W)$, each gap containing one or more intersections of $P^+(W)$ with a 
curve in $\cS_{-n}$, $n \le \lfloor \frac{\up}{\tau_{\min}} \rfloor + 1$.

Now fix one gap and choose a segment $V_0$ of $P^+(W)$ with surviving foliation on
one side of the gap.  This fixes $n = n(V_0)$ as the number of iterates appearing in the
definition of the foliation on $V_0$.  Let $j$ be the least integer $j' \ge1$
such that an element of $T^{j'}(\cS_0)$ intersects $P^+(W)$ in the gap.  
(Note that $j \le n$ must exist otherwise there would be no gap.)  Thus
$V_0$ is contained in a longer curve $V_1 \subset P^+(W)$ 
such that $T^{-m}(V_1)$ is a homogeneous stable curve for $m = 0, 1, \ldots, j-1$ and
which is cut for the 
first time at time $-j$.  We prove (vi) for the interpolated foliation by considering two cases.

\smallskip
\noindent
{\em Case 1.}  $T^{-j}(V_1) \cap \cS_0 \neq \emptyset$.  This implies that $T^{-j}(V_1)$
crosses countably many homogeneity strips ($T^{-j}(V_1)$ is not assumed to
be homogeneous).  At this step, each $\bar\gamma$ intersecting $V_0$
is such that $T^{-j}(\bar\gamma)$ is a homogeneous unstable curve
of length $k_W^2 \rho/ J^u_{T^{-j}(\bar\gamma)}T^j(z)$, for some $z \in T^{-j}(\bar\gamma)$.
This implies that the index $k_j$ of the homogeneity strip containing this curve
satisfies 
\begin{equation}
\label{eq:new(2)}
k_j^{-3} = C^{\pm 1} \frac{k_W^2 \rho}{J^u_{T^{-j}(\bar\gamma)}T^j(z)}\, ,
\end{equation}
for some uniform constant $C$ and any $z \in T^{-j}(\bar\gamma)$.
(Compare with (\ref{eq:strip bound}).)
Thus the length of the image of the gap in $T^{-j}(V_1)$ is of order $k_j^{-2}$.
Since this is the first time $V_1$ is being cut, $T^{-j+1}(V_1)$ is still a homogeneous
stable curve so, using the fact that the expansion from $T^{-j+1}(V_1)$ to $T^{-j}(V_1)$ in the
gap is of order $k_j^2$,
we conclude that for any $x \in V_1$ the length of the gap in $V_1$ is
\begin{equation}
\label{eq:gap length}
|V_1^{\gap}| = C^{\pm 1} k_j^{-4} /J^s_{V_1}T^{-j+1}(x) \,  .
\end{equation}

In order to estimate the Lipschitz constant of the foliation with which we want to fill in the gap,
we must calculate the largest possible angle between the edge of the surviving foliation
in $V_0$ and the part of $T^j(\cS_0)$ crossing $P^+(W)$.  Note that $T^j(\cS_0)$ is necessarily
an unstable curve, and indeed, the tangent vector to $\cS_0$ is mapped
into the unstable cone in one step.

Using \cite[eq. (3.39)]{chernov book}, we see that the image under $DT$ of an unstable
curve has slope $\cV_1$ (in $(r,\vf)$ coordinates) given by
\[
\cV_1 = \frac{\cos \vf_1}{\tau_0 + \frac{\cos \vf_0}{\cV_0 + \cK_0}} + \cK_1 \, ,
\]
where objects with subscript 0 correspond to a base point $x$ and objects with subscript
1 correspond to $T(x)$.

Let $x = \cS_0 \cap T^{-j}V_1$.  Now since the slope $\cV_0(x) = 0$ and 
$\vf_0(x) = \pi/2$, the slope of the image of this vector after one step is given by
\[
\cV_1(x) = \cK_1 + \frac{\cos \vf_1}{\tau_0} \,  .
\]

Let $y$ denote the endpoint of $T^{-j}(V_0)$ adjacent to the gap and let $\bar\gamma$
denote the unstable curve in the constructed foliation containing $T^j(y)$.  
Recall that $\bar\gamma = T^n(\ell)$ for some curve $\ell$ of the unstable seeding foliation
of a neighborhood of a curve $W_i \in \cG_\up(W)$; thus we do not make
any assumption on the slope of $T^{-j}\bar\gamma = T^{n-j}(\ell)$, other than
that it lies in the unstable cone at $y$. 

Letting $\phi^{\bar \gamma}(T(x), T(y))$ denote
the angle between $T(\cS_0)$ and $T^{-j+1}(\bar\gamma)$ at $T(x)$ and $T(y)$ respectively, we estimate,
\[
\begin{split}
\phi^{\bar \gamma}(Tx,Ty) & \le |\cV_1(x) - \cV_1(y)|
\le |\cK_1(x) - \cK_1(y)| + \left| \frac{\cos \vf_1(x)}{\tau_0(x)} 
- \frac{\cos \vf_1(y)}{\tau_0(y) + \frac{\cos \vf_0(y)}{\cV_0(y) + \cK_0(y)}} \right| \\
& \le  |\cK(Tx) - \cK(Ty)| + \frac{1}{\tau_{\min}} | \cos \vf(Tx) - \cos \vf(Ty)| \\
& \; \; \; \; \; + \cos \vf(Ty) \left| \frac{1}{\tau(x)} - \frac{1}{\tau(y)} \right| 
 + \cos \vf(Ty) \left| \frac{1}{\tau(y)} - \frac{1}{\tau(y) + \frac{\cos \vf_0(y)}{\cV_0(y) + \cK_0(y)}} \right| 
 \, .
\end{split}
\]
The first two differences above are bounded by a uniform constant times $d(T(x),T(y))$, which is of order
$k_j^{-4}$.  As before, $|\tau(x) - \tau(y)| \le C(d(x,y) + d(T(x),T(y))) \le C' d(x,y) \le C'k_j^{-2}$. 
Finally, the last difference is of order $\cos \vf(y) = Ck_j^{-2}$.  Putting these estimates together,
we obtain 
\begin{equation}
\label{eq:angle diff}
\phi^{\bar \gamma}(T(x),T(y)) \le C k_j^{-4} + C k_j^{-2} \cos \vf(Ty)\,  .
\end{equation}

The quantity to be estimated is $\frac{\phi^{\bar \gamma}(T^j(x), T^j(y))}{d(T^j(x), T^j(y))}$.  This will yield the 
Lipschitz constant of our interpolated foliation in the gap.
To do this, we note that
\cite[eq. (5.27) and following p.~122]{chernov book} imply
\[
\begin{split}
\phi^{\bar \gamma}(T^j(x), T^j(y)) &= Q_j + \cos \vf(T^j(y)) Q_{j-1} \\
&\qquad\qquad+ \tfrac{\cos \vf(T^j(y))}{\cos \vf(T^{j-1}(y))}
\tfrac{\phi^{\bar \gamma}(T^{j-1}(x), T^{j-1}(y))}{(1+\tau(T^{j-1}(y)) B^+(T^{j-1}y))(1 + \tau(T^{j-1}x) B^+(T^{j-1}(x)))} \, ,
\end{split}
\]
where $B^+$ denotes the post-collisional curvature of the unstable wavefront and
the functions $Q_j, Q_{j-1}$ satisfy $|Q_j| \le C d(T^j(x), T^j(y))$ and $|Q_{j-1}| \le C d(T^{j-1}(x), T^{j-1}(y))$.
Proceeding inductively and recalling the
expressions for the stable and unstable Jacobians given by \eqref{eq:jac factors} 
and \eqref{eq:translate factors} , we obtain the following expression for the difference in angles,
\begin{equation}
\label{eq:angle expand}
\begin{split}
\phi^{\bar \gamma}(T^jx, T^jy) & \le C \sum_{m=0}^{j-2} \tfrac{d(T^{j-m}x, T^{j-m}y) + \cos \vf(T^{j-m}y) d(T^{j-m-1}x, T^{j-m-1}y)}{J^u_{T^{m-j}\bar\gamma}T^m(T^{j-m}y) J^s_{V_0}T^{-m}(T^jy)} \\
&\qquad\qquad+ \tfrac{\phi^{\bar \gamma}(Tx,Ty)}{J^u_{T^{1-j}\bar\gamma}T^{j-1}(Ty) J^s_{V_0}T^{-j+1}(T^jy)}\,  .
\end{split}
\end{equation}
For each $m$ in the sum we have
the following two bounds: For the first term in the numerator, 
$d(T^{j-m}(x), T^{j-m}(y)) / J^s_{V_0}T^{-m}(T^jy) \le C d(T^j(x), T^j(y))$.
For the second term in the numerator, there is a gap in the expansion factors, so we need to 
add the stable Jacobian,  $J^s_{T^{-m}V_0}T^{-1}(T^{j-m}(y))$ as follows:
\[
\frac{\cos \vf(T^{j-m}y) d(T^{j-m-1}x, T^{j-m-1}y)}{J^s_{V_0}T^{-m}(T^jy)}
\le \cos \vf(T^{j-m}y) d(T^j x, T^j y) J^s_{T^{-m}V_0}T^{-1}(T^{j-m}y)\,  .
\]
Now since $T^{j-m-1}(y) \in \bH_{k_{m+1}}$, we have
$J^s_{T^{-m}V_0}T^{-1}(T^{j-m}y) \le C k_{m+1}^2$.  
Also, in preparation for using \eqref{eq:new(2)}, we note that
\[
\begin{split}
J^u_{T^{-m-1}\bar\gamma}T^{m+1}(T^{j-m-1}(y)) &= J^u_{T^{-m}\bar\gamma}T^m(T^{j-m}y) J^u_{T^{-m-1}\bar\gamma}T(T^{j-m-1}(y))\\
& \le C J^u_{T^{-m}\bar\gamma}T^m(T^{j-m}y) k_m^2 \, .
\end{split}
\]
Putting these facts together and using \eqref{eq:new(2)}, we estimate
\[
\begin{split}
&\frac{\cos \vf(T^{j-m}(y)) d(T^{j-m-1}(x), T^{j-m-1}(y))}{J^u_{T^{m-j}\bar\gamma}T^m(T^{j-m}(y)) J^s_{V_0}T^{-m}(T^j(y))} 
 \le \frac{\cos \vf(T^{j-m}y) d(T^jx, T^jy) J^s_{T^{-m}V_0}T^{-1}(T^{j-m}y)}
{J^u_{T^{m-j}\bar\gamma}T^m(T^{j-m}y)} \\
&\qquad\quad\le C \frac{k_m^{-2}d(T^jx, T^jy) }
{J^u_{T^{m-j}\bar\gamma}T^m(T^{j-m}y)}  k_{m+1}^2
\; \le \; C \frac{k_m^{-2} d(T^j(x), T^j(y))}{J^u_{T^{m-j}\bar\gamma}T^m(T^{j-m}y)}
\frac{(J^u_{T^{m-j}\bar\gamma}T^m(T^{j-m}y))^{2/3} k_m^{4/3}}{k_W^{4/3} \rho^{2/3}} \\
&\qquad\quad \le C d(T^jx, T^jy) \Lambda_0^{-m/3} k_W^{-4/3} \rho^{-2/3} \,  .
\end{split}
\]
Collecting this bound together with  \eqref{eq:gap length},
 \eqref{eq:angle diff}, and \eqref{eq:angle expand}, we obtain 
 (note that $|V_1^{\gap}| = d(T^j(x), T^j(y))$),
\[
\begin{split}
\frac{\phi^{\bar \gamma}(T^j(x), T^j(y))}{d(T^j(x), T^j(y))}
& \le C \sum_{m=0}^{j-2} \frac{k_W^{-4/3}}{\Lambda_0^{m/3} \rho^{2/3}} + C \frac{k_j^{-4} + \cos \vf(Ty) k_j^{-2}}{J^u_{T^{1-j}\bar\gamma}T^{j-1}(Ty) J^s_{V_0}T^{-j+1}(T^jy)} \cdot \frac{J^s_{V_0}T^{-j+1}(T^jy)}{k_j^{-4}} \\
& \le C k_W^{-4/3} \rho^{-2/3} + \frac{C}{J^u_{T^{1-j}\bar\gamma}T^{j-1}(Ty)} 
+ \frac{k_j^2 \cos \vf(Ty)}{J^u_{T^{1-j}\bar\gamma}T^{j-1}(Ty)} \,  .
\end{split}
\]
Now using \eqref{eq:new(2)}, we bound the last term by
\[
\frac{k_j^2 \cos \vf(Ty)}{J^u_{T^{1-j}\bar\gamma}T^{j-1}(Ty)}
\le \frac{\cos \vf(Ty) (J^u_{T^{-j}\bar\gamma}T(y))^{2/3}}{(J^u_{T^{1-j}\bar\gamma}T^{j-1}(Ty))^{1/3}
k_W^{4/3} \rho^{2/3}}  \, .
\]
Finally using the fact that $J^u_{T^{-j}\bar\gamma}T(y)) \le C/\cos \vf(Ty)$, we conclude
\begin{equation}\label{Lipbd}
\frac{\phi^{\bar \gamma}(T^jx, T^jy)}{d(T^jx, T^jy)}
\le C + Ck_W^{-4/3} \rho^{-2/3} \, .
\end{equation}
This ends the proof of (vi) for the interpolation in Case~1.

\smallskip
\noindent
{\em Case 2.} $T^{-j}(V_1) \cap \cS_0 = \emptyset$.  This implies that $T^{-j+1}(V_1)$
lies on one side of a singularity curve in $\cS_{-1}$, and that one endpoint of
$T^{-j}(V_1)$ lies on a singularity
curve in $\cS_1$, but not on $\cS_0$.  If $V_1$ is such a curve containing a gap,
then the other side of the gap necessarily contains a curve satisfying the condition
of Case 1 above. Since the estimate of the angle difference carried out there was a
worst-case estimate (in fact, we started with an angle outside the unstable cone, while
in Case 2, the angle of the singularity curve in $\cS_{-1}$ will necessarily be inside the
unstable cone), the interpolated foliation across the entire gap can only have
a better Lipschitz constant since we are increasing the distance we have to carry out
the interpolation while not making the difference in angles any larger.

\smallskip
Other curves in $\cS_{-n}$ may also cross the same gap.  But these additional cuts can only
make the gap longer, while again making the difference in angles no worse.
Thus they do not make the Lipschitz constant any larger than the bound derived in Case 1.

\smallskip

We next show that the Lipschitz control
\eqref{Lipbd} that we obtained gives the $L^\infty$ bound (vi) 
on the foliation,
essentially following the construction from Appendix D of 
\cite{BaL} (the main difference is that we shall not use induction).
\smallskip

Let $(0, \bxs)$, $\bxs \in [a',b']$, denote the stable curve $P^+(W)$
and the interval $[a,b]$, $a' < a < b < b'$, denote the interval on which the foliation 
must be interpolated.
The surviving foliation on both sides of the gap is given by
$\bgamma_a(\bxu, \bxs) = (\bxu, \barG_a(\bxu, \bxs))$ and
$\bgamma_b(\bxu, \bxs) = (\bxu, \barG_b(\bxu, \bxs)$.
Next, fix $\bvf \in \cC^2([0,1], [0,1])$ satisfying
$\bvf(0) = \bvf'(0) = \bvf'(1) = 0$ and $\bvf(1) =1$.  
Define $\overline{\psi} = (1-c_1)\bvf'$, where
$c_1$ will be chosen later small enough to ensure the interpolated foliation comprises unstable
curves.  Then $\bpsi(0)=\bpsi(1)=0$ and  $\int_0^1 \overline{\psi} = 1-c_1$.

Define $\vf(\bxs) = \bvf (\frac{\bxs -a}{b-a})$ and $\psi(\bxs) = \overline{\psi}(\frac{\bxs-a}{b-a})$.
It follows that $\vf(a) = \partial_{\bxs} \vf(a) = \partial_{\bxs} \vf(b) = \psi(a)
= \psi(b) = 0$, $\vf(b) =1$, $\int_a^b \psi = (1-c_1)(b-a)$.  Now define
for all $\bxs \in [a,b]$ and $|\bxu| \le k_W^2 \rho$,
\[
\begin{split}
\theta_0(\bxu, \bxs) & = \left( \partial_{\bxs} \barG_b(\bxu, b) \frac{\bxs - a}{b-a}
+ \partial_{\bxs} \barG_a(\bxu, a) \frac{b-\bxs}{b-a} \right) (1-\psi(\bxs)) \\
\theta(\bxu, \bxs) & = \theta_0(\bxu, \bxs) - \frac{\psi(\bxs)}{(1-c_1)(b-a)} \int_a^b \theta_0(\bxu, z) \, dz \\
\label{formint} \sigma(\bxu, \bxs) & = \barG_b(\bxu, b)\vf(\bxs) + \barG_a(\bxu, a)(1-\vf(\bxs)) 
+ \int_a^{\bxs} \theta(\bxu, z) \, dz \,  .
\end{split}
\]
The foliation in the gap is then defined by $\bgamma(\bxu, \bxs) = (\bxu, \sigma(\bxu, \bxs))$.
With these definitions, $\sigma(\bxu, a) = \barG_a(\bxu, a)$ and 
$\sigma(\bxu, b) = \barG_b(\bxu, b)$.  As shown in \cite[Appendix D]{BaL}, 
$\bgamma \in \cC^{1+\mbox{\scriptsize Lip}}$, and in particular, 
$C^{-1} \le |\partial_{\bxs} \sigma|
\le C$, proving the analogue of (v) for the interpolated foliation; moreover, 
choosing $c_1$ small enough guarantees that
the curves of the interpolated foliation are unstable curves.  

To prove (vi), note that
\begin{equation}
\label{eq:C0 est}
\partial_{\bxu} \partial_{\bxs} \sigma(\bxu, \bxs) = (\partial_{\bxu} \barG_b(\bxu, b) 
- \partial_{\bxu} \barG_a(\bxu, a))
\vf'(\bxs) + \partial_{\bxu} \theta(\bxu, \bxs) \, . 
\end{equation}

Now, the second term on the right-hand side is bounded by the supremum of the surviving foliations
$\partial_{\bxu} \partial_{\bxs} \barG_a$ and $\partial_{\bxu} \partial_{\bxs} \barG_b$ 
by a straightforward calculation of $\partial_{\bxu} \theta$,
and this is bounded by $Ck_W^{-2} \rho^{-2/3}$, by 
Lemma~\ref{lem:c} .

To bound the first term on the right-hand side, we use our Lipschitz bound 
\eqref{Lipbd}.
Note that $\barG_a(\eta, a)$ parametrizes the curve on one side of the gap, while
$\barG_b(\eta, b)$ parametrizes the curve on the other.  Equation
\eqref{Lipbd} says
\[
\frac{\phi(\barG_a(\bxu, a), \barG_b(\bxu, b))}{d(\barG_a(\bxu, a), \barG_b(\bxu,b))} \le C k_W^{-4/3} \rho^{-2/3} \, , 
\]
where $\phi$ represents the angle between the two unstable curves at the given points and
the distance between points is measured along a stable curve connecting the two points.  Since
stable curves have uniformly bounded curvature, this distance is uniformly equivalent to 
Euclidean distance on the scatterer.

Since $\partial_{\bxu} \barG$ represents the slope of an unstable curve, we have
\[
\begin{split}
\phi(\barG_a & (\bxu, a),  \barG_b(\bxu, b))  = | \tan^{-1}(\partial_{\bxu} \barG_a(\bxu, a)) 
- \tan^{-1}(\partial_\eta \barG_b(\bxu, b)) |    \\
&= \frac{1}{1+z^2} | \partial_{\bxu} \barG_a(\bxu, a) - \partial_{\bxu} \barG_b(\bxu, b)| 
 \ge \frac{1}{1 + K_0^2} | \partial_{\bxu} \barG_a(\bxu, a) - \partial_{\bxu} \barG_b(\bxu, b)|
\end{split}
\]
for some $z \in \mathbb{R}$, and where
$K_0$ denotes the maximum slope in the unstable cone (recall that we have global
stable and unstable cones for the map). This estimate together with 
\eqref{Lipbd}
yields,
\[
| \partial_{\bxu} \barG_a(\bxu, a) - \partial_{\bxu} \barG_b(\bxu, b)| 
\le C K_0^2 k_W^{-4/3} \rho^{-2/3}  d(\barG_a(\bxu, a), \barG_b(\bxu,b)) \, .
\]
Now using this together with the fact that 
$|\vf'| \le C \big( d(\barG_a(\bxu, a), \barG_b(\bxu,b)) \big)^{-1}$ 
due to the rescaling, we estimate the first term on the right-hand side of \eqref{eq:C0 est} by
$C k_W^{-4/3} \rho^{-2/3}$.  This completes the estimate we need for the $L^\infty$ norm (vi)
for the interpolated map-foliation.
Continuity of $\partial_{\bxu} \partial_{\bxs} \sigma$ follows immediately from 
\eqref{eq:C0 est} since all functions appearing on the right-hand side are continuous
(using the continuity of $\partial_{\bxu} \partial_{\bxs} \barG$ in the surviving foliation
by Lemma~\ref{lem:c}).
The continuity extends to the boundary of the gap since
$\partial_{\bxu} \partial_{\bxs} \sigma(\bxu, a) = \partial_{\bxu} \theta_0(\bxu, a) = \partial_{\bxu} \partial_{\bxs} \barG_a(\bxu, a)$,  and similarly for $\partial_{\bxu} \partial_{\bxs} \sigma(\bxu, b)$.


\subsection
{Checking (vii) for the map-foliation interpolated across the gaps}
\label{sec:step4}

In this step, we prove the analogue of Lemma~\ref{lem:holder jac} across gaps and in the
part of the foliation that has been filled by interpolation, this will give
the four-point estimate (vii)  across the gaps.  We begin by proving a lemma which will allow 
us to control the H\"older continuity of $\partial_{\bxs} \sigma$ defined by interpolation in each gap,
specifically that
 the average slope of the interpolated foliation across the
gap and the derivatives $\partial_{\bxs} \barG$ on either side of the gap are close in the length
scale of the gap.

\begin{lemma}
\label{lem:interpolate}
Let $P^+(W)$ be a stable curve parametrized by $\bxs \in [a',b']$ as in Subsection ~ \ref{sec:step3} and suppose
the interval $(a,b) \subset [a', b']$ is a gap in the surviving foliation pushed forward $n$ steps.
Then there exists $C>0$ such that for any $\bxu$ with $|\bxu| \le k_W^2 \rho$,
\[
\left| \partial_{\bxs} \barG_a(\bxu, a) - \frac{\barG_b(\bxu, b) - \barG_a(\bxu, a)}{b-a} \right| 
\le C \rho^{-31/105} |a-b|^{1/7} \,  .
\]
A similar bound holds for $\partial_{\bxs} \barG_b(\bxu, b)$.
\end{lemma}

Lemma~\ref{lem:interpolate} is proved in Appendix~\ref{postponeproof}.
It  allows us to prove the four-point estimate on the 
interpolated foliation in the gap.  Let 
$a \le \bxs < \bys \le b$, and $-k_W^2 \rho \le \bxu < \byu \le k_W^2 \rho$ be the coordinates
of arbitrary points in the gap.  For our first estimate,  the estimate on 
$\partial_{\bxu} \partial_{\bxs} \sigma$ from Subsection ~ \ref{sec:step3} allows us to write,
\begin{equation}
\label{eq:first 4}
\begin{split}
|\partial_{\bxs} \sigma(\bxu, \bxs) & - \partial_{\bxs} \sigma(\bxu, \bys) - \partial_{\bxs} \sigma(\byu, \bxs)
+ \partial_{\bxs} \sigma(\byu, \bys)| \\
& = \left| \int_{\bxu}^{\byu} \partial_{\bxu} \partial_{\bxs} \sigma(z, \bys) - \partial_{\bxu} \partial_{\bxs} \sigma(z, \bxu) \, dz \right|
\le C k_W^{-4/3} \rho^{-2/3} |\bxu - \byu| \,  .
\end{split}
\end{equation}

We will use Lemma~\ref{lem:interpolate} to produce a second estimate on the four-point difference above.
In the following calculation, for brevity, we set $M_a = \barG_a(\bxu, a)$, $M_b = \barG_b(\bxu, b)$,
$M_a' = \partial_{\bxs} \barG_a(\bxu, a)$ and $M_b' = \partial_{\bxs} \barG_b(\bxu, b)$.
According to the definition of $\sigma$, we have for $\bxs \in [a,b]$,
\[
\partial_{\bxs} \partial_{\bxs} \sigma (\bxu, \bxs) = ( M_b - M_a) \vf''(\bxs) + \partial_{\bxs} \theta_0(\bxu, \bxs) - \tfrac{\psi'(\bxs)}{(1-c_1)(b-a)} \int_a^b \theta_0(\bxu, z) \, dz \, .
\]
In addition, using the definition of $\theta_0$,
\[
\begin{split}
\partial_{\bxs} \theta_0(\bxu, \bxs) & = \tfrac{M_b' - M_a'}{b-a}
(1- \psi(\bxs)) - \left( M_b' \tfrac{\bxs - a}{b-a} +
M_a' \tfrac{b-\bxs}{b-a} \right) \psi'(\bxs) \qquad \mbox{and} \\
\int_a^b \theta_0(\bxu, z) \, dz & = \int_a^b M_b' \tfrac{z-a}{b-a} + M_a' \tfrac{b-z}{b-a} \, dz
 - \int_a^b (M_b' \tfrac{z-a}{b-a} + M_a' \tfrac{b-z}{b-a}) \psi(z) \, dz \\
 & = (M_b'+M_a') \tfrac{b-a}{2} - \int_a^b (M_b' \tfrac{z-a}{b-a} + M_a' \tfrac{b-z}{b-a}) \psi(z) \, dz
 \,   .
\end{split}
\]
Now combining these expressions, and recalling that by construction $\psi' = (b-a)^{-1} \bpsi'$ and
$\vf'' = \bvf'' (b-a)^{-2} = \frac{\bpsi'}{(1-c_1) (b-a)^2}$, we have
\[
\begin{split}
\partial_{\bxs} \partial_{\bxs} \sigma (\bxu, \bxs) 
& = ( M_b - M_a) \vf''(\bxs) + \tfrac{M_b' - M_a'}{b-a}(1- \psi(\bxs)) 
- \left( M_b' \tfrac{\bxs - a}{b-a} +
M_a' \tfrac{b-\bxs}{b-a} \right) \psi'(\bxs) \\
& \; \; \; \; - \tfrac{\psi'(\bxs)}{(1-c_1)(b-a)} \left[  (M_b'+M_a') \tfrac{b-a}{2} - \int_a^b (M_b' \tfrac{z-a}{b-a} + M_a' \tfrac{b-z}{b-a}) \psi(z) \, dz \right] \\
& = \tfrac{\bpsi'(\bxs)}{(1-c_1)(b-a)} \left[ \tfrac{M_b - M_a}{b-a} - \tfrac{M_b' + M_a'}{2} \right]
+  \tfrac{M_b' - M_a'}{b-a}(1- \psi(\bxs)) \\
& \; \; \; \; + \tfrac{\bpsi'(\bxs)}{b-a} \left[ \int_a^b  \big( M_b' \tfrac{z-a}{b-a} + M_a' \tfrac{b-z}{b-a} \big) \tfrac{\psi(z)}{(1-c_1)(b-a)} \, dz -  \left( M_b' \tfrac{\bxs - a}{b-a} +
M_a' \tfrac{b-\bxs}{b-a} \right) \right] \\
& =: \hbox{\circled{A} }+ \hbox{\circled{B}} + \hbox{\circled{C}} \, .
\end{split}
\]
To estimate \circled{A}, we use Lemma~\ref{lem:interpolate} to conclude that both $M_a'$ and
$M_b'$ are close to $\tfrac{M_b-M_a}{b-a}$, so that
\[
|\hbox{\circled{A}}|  \le \tfrac{|\bpsi'|}{(1-c_1)(b-a)} 2 C \rho^{-31/105} |a-b|^{1/7}
\le C' \rho^{-31/105} |a-b|^{-6/7} \,  .
\]
For \circled{B}, we again use that $|M_b' - M_a'| \le 2 C \rho^{-31/105} |a-b|^{1/7}$ by
Lemma~\ref{lem:interpolate}, so that
\[
|\hbox{\circled{B}}| \le C'  \rho^{-31/105} |a-b|^{-6/7} \,  .
\]
Finally, to estimate \circled{C}, note that the first term inside the brackets  is the average value of
$M_b' \tfrac{z - a}{b-a} + M_a' \tfrac{b-z}{b-a}$ with respect to the smooth probability measure
having density $\tfrac{\psi}{(1-c_1)(b-a)}$.  Thus there exists $\bys \in [a,b]$ such that this average
value equals the function value at $\bys$. 
\[
\begin{split}
|\hbox{\circled{C}}| & = \tfrac{|\bpsi'|}{b-a} \left|  M_b' \tfrac{\bys - a}{b-a} +
M_a' \tfrac{b-\bys}{b-a} - \big( M_b' \tfrac{\bxs - a}{b-a} +
M_a' \tfrac{b-\bxs}{b-a} \big) \right| \\
& = \tfrac{|\psi'|}{b-a} \tfrac{|\bys - \bxs|}{b-a} |M_b' - M_a'|
\le C' \rho^{-31/105} |a-b|^{-6/7} \, ,
\end{split}
\]
where we have again used Lemma~\ref{lem:interpolate} to bound the difference $|M_b'-M_a'|$.
Collecting our estimates for \circled{A}, \circled{B} and \circled{C}, we have the following bound,
\begin{equation}
\label{eq:d^2 sigma}
|\partial_{\bxs} \partial_{\bxs} \sigma (\bxu, \bxs)| \le C \rho^{-31/105} |a-b|^{-6/7} \, .
\end{equation}

Now we return to the four-point estimate for $\partial_{\bxs} \sigma$.   Grouping the terms according to their unstable coordinates and using \eqref{eq:d^2 sigma}, we estimate,
\[
\begin{split}
|\partial_{\bxs} \sigma(\bxu, \bxs) - \partial_{\bxs} \sigma(\bxu, \bys) - \partial_{\bxs} \sigma(\byu, \bxs)
+ \partial_{\bxs} \sigma(\byu, \bys)|
& \le C \rho^{-31/105} |\bys-\bxs| |a-b|^{-6/7} \\
&  \le C \rho^{-31/105} |\bys - \bxs|^{1/7} \,  .
\end{split}
\]
Putting this estimate together with \eqref{eq:first 4}, we see that the four-point difference is bounded by
the minimum of these two quantities; let us call them $X$ and $Y$.  
But if a quantity $Q$ satisfies $0 \le Q \le X$ and $0 \le Q \le Y$, then
\[
Q = Q^{1 - 7\varpi} Q^{7 \varpi} \le X^{1-7\varpi} Y^{7 \varpi},
\]
provided $0 \le 7 \varpi \le 1$.
Using this we obtain
\begin{equation}
\label{eq:4 gap}
\begin{split}
|\partial_{\bxs} \sigma(\bxu, \bxs) & - \partial_{\bxs} \sigma(\bxu, \bys) - \partial_{\bxs} \sigma(\byu, \bxs)
+ \partial_{\bxs} \sigma(\byu, \bys)| \\
& \le C k_W^{-4/3 + 28\varpi/3} \rho^{-2/3 + 13 \varpi/5} |\byu - \bxu|^{1-7\varpi} |\bys - \bxs|^\varpi
\,  ,
\end{split}
\end{equation}
which completes the required four-point estimate in the gap for the foliation on the scatterer.

Up to this point, we have proved the four-point estimate both in the gaps in the current subsection
and within each 
interval containing
a surviving part of the foliation in Lemma~\ref{lem:holder jac}.
However, this is not sufficient to prove a uniform four-point estimate along the whole of
$P^+(W)$ since, given two stable coordinates $\bxs, \bys$ with multiple gaps between them,
we would have to apply the triangle inequality once for each gap.  Since the number of gaps
grows exponentially with $n$, the H\"older control established by the four-point estimate
would not extend uniformly across $P^+(W)$.  To remedy this situation, our next
lemma proves a bound analogous to Lemma~\ref{lem:holder jac}
across gaps in the surviving foliation.

\begin{lemma}[H\"older continuity across gaps]
\label{lem:gap holder} 
Let $P^+(W)$ be a stable curve parametrized by $\bxs \in [a',b']$ as above and suppose
the interval $(a,b) \subset [a', b']$ contains one or more gaps in the surviving foliation pushed 
forward $n$ steps.
Fix $C'>0$ and consider the set of stable curves $V_1, V_2 \in \cW^s$ 
which are connected by the surviving foliation on either side of the gap such that the angles
between $V_1$ and $V_2$ satisfy $\phi(x, \bh_{12}(x)) \le C' d(x , \bh_{12}(x))$ for all $x \in V_1$,
where $\bh_{12}$ denotes the holonomy from $V_1$ to $V_2$.  
We require that $P^+(W)$ be in this set.
For any $0 \le \varpi \le 1/20$, \footnote{ Equation \eqref{largest} below  shows that we can get away with $\varpi \le 1/15$, up to using a more complicated
expression for the upper bound.} 
there exists $C>0$, independent of $W$ and $n$, such that for 
any two such curves $V_1, V_2$ and two
points, $x_a = (\barG(x^u, a), a), x_b = (\barG(x^u, b), b) \in V_1$, we have
\[
\ln \frac{J \bh_{12}(x_a)}{J \bh_{12}(x_b)} \le C k_W^{-2} \rho^{-4/5 - 11\varpi/15}  d(x_a, \bh_{12}(x_a))^{1-\varpi} |a-b|^\varpi \,  .
\]
\end{lemma}

\begin{proof}
Using similar notation to Subsection~ \ref{sec:step3} (and the proof of Lemma~\ref{lem:interpolate}), 
let $j+1$ denote the least integer $j' \ge 1$ such that
an element of $T^{j'}(\cS_0)$ intersects $P^+(W)$ in the subcurve defined by $(a,b)$.  This implies in
particular, that the surviving parts of the foliation containing $x_a$ and $x_b$
lie in the same homogeneity
strip for the first $j$ interates of $T^{-1}$.  Let $V_i^j= T^{-j} V_i$, $i=1,2$, and let
$\bh_{-j}$ denote the holonomy map from $V_1^{j}$ to $V_2^{j}$.  Then,
\begin{equation}
\label{eq:gap split}
\ln \frac{J\bh_{12}(x_a)}{J\bh_{12}(x_b)}
= \ln \frac{J^s_{V_1}T^{-j}(x_a)}{J^s_{V_2}T^{-j}(\bh_{12}(x_a))}
- \ln \frac{J^s_{V_1}T^{-j}(x_b)}{J^s_{V_2}T^{-j}(\bh_{12}(x_b))}
+ \ln \frac{J\bh_{-j}(T^{-j}x_a)}{J\bh_{-j}(T^{-j}x_b)} \, .
\end{equation}
Since $T^{-i}(x_a)$ and $T^{-i}(x_b)$ lie in the same homogeneity strip for $i=0, \ldots j$, the
difference of the first two terms in \eqref{eq:gap split} can be estimated in precisely the same
way as the difference of Jacobians in \eqref{eq:holder split} in the proof of
Lemma~\ref{lem:holder jac}.  Thus fixing $\varpi \le 1/15$,
there exists $C>0$ such that
\begin{equation}
\label{eq:first gap}
\begin{split}
& \left| \ln \frac{J^s_{V_1}T^{-j}(x_a)}{J^s_{V_2}T^{-j}(\bh_{12}(x_a))} 
 - \ln \frac{J^s_{V_1}T^{-j}(x_b)}{J^s_{V_2}T^{-j}(\bh_{12}(x_b))} \right| \\
& \qquad \le C |a-b|^\varpi d(x_a, \bh_{12}(x_a))^{1-\varpi} \max \{ k_W^{-4/3 - 10\varpi/3} \rho^{-2(1+\varpi)/3}, 
k_W^{-82 \varpi/3} \rho^{-38\varpi/3} \} \, .
\end{split}
\end{equation}

Next we estimate $\ln J \bh_{-j}(T^{-j}x_a)$ using \eqref{eq:h bound}.  Let $x_j = T^{-j}(x_a)$,
$x_0 = x_a$ and $\tx_0 = \bh_{12}(x_a)$.
Due to \eqref{eq:first gap}, we cannot use an estimate with a factor better than $|a-b|^\varpi$ and
since $\varpi \le 15$, we do not convert the full power of
$J^u_{T^{-j}\bgamma}T^j(x_j)$ in the denominator of the bound in \eqref{eq:h bound}
using Sublemma~\ref{sub:gap}.
Rather, we convert $J^u_{T^{-j}\bgamma}T^j(x_j)^{35\varpi/9}$, 
after noting that $35\varpi/9 \le 7/27 < 3/5$.  Thus,
\[
\begin{split}
\ln J\bh_{-j}(x_j) & \le C \left( \frac{d(x_0, \tx_0) |a-b|^{\varpi}}{\rho^{2/3 + 26\varpi/9} k_W^{4/3+ 70\varpi/9}} + \frac{d(x_0, \tx_0)|a-b|^\varpi k_W^{4/5}}{\rho^{2/5 + 26\varpi/9} k_W^{70\varpi/9} k_W^{6/5 - 70\varpi/9}}    \right) \\
& \le C  d(x_0, \tx_0)^{1-\varpi} |a-b|^\varpi \Big( k_W^{-4/3 - 52\varpi/9} \rho^{-2/3 - 17\varpi/9} 
+ k_W^{-2/5 + 2\varpi} \rho^{-2/5 - 17\varpi/9} \Big) \, ,
\end{split}
\]
where in the second term on the first line we have converted the remaining power of
$J^u_{T^{-j}\bgamma}T^j(x_j)$ to $k_W^{6/5 - 70\varpi/9}$ to cancel the power of $k_W$ in the
numerator of that fraction; and in the second line we have used $d(x_0, \tx_0) \le k_W^2 \rho$.
A similar estimate holds for $\ln J\bh_{-j}(T^{-j}(x_b))$. 
This estimate together with
\eqref{eq:first gap} in \eqref{eq:gap split}
yields four factors involving $k_W$ and $\rho$.  In anticipation of Subsection ~ \ref{sec:step5}
and to simplify these terms, we factor out $k_W^{-2}$ from each term and convert the remaining powers
of $k_W$ using the inequality, $k_W \le \rho^{-1/5}$.  The factor involving $\rho$
in our bound for $\ln \frac{J\bh_{12}(x_a)}{J\bh_{12}(x_b)}$
is then,
\begin{equation}\label{largest}
\max \{ \rho^{-2/5 - 36\varpi/5}, \rho^{-4/5- 11\varpi/15}, \rho^{-18/25 - 103\varpi/45} \} \, ,
\end{equation}
where we have dropped one of the terms, $\rho^{-4/5}$, as being clearly less than the middle term
above.  
Unfortunately, the remaining three exponents intersect for $\varpi < 1/15$, so there is no clear maximum
in this range; however, if we restrict to $0<\varpi \le 1/20$, the middle factor is the largest so we may drop
the other two, completing the proof of Lemma~\ref{lem:gap holder}. 
\end{proof}

Finally, we use 
Lemma~\ref{lem:gap holder} to extend the four-point estimate uniformly
across gaps along the whole of $P^+(W)$.  
If we are given two stable coordinates $\bxs, \bys$ corresponding to surviving 
parts of the foliation between which there may be multiple gaps and multiple bits of surviving foliation,
Lemma~\ref{lem:gap holder} immediately implies the four-point estimate holds between $\bxs$ 
and $\bys$.  Now suppose $\bxs$ and $\bys$ belong to different gaps in which the foliation
has been interpolated.  Then using \eqref{eq:4 gap}
we have a four-point estimate from $\bxs$ to the 
edge of its gap closest to $\bys$, call that stable coordinate $a$; similarly, the four-point estimate
holds from $\bys$ to the edge of its gap closest to $\bxs$, 
call this stable coordinate $b$.  Now from $a$ to $b$ there may be multiple gaps, but by 
Lemma~\ref{lem:gap holder}, 
the four-point estimate holds from $a$ to $b$ and so by the triangle inequality, it extends from
$\bxs$ to $\bys$.  
Notice that decomposing the distance from $\bxs$ to $\bys$ in this way, we only have
to add three terms, using the triangle inequality twice, making the estimate uniform on all of $P^+(W)$.
Similarly, we need only use the triangle inequality once
to obtain the four-point estimate given $\bxs$ in a gap and $\bys$ in a surviving piece of foliation
or vice versa.

\subsection
{Lifting the map-foliation to a flow foliation, checking (i--viii)}
\label{sec:step5}

In this step, we lift the map foliation $\{ \bar\gamma \}$ to $W$.

Given a curve $\bar\gamma$ in our foliation, we want to define a ``lift'' of $\bar\gamma$
to a curve $\gamma \in \cW^u$ which intersects $W$ in a point and such that 
$P^+(\gamma) = \bar\gamma$.  Adopting
the coordinates $(x^u, x^s, x^0)$ defined by Remark ~\ref{Rk1.4},
we want to find a parametrisation
$\gamma(v) = (x^u(v), x^s(v), x^0(v))$, with $v \in I$, an interval, which has the desired properties.  Note that the functions
$x^u(v)$ and $x^s(v)$ are uniquely determined by the projection $P^+(W)$ 
(they do not depend on $x^0$) so the only
degree of freedom is in $x^0(v)$.  However, since $\gamma$ must lie in the kernel
of the contact form, we must have $\mathbf{d}x^0 - x^s \mathbf{d}x^u = 0$, 
i.e., $(x^0)' = (x^u)' \cdot   x^s$.  This, together with the initial condition that
$x^0(0)$ must be an endpoint of $\gamma$ uniquely determines $\gamma$.

Now according to Lemma ~\ref{lem:smooth}, the map from $P^+(W)$ to $W$ is
$\cC^{1+1/2}$ with uniformly bounded norm (not depending on $k_W$).  Thus
the measure of $\bDelta_\up$ in $P^+(W)$ is comparable to the measure of 
$\Delta_\up$ in $W$, proving item (iv).  Items (i)-(iii) follow directly
from the definition of $\{ \gamma \}$ and the construction of $\{ \bar\gamma \}$,
recalling (\ref{cond(1)BLu}).

In order to determine the smoothness of the lifted foliation, we must consider the
action of lifting $\bar\gamma$ off the scatterer in the unstable direction as well.
In the unstable direction, the curves $\bar\gamma$ undergo a contraction on the
order of $1/J^u_{T^{-1}\bar\gamma}T \approx k_W^{-2}$.  Using the usual distortion
bounds together with the expression for the Jacobian given by 
\eqref{eq:jac factors},
we see that
\[
\left| \ln \frac{J^u_{T^{-1}\bar\gamma}T(x)}{J^u_{T^{-1}\bar\gamma}T(y)} \right|
\le C d(x,y) k_W^2 \, ,
\]
for any $x, y \in \bar\gamma$ and a similar estimate holds between different curves
by estimates similar to those used in the proof of Lemma~\ref{lem:jac holonomy}.
Thus the Lipschitz constants of the foliation increase by a factor of 
$k_W^2 \le C \rho^{-2/5}$.

For $\bar\gamma$
on the domain $\bDelta_\up$, we have $|\partial_{\bxu} \partial_{\bxs} \barG| \le C\rho^{-2/3} k_W^{-2}$
by Section~\ref{sec:step2},
so that $|\partial_{x^u} \partial_{x^s} G| \le C \rho^{-2/3}$.
On the complement of the domain $\bDelta_\up$, the Lipschitz constant of the
projected foliation is bounded by $C + Ck_W^{-4/3} \rho^{-2/3}$
by (\ref{Lipbd}) in Section~ \ref{sec:step3}.  Thus the Lipschitz
constant of the lifted foliation is bounded by 
$Ck_W^2 + C k_W^{2/3} \rho^{-2/3} \le C \rho^{-4/5}$, which yields
item (vi) of the foliation.
We can then integrate $\partial_{x^u} \partial_{x^s} G$ to obtain the uniform bound
on $\partial_{x^s}G$ needed for item (v) of the foliation, recalling (ii).

Finally,
using \eqref{eq:4 gap} and Lemmas~\ref{lem:holder jac} and \ref{lem:gap holder},
and collecting the relevant terms, we see that the constant
in the right-hand side of the four-point condition (vii) for the Jacobian  
 is at most 
 \[
 \max \{ \rho^{-4/5}, \rho^{-2/5-36\varpi/5}, \rho^{-4/5 + 37\varpi/15}, \rho^{-4/5-11\varpi/15} \} 
 \, .
 \]
 If we restrict to $\varpi \le 1/20$, the last term above is the largest, completing the proof of (vii).

Finally, we turn to condition (viii).  Note that since $\gamma$ lies in the kernel of the
contact form, we have $H(x^u, x^s) = \int_0^{x^u} G(z, x^s) \, dz$.  In particular,
$\partial_{x^s} H(x^u, x^s) = \int_0^{x^u} \partial_{x^s} G(z, x^s) \, dz$, so that since $\partial_{x^s}G$
uniformly bounded, we have $|\partial_{x^s} H|_{C^0} \le C \rho$.

Next, we note that due to the normalization $W = (0, x^s, 0)$, we have
$\partial_{x^s} G(0, x^s) = 1$ for all $x^s \in [0, |W|]$.  Thus using the four-point estimate (vii) 
for $G$ at the
points $(x^u, x^s)$, $(x^u, y^s)$, $(0, x^s)$ and $(0, y^s)$,
we have
\[
|\partial_{x^s} G(x^u, x^s) - \partial_{x^s} G(x^u, y^s)| \le C \rho^{-4/5 - 11\varpi/15} |x^s - y^s|^\varpi 
|x^u|^{1-7 \varpi} \, .
\] 
This implies immediately that 
$|\partial_{x^s} H|_{\cC^\varpi} \le C \rho^{6/5-116\varpi/15}$,
which completes the proof of (viii).

\section{Mollification operators and embeddings}
\label{molsec}

In this section we prove some relations between our Banach spaces and standard spaces of
distributions, and establish several key inequalities which indicate that 
mollification operators provide good approximations in the norms we have defined.
We begin by proving Lemma~\ref{relating}, relating our norms to the dual
spaces of continuous functions.

\begin{proof}[Proof of Lemma~\ref{relating}]
We will prove that  
\begin{equation}
| \oldh|_{(\cC^{\oldp}(\Omega_0))^*}\le C |\oldh|_{w} \qquad \forall \oldh\in \cC^0(\Omega_0)
\end{equation}
for some $C > 0$.  The analogous inequality for the strong norm,
$| \oldh |_{(\cC^{\oldq}(\Omega_0))^*}\le C \|\oldh\|_s$, is proved similarly.

For $f \in \cC^0(\Omega_0)$ and $\psi \in \cC^\alpha(\Omega_0)$ with 
$|\psi|_{\cC^\alpha(\Omega_0)} \le 1$, we must estimate
$\int_{\Omega_0} f \, \psi \, dm$.  To estimate this, we decompose Lebesgue measure
over the Poincar\'e section $\cM$ defined in Section~\ref{defspace}.

Our first step is to decompose $\cM$ into boxes foliated by a smooth family of homogeneous map-stable curves.
(For the present lemma, it would  be enough to consider a smooth family of map-stable curves,
the further decomposition into homogeneous curves is useful in view of the proof
of Lemma~\ref{mollbound1} below.)
On each connected component $\bH_{k,i}$ of $\bH_k$, $k > k_0$,
we define a smooth foliation $\{ V_\xi \}_{\xi \in E_{k,i}}$ of 
map-stable curves; indeed, we may choose a foliation of straight line segments due to the
global stable cones for the map.  Using this foliation, 
we desintegrate the probability measure $\bar{m} = \bar{c} dr d\vf$ into
$c' d\bar{m}_{V_\xi} d\xi$, where $c'$ is a constant depending smoothly on the angle 
between the foliation
of stable curves and the boundary of $\bH_k$, and $\bar{m}_V$ is arclength measure on the curve
$V$.
On the set $\bH_{k_0} := \cM \setminus \cup_{k > k_0} \bH_k$, 
we perform a similar decomposition, after
first subdividing the space into boxes $B_i$ which are foliated by parallel stable line segments 
of length at most $L_0$.

Next, we lift this decomposition to $\Omega_0$.
For $k > k_0$, define 
\[
H_k^{0-} = \{ Y \in \Omega_0 : P^+(Y) \in \bH_k \},
\]
and $H_{k_0}^{0-} = \Omega_0 \setminus \cup_{k > k_0} H_k^{0-}$. 
Over each box $B_i$ define the flow region 
$B_i^{0-} = \{ Z \in \Omega_0 : P^+(Z) \in B_i \}$.  

In each $B_i^{0-}$ or $H_k^{0-}$, we represent Lebesgue measure as\footnote{ The projection on the scatterers of Lebesgue measure is the $T$-invariant probability measure $\mu_0=c \cos \vf dr d\vf$, where $c$ is a normalization constant.}
$c \cos \vf ds dr d\vf$, where $r, \vf$ range over the box $B_i$ or strip $\bH_k$ and $s$ ranges from
0 to the maximum free flight time under the backwards flow of any point in the box $B_i$, which
we denote by $\tau^-_{\max,i}$.  For each $s$ in this range and each curve $V_\xi$,
let $W^s_\xi = \Phi_{-t(s)}(V_\xi)$, where the function $t(s,z)$ is defined  for $z \in V_\xi$
so that $W^s_\xi$ lies in
the kernel of the contact form.  Note that for $s < L_0$, it may be that some points in $V_\xi$ 
have not yet lifted off of $\cM$.  For such small times, $W^s_\xi$ denotes only those points which have
lifted off of $\cM$ and so may be the union of at most\footnote{ Just like  for \eqref{esscurve}, see the paragraph containing \eqref{cardinality}.}  two flow-stable curves.  Also, for 
$s > \tau_{\min}$, it may be that part of $\Phi_{-t(s)}(V_\xi)$ has made a collision with a scatterer.
In this case, again, $W^s_\xi$ denotes only those points which have not yet made a first collision.
Thus, using the disintegration of
$\bar {m}$ described above, $c \cos \vf ds dr d\vf = \cos \vf(P^+Z) \rho_{\xi}(Z) dm_{W_\xi^s}(Z) d\xi ds$,
where $\xi \in E_{k,i}$ and (using that the disintegration factor $c'$ 
from the third paragraph of the proof is smooth) $|\rho_\xi|_{\cC^1(W^s_\xi)} \le C$, uniformly in $\xi$ and $s$.  Now,
\begin{equation}
\label{eq:cos decomp}
\begin{split}
\left| \int_{\Omega_0} f \psi \, dm \right| & = \left| \sum_{k \ge k_0} \sum_{i} \int_0^{\tau^-_{\max,i}}
\int_{E_{k,i}} \int_{W^s_\xi} f \psi \, \rho_{\xi} \cos \vf(P^+) \, dm_{W^s_\xi} d\xi ds \right| \\
& \le \sum_{k \ge k_0} \sum_i \int_0^{\tau^-_{\max,i}} \int_{E_{k,i}} 
| f |_w |\rho_\xi|_{\cC^\alpha(W^s_\xi)} 
|\cos \vf(P^+)|_{\cC^\alpha(W^s_\xi)} |\psi|_{\cC^\alpha(\Omega_0)} d\xi ds .
\end{split}
\end{equation}
Now $|\cos \vf|_{\cC^0(\bH_k)} \le C k^{-2}$, while for $y, z \in V_\xi \subset \bH_k$, 
\[
|\cos \vf(y) - \cos \vf(z) | \le |y-z|^\alpha |y-z|^{1-\alpha} \le |y-z|^\alpha C k^{-3(1-\alpha)}.
\]
Since $\alpha \le 1/3$, the H\"older constant of $\cos \vf$ is bounded by $Ck^{-2}$ as well.
Since $P^+$ is $\cC^1$ along stable curves by Lemma~\ref{lem:smooth}, we have
$|\cos \vf(P^+)|_{\cC^\alpha(W^s_\xi)} \le Ck^{-2}$, independently of $\xi$ and $s$.
Also, for each $k \ge k_0$, the number of boxes $B_i$ in $\bH_k$ is finite, depending on 
$L_0$.  Thus,
\[
\left| \int_{\Omega_0} f \psi \, dm \right| \le \sum_k C \tau_{\max} |f|_w  k^{-2} \le C' |f|_w,
\]
completing the proof of the lemma.
\end{proof}

\begin{lemma}
\label{lem:injective}
If $\beta < 1/q$, the inclusions $\cB \subset (\cC^\beta(\Omega_0))^*$ 
and $\cB \subset (\cC^1(\Omega_0))^*$ are injective.
\end{lemma}

\begin{proof}
Our proof has two steps:  (1) for $f \in \cB$, if $\| f \|_{\cB} \neq 0$, then $\| f \|_s \neq 0$; (2) if $\| f \|_s \neq 0$,
then $f \neq 0$ as an element of $(\cC^\beta(M))^*$.

For claim (1), note that $\| f \|_s = 0$ implies immediately that $\| f \|_u = 0$ since the test
functions for $\| \cdot \|_u$ are in $\cC^\alpha(W)$, while those for $\| \cdot \|_s$ are in
$\cC^\beta(W)$ and $\alpha > \beta$.  To see that $\| f \|_0 = 0$ as well, observe that for fixed 
$W \in \cW^s$, $\psi \in \cC^\alpha(W)$, the functional $F_s(f) = \int_{\Phi_{-s}(W)} \frac{d}{dr}(f \circ \Phi_r)|_{r=0} \, \psi \circ \Phi_s \, J_{\Phi_{-s}(W)}\Phi_s \, dm_{\Phi_{-s}(W)}$ is continuous as a function of $s$
as long as $\Phi_{-s}(W)$ undergoes no collisions.  (It is clearly continuous for 
$f \in \cC^2(\Omega_0)$ and extends to $f \in \cB$ by density since the map $f \to  F_s(f)$ is
continuous in the $\| \cdot \|_{\cB}$ norm.)
Thus 
$$
\lim_{t \downarrow 0} \frac 1t \int_0^t \int_{\Phi_{-s}(W)} \frac{d}{dr}(f \circ \Phi_r)|_{r=0} \, \psi \circ \Phi_s \, J_{\Phi_{-s}(W)}\Phi_s \, dm_{\Phi_{-s}(W)} \, ds
= \int_{W} \frac{d}{dr}(f \circ \Phi_r)|_{r=0} \, \psi \, dm_W .
$$
On the other hand,
\[
\begin{split}
\int_0^t \int_{\Phi_{-s}(W)} & \frac{d}{dr}(f \circ \Phi_r)|_{r=0} \, \psi \circ \Phi_s \, J_{\Phi_{-s}(W)}\Phi_s \, dm_{\Phi_{-s}(W)} \, ds
= \int_W \int_0^t \frac{d}{dr} (f \circ \Phi_r)|_{r=0} \circ \Phi_{-s} \, \psi \, ds \, dm_W \\
& = \int_W \int_0^t \frac{d}{ds}(f \circ \Phi_{-s}) \, \psi \, ds \, dm_W
= \int_{\Phi_{-t}(W)} f \, \psi \circ \Phi_t \, J_{\Phi_{-t}(W)}\Phi_t  - \int_W f \, \psi = 0,
\end{split}
\]
by assumption on $f$, where we have used \eqref{eq:flow derivative} for the second equality.
Thus $\| f \|_0 = 0$ as claimed.

Our proof of claim (2) follows closely \cite[Lemma 3.8]{demers zhang 3}.  For $f \in \cC^2(\Omega_0)$
and $W \in \cW^s$, the expression,
\[
\langle D^\beta_W(f), \psi \rangle = \int_W f \psi \, dm_W, \quad \psi \in \cC^\beta(\Omega_0),
\]
satisfies $| \langle D^\beta_W(f), \psi \rangle | \le \| f \|_s |W|^{1/q} |\psi_{\cC^\beta(W)}$, so that
$D^\beta_W(f) \in (\cC^\beta(\Omega_0))^*$.  Since the map $f \to D^\beta_W(f)$ is continuous
in the $\| \cdot \|_{\cB}$ norm, by density it can be extended to $\cB$.

Now assume $\| f \|_s \neq 0$.  There exists $W \in \cW^s$, $\psi \in \cC^\beta(\Omega_0)$
such that $\langle D^\beta_W(f), \psi \rangle =: \delta > 0$.  Again, the map $W \to \langle D^\beta_W(f), \psi \rangle$ is continuous for $f \in \cB$.  Thus there exists an open set $E$,
foliated by invariant curves $W' \in \cW^s$ close to $W$ such that $\langle D^\beta_{W'}(f), \psi \rangle \ge \delta/2$ for each $W' \subset E$.

In order to localize the support of $\psi$ to the set $E$, we extend each stable curve $W'$ in $E$
by length $\ve >0$ at both ends to form a larger set $E' \supset E$.  Call such extended curves
$W'_\ve$.  Next we choose a bump function $\rho_\ve$ such that $\rho_\ve = 0$ on
$\Omega_0 \setminus E'$ and $\rho_\ve = 1$ on $E$.  We may choose $\rho_\ve$ so that
$| \rho_\ve|_{\cC^\beta(W'_\ve)} \le C \ve^{-\beta}$ for some uniform constant $C$.  Now,
\[
\begin{split}
\langle D^\beta_{W'_\ve}(f), \rho_\ve \psi \rangle 
& = \langle D^\beta_{W'}(f), \rho_\ve \psi \rangle
+ \langle D^\beta_{W'_\ve \setminus W'}(f) , \rho_\ve \psi \rangle \\
& \ge \delta/2 - C | \psi |_{\cC^\beta(\Omega_0)} \ve^{-\beta} |W'_\ve \setminus W'|^{1/q} \| f \|_s
\ge \delta/2 - C |\psi|_{\cC^\beta(\Omega_0)} \| f \|_s \ve^{1/q - \beta} .
\end{split}
\]
This difference can be made larger than $\delta/3$ by choosing $\ve$ small since $1/q < \beta$.
Thus the function $\rho_\ve \psi \in \cC^\beta(\Omega_0)$ satisfies 
$f (\rho_\ve \psi) \neq 0$ and so $f \neq 0$ as an element of 
$(\cC^\beta(\Omega_0))^*$.

The injectivity of $\cB \subset (\cC^1(\Omega_0))^*$ follows by a similar argument since
we may take $\psi \in \cC^1(\Omega_0)$ in the proof of claim (2).  This holds
since we have defined $\cC^\beta(W)$ to be the closure of $\cC^1$ functions in the $\cC^\beta$
norm. 
\end{proof}

\subsection{Mollification operators}

Since the right-hand side of  the Dolgopyat Lemma~\ref{dolgolemma} will 
involve the Lipschitz and supremum norm of $\oldh$, it will be convenient 
to use mollification  operators $\bM_\epsilon$.

We start by defining $\bM_\epsilon$:
Fix $\epsilon_0$ small.\footnote{ In Section~\ref{dodo}, we shall take $\epsilon_0$ small enough as a
function of $n=\lceil c \ln |b| \rceil$.} 
Let $\eta: \bR^{3} \to [0,\infty )$ be a 
$\cC^\infty$ function, supported in $|Y|\le 1$
and bounded away from zero on  $|Y|\le 1/3$,
with $\int \eta \, dm =1$, and set, for $0<\epsilon<\epsilon_0$,
$$
\eta_\epsilon(Y)=\frac{1}{\epsilon^{3}} \eta\left( \frac{Y}{\epsilon} \right)
\, .
$$

Let $\Omega_1$ be an $\epsilon_0$ neighborhood
of $\Omega_0$, assuming that $\epsilon_0$ is small enough so that $\Omega_1$
is still in the torus.

\begin{defin}[The mollifier operator $\bM_\epsilon$]
Fix global periodic coordinates on $\mathbb T^3$, i.e., view $\mathbb T^3$ as 
a subset of $\bR^3$,
extending functions periodically. 
For   $0<\epsilon<\epsilon_0$,  set
for $\oldh \in L^\infty(\Omega_0)$
\begin{equation}
\label{eq:moll def}
\bM_\epsilon (\oldh) (Z)= \int_{\bR^{3}} \eta_\epsilon ( Z-Y) 
\oldh (Y)\, dm(Y) = [\eta_\epsilon *  \oldh ](Z)
\, .
\end{equation}
\end{defin}

Since $\oldh$ is supported in $\Omega_0$, we have that
$\bM_\epsilon(\oldh)$ is supported in a small neighborhood of $\Omega_0$ contained
in $\Omega_1$.  For $\oldh\in \cC^1(\mathbb T^3)$, not necessarily supported in $\Omega_0$, we let
\begin{equation}
\label{deflip}
|\oldh|_{H^1_\infty(\Omega_0)}:=| \nabla (\oldh) |_{L^\infty(\Omega_0)}
\le
 |\oldh|_{H^1_\infty(\mathbb T^3)}:=| \nabla (\oldh) |_{L^\infty(\mathbb T^3)} \, .
\end{equation}

We have the following  bounds for $\bM_\epsilon$:

\begin{lemma}\label{mollbound1}
There exists $C$ so that for all small
enough $\epsilon >0$, all $\oldq$ and all $\oldbetaparam$,
every admissible stable curve   and every $\oldh\in \cC^0(\mathbb T^3)$, supported
in $\Omega_0$,
\begin{equation}\label{claim1}
|\bM_\epsilon (\oldh)|_{L^\infty(\Omega_0)} 
\le C \epsilon^{-\oldq-1+\oldbeta} \|\oldh\|_{s}  \, ,
\end{equation}
and
\begin{equation}\label{claim0}
|\bM_\epsilon (\oldh)|_{H^1_\infty(\Omega_0)} 
\le C \epsilon^{-\oldq-2+\oldbeta} \|\oldh\|_{s}  \, .
\end{equation}
\end{lemma}

The following lemma shows that $\bM_\epsilon$ is in some sense
an approximation of the identity.

\begin{lemma}\label{mollbound2}
There exists $C >0$ so that for all  small enough $\epsilon >0$,
all $\oldp$, $\oldq$, $\gamma$,  $\oldbetaparam$, and all $\oldh\in \cC^0(\mathbb T^3)$,
supported in $\Omega_0$
\begin{equation}\label{lemma1.1b}
|\bM_\epsilon (\oldh)-\oldh  |^{\mathbb H}_{w}
\le C \epsilon^{\delta} \|\oldh\|_{\cB} \, ,
\end{equation}
where $\delta = \min \{ \gamma, 1/(2q), 1/q - 2/5 - \beta \}$, and
where the homogeneous weak norm is defined by
 \begin{equation}\label{onemorenorm}
 |\oldh|_w^{\mathbb H}:=
 \sup_{\substack{W \in \cW^s\\ W \, \,\,  { homogeneous}} }\, 
\sup_{\substack{\psi \in \cC^\oldp(W) \\ |\psi|_{\cC^\oldp(W)} \le 1}}
\int_W \oldh \psi \, dm_W \, .
\end{equation}
\end{lemma}

The last lemma of this section completes the proof
of Lemma ~ \ref{lem:embedding},
showing that, although we have taken $\cB$ to be the closure
of $\{\cL_t(\cC^2(\Omega_0) \cap \cC^0_\sim)\mid t\ge 0\}$, in fact we have $\cC^1(\Omega_0) \subset \cB$ as well.

\begin{lemma}
\label{lem:C1}
$
\cC^1(\Omega_0) \subset \cB$.
\end{lemma}

\subsection{Proofs of Lemmas~\ref{mollbound1} - \ref{lem:C1}}

\begin{proof}[Proof of Lemma~ \ref{mollbound1}]
We first bound $|\bM_\epsilon \oldh(Z)|$ for all $Z \in \Omega_0$.  
Recalling the notation used in the proof of Lemma~\ref{relating}, we decompose the 
$\epsilon$-neighborhood $N_\epsilon(Z)$
of $Z$ inside $\Omega_0$, into homogeneous stable curves $W_\xi^s$ over each component of $N_{\epsilon}(Z) \cap \bH_k^{0-}$ and $N_{\epsilon}(Z) \cap B^{0-}_i$. 
We estimate the relevant integral over one such component at a time, using \eqref{eq:cos decomp},
\[
\begin{split}
\int_{B^{0-}_i} f(Y) & \eta_\epsilon(Z-Y) dm(Y) 
 = \int_0^{\tau^-_{\max,i}} \int_{E_k,i} \int_{W^s_\xi} f \, \eta_\epsilon(Z-\cdot) \rho_{\xi} \cos \vf (P^+)
dm_{W^s_\xi} d\xi ds \\
& \le C \eps k^{-3} \| f \|_s \sup_{\xi, s} |W^s_\xi \cap \mbox{supp}(\eta_\epsilon(Z - \cdot))|^{1/q}
|\eta_\epsilon|_{\cC^\beta} |\cos \vf(P^+)|_{\cC^\beta(W^s_\xi)} |\rho_{\xi}|_{\cC^\beta(W^s_\xi)},  
\end{split}
\]
where the factor $\epsilon k^{-3}$ comes from the fact that the support of $\eta_\epsilon$ is of order 
at most $\epsilon$ in the transverse integral $ds$ and of order $k^{-3}$ in the integral $d\xi$.  
Next, 
$|\cos \vf(P^+)|_{\cC^\beta(W^s_\xi)} \le Ck^{-2}$, using the estimate following \eqref{eq:cos decomp}
and the fact that $\beta < 1/3$.  Also, $|\eta_\epsilon|_{\cC^\beta} \le C \epsilon^{-3-\beta}$, and
as before, $|\rho_\xi|_{\cC^1} \le C$.  Putting these estimates together yields,
\begin{equation}
\label{eq:one box}
\left| \int_{B^{0-}_i} f(Y)  \eta_\epsilon(Z-Y) dm(Y) \right| \le C \epsilon^{-2-\beta + 1/q} k^{-5} \| f \|_s .
\end{equation}
It remains to sum over the relevant $k$ and $i$.  Since there are only a finite number of flow boxes
$B^{0-}_i$, independent of $\epsilon$, this sum is uniformly bounded.  Also, 
$N_\epsilon(Z)$ can cross only a uniformly bounded number of the singularity surfaces
$S^{0-} = \{ Y \in \Omega_0 : \vf(P^+(Y)) = \pm \frac \pi2 \}$ (in fact, no more than
$\tau_{\max}/\tau_{\min}$ of them).  
Thus the only index which may be unbounded or infinite is $k$.  Suppose $N_\eps(W)$ intersects
a range of $k$, from $k_1$ to $k_2$ (note that $k_2 = \infty$ is allowed).  
Since the width in the unstable direction of
$\bH_k^{0-}$ is approximately $C k^{-5}$, we must have $\epsilon$ of order 
$k_1^{-4} - k_2^{-4}$,
which is precisely the bound we obtain summing over the above estimate, yielding one more
power of $\epsilon$ and completing
the proof of \eqref{claim1}.

For the $H^1_\infty$ bound \eqref{claim0}, we first differentiate once (which produces an
extra factor $\epsilon^{-1}$) and then proceed in the same way as above.
\end{proof}

\begin{proof}[Proof of Lemma~ \ref{mollbound2}]
Let $W$ be a homogeneous stable curve and let $\psi$ be a function on
$W$ with $|\psi |_{C^\oldp}(W)\le 1$.  

Let $A$ be a constant to be determined below.  
We assume for the moment that $P^+(W)$ does not lie within a
distance of $A \epsilon^{3/5}$ of the boundary of any homogeneity strip. 
We may also assume that $|W| \le L_0/2$, since otherwise, we may simply subdivide $W$ into
two components and perform the estimate on each component separately.

 In order to compare $\bM_\epsilon(f)$ with $f$, we will
adopt a new coordinate system in $N_{\epsilon}(W)$
(from the proof of Lemma~ \ref{mollbound1}) so that small translations of $W$ are
again stable curves.  Let $W^\epsilon$ be a $\cC^2$ extension of $W$ of length $\epsilon$
at each end.   Note that if $P^+(Z)$ and $P^+(Y)$ lie in the same homogeneity strip $\bH_k$, then
$d(P^+(Z), P^+(Y)) \le C k^2 d(Z,Y)$, since distances along stable curves are contracted by 
a factor of order 1 under $P^+$, while distances along unstable curves are expanded by a factor proportional
to $k^2$.  Since the maximum length of map-stable and -unstable curves in $\bH_k$ is at most
$k^{-3}$, this implies that 
\begin{equation}
\label{mysterious}
d(P^+(Z), P^+(Y)) \le B' d(Z,Y)^{3/5}
\end{equation}
 for some $B'>0$ depending only on the 
maximum curvature and $\tau_{\max}$. 
Thus, $P^+(N_\epsilon(W))$ lies in a $B' \epsilon^{3/5}$-neighborhood of $P^+(W)$. 
Since we have assumed the homogeneity strip
containing $P^+(W)$ has width at least $2A\epsilon^{3/5}$, we choose $A$ large enough 
compared to $B'$ that 
$P^+(N_\epsilon(W))$ lies
in the same homogeneity strip $\bH_k$ as $W$; indeed, our assumptions
imply $Ck^{-3} \ge 2A\epsilon^{3/5}$, so that the width of the stable cone in this
neighborhood of $W$ is at least $\cos \vf(P^+) \ge A' \epsilon^{2/5}$, for some constant 
$A'$, increasing with $A$. 

Now using Remark~\ref{suitcharts0} and proceeding like in the decomposition in the
proof of Lemma~\ref{relating}, we may define a smooth foliation of $N_\epsilon(W)$ by 
admissible stable curves with the following
properties: (i) $W^\epsilon$ belongs to the foliation; (ii) the foliation is trivial in the flow direction,
i.e., $W^\epsilon$ lies in a smooth surface formed by such stable curves, which is then flown
a distance proportional to $\pm \epsilon$ to define the foliation of $N_\epsilon(W)$.
It may be that this flow surface reaches $\partial \Omega_0$.  In this case, we simply stop
the foliation at $\partial \Omega_0$ so that it does not contain any collision points.  

Fixing a transverse curve $\gamma_0$ in the surface containing $W^\epsilon$, we define 
local coordinates $(s,u,t)$, associating $W^\epsilon$ with the first coordinate,
$\{ (s,0,0) : s \in [-\epsilon - |W|/2, |W|/2 +\epsilon] \}$, parametrized by arclength,
$\gamma_0$ with the second coordinate,
and the flow direction with the third.  Due to the smoothness of the foliation, the
Jacobian $J_0$ of the change of variables $(x,y,\omega) \mapsto (s,u,t)$ is uniformly $\cC^1$.
Also, it may be that in this coordinate system near $\partial \Omega_0$, some coordinate choices
lie inside one of the scatterers.  This is immaterial to the integral we must estimate since
$f$ is taken to be simply 0 inside such scatterers.

We are now ready to proceed with the required estimate.
We  have, by Fubini,
\begin{equation}
\label{chain0}
\begin{split}
 \int_W  (\bM_\epsilon(\oldh) - & f) \psi \, dm_W 
= \int_W  \int_{\bR^3} \eta_\epsilon (Z-Y)
(f(Y) - f(Z))
\, dm(Y) \psi(Z)\, dm_W(Z) \\
& =  \int_{\bR^3} \eta_\epsilon(Y)
\int_W 
(\oldh(Z-Y) - f(Z)) \psi(Z)  J_0(Z-Y) \,
dm_W(Z)\, dm(Y)\\
& \le | \eta_\epsilon |_{L^1} \sup_{Y \in N_\epsilon(0)} \left |\int_W 
(\oldh(Z-Y) - f(Z)) \psi(Z)  J_0(Z-Y) \, dm_W(Z) \right|\, ,
\end{split}
\end{equation}
using that $|\eta_\epsilon |_{L^1}=1$.
We next want to apply the definition of the unstable norm.  
For $Y \in N_\ve(0)$, let $W_Y$ denote the stable curve corresponding to $W - Y$ in the
adapted coordinates $(s,u,t)$ and lying in the interior of $\Omega_0$.  
$W_Y$ is necessarily an admissible stable curve by our choice
of adapted coordinates.  
Let $\bh_Y$ denote the translation map from $W_Y$ to $W$.  Then,
\[
\int_W f(Z-Y) \psi(Z) J_0(Z-Y) \, dm_W = \int_{W_Y} f \psi \circ \bh_Y J_0 \, dm_{W_Y}.
\]
Note that $J\bh_Y = 1$.
In light of \eqref{chain0}, the following sublemma is the main estimate in the proof of 
Lemma~ \ref{mollbound2}.

\begin{sublem}
\label{lem:close Y}
There exists $C>0$ such that for each $Y \in N_\epsilon(W)$,
\[
\left| \int_{W_Y} \oldh \psi \circ \bh_Y J_0 \, dm_{W_Y} - \int_W \oldh \psi J_0 \circ \bh_Y^{-1} 
\, dm_W \right|
\le C {\epsilon}^{r_0} \| \oldh \|_\cB \,  ,
\]
where $r_0=  \min \{ \gamma, \oldbetatwo \}$.
\end{sublem}

\begin{proof}[Proof of Sublemma~\ref{lem:close Y}]
Notice that not every $W_Y$ has 
$d_{\cW^s}(W_Y, W) < \infty$.
To remedy this, let $\tau^{\pm}(W) = \min_{W} \{ \tau(\cdot ), \tau_{-1}(\cdot ) \}$ and
define the local surface $W^0 = \{ \Phi_t (W) : |t| \le \tau^{\pm}(W)  \}$
obtained by flowing $W$.  
By arguments similar to those used to ensure
(ii) in the proof of
Lemma~\ref{lem:compact}, for each $Y$ we may choose a stable curve 
$V_Y = \Phi_{t_{Y}}(W) \subset W^0$ such that
$d_{\cW^s}(V_{Y}, W_{Y}) < \sqrt{\epsilon}$.  
Indeed, there exists a constant $C_\kappa > 0$, depending
on the maximum curvature of stable and unstable curves, such that 
\begin{equation}
\label{eq:curve const}
|t_Y| \le C_\kappa \epsilon, \; \mbox{ for all } Y \in N_\ve(W)
\, .
\end{equation}
It may be that $\Phi_{t_Y}(W)$ is in the midst of a collision.  But in this case, we let
$V_Y$ denote only that portion of $\Phi_{t_Y}(W)$ that has not undergone a collision
during time $t \in [0, t_Y]$.  The mismatch between $W$ and $V_Y$ is then of length
at most $C \sqrt{\epsilon}$ and can be estimated as in \eqref{eq:missed} below.

Now we estimate,
\begin{equation}
\label{eq:Y split}
\begin{split}
\int_{W_Y} f \psi \circ & \bh_Y J_0 \,  dm_{W_Y}  - \int_W f \psi J_0 \circ \bh_Y^{-1} \, dm_W \\
& = \int_{W_Y} f \psi \circ \bh_Y J_0 \, dm_{W_Y} 
  - \int_{V_Y} f (\psi  \, 
J_0 \circ \bh_Y^{-1}) \circ \Phi_{-t_Y} \, J_{V_Y} \Phi_{-t_Y} \, dm_{V_Y} \\
&  \; \;\; \;+ \int_{V_Y} f (\psi  \, 
J_0 \circ \bh_Y^{-1}) \circ \Phi_{-t_Y} \, J_{V_Y} \Phi_{-t_Y} \, dm_{V_Y} 
- \int_W f \psi J_0 \circ \bh_Y^{-1} \, dm_W \,  .
\end{split}
\end{equation}

We begin with the first difference above.
Denote by $S_{W_Y}$ the natural map from the $r$-interval $I_{W_Y}$ (on the scatterer) to
$W_Y$ as defined at the
end of Subsection~\ref{stable curves}, and let $S_{V_Y}$ denote the analogous map
for $V_Y$.   Set $I_Y$ to be the common $r$-interval on which $P^+(W_Y)$ and $P^+(V_Y)$ are defined, and denote the
corresponding matched subcurves 
by $W'_Y$ and $V'_Y$, as defined in Section~\ref{unstable norm}.
 The unmatched (at most) two ends 
may be estimated separately, using the strong stable norm since these
unmatched  pieces have length at most $\sqrt{\epsilon}$:
\begin{equation}
\label{eq:missed}
\begin{split}
\int_{W_Y \setminus W'_Y} f \psi \circ \bh_Y J_0 \, dm_{W_Y} 
&\le |W_Y \setminus W'_Y|^{\oldbeta} \| \oldh \|_s |J_0|_{\cC^\beta(W_Y)} |\psi \circ \bh_Y|_{\cC^\beta(W_Y)} \\
&\le  C\epsilon^{\oldbetatwo}\|\oldh \|_s  \,  ,
\end{split}
\end{equation}
where we have used the smoothness of both $J_0$ and $\bh_Y$.
A similar estimate holds for the integral over $V_Y \setminus V'_Y$.

On $I_Y$, define $\phi = (\psi \circ \bh_Y \cdot J_0) \circ S_{W_Y} \circ S_{V_Y}^{-1}$, so that
$d(\psi \circ \bh_Y \cdot J_0, \phi) =0$.  
Since the derivatives of $S_{W_Y}$ and
$S^{-1}_{V_Y}$ are uniformly bounded by Lemma~\ref{lem:smooth}, we get
$|\phi|_{\cC^\oldp(V'_Y)} \le C |\psi|_{\cC^\oldp(W'_Y)} \le C'$.
Now we estimate the first difference in the matched part of equation
\eqref{eq:Y split},
\begin{equation}
\label{eq:first Y}
\begin{split}
& \left| \int_{W'_Y} \oldh \psi \circ \bh_Y \, J_0 \, dm_W - \int_{V'_Y} f (\psi  \, 
J_0 \circ \bh_Y^{-1}) \circ \Phi_{-t_Y} \, J_{V_Y} \Phi_{-t_Y} \, dm_{V_Y} \right| \\
& \le \left|  \int_{W'_Y} \oldh \psi \circ \bh_Y \, J_0 \, dm_W - 
\int_{V'_Y} f \phi \, dm_W \right|  \\
& \; \; \; \; +
\left| \int_{V'_Y} f \phi  \, dm_W - \int_{V'_Y} f (\psi  \, 
J_0 \circ \bh_Y^{-1}) \circ \Phi_{-t_Y} \, J_{V_Y} \Phi_{-t_Y} \, dm_{V_Y} \right| \\
& \le C \epsilon^\gamma \| \oldh \|_u + \| \oldh \|_s |\phi - (\psi  \, 
J_0 \circ \bh_Y^{-1}) \circ \Phi_{-t_Y} \, J_{V_Y} \Phi_{-t_Y} |_{\cC^\oldq(V'_Y)}
\,  .
\end{split}
\end{equation}
where we used that  $d_{\cW^s}(P^+(W'_Y), P^+(V'_Y)) < \epsilon$ 
(recall we have removed the endpoint discrepancy).
We proceed to estimate the norm of the test function in the second term.  Since each term in the
difference is bounded, we may estimate each difference separately.  For $Z \in V'_Y$, we have
\[
|\psi \circ \bh_Y \circ S_{W_Y} \circ S_{V_Y}^{-1}(Z) - \psi \circ \Phi_{-t_Y}(Z)| 
\le |\psi|_{\cC^\alpha(W)} d(\bh_Y \circ S_{W_Y} \circ S_{V_Y}^{-1}(Z), \Phi_{-t_Y}(Z))^\alpha \, .
\]
Using again  $d_{\cW^s}(P^+(W'_Y), P^+(V'_Y)) < \epsilon$, we have,
$d(S_{W_Y} \circ S_{V_Y}^{-1}(Z), Z) \le C \epsilon$.  Moreover, $\bh_Y$ is simply translation
by $Y$, and $|Y| \le C\epsilon$ while $|t_Y| \le C_\kappa \epsilon$ by \eqref{eq:curve const}.
Then using the triangle inequality twice gives
$d(\bh_Y \circ S_{W_Y} \circ S_{V_Y}^{-1}(Z), \Phi_{-t_Y}(Z)) \le C \epsilon$.
Thus,
\[
|\psi \circ \bh_Y \circ S_{W_Y} \circ S_{V_Y}^{-1}(Z) - \psi \circ \Phi_{-t_Y}(Z)| 
\le C \epsilon^\alpha ,
\]
for each $Z \in V'_Y$.  For brevity, set $\psi_1 = \psi \circ \bh_Y \circ S_{W_Y} \circ S_{V_Y}^{-1}$
and  $\psi_2 = \psi \circ \Phi_{-t_Y}$. 
Now given $Z, Z' \in V'_Y$, we estimate on the one hand using the above,
\[
|\psi_1(Z) - \psi_2 (Z) - \psi_1(Z') + \psi_2 (Z')| \, d(Z,Z')^{-\beta} 
\le 2 C \epsilon^\alpha d(Z,Z')^{-\beta} \, ,
\]
while on the other, using the H\"older continuity of $\psi_1$ and $\psi_2$ separately
(since $\bh_Y, S_{W_Y}, S_{V_Y}^{-1}$ and $\Phi_{-t_Y}$ are all smooth functions),
\[
|\psi_1(Z) - \psi_2 (Z) - \psi_1(Z') + \psi_2 (Z')| \, d(Z,Z')^{-\beta}  \le 2 C|\psi|_{\cC^\alpha(W)} d(Z,Z')^{\alpha-\beta} \,  .
\]
The H\"older constant is bounded by the minimum of the two estimates, which is maximized
when the two are equal, i.e., when $d(Z,Z') = \epsilon$.  Thus,
\begin{equation}
\label{eq:psi diff}
|\psi_1 - \psi_2 |_{\cC^\beta(V'_Y)} \le C \epsilon^{\alpha-\beta}\,  .
\end{equation}
Next, we must estimate,
\[
|J_0 \circ S_{W_Y} \circ S_{V_Y}^{-1} - J_0 \circ \bh_Y^{-1} \circ \Phi_{-t_Y}|_{\cC^\beta(V'_Y)} .
\]
But note that setting $\tilde \psi = J_0 \circ \bh_Y^{-1}$, this difference has precisely the same
form as the difference just estimated concerning $\psi$.  Since $J_0$ and $\bh_Y^{-1}$ are both
smooth functions, this estimate has the same bound as \eqref{eq:psi diff}.

Finally, we must estimate $| 1 - J_{V_Y} \Phi_{-t_Y} |_{\cC^\beta(V'_Y)}$.  Using the
linearity of the flow between collisions and again \eqref{eq:curve const}, 
we have $| 1- J_{V_Y} \Phi_{-t_Y} (Z)| \le C \epsilon$ for all $Z \in V_Y'$.  Then since 
$J_{V_Y} \Phi_{-t_Y}$ is uniformly $\cC^1$ on $V_Y$ due to the fact that $V_Y$ has
uniformly bounded curvature, we may use the same technique as that above
\eqref{eq:psi diff} to estimate $|1 - J_{V_Y} \Phi_{-t_Y} |_{\cC^\beta(V'_Y)} \le C \epsilon^{1-\beta}$.

Using these estimates on the test functions in \eqref{eq:first Y} together with \eqref{eq:missed}
 yields the following estimate on the first difference
in \eqref{eq:Y split},
\begin{equation}
\label{eq:second Y}
\begin{split}
& \left| \int_{W_Y} f \psi \circ \bh_Y J_0 \, dm_{W_Y} 
  - \int_{V_Y} f (\psi  \, 
J_0 \circ \bh_Y^{-1}) \circ \Phi_{-t_Y} \, J_{V_Y} \Phi_{-t_Y} \, dm_{V_Y} \right| \\
& \qquad \le C \epsilon^{1/(2q)} \| f \|_s+ C \epsilon^{\gamma} \| f \|_u + C \epsilon^{\alpha-\beta} 
\| f \|_s 
\,  .
\end{split}
\end{equation}

We proceed to estimate the second difference in  \eqref{eq:Y split}.
Using
\eqref{eq:unstable third one} and following,
\[
\begin{split}
\int_{\Phi_{t_Y}(W)} f \, (\psi  \, 
& J_0 \circ  \bh_Y^{-1}) \circ \Phi_{-t_Y} \, J_{V_Y} \Phi_{-t_Y} \, dm_{V_Y}  - \int_W f \, \psi \, J_0 \circ \bh_Y^{-1} \, dm_W  \\
& = \int_0^{t_Y} \partial_t \int_{\Phi_t(W)} f \, (\psi J_0 \circ \bh_Y^{-1}) \circ \Phi_{-t} \, J_{\Phi_t(W)} \Phi_{-t} \, dm_{\Phi_t(W)} \, dt \\
& = \int_0^{t_Y} \partial_t \int_W f \circ \Phi_t \cdot \psi \cdot J_0 \circ \bh_Y^{-1} \, dm_W \, dt \\
& = \int_0^{t_Y} \int_W \partial_r (f \circ \Phi_r)|_{r=0} \circ \Phi_t \cdot \psi \cdot J_0 \circ \bh_Y^{-1}  \, dm_W \, dt \\
& = \int_0^{t_Y} \int_{\Phi_t(W)} \partial_r (f \circ \Phi_r)|_{r=0} \, (\psi J_0 \circ \bh_Y^{-1}) \circ \Phi_{-t} \, \cdot J_{\Phi_t(W)}\Phi_{-t} \, dm_{\Phi_t(W)} \, dt \\
& \le \int_0^{t_Y} \| f \|_0 | (\psi J_0 \circ \bh_Y^{-1}) \circ \Phi_{-t} \cdot J_{\Phi_t(W)} \Phi_{-t} |_{\cC^\alpha(\Phi_t(W))} \, dt \, .
\end{split}
\]
On each stable curve $\Phi_t(W)$
the norm of the test function is bounded by a uniform constant since
the flow is linear in $t$ between collisions, and both $\bh_Y^{-1}$ and $J_0$ are $\cC^1$
functions.  Using these estimates and  
$|t_Y| \le C_\kappa \epsilon$, we obtain the following estimate for the second difference in 
\eqref{eq:Y split},
\begin{equation}
\label{eq:flow diff}
\left| \int_{\Phi_{t_Y}(W)} f (\psi  \, 
J_0 \circ \bh_Y^{-1}) \circ \Phi_{-t_Y} \, J_{V_Y} \Phi_{-t_Y} \, dm_{V_Y}  - \int_W f \, \psi \, J_0 \circ \bh_Y^{-1} \, dm_W \right|
\le C \epsilon  \| f \|_0  \,  .
\end{equation}
Putting this together with \eqref{eq:second Y} and \eqref{eq:Y split} completes the proof
of Sublemma~\ref{lem:close Y} since $\gamma \le \alpha - \beta$ by definition of the norms. 
\end{proof}

The sublemma being proved, we may return to the proof of Lemma~ \ref{mollbound2}:
Sublemma~\ref{lem:close Y} and \eqref{chain0} provide the required estimate
under the assumption that 
$P^+(W)$ does not lie within a distance
$A \epsilon^{3/5}$ of the boundary of a homogeneity strip.  

Now suppose that $P^+(W)$ lies within a distance of $A \epsilon^{3/5}$ of the boundary of
a homogeneity strip $\bH_{k_1}$.  
Let $\overline{W}$ denote the one connected component of $W$ which
does not project to within a distance of $A\epsilon^{3/5}$ of the boundary of $\bH_{k_1}$
(note that this may be empty, for example in homogeneity strips of high index).
We may perform the estimate on $\overline{W}$ precisely as in Sublemma~\ref{lem:close Y}.

By construction, $\overline{W}\setminus W$ consists of at most two components
of length at most $A \epsilon^{3/5}$.  Call one of them $W_1$.  Although now
$N_\ve(W_1)$ may cross countably many homogeneity regions, we proceed using \eqref{eq:one box}
and the decomposition of Lebesgue measure
from the proof of Lemma~\ref{mollbound1}.  Now,
\[
\left| \int_{W_1} \psi(Z) dm_W \sum_{k,i} \int_{B^{0-}_i} f(Y) \eta_\eps(Z-Y) dm(Y) \right|
\le \sum_{k,i} |W_1| C \eps^{-2 -\beta +1/q} k^{-5} \| f\|_s .
\]
As in the proof of Lemma~\ref{mollbound1}, the sum over $k$ yields a factor proportional to
$\epsilon$.  Finally, $|W_1| \le A \eps^{3/5}$, completing the proof of the lemma.
\end{proof}

\begin{proof}[Proof of Lemma~\ref{lem:C1}]
Let $f \in \cC^1(\Omega_0)$ and define 
$\bM_\epsilon^\sim f = \bM_\epsilon (f \cdot 1_{\Omega_0 \setminus \partial_\eps \Omega_0})$, 
where $\partial_\eps \Omega_0$ denotes the $\eps$-neighborhood of $\partial \Omega_0$ in
$\Omega_0$.  Note that $\bM_\epsilon^\sim f \equiv 0$ on $\partial \Omega_0$, so that
$\bM_\eps^\sim f \in \cC^2(\Omega_0) \cap \cC^0_\sim$.  We will show that
$\bM_\eps^\sim f$ is a good approximation of $f$ in $\cB$. 
Note that $\nabla \bM_\epsilon^\sim f= \bM_\epsilon^\sim(\nabla f)$ so
$| \bM_\epsilon^\sim f|_{\cC^1(\Omega_0)}\le |f|_{\cC^1(\Omega_0)}$.
By the proof of Lemma~\ref{lem:embedding}, $\| \bM_\epsilon^\sim f\|_{\cB}
\le C |f|_{\cC^1(\Omega_0)}$ for all $\epsilon >0$.

Let $W \in \cW^s$, $\psi \in \cC^\beta(W)$ with $|\psi|_{\cC^\beta(W)} \le |W|^{-1/q}$.  
Assume for the moment
that $W \cap \partial_{2\eps} \Omega_0 = \emptyset$.  Then using \eqref{eq:moll def} and the fact
that $\int \eta_\eps = 1$,
\[
\begin{split}
\int_W (f - \bM_\eps^\sim f) & \, \psi \, dm_W 
 = \int_W \psi(Z) \int_{\bR^3} \eta_\eps(Z-Y)  (f(Z) - f(Y)) \, dm(Y) dm_W(Z) \\
& \le  |f|_{\cC^1(\Omega_0)} |W|^{-1/q} \int_W \int_{\bR^3} \eta_\eps(Z-Y) |Z- Y| dm(Y) dm_W(Z) 
\le \eps |f|_{\cC^1(\Omega_0)},
\end{split}
\]
where in the last estimate, we have used that $q>1$ and $\eta_\eps(Z-Y) = 0$ if
$|Z-Y| > \eps$.

Next, if $W \cap \partial_{2\ve} \Omega_0 \neq \emptyset$, then we subdivide $W$ into
at most two components $W' = W \cap \partial_{2\ve} \Omega_0$ and at most one component
$W \setminus W'$.  On $W \setminus W'$, the above estimate holds.  On $W'$, we use the 
fact that $|W'| \le C \eps^{1/2}$ (see, for example, Lemma~\ref{lem:discarded}) to estimate,
\[
\int_{W'} (f - \bM_\eps^\sim f) \, \psi \, dm_W \le |W'|^{1-1/q} 2 |f|_{\cC^0(\Omega_0)} 
\le C \eps^{\frac 12(1 - \frac 1q)} |f|_{\cC^0(\Omega_0)} . 
\]
Putting these two estimates together, we have 
\begin{equation}
\label{eq:stable moll}
\| f - \bM_\eps^\sim f \|_s \le C \eps^{\frac 12 (1 - \frac 1q)} |f|_{\cC^1(0)} .
\end{equation}

Next, we estimate the strong unstable norm.  For $\ve > 0$, let $W_1, W_2 \in \cW^s$ with
$d_{\cW^s}(W_1, W_2) < \ve$.  Let $\psi_i \in \cC^\alpha(W_i)$ with $|\psi_i|_{\cC^\alpha(W_i)} \le 1$
and $d(\psi_1, \psi_2) = 0$.

Let $s>0$ be a small number to be chosen below. 
First assume that $\eps^{\frac 12 - s} < \ve^\gamma$.  Then using the estimates for the
strong stable norm on each curve separately, we have,
\[
\ve^{-\gamma} \left| \int_{W_1} \! \! (f - \bM_\eps^\sim f) \psi_1 \, dm_{W_1} -
\int_{W_2} \! \! (f - \bM_\eps^\sim f) \psi_2 \, dm_{W_2} \right| 
\le C \ve^{-\gamma} \eps^{1/2} |f|_{\cC^1(\Omega_0)} 
\le C \eps^s |f|_{\cC^1(\Omega_0)}  .
\]
Notice that we do not have the exponent $1-1/q$ in this estimate since the test functions used in the
strong unstable norm are nicer than those used in the strong stable norm.

Now suppose $\eps^{\frac 12 -s} \ge \ve^\gamma$.  Recalling the notation of 
Definition~\ref{def:dist}, we let
\[
W_i = \{ S_{W_i}(r) = \Phi_{-t(r)} \circ G_{W_i}(r) : r \in I_i \} ,
\]
where $G_{W_i}(r) = (r, \phi_i(r))$ is the graph of $P^+(W_i)$ over the $r$-interval $I_i$.
Denote by $U_i \subset W_i$ the maximal subcurve such that $P^+(U_1)$ and 
$P^+(U_2)$ are defined as graphs over the interval $I = I_1 \cap I_2$.  We have at most
two pieces $V_i = W_i \setminus U_i$ and by definition of $d_{\cW^s}(\cdot, \cdot)$ and
Lemma~\ref{lem:smooth}, we have $|V_i| \le C|P^+(V_i)| \le C\ve$.  Thus the estimate over the
(at most two)
unmatched pieces $V_i$ is,
\[
\ve^{-\gamma} \left| \int_{V_i} (f - \bM_\eps^\sim f) \psi_1 \, dm_{W_1} \right|
\le C \ve^{1-\gamma} 2 |f|_{\cC^0(\Omega)} \le C \eps^{(\frac 12 -s)(\frac 1\gamma -1)} .
\]
It remains to estimate the norm of $f - \bM_\eps^\sim f$ on matched pieces.
\begin{equation}
\label{eq:moll match}
\begin{split}
\int_{U_1} &(f - \bM_\eps^\sim f) \psi_1 \, dm_{W_1} - 
\int_{U_2} (f - \bM_\eps^\sim f) \psi_2 \, dm_{W_2}
= \int_{U_1} f \psi_1 \, dm_{W_1}  \\
& -\int_{U_2} f \psi_2 \, dm_{W_2}+ \int_{U_2} \bM_\eps^\sim f \,  \psi_2 \, dm_{W_2} -
\int_{U_1} \bM_\eps^\sim f \, \psi_1 \, dm_{W_1} \, .
\end{split}
\end{equation}
To estimate the difference of integrals in $f$, we write, using the fact that
$\psi_1 \circ S_{W_1} = \psi_2 \circ S_{W_2}$ by assumption,
\[
\begin{split}
 \int_{U_1} f \psi_1 \, dm_{W_1} -
& \int_{U_2} f \psi_2 \, dm_{W_2}
= \int_I (f \circ S_{W_1} JS_{W_1} - f \circ S_{W_2} JS_{W_2}) \, \psi_1 \circ S_{W_1} \, dr  \\
& \le C |I| |f \circ S_{W_1} \cdot JS_{W_1} - f \circ S_{W_2} \cdot JS_{W_2}|_{\cC^0(I)} \\
& \le C|I| \big( |f|_{\cC^1(\Omega_0)} |S_{W_1} - S_{W_2}|_{\cC^0(I)} + |f|_{\cC^0(\Omega_0)} 
|JS_{W_1} - JS_{W_2}|_{\cC^0(I)} \big) \, .
\end{split}
\]
For the first term above, since expansion in the stable direction from 
$P^+(W_i)$ to $W_i$ is of order 1 by Lemma~\ref{lem:smooth}, and since the unstable direction
only contracts under this map, we have 
\[
|S_{W_1} - S_{W_2}|_{\cC^0(I)} \le C |G_{W_1} - G_{W_2}|_{\cC^0(I)} \le C \ve,
\]
by assumption on $W_1$ and $W_2$.  For the second term,
$|JS_{W_1} - JS_{W_2}|_{\cC^0(I)} \le C \ve$,  
again using Lemma~\ref{lem:smooth} and that $U_1$ and $U_2$ are $\cC^1$-close as graphs
over $I$ by definition of $d_{\cW^s}(W_1, W_2)$.   Thus,
\[
\ve^{-\gamma} \left| \int_{U_1} f \psi_1 \, dm_{W_1} -
 \int_{U_2} f \psi_2 \, dm_{W_2} \right| \le C \ve^{1-\gamma} |f|_{\cC^1(\Omega_0)} 
 \le C \eps^{(\frac 12 - s)(\frac 1\gamma-1)} |f|_{\cC^1(\Omega_0)},
\]
completing the estimate on the first term of \eqref{eq:moll match}.  The second term is estimated
similarly, using the fact that $|\eta_\eps|_{\cC^1(\Omega_0)} \le C \eps^{-1} |f|_{\cC^0(\Omega_0)}$,
\begin{align*}
\ve^{-\gamma} \left| \int_{U_1} \bM_\eps^\sim f \, \psi_1 \, dm_{W_1} -
 \int_{U_2} \bM_\eps^\sim f \, \psi_2 \, dm_{W_2} \right| 
&\le C \ve^{1-\gamma} |\bM_\eps^\sim f|_{\cC^1(\Omega_0)} \\
&\le C \eps^{(\frac 12 - s)(\frac1\gamma - 1) -1} |f|_{\cC^0(\Omega_0)} \, .
 \end{align*}
Putting together these estimates on the terms in  \eqref{eq:moll match} with the estimate on unmatched pieces
yields an exponent of $\eps$ at worst $(\frac 12 - s)(\frac 1\gamma -1) -1$.  Since $\gamma < 1/3$,
we have $\frac 1\gamma -1 > 2$, and so we may choose $s>0$ sufficiently small that the 
above exponent is positive.

Finally, we estimate the neutral norm.  Fix $W \in \cW^s$ and $\psi \in \cC^\alpha(W)$ with
$|\psi|_{\cC^\alpha(W)} \le 1$.  Recall that $\heta$ denotes the unit vector in the flow direction.  Then,
$\partial_t (f \circ \Phi_t)|_{t =0} = \nabla f \cdot \heta$ and also for 
$Z \in \Omega_0 \setminus \partial_{2 \eps} \Omega_0$,
\[
\partial_t ((\bM_\eps^\sim f) \circ \Phi_t)|_{t=0}(Z) = \bM_\eps^\sim (\nabla f \cdot \heta)(Z) .
\]
Now suppose $W \cap \partial_{2 \eps} \Omega_0 = \emptyset$.  Then,
\[
\int_W \partial_t (f - \bM_\eps^\sim f) \circ \Phi_t |_{t=0} \, \psi \, dm_W
= \int_W \psi(Z) \int_{\bR^3} \eta_\eps(Z-Y) (\nabla f \cdot \heta(Z) - \nabla f \cdot \heta (Y)) \, . 
\]
Now since $\nabla f$ is uniformly continuous on $\Omega_0$, there exists a function
$\rho(\eps)$ with $\rho(\eps) \downarrow 0$ as $\eps \downarrow 0$, such that
$|\nabla f (Z) - \nabla f(Y) | \le \rho(\eps)$ whenever $|Z-Y| < \eps$.  Since the flow direction
$\heta$ also changes smoothly on $\Omega_0$, we have
\[
\left| \int_W \partial_t (f - \bM_\eps^\sim f) \circ \Phi_t |_{t=0} \, \psi \, dm_W \right|
\le C \rho(\eps) \, . 
\]
If, on the other hand, $W \cap \partial_{2 \eps} \Omega_0 \neq \emptyset$, then using the same
decomposition of $W$ as in the proof the strong stable norm estimate, we have
\[
\int_{W'}  \partial_t (f - \bM_\eps^\sim f) \circ \Phi_t |_{t=0} \, \psi \, dm_W
\le 2 |W'|  |\nabla f|_{\cC^0(\Omega_0)} \le C \eps^{1/2} |f|_{\cC^1(\Omega_0)} \, .
\]
Since $f$ can be approximated by functions in $\cC^2(\Omega_0) \cap \cC^0_\sim$
with uniformly bounded $\| \cdot \|_{\cB}$-norms, it follows
that $f \in \cB$.
\end{proof}


\section{The Dolgopyat cancellation estimate (Lemma ~\ref{dolgolemma})}
\label{dodo}

Recall the hyperbolicity exponent $\Lambda =\Lambda_0^{1/\tau_{\max}}>1$ from \eqref{Lambda}.
\begin{lemma}[Dolgopyat bound]\label{dolgolemma}
There exists  $\Cs>0$, and, for any $0<\oldp\le 1/3$, there exist  $C_{Do}>0$
and $\gamma_{Do}>0$ so 
that for   any homogeneous curve $W\in \cW^s$ of length
$|W|=m_W(W)\leq 1$ and  any\footnote{ We do not assume here that $\oldh$ is
supported in $\Omega_0$, recalling \eqref{deflip}.} $\oldh \in \cC^1(\mathbb T^3)$
\begin{align*}
 &\sup_{\substack{\psi \in \cC^\oldp(W) \\ |\psi|_{\cC^\oldp(W)} \le 1}}\int_W \psi  \cR(a+ib)^{2m} (\oldh)\,   dm_W 
\leq 
\frac{\Cs }{a^{2m} b^{\gamma_{Do}}}
\big(|\oldh|_{L^\infty(\Omega_0)}+
(1+a^{-1}\ln\Lambda)^{-m}|\oldh|_{H^1_\infty(\Omega_0)} \big)\, ,\\
&\qquad\qquad\qquad  
\forall \, 1<a<2\, , \quad \forall \, b>1\, ,  \quad \forall \, m\geq C_{Do}\ln b \, .
\end{align*}
\end{lemma}

Since $\overline{\cR(a+ib)}=\cR(a-ib)$ the obvious counterpart of the above lemma holds  for $b<0$.

Note that we do not require the $\delta$-averaging operator used in
\cite{Li04, BaL}: Since the weak norm in the right-hand side of the
Lasota--Yorke estimates already involves  averaging over stable curves, we can
use the weak norm directly.

The rest of the section contains the proof of Lemma \ref{dolgolemma} and consists of a direct, 
but lengthy and highly non trivial, computation.  In the present section, $\Cs$ denotes a 
constant depending only on the billiard dynamics
and not on $\gamma$, $\oldbeta$,  $\oldq$, or $\oldp$.

First of all note that if $|W|\leq b^{-\gamma_{Do}}$, then the statement of the lemma is trivially true. We can thus assume that $|W|\geq b^{-\gamma_{Do}}$.

It is more convenient to work directly with the flow rather than with the Poincar\'e sections,
and  we will consider time steps $\vu\ll \tau_{\min}$.
The precise value of $\vu$ will be chosen later in \eqref{eq:vu}.

To compute the integral, it is helpful to localize it in space-time. To localize in time, consider a smooth function $\tilde p:\bR\to\bR$ such that $\operatorname{supp} \tilde p\subset (-1,1)$, $\tilde p(s)=\tilde p(-s)$, and $\sum_{\ell\in \bZ}\tilde p(t-\ell)=1$ for all $t\in\bR$. 
For  $\oldh\in L^\infty(\Omega_0, \text{vol})$, we set $p(s)=\tilde p(\vu^{-1}s)$ and write
\begin{align}
\nonumber
\cR(z)^m(\oldh)&=\sum_{\ell\in\bZ}\int_{0}^\infty p(t-\ell\vu) \frac{t^{m-1}}{(m-1)!}e^{-zt}\cL_t \oldh\, dt\\
\label{eq:step-1}&=\sum_{\ell\in\bN^*}\int_{-\vu}^{\vu}  p(s) \frac{(s+\ell\vu)^{m-1}}{(m-1)!}e^{-z\ell\vu-zs}\cL_{\ell\vu}\cL_s \oldh\, ds
+\int_{0}^{\vu} p(s) \frac{s^{m-1}}{(m-1)!}e^{-zs}\cL_s \oldh\, ds\, .
\end{align}
Next, setting for $\ell \ge 1$,
\begin{equation}\label{eq:pstuff0}
p_{m,\ell,z}(s)=  p(s) \frac{(s+\ell\vu)^{m-1}}{(m-1)!}e^{-z\ell\vu-zs},\; \; \mbox{and} \; \; p_{m,0,z}(s)=  p(s) \frac{s^{m-1}}{(m-1)!}e^{-zs}\Id_{\{s\geq 0\}} \, , 
\end{equation} 
we fix $\psi \in \cC^\alpha(W)$ with $|\psi|_{\cC^\alpha(W)} \le 1$ and write
\begin{equation}\label{eq:step0}
\int_W \psi \cdot ( \cR(z)^m (\oldh)) \, dm_W=
 \sum_{\ell\in\bN}\int_{-\vu}^{\vu}
p_{m,\ell,z}(s)\int_{W} \psi \cdot \cL_{\ell\vu}\cL_s \oldh \, dm_W ds \, .
\end{equation} 
Changing variables, we rewrite \eqref{eq:step0} as
\begin{equation}\label{eq:step02}
\begin{split}
\int_W \psi \cdot &(\cR(z)^m (\oldh))\, dm_W =
\sum_{\ell\in\bN}\,  \int_{-\vu}^{\vu}  p_{m,\ell,z}(s)
\int_{\Phi_{-\ell\vu}W} J^s_{\ell\vu}\cdot \psi\circ \Phi_{\ell \vu} \cdot  \cL_s \oldh \, ds \\
=&\sum_{\ell\in\bN}\,\, \sum_{W_A\in  \hG_{\ell\vu}(W)}\int_{-\vu}^{\vu}\, \,
 p_{m,\ell,z}(s)\!\!\int_{W_A} J^s_{\ell\vu}\cdot  \psi\circ \Phi_{\ell \vu} \cdot  \cL_s \oldh \, dm_{W_A} \, ds\, ,
\end{split}
\end{equation}
where  $J^s_{\ell\vu} = J^s_{W_A}\Phi_{\ell \vu}$
 is the (stable) Jacobian of the change of variable and 
$\hG_{\ell\vu}(W):=\{W_A\}_{A\in A_\ell}$ is the decomposition of 
$\Phi_{-\ell\vu}W$ specified in Definition~\ref{cG_t}. 
By Lemma~\ref{lem:preserved}, for each $A\in A_\ell$ there exists $t_A\in[0,\tau_{\min}]$ such that  
$\Phi_{-t_A}W_A$ is a $\cC^2$ curve with uniform $\cC^2$ norm and $\Phi_{\ell\vus-t_A}$, restricted to $\Phi_{-t_A}W_A$ is a $\cC^2$ map.

Lemma \ref{lem:distortion} implies that $\Cs^{-1}\frac{|\Phi_{\ell\vu}W_A|}{|W_A|}\leq J^s_{\ell\vu}\leq \Cs\frac{|\Phi_{\ell\vu}W_A|}{|W_A|}$, provided $\Cs$ is chosen large enough.
Note that  $\sum_{A}\int_{W_A}  J^s_{\ell\vu}=|W|$.

Next, we shall localize in space as well, introducing in
\eqref{starstar} below a sequence of smooth partitions
of unity parametrised by
\[
\theta\in (0,1) \text{ and } \epsilonr \in (0, L_0)
\]
to be chosen later.  First, we need some preparations: 
We shall exclude a neighborhood of $\partial \Omega_0$ since\footnote{ In such a way in the following we will never need to consider a curve in the the midst of a collision.} the angle of the
stable and unstable cones can change quickly near this boundary.  
We shall also exclude
a neighborhood of the surfaces on which either the stable or unstable cones we have defined
are discontinuous.  Define
\begin{equation}\label{beware}
S^{0-} = \{ Z \in \Omega_0 : \vf(P^+(Z)) = \pm \pi/2 \}
\; \; \mbox{ and } \; \;
S^{0+} = \{ Z \in \Omega_0 : \vf(P^-(Z)) = \pm \pi/2 \} ,
\end{equation}
and their $\epsilonr$-neighborhoods, $S_\epsilonr^{0\pm} = \{ Z \in \Omega_0 : d(Z, S^{0\pm}) < \epsilonr \}$.
Note that $S^{0-}$ is the flow of the singularity curve for the map $\cS_0 = \{ \pm \pi/2 \}$ backwards until its 
first collision, while $S^{0+}$ is the flow of $\cS_0$ forwards until its next collision.
Stable cones are discontinuous across the surface $S^{0-}$ while unstable cones are
discontinuous across $S^{0+}$.
Similarly, let $\partial_\epsilonr(\Omega_0)$ denote the $\epsilonr$-neighborhood of $\partial \Omega_0$ 
{\em within} $\Omega_0$.  The following lemma shows that the curves in $\cW^s$ have small intersections with such sets. We will then be able to discard such intersections and this will allow us to introduce special charts as in Remark~\ref{suitcharts0}; see Remark \ref{suitcharts} below.  This lemma has no analogue in \cite{BaL}, where the situation was much easier.

\begin{lemma}
\label{lem:discarded}
There exists $\Cs>0$ such that for any $\epsilonr \in [0,L_0]$ and any homogeneous $W \in \cW^s$,
\[
|W \cap S_\epsilonr^{0-}| \le \Cs \epsilonr^{3/5}, \quad |W \cap S_\epsilonr^{0+}| \le \Cs \epsilonr, \quad
|W \cap \partial_\epsilonr \Omega_0| \le \Cs \epsilonr^{1/2} .
\]
Similar bounds hold for unstable curves with the estimates on 
$S_\epsilonr^{0-}$ and $S_\epsilonr^{0+}$ reversed.
\end{lemma}

\begin{proof}
Let $V = W \cap S_\epsilonr^{0-}$. The curve $P^+(W)$ is a homogeneous stable curve
by assumption.  Due to the uniform transversality of the map-stable cones
with horizontal lines, $P^+(V)$ is uniformly transverse to $\cS_0 = \{ \vf = \pm \pi/2 \}$.
Suppose $P^+(V)$ lies in a homogeneity strip of index $k$.  The expansion in the
stable cone from $P^+(V)$ to $V$ is of order 1, by Lemma~\ref{lem:smooth}.  On the other hand,
vectors in the unstable cone undergo an expansion of order $k^2$.  Thus the angle between
$V$ and $S^{0-}$ is no smaller than order $k^{-2}$, which means that $|V|$ is bounded
by a uniform constant times $\epsilonr k^2$.  But again using Lemma~\ref{lem:smooth},
we have the length of $|V|$ comparable to $|P^+(V)|$, which is at most $k^{-3}$.  
Since this holds for any $V'$ in the same homogeneity strip, in particular it is true of a curve
$V'$ of length $k^{-3}$ (even if $V$ itself is shorter).  Thus up to
a uniform constant, we must have $k^{-3} \leq \epsilonr k^2$, which implies $k \geq \epsilonr^{-1/5}$.  This means
$|P^+(V)|$ and so $|V|$ is at most a uniform constant times $\epsilonr^{3/5}$.

The analysis is similar for $V = W \cap S_\epsilonr^{0+}$.  
Note that $S^{0+}$ is a surface which may divide $V$ into several
components.  Let $V_1$ denote a component of $V$ on one side of $S^{0+}$, so that
$P^-(V_1)$ is contained in a single scatterer.  Since $P^-(V_1)$ is a stable curve, even though 
it may not be homogeneous, it is uniformly transverse to $\cS_0$.  Since 
$S^{0+}$ is the surface defined by flowing $\cS_0$ forward,
$S^{0+}$ will remain
uniformly transverse to $V_1$.  Thus $|V_1| \le \Cs \epsilonr$. Since the number of such
components of $V$ is uniformly bounded by $\lfloor \frac{\tau_{\max}}{\tau_{\min}} \rfloor +1$
\cite[\S 5.10]{chernov book},
the bound on $|V|$ follows.

Finally, let $V = W \cap \partial_\epsilonr \Omega_0$.  Recall our global coordinates $(x,y,\omega)$
from Section~\ref{setting}.
Although $V$ may not be uniformly transverse to
$\partial \Omega$ (consider when $V$ is close to making a normal collision with a 
scatterer), the estimate will follow from the fact that its curvature is uniformly bounded away from the
curvature of the scatterer.  

Note that the (absolute value) of the curvature of a stable wavefront (the projection of a stable curve 
in the $(x,y)$-plane) just after a collision
is given by $B^+ = B^- + \frac{2 K}{\cos \vf}$, where $B^-$ is the curvature just before collision;
between collisions, the curvature evolves according to $B_t = \frac{B_0^+}{B_0^+ + t}$
\cite[\S 3.8]{chernov book}.  Putting these together and using the definition of our stable
cones for the flow, we see that the minimum curvature for any stable curve is
$B_{\min} = \frac{2 \cK_{\min}}{2 \cK_{\min} + \tau_{\max}}$.

Now if $V$ is a component of a stable curve about to make a nearly normal collision
in backward time, then the curvatures of $V$ and the scatterer are convex in opposition, both
with minimum curvatures bounded away from 0 by our calculation above.  Thus 
the length of the projection of $V$ in the $(x,y)$-plane is bounded by $C\epsilonr^{1/2}$
and since the slope of $|V|$ in the $(d\xi, d\omega)$-plane just before collision is bounded 
above by $2/\tau_{\min}$, we have
$|V| \le C \epsilonr^{1/2}$ as well.  If, on the other hand, $V$ has just made a nearly normal collision,
then the curvature of its projection for a short time afterward
is at least 
\begin{equation}
\label{cardinality}
B^+ \ge B_{\min} + \frac{\cK}{\cos \vf}
\, ,
\end{equation}
 while the curvature of the scatterer is $\cK$.
Thus the curvatures of the two projections are bounded away from one another, and
so $|V| \le C \epsilonr^{1/2}$ as before.

The last case to consider is if $V$ is close to a tangential collision.  In this case, the projection
of $V$ in the $(x,y)$-plane is uniformly transverse to the boundary of the scatterer; however,
by the proof of Lemma~\ref{lem:preserved}, the length of $V$ in the $\omega$-coordinate 
undergoes an expansion of order $1/\cos \vf$ near such tangential collisions, 
and so again, $|V| \le C\epsilonr^{1/2}$. 
 \end{proof}

Let $c>2$ be a  constant that will be chosen shortly and let 
$N_{\epsilonr^\theta}(Z)$ denote the neighborhood of $Z$ in $\Omega_0$ 
of size $\epsilonr^\theta$ in the standard $(x,y, \omega)$ coordinates.
There exists $\Cs>0$ such that, for each $\epsilonr\in(0,L_0)$, there exists a 
$\cC^{\infty}$ partition of unity  of $\Omega_0$
\begin{equation}\label{starstar}
\{\phi_{\epsilonr,i}\}_{i=0}^{q(\epsilonr)}
\end{equation}
enjoying the following properties
\begin{enumerate}
\item[(i)] for each $i\in\{1,\dots,q(\epsilonr)\}$, there exists $x_i\in \Omega_0$ such that $\phi_{\epsilonr,i}(z)=0$ for all $z\not\in N_{\epsilonr^\theta/c}(x_i)$;
\item[(ii)] for each $\epsilonr,i$ we have $|\nabla\phi_{\epsilonr,i}|_{\infty}\leq\Cs \epsilonr^{-\theta}$;
\item[(iii)]  $q(\epsilonr)\leq \Cs \epsilonr^{-3\theta}$;
\item[(iv)] for each $\epsilonr > 0$ we have supp$(\phi_{\epsilonr, 0}) \subset \partial_{3c\epsilonr^\theta} \Omega_0 \cup S_{3c\epsilonr^\theta}^{0-}
\cup S_{3c\epsilonr^\theta}^{0+}$, and, for every $i \ge 1$, supp$(\phi_{\epsilonr,i}) \subset \Omega_0 \setminus (\partial_{2c\epsilonr^\theta} \Omega_0 \cup S_{2c\epsilonr^\theta}^{0-} \cup S_{2c\epsilonr^\theta}^{0+})$.
\end{enumerate}

\begin{remark}[$\cC^2$ cone-compatible Darboux charts]\label{suitcharts}
Following Remark~\ref{suitcharts0}, we setup a similar family of local charts 
as announced in Remark~\ref{Rk1.4}:
For any $Z \in \Omega_0 \setminus (\partial_{2c\epsilonr^\theta} \Omega_0 \cup S_{2c\epsilonr^\theta}^{0-}
\cup S_{2c\epsilonr^\theta}^{0+})$, it follows from the proof of Lemma~\ref{lem:discarded}
that $\cos \vf(P^+(Z)) \ge \epsilonr^{\theta/2}$ for $c$ sufficiently large.
Similarly, if $Z' \in \Omega_0$ satisfies $d(Z, Z') < c\epsilonr^\theta$, we have $\tau(Z') \ge c\epsilonr^\theta$
and $\cos \vf(P^+(Z')) \ge \epsilonr^{\theta/2}$.
Thus it follows from \eqref{eq:stable cone} that the width of the stable cones at
$Z$ and $Z'$ has angle at least of order $\epsilonr^{\theta/2}$, and the maximum and minimum
slopes in $C^s(Z)$ and $C^s(Z')$ are uniformly bounded multiples of one another.
Similar considerations hold for the unstable cones.
Indeed, given the relation $k \le \epsilonr^{-\theta/5}$ and the fact that 
$d(P^+(Z), P^+(Z')) < \Cs \epsilonr^\theta$ by Lemma~\ref{lem:smooth}, we conclude that
$P^+(Z)$ and $P^+(Z')$ must lie either in the same homogeneity strip or in adjacent 
homogeneity strips for small $\epsilonr$; specifically, $\epsilonr^{\theta} \le C' k^{-5}$,
 where $C'$ depends only
on $c$, the distortion of $P^+$ given by Lemma~\ref{lem:smooth} and the spacing of
the homogeneity strips.
Thus for each $i$, we may adopt local coordinates
in a $c\epsilonr^\theta$ neighborhood
of $x_i$,  $(x^u, x^s, x^0)$  as introduced in Remark~\ref{Rk1.4},  for which the contact
form is in standard form and $x_i$ is the origin, and where $c$
is large enough so that $N_{\epsilonr^\theta/c}(x_i)\subset
B_{\epsilonr^\theta}(x_i)$, where $B_{\epsilonr^\theta}(x_i)$ denotes the ball of radius $\epsilonr^\theta$ centered at $x_i$,
in the sup norm of the $(x^u, x^s, x^0)$ coordinates.   By the above discussion, the charts
effecting this change of coordinates are uniformly $\cC^2$.
\end{remark}

We will refer to the two sides of the box
$B_{\epsilonr^\theta}(x_i)$ comprising approximate weak stable manifolds as ``stable sides,'' the 
two sides comprising approximate weak unstable manifolds as ``unstable sides'' and the
two remaining sides as ``flow sides.''

Fix $c>2$ large enough and choose $\epsilonr$ small enough so that any manifold $W_A$ 
intersecting $B_{\epsilonr^\theta}(x_i)$ can intersect only (see \cite[Fig. 1]{BaL}) the weak unstable sides
of the boundary of $B_{c\epsilonr^\theta}(x_i)$. (This is possible since the $W_{A}$ belong to the stable cone
and the cones vary continuously in 
$\Omega_0 \setminus (\partial_{2c\epsilonr^\theta} \Omega_0 \cup S_{2c\epsilonr^\theta}^{0-}
\cup S_{2c\epsilonr^\theta}^{0+})$.) 
Define for $\ell \ge 0$, 
\[
A_{\ell,0}=D_{\ell, 0} = \{ A \in A_\ell \; : \; W_A \cap (\partial_{2c\epsilonr^\theta} \Omega_0 \cup S_{2c\epsilonr^\theta}^{0-}
\cup S_{2c\epsilonr^\theta}^{0+}) \neq \emptyset \},
\]
and for each $x_i$, $i \ge 1$, let  
\[
\begin{split}
&A_{\ell, i}=\{A\in A_\ell\;:\; W_A\cap B_{\epsilonr^\theta}(x_i)\neq \emptyset\},\\
D_{\ell,i}=\{A\in A_{\ell,i}\;:\;&\partial(W_A\cap B_{c\epsilonr^\theta}(x_i))\not\subset\partial B_{c\epsilonr^\theta}(x_i)\}\, , \qquad
E_{\ell,i}=A_{\ell,i}\setminus D_{\ell,i} \, .
\end{split}
\] 
The manifolds with index in $A_{\ell,i}$ are those which intersect the small box $B_{\epsilonr^\theta}(x_i)$, while
$E_{\ell,i}\subset A_{\ell,i}$ consists of the indices of those manifolds which go completely across the big
box $B_{c\epsilonr^\theta}(x_i)$.
The remaining manifolds of $A_{\ell,i}$, with indexes in $D_{\ell,i}$, are called the {\em discarded manifolds.}  
We set $W_{A, i}=W_{A}\cap B_{c\epsilonr^\theta}(x_i)$ for $i \ge 1$, 
$W_{A,0} = W_A \cap \partial_{3c\epsilonr^\theta} \Omega_0 \cup S_{3c\epsilonr^\theta}^{0-}
\cup S_{3c\epsilonr^\theta}^{0+}$,
 and
\begin{equation}\label{eq:ro-la1} 
Z_{\ell, A,i}= \int_{W_{A,i}}  J^s_{\ell\vu} \, dm_{W_{A,i}} \, , \qquad i \ge 0 \, .
\end{equation}
It is now natural to choose
\begin{equation}\label{eq:vu}
\vu=c \epsilonr^\theta\, .
\end{equation}

We are going to worry about the small $\ell$ first. Let $a=\Re(z)$. Remembering \eqref{eq:pstuff0}, setting
\begin{equation}\label{eq:l0conv}
\ell_0:=\frac{m}{ae^{2}\vu } \, ,
\end{equation}
and using Stirling's formula, we have
\begin{equation}\label{Stirling}
\begin{split}
\left| \sum_{\ell\leq  \ell_0 } \right. &  \left. \int_{-\vu}^{\vu}
p_{m,\ell,z}(s)\int_{W} \psi \cdot \cL_{\ell\vu+s} \oldh \, dm_W ds \right| 
\leq |f|_\infty |\psi|_\infty |W| \int_0^{\ell_0 \vu} e^{-as} \frac{s^{m-1}}{(m-1)!} \, ds \\
& \leq \Cs \frac{|\oldh|_\infty|\psi|_\infty}{(m-1)!a^{m}}\int_0^{e^{-2}m }e^{-x}x^{m-1} dx
\leq \Cs \frac{|\oldh|_\infty|\psi|_\infty(e^{-2}m )^m}{m!a^{m}}\leq \Cs \frac{|\oldh|_\infty|\psi|_\infty}{a^{m}e^{m}} ,  
\end{split}
\end{equation}
where we have made the substitution $x = as$ in the second line.
We impose  that
\begin{equation}\label{eq:h1}
e^{-m}\leq  \epsilonr^{\frac\theta 2} \, .
\end{equation}

Our next step is to estimate the contribution of the discarded manifolds corresponding to
indices in $D_{\ell,i}$
and $\ell \ge \ell_0$.
For fixed $\ell$, we use the fact that the $B_{\epsilonr^\theta}(x_i)$ have a bounded number of overlaps
as well as the fact that the number of components corresponding to a $D_{\ell,i}$
for each curve $W_A \in \cG_{\ell \vu}(W)$ is uniformly bounded to estimate
\begin{equation}
\begin{split}\label{eq:discard}
\sum_{i \ge 0}\sum_{A\in D_{\ell,i}} Z_{\ell,A,i} 
& \leq \Cs \sum_{i \ge0} \sum_{A\in D_{\ell,i}} |W_{A,i}| |J^s_{\ell \vu}|_{L^\infty(W_{A,i})} \\ 
& \le \Cs \epsilonr^{\theta/2} \sum_{W_A \in \cG_{\ell \vu}(W)}  |J^s_{\ell \vu}|_{L^\infty(W_A)}
\le \Cs \epsilonr^{\theta/2},
\end{split}
\end{equation}
where we have used Lemma~\ref{lem:discarded} to bound the lengths
of the discarded curves and Lemma~\ref{lem:growth} to bound the sum over the 
Jacobians.

\begin{remark}[About  Lemma 6.3 in \cite{BaL}]\label{corrigendum}
The statement of \cite[Lemma 6.3]{BaL}  is incorrect due to  a missing
sum over $\ell$ there (the error occurs in \cite[(E.6)]{BaL} since
$\Omega_{\beta,i}$ is two-dimensional while the neighborhood of $\widetilde O$
is three-dimensional). 
The argument of \cite[\S 6]{BaL} can be fixed either by using Lemma~ 6.8
of \cite{BaL}, or by providing a direct argument in the spirit of the one above.
\end{remark}

Now using \eqref{eq:discard} and remembering \eqref{eq:pstuff0}, we have
\begin{equation}\label{eq:step03}
\begin{split}
& \left|\sum_{\ell\geq \ell_0}\sum_{i\ge 0}\sum_{A\in D_{\ell,i}}\int_{-\vu}^{\vu}\!\!\!\!\! p_{m,\ell,z}(s)\int_{W_{A,i}}\!\!\!\!\!\ \psi \circ \Phi_{\ell \vu} \cdot
J^s_{\ell\vu}\cdot \phi_{\epsilonr,i} \cdot  \cL_s \oldh\right| \\
& \leq \Cs \epsilonr^{\frac \theta 2}|\oldh|_\infty|\psi|_\infty\, 
\int_{\ell_0 \vu}^\infty \frac{t^{m-1}}{(m-1)!} e^{-at} dt 
 \le \Cs \epsilonr^{\frac \theta 2} a^{-m} |f|_\infty |\psi|_\infty .
\end{split}
\end{equation}

We are left with the elements of $E_{\ell,i}$ for $\ell\geq \ell_0$. To study them, it is convenient to set
\begin{equation}\label{flowedleaf}
W^{0}_{A}=\cup_{t\in[-\vu,\vu]}\Phi_tW_{A} \quad
\forall  A\in A_{\ell} \, , \quad W^0=\cup_{t\in[-\vu,\vu]}\Phi_tW \, .
\end{equation}

To continue, for each $x_i$ we consider the line $x_{i}+(x^u,0,0)$ for $x^u\in[-\epsilonr^{\theta},\epsilonr^{\theta}]$, and we partition it in intervals of length $\epsilonr/3$. To each such interval $I_{i,j}$ we associate a point 
\begin{equation}\label{defxij}
x_{i,j}\in \cup_{A\in E_{\ell,i}}W_{A}^{0}\cap I_{i,j}\, , 
\end{equation}
 if the intersection is not empty.  For each such point $x_{i,j}$, we choose an index $A$ so that
$x_{i,j}\in W_{A}^{0}$, and
we associate  Reeb coordinates $\tilde\kappa_{x_{i,j}}$
to $x_{i,j}$  as follows: We
require that $x_{i,j}$ is at the origin in the $\tilde\kappa_{x_{i,j}}$ coordinates, that $\tilde\kappa_{x_{i,j}}(W_{A}^{0})\subset \{(0,x^s,x^0)\mid x^s,x^0\in\bR\}$, and that the vector 
$D_{x_{i,j}}\Phi_{-\ell \vu}(1,0,0)$ belongs to the unstable cone. 
Such changes of coordinates exist and are uniformly $\cC^2$ by \cite[Lemma A.4]{BaL}. 

Letting $\cB_{\epsilonr}=\{(x^u,x^s,x^0)\;:\; |x^u|\leq \epsilonr,\,|x^s|\leq \epsilonr^{\theta}, \, |x^0|\leq \epsilonr^\theta\}$,
we consider for each $x_{i,j}$   the box (in the coordinates $\tilde\kappa_{x_{i,j}}$)
\begin{equation}\label{defbox}
\cB_{\epsilonr,i,j}=\tilde\kappa_{x_{i,j}}^{-1}(\cB_\epsilonr) \, .
\end{equation} 
Up to taking larger $c$, we may ensure that the support of the element of the partition
of unity corresponding to $x_i$ does not intersect the flow and weak-unstable boundaries
of $\cB_{\epsilonr,i,j}$.

We next control regular pieces intersecting $\cB_{\epsilonr}$:

\begin{lemma}[Controlling regular pieces intersecting $\cB_{\epsilonr}$] \label{lem:smallbox}
For each $i$, $j$ and any $B\in E_{\ell,i}$ so that  $W_{B} \cap \cB_{\epsilonr,i,j}\ne \emptyset$, 
if $\epsilonr$ is sufficiently small so that $C\epsilonr^{\theta/5} \le 1/2$, where
$C$ is from Section~\ref{Lipschitz}, then 
the flowed leaf $W_{B}^0 \cap\partial \cB_{2\epsilonr,i,j}$ does not intersect the 
stable side of $\cB_{2\epsilonr,i,j}$.
\end{lemma}
\begin{proof}
By construction, for each $\cB_{\epsilonr,i,j}$ there is a manifold $W_{A}^0$ going through  its center and 
perpendicular to $(1,0,0)$ (in the $\tilde\kappa_{x_{i,j}}$ coordinates),
for some $A \in E_{\ell,i}$. Let $\tau_{A}\in [-\vu,\vu]$ be such that 
$\widetilde W_{A}:=\Phi_{\tau_{A}}W_{A}$ satisfies 
$\widetilde W_{A}  \cap I_{i,j}=\{(x^u_A,0,0)\}\neq \emptyset$.\footnote{ When no confusion arises, to ease notation, we will identify $\tilde\kappa_{x_{i,j}}(W_{A}^0)$ and $W_{A}^0$.}

Next,  consider  $B\in E_{\ell,i}$ with $B\ne A$
so that $W_{B}^{0}$ intersects  $\cB_{\epsilonr,i,j}$, and let $x^u_B$ be the intersection point with $I_{i,j}$. Again, let $\widetilde W_{B}=\Phi_{\tau_{B}}W_{B}$. 
By construction $|x^u_A-x^u_B|\leq \epsilonr$. Our first goal is to estimate $d(\widetilde W_{A},\widetilde W_{B})$. To this end we will use the stable version
of the fake foliations constructed in Section~\ref{Lipschitz}. Observe that each result concerning the unstable curves can be transformed into a result for the stable curves by time reversal. 
More precisely, we construct a foliation in the neighborhood of $I_{i,j}$ made of stable curves of length 
$\rho=c\epsilonr^\theta$.

To do this, consider $P^+(\Phi_{\ell\vu}(I_{i,j}))$, which is a countable union of homogeneous unstable curves,
defined analogously to the set $\tG_{n}(V)$ for a map-stable curve $V$ in the proof of 
Lemma~\ref{lem:growth}.  
 Note that by construction, both $\widetilde W_{A,i}$ and $\widetilde W_{B,i}$ can be iterated 
 forward for a 
 time $\ell\vu$, and $\Phi_{\ell\vu}(\widetilde W_{A,i}\cup \widetilde W_{B,i})\subset W^0$. 
Thus both $P^+(\Phi_{\ell\vu}(\widetilde W_{A,i}))$ and $P^+(\Phi_{\ell\vu}(\widetilde W_{B,i}))$
are subcurves of $P^+(W)$, and they
intersect $P^+(\Phi_{\ell\vu}(I_{i,j}))$.
 
Now using the (time reversed) construction detailed in  Section~\ref{Lipschitz}, we choose a seeding foliation, defined on homogeneous subcurves of $P^+(\Phi_{\ell\vu}(I_{i,j}))$, that contains 
both $P^+(\Phi_{\ell\vu}(\widetilde W_{A,i}))$ and $P^+(\Phi_{\ell\vu}(\widetilde W_{B,i}))$ as 
two of its curves.  Notice that this choice still allows the seeding foliation to be uniformly $\cC^2$
since both $P^+(\Phi_{\ell\vu}(\widetilde W_{A,i}))$ and $P^+(\Phi_{\ell\vu}(\widetilde W_{B,i}))$ 
are subcurves of $P^+(W)$ and since we assumed that
$W$ is homogeneous, by definition of $\cW^s$ and Lemma~\ref{lem:smooth},
$P^+(W)$ is a single homogeneous map-stable curve with uniformly bounded curvature.

By the definition of $\Delta_{\up}$ in Section~ \ref{sec:step1}, this means that both 
$\widetilde W_{A,i}$ and $\widetilde W_{B,i}$ belong to the fake stable foliation in a 
neighborhood of $I_{i,j}$.  Note in particular that $\widetilde W_{A,i}$ and $\widetilde W_{B,i}$
do not belong to the gaps in the foliation since by construction, they are images of stable curves
on which $\Phi_{-\ell\vu}$ is smooth and ${A}, {B} \in E_{\ell,i}$.
Also, they are
guaranteed to have length at least $\rho = c\epsilonr^\theta$ because, having their
indices in $E_{\ell,i}$, they
completely cross $B_{c\epsilonr^\theta}(x_i)$.
In the coordinates used in Section~\ref{Lipschitz} (with the stable and unstable interchanged), 
let $(G(x_A^u,x^s), x^s, H(x_A^u,x^s))$  and $(G(x_B^u,x^s), x^s, H(x_B^u,x^s))$ be the 
curves $\widetilde W_{A,i}, \widetilde W_{B,i}$, respectively. By properties (ii) and (vi) of the 
fake foliation, it follows that
\begin{equation}\label{eq:shortcutG}
\begin{split}
|G(x_B^u,x^s)-G(x_A^u,x^s)|&=\left|\int_{x_A^u}^{x_B^u}\partial_{x^u}G(u,x^s) du\right|
=\left|\int_{x_A^u}^{x_B^u}\left[ 1+\int_{0}^{x^s}\partial_{x^s}\partial_{x^u}G(u,s) ds\right]du\right|\\
&\leq \int_{x_A^u}^{x_B^u}\left[ 1+C\epsilonr^{\theta/5}\right]du\leq \epsilonr(1 + C \epsilonr^{\theta/5}) \le (3/2) \epsilonr 
< 2\epsilonr \, ,
\end{split}
\end{equation}
where we have chosen $\epsilonr$ sufficiently small that $C\epsilonr^{\theta/5} \le 1/2$.
Analogously, by properties (ii) and (v) of the foliation,
\begin{equation}\label{eq:shortcutH}
\begin{split}
|H(x_B^u,x^s)-H(x_A^u,x^s)|&=\left|\int_{x_A^u}^{x_B^u}\int_{0}^{x^s}\partial_{x^s}\partial_{x^u}H(u,s) ds du\right|\leq \Cs \epsilonr^{1+\theta}\, .
\end{split}
\end{equation}
Hence $W^0_{B}$ cannot intersect the stable boundary of $\cB_{2\epsilonr,i,j}$, again assuming 
$\Cs \epsilonr^\theta \le 3/2$.
 \end{proof}

\begin{remark} The  above estimates on the distances of nearby manifolds can be applied to the ``central'' manifolds of the boxes as well. This readily implies that the covering $\{\cB_{\epsilonr,i,j}\}$ has a uniformly bounded number of overlaps.
\end{remark}
After  this preparation, we can summarise where we stand: From \eqref{eq:step02},  \eqref{Stirling}, \eqref{eq:step03} and assumption \eqref{eq:h1} we have
\begin{equation}\label{eq:step04}
\begin{split}
&\int_W \psi \cdot (\cR(z)^m (\oldh))\, dm_W =\sum_{\ell, i\in\bN}\, \sum_{A \in  A_{\ell,i}}\int_{-\vu}^{\vu}\!\!\!
 p_{m,\ell,z}(s)\int_{W_{A,i}}\!\!\! J^s_{\ell\vu}\cdot  \psi\circ \Phi_{\ell \vu} \cdot 
 \phi_{\epsilonr,i} \cdot  \cL_s \oldh  \\
&= \sum_{\ell\geq \ell_0}\,\sum_{i=1}^\infty\sum_j\, \sum_{A \in  E_{\ell,i,j}}\int_{-\vu}^{\vu}\!\!\!
 p_{m,\ell,z}(s)\!\!\int_{W_{A,i}}\!\! J^s_{\ell\vu}\cdot  \psi\circ \Phi_{\ell \vu} \cdot  
 \phi_{\epsilonr,i} \cdot \cL_s \oldh\\
 &\qquad\quad+ \cO(\epsilonr^{\frac \theta 2}a^{-m}|\oldh|_\infty|\psi|_\infty)\, ,
\end{split}
\end{equation}
where\footnote{ Some manifolds can be attributed either to a box or to
an adjacent box and a choice can be made to resolve the ambiguity.} 
$E_{\ell,i,j}\subset \{A\in E_{\ell,i}\;:\; W_{A,i}\cap \cB_{\epsilonr,i,j}\neq \emptyset\}$, $E_{\ell,i,j}\cap E_{\ell, i,k}\neq \emptyset$  implies $j=k$, and $\cup_jE_{\ell,i,j}=E_{\ell,i}$.
To continue, we need some notation that, for each $\ell, i,j$, is better stated in the above mentioned Reeb charts $\tilde \kappa_{x_{i,j}}$. 
In fact, from now on, we identify the coordinate charts and the manifold notation. For each $A\in E_{\ell,i,j}$, we now view $W_{A,i}^0\cap \cB_{c\epsilonr}$ (directly) as the graph of 
$\bW_A^0 (x^s, x^0) := \bW_A(x^s)+(0,0,x^0)$ where
\begin{equation}
\label{eq:board W}
\bW_A(x^s):=(M_A(x^s),x^s,N_A(x^s))\, ,\quad  |x^s|\leq c \epsilonr^\theta\, ,
\end{equation}
and $M_A$ and $N_A$ are uniformly $\cC^{2}$ functions. By the same exact arguments used in the proof of Lemma \ref{lem:smallbox}, 
we have\footnote{ The manifold $W_B$ in the proof of Lemma \ref{lem:smallbox} corresponds here to the central manifold of the box, which in the current coordinates reads $\{(0,x^s,0)\}$, while $M_A(x^s)$ corresponds to $G(x^u_A,x^s)$.}
\begin{equation}\label{condF_A}
|M_{A}'|\leq \Cs \epsilonr^{1-\frac{4\theta} 5} \, .
\end{equation} 
Since both $\bal$ and $\extd \bal$ are invariant under the flow, and the manifolds are the images of manifolds with tangent space in the kernel of both forms, we have that 
\[
N_A'(x^s)=x^s M_A'(x^s)\, .
\]
With the above notation, the integrals in equation \eqref{eq:step04} can be rewritten as 
\begin{equation}\label{eq:step1}
\int_{-\vu}^{\vu} p_{m,\ell,z}(s)\int_{W_{A,i}}\psi \circ \Phi_{\ell \vu} \cdot 
J^s_{\ell\vu}\cdot \phi_{\epsilonr,i}\cdot  \cL_s \oldh \, ds =\int_{\cQ_{c\epsilonr^\theta}} p_A
\cdot \psi_A  \cdot \bar J^s_{A, \ell \vu} \cdot \phi_{\epsilonr, A} \cdot \oldh_A \, dx^s dx^0 ,
\end{equation}
where  $\cQ_\delta=\{(x^s,x^0)\;:\;|x^s|\leq \delta, \,|x^0|\leq \delta\}$ and
\[
\begin{split}
p_{A}(x^s,x^0)& = p_{m,\ell,z}(-x^0) \, ,
\qquad
\phi_{\epsilonr,A}(x^s, x^0) = \phi_{\epsilonr,i} \circ \bW_A^0(x^s, x^0) 
\cdot \| \bW'_A(x^s) \|  \, , \\
\psi_{A}(x^s,x^0) & =
\psi \circ \Phi_{\ell \vu}\circ\bW_{A}(x^s) \, , \quad \bar J^s_{A, \ell \vu}(x^s,x^0) = J^s_{\ell \vu} \circ \bW_A(x^s) \, , \\
f_A(x^s, x^0) &= f \circ \bW_A^0(x^s, x^0) \, .
\end{split}
\]
Our strategy will be based on the fact that an oscillatory integral of a Lipschitz function is small. Unfortunately, the integrands above are not Lipschitz. To deal with this, let 
\begin{equation}\label{eq:xizeta}
\Xi_{\ell,A,i,j}=\psi_{A}(0,0) \cdot Z_{\ell,A,i} \, ,
\end{equation}
recalling $Z_{\ell,A,i}$ defined by \eqref{eq:ro-la1}.
By Lemma \ref{lem:distortion} and the fact that $|\psi|_{\cC^\oldp(W)}\leq 1$, we have
\begin{equation}\label{eq:toconstant}
\left|\psi_{A}\cdot \bar J^{s}_{A,\ell\vu}-\Xi_{\ell,A,i,j}|W_{A,i}|^{-1}\right|_\infty\leq \Cs \epsilonr^{\oldp\theta}\frac{Z_{\ell,A,i}}{|W_{A,i}|} \, .
\end{equation}
Therefore, using  \eqref{eq:toconstant} in \eqref{eq:step1} and substituting it in \eqref{eq:step04}, yields
\begin{equation}\label{eq:step3}
\begin{split}
\int_W\psi \cR(z)^m (\oldh)
\, dm_W=&\sum_{\ell\geq \ell_0}\sum_{i,j}\sum_{A\in E_{\ell,i,j}}\frac{\Xi_{\ell,A,i,j}}{|W_{A,i}|}
\int_{\cQ_{c\epsilonr^\theta}} p_{A} \cdot \phi_{\epsilonr,A}\cdot   \oldh_A\\
&\;+\cO(|\oldh|_{\infty}|\psi|_\infty \epsilonr^{\oldp\theta }a^{-m})\, ,
\end{split}
\end{equation}
since $\alpha \le 1/3$.
Recalling $p(s)=p(-s)$, we
introduce  (to ease notation we suppress some indices):
\begin{equation}\label{eq:defF}
\begin{split}
&{\bf G}_{\ell,m,i,A}(x^s,x^0)= p(x^0)\frac{(\ell\vu-x^0)^{m-1}}{|W_{A,i}|(m-1)! }e^{-z\ell\vu+ax^0}\\
&\qquad\qquad\qquad\qquad\cdot
\phi_{\epsilonr,i} (M_A(x^s),x^s,N_A(x^s)-x^0) \cdot\Theta_A(x^s)\, ,
\end{split}
\end{equation}
where $\Theta_A ds\wedge dx^s$ is the volume form on $W_A^{0}$ in the coordinates   
$\bW^0_A$. Note that
\begin{equation}\label{eq:Xi}
\begin{split}
&|{\bf G}_{\ell,m,i,A}|_{\infty}\leq \Cs \epsilonr^{-\theta}\frac{(\ell\vu)^{m-1}}{(m-1)!}e^{-a\ell\vu}\\
&|{\bf G}_{\ell,m,i,A}|_{\text{Lip}}\leq \Cs \epsilonr^{-2\theta}\frac{(\ell\vu)^{m-1}}{(m-1)!}e^{-a\ell\vu}\, .
\end{split}
\end{equation}
Indeed, by construction $|W_{A,i}|\geq \epsilonr^\theta$, and aside from $\ell$, 
the only large contribution to the Lipschitz  norm comes from $\phi_{\epsilonr,A}$
(which  only differs from $\phi_{\epsilonr,i}$ by the uniformly  smooth change of
coordinates $\bW^0_A$), and $p$ and is of order $\epsilonr^{-\theta}$.

With the above notation we can write,
\begin{equation}\label{eq:step4}
\begin{split}
&\int_W\psi \cR(z)^m (\oldh) \, dm_W=\cO(|\oldh|_{\infty}|\psi|_\infty \epsilonr^{\oldp\theta }a^{-m})\\
&+\sum_{\ell\geq \ell_0}\sum_{i,j}\sum_{A\in E_{\ell,i,j}}\!\!\Xi_{\ell,A,i,j}
\int_{\cQ_{c\epsilonr^\theta}} e^{ibx^0}{\bf G}_{\ell,m,i,A}(x^s,x^0) \oldh(\bW_A(x^s)+(0,0,x^0)) \, .
\end{split}
\end{equation}
At this point, we would like to compare different manifolds in the same box by sliding them along an approximate unstable direction.
For this we fix $\ell, i, j$ and use the approximate unstable fibres $\gamma_{i,j, \epsilonr}^\up$ constructed in Section \ref{Lipschitz}
(the parameter $\up$ represents the maximum time  in $\cF_{\up,x_0}$ there). In short, for each coordinate $\tilde\kappa_{x_{i,j}}$, we can construct a Lipschitz foliation in $\cB_{c\epsilonr}$ in a 
\[
\rho=\epsilonr^\varsigma
\]
 neighborhood  of $W_*=\{(0,x^s,x^0)\:\;|x^s|\leq c \epsilonr^\theta,|x^0|\leq c \epsilonr^\theta \}$. 
In order to have the foliation defined in all $\cB_{c\epsilonr}$ we need $\varsigma<1$, while for the foliation to have large part where it can be smoothly iterated backward as needed it is necessary that $\varsigma>\theta$. We thus impose
\begin{equation}\label{eq:varsigma}
\theta<\varsigma<1 \, .
\end{equation}
The foliation constructed in Section ~\ref{Lipschitz} can be described by the coordinate change
\[
\bF_{i,j,\up}(x^u,x^s,x^0)=(x^u,G_{i,j,\up}(x^u,x^s), H_{i,j,\up}(x^u,x^s)+x^0) \, .
\]
Also, in equation \eqref{eq:old8.48}
 it will be essential that $\up$ be large enough. It turns out that
\begin{equation}\label{eq:up-choice}
\up=10 \,ma^{-1}
\end{equation}
suffices.
The leaf $\gamma_{i,j,\epsilonr}^{\up}(x^s,x^0)$ is thus the graph of $\bF_{i,j,\up}(\cdot, x^s,x^0)$. By construction, we require that the vector $(1,0,0)$ is tangent to the 
curve in the foliation passing through $(0,0,0)$. 
Thus,\footnote{ We cannot require that the leaf $(x^u,0,0)$ belongs to the foliation, 
but  the proof of \cite[Lemma ~A.4]{BaL} ensures that we can obtain  the desired tangency.}
\begin{equation}\label{eq:straight-unst}
\begin{split}
G_{i,j,\up}(0,0) & = H_{i,j,\up}(0,0) = \partial_{x^u}G_{i,j,\up}(0,0) = \partial_{x^u}H_{i,j,\up}(0,0)=0 \\
|G_{i,j,\up}(x ^u,0)| & \le C \rho^2,\quad |H_{i,j,\up}(x^u,0)| \le C \rho^2.
\end{split}
\end{equation}
For $A \in E_{\ell, i, j}$, we consider the holonomy 
\[
\bh_{i,j,A, \up}: W_A\cap \cB_{\epsilonr}\to W_*
\]
defined by $\{z\}=\gamma_{i,j,\epsilonr}^\up(\bh_{i,j,A,\up}(\{z\}))\cap W_A$. Recalling \eqref{eq:board W},
we shall use the notation
\begin{equation}\label{defcoord}
\bh_{i,j,A,\up}\circ \bW_A(x^s)=(0,\bh^s_A(x^s), \bh^0_A(x^s)) \, ,
\end{equation}
where, by construction and provided $c$ has been chosen large enough, for all $|x^s|\leq \epsilonr^\theta$,
\begin{equation}\label{trivialbound}
|\bh^s_A(0)| \le C\epsilonr^2\, , \; |\bh^0_A(0) - N_A(0)| \le C\epsilonr^2\, , \;|\bh^s_A(x^s)|\le  c \epsilonr^\theta \, , 
\, |\bh^0_A(x^s)|\le c \epsilonr^\theta \, .
\end{equation}
In other words, $\bF_{i,j,\up}(M_A(x^s),\bh^s_A(x^s),\bh^0_A(x^s))=\bW_A(x^s)$, that is
\begin{equation}\label{eq:holo}
\begin{split}
&G_{i,j,\up}(M_A(x^s),\bh^s_A(x^s))=x^s\, ,
\quad H_{i,j,\up}(M_A(x^s),\bh^s_A(x^s))+\bh^0_A(x^s)=N_A(x^s)\, .
\end{split}
\end{equation}

We shall need the following bounds on the regularity of $\bh^s_A$.
The proof of Lemma ~\ref{lem:hpr}  is provided in Appendix \ref{sec:hoihoi}.

\begin{lemma}\label{lem:hpr}
There exists $\Cs>0$  such that for each $i$, $j$, $\ell$, 
and every  $A\in E_{\ell,i,j}$,  we have
\[
|\bh^s_A(x^s)-x^s|+|(\bh^s_A)^{-1}(x^s)-x^s|\leq \Cs \epsilonr^{1-\frac 45\varsigma+\theta}\; ,
\qquad\left|1- (\bh^s_A)'\right|\leq \Cs \epsilonr^{1-\frac 45\varsigma} \, .
\]
\end{lemma}

Next, remember that the fibers $\gamma_{i,j,\epsilonr}^\up$ in the domain $\Delta_\up$ can be iterated backward a time $\up$ and still remain in the unstable cone. In the following we will use the notation
\begin{equation}\label{defnot}
\partial_{\up,i}\oldh=\sup_{(x^s,x^0)\in \Delta_\up}\esssup_{|x^u|\leq \epsilonr}|\langle\partial_{x^u}\bF_{i,j,\up}(x^u,x^s,x^0), (\nabla \oldh)\circ \bF_{i,j,\up}(x^u,x^s,x^0)\rangle|\, .
\end{equation}
With the above construction and notations, using \eqref{eq:Xi}, \eqref{trivialbound}, \eqref{eq:holo} and Lemma \ref{lem:hpr}, we can continue our estimate left at \eqref{eq:step4}
(recalling \eqref{defcoord})
\begin{equation}\label{eq:step5}
\begin{split}
&\int_{\cQ_{c\epsilonr^\theta}}\!\!\!\!dx^0\,dx^s\,e^{ibx^0}{\bf G}_{\ell,m,i,A}(x^s,x^0) \oldh(\bW_A(x^s)+(0,0,x^0))
\\&\quad
=\int_{\cQ_{c\epsilonr^\theta}}\!\!\!\!dx^0\,dx^s\; e^{ib x^0}{\bf G}_{\ell,m,i,A}(x^s,x^0) 
\oldh(0, \bh^s_A(x^s),\bh_A^0(x^s)+x^0)\\
&\qquad\quad+ \cO(\epsilonr^{1+\theta}\partial_{\up,i}\oldh +
\epsilonr^{\varsigma}|\oldh|_{\infty})\cdot\frac{(\ell\vu)^{m-1}}{(m-1)!}e^{-a\ell\vu}\\
&=\int_{\cQ_{c\epsilonr^\theta}}\!\!\!\!dx^0\,dx^s\; e^{ib(x^0-\bomega_A(x^s))}
\frac{{\bf G}^*_{\ell,m,i,A}(x^s,x^0)}
{ |(\bh^s_A)'\circ (\bh^s_A)^{-1}(x^s)|}
 \oldh(0, x^s,x^0)\\
&\qquad\quad+ \cO(\epsilonr^{1+\theta}\partial_{\up,i}\oldh +\epsilonr^{\varsigma}|\oldh|_{\infty})\cdot\frac{(\ell\vu)^{m-1}}{(m-1)!}e^{-a\ell\vu}\\
&=\int_{\cQ_{c\epsilonr^\theta}}\!\!\!\!dx^0\,dx^s\; e^{ib(x^0-\bomega_A(x^s))}{\bf G}^*_{\ell,m,i,A}(x^s,x^0) \oldh(0, x^s,x^0)\\
&\qquad\quad+ \cO(\epsilonr^{1+\theta}\partial_{\up,i}\oldh+(\epsilonr^{\varsigma}+\epsilonr^{1-4\varsigma/5+\theta})|\oldh|_{\infty})\cdot\frac{(\ell\vu)^{m-1}}{(m-1)!}e^{-a\ell\vu} \, ,
\end{split}
\end{equation}
where in the first equality, we have expanded $f$ at the reference manifold $W_*$,
using $\partial_{\up,i}\oldh$ on $\Delta_{\up}$ and property (iv) of the foliation on
$W_* \setminus \Delta_{\up}$; in the second equality, we have changed variables,
$x^0 \mapsto x^0 - h^0_A(x^s)$ and then $x^s \mapsto (h^s_A)^{-1}(x^s)$; 
and in the third equality we have used 
Lemma \ref{lem:hpr}, setting 
$\bomega_A(x^s)=\bh^0_A\circ (\bh^s_A)^{-1}(x^s)$,
and  
\begin{equation}\label{eq:defFs}
\begin{split}
&{\bf G}^*_{\ell,m,i,A}(x^s,x^0)={\bf G}_{\ell,m,i,A}((\bh_A^s)^{-1}(x^s),x^0 - \bomega_A(x^s))\, .
\end{split}
\end{equation}

Note that Lemma \ref{lem:hpr} and \eqref{eq:Xi} imply that there exists $\Cs>0$ such that, 
for all $A$,
\begin{equation}\label{eq:hoG}
\begin{split}
&|{\bf G}^*_{\ell,m,i,A}|_{\infty}\leq \Cs \epsilonr^{-\theta}\frac{(\ell\vu)^{m-1}}{(m-1)!}e^{-a\ell\vu}\, ,
\quad |{\bf G}^*_{\ell,m,i,A}|_{\text{Lip}}\leq \Cs \epsilonr^{-2\theta}\frac{(\ell\vu)^{m-1}}{(m-1)!}e^{-a\ell\vu}\, .
\end{split}
\end{equation}
At last, we can substitute \eqref{eq:step5} into \eqref{eq:step4}:  We note that \eqref{eq:ro-la1}, \eqref{eq:xizeta} imply 
\[
\sum_{i, j}\sum_{A\in E_{\ell,A,i,j}}\Xi_{\ell,A,i,j}\leq \Cs \, ,
\]
and use the Schwarz inequality (first with respect to the integrals
and then with respect to the sum on $i,j$)  together with
the trivial identity $|\sum_A z_A|^2=(\sum_A z_A )(\sum_B \bar z_B)$
on sums
of complex numbers, to obtain\footnote{ Remember that $i$ runs from $1$ to $\Cs \epsilonr^{-3\theta}$, while $j$ from $1$ to $\Cs \epsilonr^{\theta-1}$, so, all in all, the sum over $i,j$ consists of $\Cs \epsilonr^{-1-2\theta}$ terms.}
\begin{equation}\label{eq:almost}
\begin{split}
&\left|\int_W\psi \cR(z)^m (\oldh) \, dm_W\right|\leq \Cs\left[(\epsilonr^{\oldp\theta}+\epsilonr^{1-4\varsigma/5}+\epsilonr^{\varsigma-\theta})|\oldh|_\infty+\epsilonr\sup_i \partial_{\up,i}\oldh\right] a^{-m}\\
&+\Cs\sum_{\ell\geq \ell_0}\sum_{i,j} |\oldh|_\infty \epsilonr^{\theta}\left\{\sum_{A,B\in E_{\ell,i,j}}\!\!\!\!\Xi_{\ell,A,i,j} \Xi_{\ell,B,i,j}\!\!\int_{\cQ_{c\epsilonr^\theta}} \!\! e^{ib[\bomega_A-\bomega_B]}\left[{\bf G}^*_{\ell,m,i,A}\cdot\overline{{\bf G}^*_{\ell,m,i,B}}\right]\right\}^{\frac 12}\\
&\leq \Cs\left[(\epsilonr^{\oldp\theta}+\epsilonr^{1-4\varsigma/5}+\epsilonr^{\varsigma-\theta})|\oldh|_\infty+\epsilonr\sup_i \partial_{\up,i}\oldh\right] a^{-m}\\
&+\Cs\sum_{\ell\geq \ell_0} |\oldh|_\infty \epsilonr^{-\frac 1 2}\left\{\sum_{i,j}\sum_{A,B\in E_{\ell,i,j}}\!\!\!\!\!\!Z_{\ell,A,i} Z_{\ell,B,i}\!\left|\int_{\cQ_{c\epsilonr^\theta}} \!\!  e^{ib[\bomega_A-\bomega_B]}\left[{\bf G}^*_{\ell,m,i,A}\overline{{\bf G}^*_{\ell,m,i,B}}\right]\right|\right\}^{\frac 12} \, .
\end{split}
\end{equation}

In view of the cancellation argument needed to conclude the
present proof of Lemma~\ref{dolgolemma}, we need a fundamental, but technical, result
whose proof can be found in Appendix~\ref{sec:hoihoi}.  

\begin{lemma}[Oscillatory integral]\label{lem:cancel} We have
\begin{equation}\label{eq:below-o}
\inf_{x^s}|\partial_{x^s}\left[\bomega_A-\bomega_B\right](x^s)|\geq \Cs d(W_{A},W_{B}) \, .
\end{equation}
In addition, provided that 
\begin{equation}\label{eq:varvar}
0<\ho \le \frac\theta 5 \mbox{ and } 4 \varsigma /5 +\ho (7+11 \varsigma/15)\leq 1\, ,
\end{equation}
we have
\begin{equation}\label{eq:amazing}
|\bomega_A-\bomega_B|_{\cC^{1+\ho}(\cQ_{c\epsilonr^\theta})}\leq \Cs \epsilonr \, ,
\end{equation}
which implies
\begin{equation}\label{eq:up-o}
\begin{split}
\bigg|\int_{|x^s|\leq c\epsilonr^\theta}\!\!\!\!\!\!\!\!\!\!\!\!dx^s\,&e^{ib(\bomega_A(x^s)-\bomega_B(x^s))}{\bf G}^*_{\ell,m,i,A}\overline{{\bf G}^*_{\ell,m,i,B}}\,\bigg|_{\infty}
\\ &\qquad  \leq \Cs\frac{(\ell\vu)^{2m-2}}{[(m-1)!]^2} e^{-2a\ell\vu}
 \epsilonr^{-\theta}\left[\frac \epsilonr{d(W_A,W_B)^{1+\ho} b^\ho}+\frac1{\epsilonr^{\theta}d(W_A,W_B) b}\right]\, .
\end{split}
\end{equation}
\end{lemma}

Note that the proof 
of the $\cC^{1+\ho}$ estimates on the holonomy in
\cite[Lemma 6.6]{BaL} used 
a H\"older bound on $\partial_{x^u} \partial_{x^s}G$ (called (6) in \cite[App. ~D]{BaL}) 
that is not available in the present context. Yet, it turns out that the four-point estimate (vii) 
from  Section ~ \ref{Lipschitz} suffices to prove  the bound
\eqref{eq:amazing} on $\bomega_A-\bomega_B$.

We have now all the ingredients to conclude
the proof of Lemma~\ref{dolgolemma}. It is  convenient to introduce a parameter $\vartheta>1$ and to define 
\begin{align*}
&\bE^{\text{close}}_{\ell,i,j}=\{(A,B)\in E_{\ell,i,j}\times E_{\ell,i,j} \;:\;   d(W_A,W_B)\leq \epsilonr^\vartheta\}\\
&\bE^{\text{far}}_{\ell,i,j}=\{(A,B)\in  E_{\ell,i,j}\times E_{\ell,i,j}\;:\;   d(W_A,W_B)> \epsilonr^\vartheta\}
\, .
\end{align*}

To estimate the sum on $\bE^{\text{close}}_{\ell,i,j}$, we will need the following lemma
(proved in Appendix~\ref{sec:hoihoi}).

\begin{lemma}
\label{lem:disco} 
For each $\ell\geq \ell_0$, $i\in\bN$ and $A\in E_{\ell,i}$, the following estimate holds,
\[
 \sum_{\{(B,i)\;:\; i\in\bN, B\in E_{\ell,i}, d(W_{B,i},W^0_{A,i})\leq \rho_{*}\}} Z_{\ell,B,i}\leq  
\Cs\left[\epsilonr^\theta\sqrt \rho_*+\lambda^{\ell\vu/2}\right] \, ,
\]
where $\lambda<1$ is from Lemma~\ref{lem:growth}.
\end{lemma}

Thus, by the above Lemma with $\rho_*=\epsilonr^\vartheta$, we have
\begin{equation}\label{eq:nearby0}
\begin{split}
\sum_{(A,B)\in \bE^{\text{close}}_{\ell,i,j}}Z_{\ell, A, i}Z_{\ell, B, i}&=\sum_{A\in E_{\ell, i,j}}Z_{\ell, A, i} \sum_{(A,B)\in \bE^{\text{close}}_{\ell,i,j}}Z_{\ell, B, i}\\
&\leq \Cs\sum_{A\in E_{\ell, i,j}}Z_{\ell, A, i}\left[\epsilonr^{\theta+\vartheta/2} +\lambda^{\ell\vu/2}\right] \, .
\end{split}
\end{equation}
Next, we assume (possibly strengthening assumption \eqref{eq:h1}),
\begin{equation}
\label{eq:msigma}
\lambda^{\ell_0\vu/2}=(e^{-m})^{\frac{\ln\lambda^{-1}}{2ae^2}}\leq \epsilonr^{\theta+\vartheta/2} \, .
\end{equation}
Hence, using estimate \eqref{eq:nearby0}, for $\ell\geq \ell_0$,
\begin{equation}\label{eq:nearby}
\begin{split}
\sum_{i,j}\sum_{(A,B)\in \bE^{\text{close}}_{\ell,i,j}}Z_{\ell, A, i}Z_{\ell, B, i}&\leq \Cs|W| \epsilonr^{\theta+\vartheta/2} \, .
\end{split}
\end{equation}
To conclude, we want to use estimate \eqref{eq:almost}, applied to $\cR(z)^m \oldh$ rather than to $\oldh$. The reason for this is to obtain an estimate in terms of the $H^1_\infty$ norm, but with a very small factor in front.  
More precisely,  we will in fact estimate,
\begin{equation}
\label{eq:est}
\int_W \cR(z)^{2m}f \, \psi \, dm_W = \int_W \cR(z)^m g_{\up, 1} \, \psi \, dm_W + \int_W \cR(z)^m g_{\up, 2} \, \psi \, dm_W  \, , 
\end{equation}
where $g_{\up,1}$ and  $g_{\up,2}$ are defined by
\[
\cR(z)^m f = \int_0^\up \frac{s^{m-1}}{(m-1)!} e^{-zs} \cL_s f\, ds 
+ \int_\up^\infty \frac{s^{m-1}}{(m-1)!} e^{-zs} \cL_s f\, ds =: g_{\up,1} + g_{\up,2} \, .
\]
To estimate the second term in \eqref{eq:est},  recalling \eqref{eq:int} and
\eqref{eq:up-choice}, notice that
\begin{align}\label
{eq:old8.48}
|g_{\up,2}|_{\infty} & \le |f|_{\infty} \int_\up^\infty \frac{s^{m-1}}{(m-1)!} e^{-as} ds
= |f|_{\infty} e^{-10m} a^{-m} \sum_{k=0}^{m-1} \frac{(10m)^k}{k!} \\
\nonumber
& \le  |f|_{\infty} e^{-10m} a^{-m} \frac{(10m)^{m-1}}{(m-1)!} \sum_{k=0}^{m-1} \frac{(m-1)\cdots k}{(10m)^{m-1-k} } \\
\nonumber
& \le |f|_{\infty} e^{-10m} a^{-m} \frac{(10m)^{m}}{m!} \sum_{k=0}^{m-1} 10^{-m+k} 
\le \Cs |f|_{\infty} a^{-m} e^{-6m} \, .
\end{align} 
Thus, 
\[
| \cR(z)^m g_{\up,2} |_{\infty} \le a^{-m} |g_{\up,2}|_{\infty}
 \le \Cs a^{-2m} e^{-6m} |f|_{\infty}
\le \Cs a^{-2m} \epsilonr^{3\theta} |f|_{\infty}\, ,
\]
where in the last inequality we have used condition \eqref{eq:h1}, $e^{-m} \le \epsilonr^{\theta/2}$.

To estimate the first term in \eqref{eq:est}, we apply the bounds derived throughout this
section, in particular \eqref{eq:up-choice} and
\eqref{defnot}, with $g_{\up,1}$
in place of $f$, noting that $| g_{\up,1}|_{\infty} \le a^{-m} |f|_{\infty}$ and
\[
\begin{split}
|\partial_{\up, i} (g_{\up, 1}) | & \le \int_0^\up \frac{s^{m-1}}{(m-1)!} e^{-as} \partial_{\up,i} (f \circ \Phi_{-s}) \, ds \\
& \le \Cs | \nabla f |_{\infty} \int_0^\infty \frac{s^{m-1}}{(m-1)!} e^{-(a+\ln \Lambda)s} \, ds 
\le \Cs (a+\ln \Lambda)^{-m} |f|_{H^1_\infty} .
\end{split}
\]

Thus, by using \eqref{eq:hoG} and \eqref{eq:nearby} to estimate the sum over $\bE^{\text{close}}_{\ell,i,j}$ and Lemma \ref{lem:cancel} to estimate the sum over $\bE^{\text{far}}_{\ell,i,j}$
in \eqref{eq:almost}, we obtain
\begin{equation}
\label{eq:done}
\begin{split}
&\left|\int_W\psi \cR(z)^{2m} (\oldh) \, dm_W\right|
\le \left| \int_W \psi \cR(z)^m (g_{\up, 1}) \, dm_W \right| + \left| \int_W \psi \cR(z)^m (g_{\up,2}) \, dm_W \right| \\
& \leq \Cs\sum_{\ell\geq \ell_0}| g_{\up, 1} |_{\infty}\frac{e^{-a\ell\vu}(\ell\vu)^{m-1}}{(m-1)!\;\epsilonr^{\frac 1 2}} \bigg[\frac{\epsilonr^{-\vartheta(1+\ho)+1}}{b^{\ho}}+\frac{\epsilonr^{-\theta-\vartheta}}{b}+\epsilonr^{\frac{\vartheta}2+\theta}\bigg]^{\frac 12}\\
&\qquad+\Cs\left[(\epsilonr^{\oldp\theta}+\epsilonr^{1-4\varsigma/5}+\epsilonr^{\varsigma-\theta})| g_{\up, 1} |_\infty+\epsilonr\sup_i \partial_{\up,i}(g_{\up, 1})\right] a^{-m} + | \cR(z)^m g_{\up, 2} |_\infty \\
&\leq  \Cs a^{-2m}\left(\frac{\epsilonr^{-\frac {(1+\ho)\vartheta+2\theta} 2}}{b^{\frac \ho 2}}+\frac{\epsilonr^{-\frac{1+3\theta+\vartheta}2}}{b^{\frac 12}} +\epsilonr^{\frac{\vartheta-2\theta-2}4}+\epsilonr^{\oldp\theta }+
\epsilonr^{1-4\varsigma/5}+\epsilonr^{\varsigma-\theta} + \epsilonr^{3\theta} \right)|\oldh|_{\infty}\\
&\qquad+\Cs  a^{-2m}\epsilonr (1+a^{-1}\ln\Lambda)^{-m}|\oldh|_{H^1_\infty}\, ,
\end{split}
\end{equation}
where we have used the fact that
\[
\sum_{\ell \ge \ell_0} e^{-a\ell \vu} \frac{(\ell \vu)^{m-1}}{(m-1)!} \le C \frac{a^{-m}}{\vu} \le C' a^{-m} \epsilonr^{-\theta} \, .
\]
Note that if we choose $\ho=\frac 1{40}$, then condition \eqref{eq:varvar} is satisfied for each $\frac 13\leq \theta<\varsigma\leq 1$. Thus all the conditions imposed along the computation are
\[
0<\theta<\varsigma<1 \;;\;\; \vartheta> 2 + 2\theta \;;\;\;\lambda^{\ell_0\vu/2}<\epsilonr^{\theta + \vartheta/2} \, .
\]
If we choose  $\vartheta=2+4\theta$, $b=\epsilonr^{-\frac{(1+\ho)\vartheta+2(1+\oldp)\theta}{\ho}}$, $\theta=\frac 1{1+2\oldp}$, $\varsigma=\frac{1+\oldp}{1+2\oldp}$, and $m\geq C a \ln b$, for some appropriate fixed constant $C$, then we satisfy all the conditions, and  Lemma~\ref{dolgolemma} follows
with $\gamma_{Do}=\frac{\oldp\ho}{8+6\ho+(6+4\ho)\oldp}$
since the largest term appearing in \eqref{eq:done} is of 
order\footnote{ Note that with the above choices if $\alpha=\frac 13$, then we have the rather small value $\gamma_{Do}=\frac 1{1222}$. Working more it is certainly possible to obtain a better estimate but it remains unclear what the optimal value is.}
$\epsilonr^{\alpha \theta} = b^{-\gamma_{Do}}$.
\qed

\section{Completing the proofs (exponential mixing of $\Phi_t$ and resonances of $X$)}
\label{theend}

\subsection{Proofs of Theorem~ \ref{main}  and \ref{resonances}}\label{subsec:theend}

The key step for exponential mixing consists in showing the following consequence of the Lasota--Yorke
estimates and the Dolgopyat bound on $\cR(z)$:

\begin{proposition}\label{dolgo}
For any $0<\gamma< \oldp\le 1/3$, if    $\oldq>0$ and $\oldbeta <1$  are such that
$\max\{\oldq, 1-\oldbeta\}$ is small enough, then there exist $1\le a_0 \le a_1 \le b_0$,
$0 <  c_1< c_2$, and $\upsilon_{Do} >0$, so that
$$
\| \cR (a+ib)^n\|_\cB \le \left ( \frac {1}{a+\upsilon_{Do}} \right )^n
$$
$c_2\le  (\ln(1+a^{-1} \upsilon_{Do}))^{-1}$ and for all $|b|\ge b_0$, $a \in [a_0 , a_1]$, and 
$n \in [c_1   \ln |b|, c_2  \ln |b|]$.
\end{proposition}

Proposition~ \ref{dolgo} will be shown in the next subsection. Together with
Corollary ~ \ref{spX}   it gives 
the following corollary:

\begin{cor}[Spectral gap for the generator $X$]\label{specX'}
Under the assumptions of Corollary~\ref{CorLY}
and Proposition~\ref{dolgo}, 
there exists  $\upsilon_0 >0$ so that
the spectrum of $X$ on $\cB$
in the half-plane
$$
\{ z \in \bC \mid \Re (z) > -\upsilon_0 \}
$$
consists only of the eigenvalue $z=0$, which has algebraic multiplicity equal to one.
\end{cor}

Corollaries ~ \ref{spX} and \ref{specX'} with Lemma~\ref{lem:C1}
Remark ~\ref{opt} give claims (a) and (b) of Theorem~ \ref{resonances}.

\begin{proof}
By Proposition \ref{dolgo} and Corollary \ref{spX},  the sets
$\{z\in \bC\;:\; \Re(z)>-\upsilon_{Do}, \, |\Im(z)|> b_0\}$ and $\{z\in\bC\;:\; \Re(z)>0\}$ are included in
 the resolvent set of $X$, in fact, $z\mapsto \cR(z)$
 is holomorphic, as a bounded operator on $\cB$,  on the union of these two sets.  (Just like for \cite[Cor. 3.10]{BaL}
or \cite[Cor. 2.13]{Li04},  we apply 
$\cR(z+\eta^{-1})=\cR(z)  (1+\eta^{-1} \cR(z))^{-1}$. See \cite[Lemma 4.2]{Butterley} for details.)

On the other hand, Corollary~  \ref{spX} implies that in the set $\{z\in \bC\;:\; \Re(z)>-\upsilon_{Do} , \, |\Im(z)|\leq b_0\}$ there can be only finitely many eigenvalues, while the intersection
of the spectrum with the imaginary axis consists in a simple eigenvalue
at $z=0$. Corollary~\ref{specX'}  is proved.
\end{proof}

\begin{proof}[Proof of Theorems~\ref{main} and  claim (c) of Theorem ~\ref{resonances}]
 We will apply results of Butterley \cite{Butterley} to obtain  
 exponential decay of correlations for H\"older observables 
(bypassing  the final argument in  \cite[Proof of Theorem 1.1]{BaL}).
Proposition~\ref{dolgo} is \cite[Assumption ~3A]{Butterley},
Lemma ~ \ref{lem:lip} is \cite[Assumption ~1]{Butterley}, 
and  Corollary~ \ref{CorLY}
is \cite[Assumption ~2]{Butterley}.
In addition, Lemma~ \ref{lem:strong c} gives  strong continuity on $\cB$
 (in particular the domain $\Dom(X)$ of $X$
is dense). Thus, \cite[Theorem 1]{Butterley} (noting also the corrigendum) provides a finite set
 \[
 \{z_{j}\}_{j=0}^{N}= \spp(X)\cap  \{z\in \bC: -\upsilon_{Do}<  \Re({z})\leq 0, |\Im(z)| \leq b_0\} \, ,
 \]
 a (nontrivial) finite rank projector $\Pi:\cB\to \cB$, and an operator-valued function $t\mapsto \cP_{t}$, where $\cP_t$ is bounded on $\cB$,  $\Pi\cP_t=\cP_t\Pi=0$, and a matrix  $\widehat X\in L(F,F)$, $F=\Pi(\cB)$, which has exactly $\{z_j\}$ as eigenvalues and 
 \begin{equation}\label{eq:result}
 \cL_t = \cP_{t} + e^{t\widehat X}\Pi  \quad \text{for all $t\geq 0$,}
 \end{equation}
and for each $\upsilon_1 < \upsilon_{Do}$ there exists $C_{\upsilon_1}>0$ such that,  for all $\oldh$
in $\Dom X \subset \cB$,
 \begin{equation}\label{eq:estimate}
 |\cP_{t}\oldh|_{w} \leq C_{\upsilon_1} e^{-\upsilon_1 t}\|X \oldh\|_{\cB}  \quad \text{for all $t\geq 0$.}
 \end{equation}
Next, 
by Corollary ~ \ref{specX'}, we know that the only $z_j$ on the imaginary axis is
$z_0=0$, corresponding to the constant fixed point of $\cL_t$ and to
the invariant volume $\Pi_0 \oldh=\int_{\Omega_0} \oldh \, d m$.
This implies claim (c) of Theorem~ \ref{resonances}.

Finally, Theorem ~\ref{main} (and Corollary ~\ref{truemain})  follows from applying \eqref{eq:result}
and  \eqref{eq:estimate} to $\oldh\in \cC^2(\Omega_0)\cap \cC^0(\Omega)$ and $\psi \in \cC^1(\Omega_0)$, since
Lemma~\ref{relating} gives
$$
\left| \int \psi \cP_t \oldh \, dm \right| \le | \cP_t \oldh|_{(\cC^{1})^*}|\psi|_{\cC^1}\le | \cP                                                              _t \oldh|_{w}|\psi|_{\cC^1} \, ,
$$
while
the proof of  Lemma~ \ref{lem:strong c} gives
$\cC^2(\Omega_0)\cap\cC^0(\Omega) \subset \Dom( X)$ and $\|X \oldh\|_{\cB}\le C |\oldh|_{\cC^2(\Omega_0)}$.
\end{proof}

\subsection{Proof of Proposition~\ref{dolgo}}
\label{dolgoproof}

It only remains to deduce Proposition~\ref{dolgo} from Lemma~ \ref{dolgolemma}.

\begin{proof}[Proof of Proposition~\ref{dolgo}]
We shall use several times the trivial observation that if 
$0<\hat \eta <  \eta$ are small enough then for any $D >1$ we have
\begin{equation}\label{trivv}
\hat \eta < \frac{\eta }{D} \Rightarrow
\frac{1}{(1+\eta)^m} \le \frac{1}{(1+\hat \eta)^{Dm}} \quad \forall m\ge 1 \, .
\end{equation}

We introduce an auxiliary norm in order to handle the\footnote{ In \cite{GLP} another approach was used to address this issue, involving
an auxiliary norm $\|\cdot\|^*$ based on exact stable leaves (bypassing 
the neutral norm, using the resolvent). See
\cite[Lemma 7.4, Lemma 7.8, just after (7.4)]{GLP}.} 
$|z|$ factor in \eqref{eq:neutral R} which is much larger
than the $|b|^{-\gamma_{Do}}$ decay from the Dolgopyat Lemma~\ref{dolgolemma}.
For fixed $|b|\ge  1$ and $\oldh\in   \cC^1(\Omega_0)$, set
\begin{equation}\label{weighted}
\|\oldh \|_* = \|\oldh\|_s + \frac{c_u}{|b|} \|\oldh\|_u + \frac {1}{|b|} \|\oldh \|_0 \,.
\end{equation}

We shall prove Proposition~ \ref{dolgo} for the norm $\|\cdot\|_*$. We first check that
this is sufficient:
For $M >L>1$, if $a \in [a_0,a_1]$ and $|b|\ge b_0\ge 1$, 
where $a_0$, $a_1$, $b_0$, $0<c_1<c_2$,  and $\upsilon_{Do}$ are as in
Proposition~ \ref{dolgo} for  $\|\cdot\|_*$, we have for $c\in [c_1, c_2]$ and 
$m = [c \ln b]$,
\begin{align}
\|\cR^{Lm+Mm}(a+ib) \oldh\|_{\cB} &\le    |b| \|\cR^{Lm+Mm}(a+ib) \oldh\|_*\\
\nonumber                    &\le  C a^{-Lm} \|\cR^{Mm}(a+ib) \oldh\|_*\\
   \nonumber                 &\le  C a^{-Lm-Mm} (1+a^{-1} \upsilon_{Do})^{-Mm} \|\oldh\|_{\cB} \, ,
\end{align}
where  we used  in the first inequality the 
fact that $ \|\cdot\|_\cB
\le |b|  \|\cdot\|_*$), in the second inequality  Proposition~ \ref{dolgo} 
with the bound
$|b |(1+ a^{-1} \upsilon_{Do})^{-Lm}\le 1$ for $L$  large
enough, and in the last inequality Proposition~ \ref{dolgo} with $\|\cdot \|_{*} \le  \|\cdot \|_{\cB}$.
Taking large enough  $M/L$  and small enough  $0<\hat \upsilon_{Do} <\upsilon_{Do}$ 
(recalling \eqref{trivv}), we have proved
$$
\|\cR^{n} (a+ib)\|_{\cB}  \le (a+\hat \upsilon_{Do})^{-n}
$$
for some  $\hat \upsilon_{Do} >0$ and $\hat c_1 
< \hat c_2$, all $|b| \ge b_0$, $a\in [a_0, a_1]$, any $\hat c_1 <\hat c 
< \hat c_2$,
and $n=[\hat c \ln b]$. 

Now we prove Proposition~\ref{dolgo} for the norm $\| \cdot \|_*$.
Let $\oldh\in \cB \cap 
\cC^0_\sim \cap 
\cC^2(\Omega_0)$.
We shall assume that $1\le a_0 \le a_1 \le b_0$, to be determined later, and take 
$b\ge b_0$ (the case $b\le - b_0$ is similar) and $a\in [a_0, a_1]$. In particular
$$
\frac{|a+ib|}{a|b|}\le  2 \,  . 
$$

We shall consider $n=2m$, the case $n=2m+1$ is similar.
Our starting point is  the Lasota-Yorke estimate \eqref{babyLY} in the proof of Corollary~\ref{CorLY}, which holds also when replacing $\|\cdot\|_\cB$
by $\|\cdot \|_*$: 
Let $0<{\tilde \lambda}<1$  be as in Proposition~\ref{prop:LY R}, 
then there exists\footnote{ We use here that the constant $C$ in \eqref{babyLY} is uniform
in $a$ and $b$, also for $\|\cdot\|_*$.}   $m_0\ge 1$ so that for all $a\ge 1$, all
$|b|>0$, and all $m\ge m_0$, we have
 \begin{equation}\label{startdolgo}
  \| \cR(a+ib)^{2m} \oldh \|_{*} 
   \le a^{-m}  \nu_a^m \|\cR(a+ib)^{m} \oldh\|_* + C  a^{1-m}  |\cR(a+ib)^{m} \oldh|^{\mathbb H}_w \, ,
 \end{equation}
 where   $\nu_a^m \le C ( 1- a^{-1} \ln {\tilde \lambda})^{-m}<1$, and where the homogeneous
 weak norm  $|\oldh|_w^{\mathbb H}$ was defined in \eqref{onemorenorm}.
(Indeed, by Definition \ref{cG_t} of the partition $\cG_t(W)$,
all the admissible stable curves appearing in the right-hand sides of  Propositions~\ref{prop:L_t}
and~\ref{prop:LY R} are  homogeneous.)

Recalling \eqref{trivv}, take $\upsilon_{Do} >0$ so that
$$
0< \upsilon_{Do} < -\ln \tilde \lambda/ 2 \,  .
$$
Since \eqref{babyLY} for $\|\cdot\|_*$ also implies
$$
 a^{-m}   \nu_a^m\| \cR^{m} \oldh\|_*
\le \Cs a^{-2m}   \nu_a^m   \|\oldh\|_* \le \Cs a^{-2m}  
\frac{ 1}{ (1 -  a^{-1} \ln {\tilde \lambda})^m } \|\oldh\|_* \, , 
$$
and since we assume $a\ge a_0\ge 1$
the first term in the right-hand side of \eqref{startdolgo} 
is bounded by 
$$
\frac{ a^{-2m}}{6(1 + a^{-1} \upsilon_{Do})^{2m}} \|\oldh\|_*=
\frac{ 1}{6(a +\upsilon_{Do})^{2m}} \|\oldh\|_* \, ,
$$
 taking $D=2$ in \eqref{trivv},
 if $b_0$ is large enough and if
$c_1 > 0 $.

We  may thus focus on the second term
in the right-hand side of \eqref{startdolgo}, that
is, the homogeneous weak norm contribution. 
For $\sigma >1$ to be determined later,
we fix 
$$\epsilon=|b|^{-\sigma} \, . $$
Then, 
using \eqref{eq:weak R}, which gives 
$| \cR(z)^{m}|_{w}\le \Cs a^{-m}$, 
 we find (recalling that $1\le a\le a_1$)
\begin{align}
\label{eq4'} 
 \Cs  a^{1-m}   |\cR(z)^{m}(\oldh)|_{w}^{\mathbb H}
&\le \Cs    
 \big[ a^{-m}
|\cR(z)^{m}( \bM_{\epsilon}(\oldh))|_{w}^{\mathbb H}
  +  \Cs a^{-2m}|\oldh-\bM_{\epsilon}(\oldh)|_w^{\mathbb H} \big] \, .
\end{align}
If $q>1$ is close enough to $1$ and $\beta>0$ is small enough, then 
$\min\{\gamma, (2q)^{-1}, 1/q-2/5-\beta\}=\gamma$, since $\gamma<1/3$.
By Lemma~\ref{mollbound2},   the second term
on the right-hand side of \eqref{eq4'} is thus bounded by
\begin{align}\label{eq5'}
\Cs L_0^\oldbeta
a^{-2m}   \epsilon^{\gamma}  \|\oldh\|_\cB\le \Cs
a^{-2m}  |b|^{1-\sigma \gamma } \|\oldh\|_* .
\end{align}
Next,  if
\begin{equation}\label{ss0}
\sigma \gamma > 1 \, , 
\end{equation}
we get
$$
\Cs 
a^{-2m} |b|^{1-\sigma \gamma  }
\le  \frac{1}{6(a+\upsilon_{Do})^{2m}}
$$
if  $b\ge b_0$ is  large enough  and if
$$
2m\le \frac{ (\sigma \gamma-1 ) \ln |b| -\ln (6\Cs)  }
{  \ln (1+ a^{-1} \upsilon_{Do})} \, .
$$
The numerator in the right-hand side above is $>    (\sigma \gamma-1 )\ln |b|/2$
for large enough  $b_0$ and $a_0/\upsilon_{Do}$, since $a\ge a_0$. This gives
a constraint satisfied  if 
\begin{equation}\label{bdlogb2}c_2 \le \frac{a_0  (\sigma \gamma-1
 )}{4 \upsilon_{Do}}\, .
\end{equation}

 The first term in \eqref{eq4'} will be more tricky to handle. Recalling
 $a\le a_1$, we must estimate:
\begin{equation}
\label{eq6'}
 \Cs a^{-m}   \sup_{\substack{\psi \in C^{\oldp}(W)\\ |\psi|_{C^{\oldp}(W)}\le 1}}
\int_W \cR(z)^{m}( \bM_{\epsilon}(\oldh)) dm_W
\end{equation}
on some admissible stable homogeneous curve $W$.
Recall $\Lambda=\Lambda_0^{1/\tau_{\max}}$ from
\eqref{Lambda}. By the Dolgopyat bound (Lemma~\ref{dolgolemma}),
there exist $C_{Do}>0$ and $\gamma_{Do}$ (independent of
$\oldbeta$ and  $\oldq$),  so that, if \footnote{ The proof gives a constant $C_{Do}$  much larger than  $\gamma_{Do}$, and
we may safely assume that $\gamma_{Do}<1$.}
\begin{equation}
\label{bdlogb0}
m \ge  C_{Do}  \ln| b| \qquad \mbox{ i.e., }\quad c_1 \ge C_{Do} \, ,
\end{equation}
then \eqref{eq6'} 
is bounded by
\begin{align}\label{eq66'}
 \Cs   a^{-2m} |b|^{-\gamma_{Do}}
\biggl (  |\bM_{\epsilon} (\oldh)|_{\infty(\Omega_0)}
+\frac{ |\bM_{\epsilon}( \oldh)|_{H^1_\infty(\Omega_0)}}{(1+ a^{-1} \ln \Lambda )^{m/2}}
\biggr )\, .
\end{align}

Now,    Lemma~\ref{mollbound1} and Lemma~\ref{relating}  give
\begin{align}\label{eq7'}
| \bM_{\epsilon}( \oldh)|_{L^\infty(\Omega_0)}
 \le \Cs\epsilon^{-\oldq-1+\oldbeta} 
\|\oldh\|_{s}
 \le \Cs\epsilon^{-\oldq-1+\oldbeta} 
\|\oldh\|_{*}\, ,
\end{align}
and  
\begin{align}
\label{eq8'}
|\bM_\epsilon(\oldh)|_{H^{1}_\infty(\Omega_0)}&
\le \Cs\epsilon^{-\oldq-2+\oldbeta} \|\oldh\|_{s}
\le \Cs\epsilon^{-\oldq-2+\oldbeta} \|\oldh\|_{*}\, .
\end{align}

On the one hand, if 
\begin{equation}\label{ssa}
\sigma (\oldq +1-\oldbeta)< \gamma_{Do} \, ,
\end{equation}
the estimate  \eqref{eq7'}   gives
the following bound for the first term of \eqref{eq66'}
$$
\Cs  a^{-2m} |b|^{- \gamma_{Do} +\sigma (\oldq+1-\oldbeta)}\|\oldh\|_{*}
\le  \frac{1}{6(a+\upsilon_{Do})^{2m}}\|\oldh\|_{*}\, ,
$$
if $b_0$ is large enough, and
$$
2m\le \frac{(\gamma_{Do}-\sigma (\oldq+1-\oldbeta))\ln |b|- \ln (6\Cs)}
{\ln (1+ a^{-1} \upsilon_{Do})}\, .
$$
which holds if $b_0$ is large enough and
\begin{equation}\label{bdlogb3}
c_2 \le a_0  \frac{\gamma_{Do}-\sigma (\oldq+1-\oldbeta)}
{2 \upsilon_{Do}} \, .
\end{equation}

On the other hand, recalling $1\le a_0\le a \le a_1$, \eqref{eq8'} gives that the second term of
\eqref{eq66'} is bounded by $\|\oldh\|_{*}$ multiplied by
(note that \eqref{ss0} implies
$\sigma (\oldq+2-\oldbeta)>\gamma^{-1}$, while $\gamma^{-1}\ge 3>\gamma_{Do}$)
$$
\Cs  \frac{ a^{-2m}|b|^{ - \gamma_{Do}+\sigma (\oldq-\oldbeta+2)} }{ (1+ a^{-1}\ln \Lambda)^{m/2}} 
\le \Cs  \frac{ a^{-2m}|b|^{ - \gamma_{Do}+\sigma (\oldq-\oldbeta+2)} }{ (1+ a^{-1} \ln \Lambda)^{m/4} (1+ a^{-1} \ln \Lambda )^{m/4}} 
\le  \frac{1}{6(1+ a^{-1} \upsilon_{Do} )^{2m}} \, , 
$$
if $b_0$ is large enough and $m$ is large enough, more precisely
\begin{equation}\label{morelower}
c_1 > 4 \frac{\sigma (\oldq+2-\oldbeta)-\gamma_{Do}}{\ln (1+ a^{-1} \ln  \Lambda )} \, ,
\end{equation}
and if, in addition, recalling \eqref{trivv}, we assume
$
 \upsilon_{Do}
< \frac{\ln  \Lambda}{8}
$.

Along the way, we have collected the lower bounds on $c_1$ given by
\eqref{morelower} and 
\eqref{bdlogb0}, the upper bounds on $c_2$ listed in
\eqref{bdlogb2}
 and  \eqref{bdlogb3}, as well as the conditions
 \eqref{ss0} (lower bound on $\sigma$) and
 \eqref{ssa}
(upper bound on $\sigma$).
Fixing $\oldp=1/3$,  Lemma ~\ref{dolgolemma}
provides values for $\gamma_{Do}$ and $C_{Do}$.
 If $b_0>1$ is large enough and
 $\upsilon_{Do}>0$ is small enough,
 then all conditions can be satisfied simultaneously if $\oldbeta <1$ is close enough
to $1$ and  $\oldq>0$ is small enough.  
In particular, recalling\footnote{ We may take $\gamma=\oldp-\oldq$, which is smaller than
(and close
to) $1/3<\oldbeta$, given our other choices, so that $\sigma$ is
larger than (and close to) $3$.} 
$\sigma^{-1}<\gamma \le \oldp-\oldq$ 
we require
\begin{equation}\label{finalcond}
\oldq +1-\oldbeta< \frac{\gamma_{Do}}{\sigma}<\gamma_{Do}(\oldp-\oldq)  \, .
\end{equation}

This completes the proof of Proposition~\ref{dolgo}.
\end{proof}

\appendix

\section{Proofs of  Lemmas~\ref{lem:holder jac} and ~ \ref{lem:interpolate}
(approximate foliation holonomy)}
\label{postponeproof}

This section contains the proofs of Lemmas~\ref{lem:holder jac} and ~\ref{lem:interpolate}
(Subsections \ref{sec:step2} and ~ \ref{sec:step4}) on the H\"older bounds of the Jacobians of the holonomy
of the fake (approximate) foliation. These bounds are
instrumental to get the four-point estimate (vii) for the fake unstable foliation.

\begin{proof}[Proof of Lemma~\ref{lem:holder jac}]
Letting $\bh_{V^n_1, V^n_2}$ denote the holonomy map from $V^n_1 = T^{-n}V_1$
to $V^n_2 = T^{-n} V_2$, we begin by noting that
\begin{equation}
\label{eq:holder split}
\begin{split}
&\ln \frac{J\bh_{12}( v_1(\bxs),\bxs)}{J\bh_{12}( v_1(\bys),\bys)}\\
 &\qquad= \ln \frac{J^s_{V_1}T^{-n}( v_1(\bxs),\bxs)}{J^s_{V_2}T^{-n}(\bh_{12}(v_1(\bxs),\bxs,))}
- \ln \frac{J^s_{V_1}T^{-n}( v_1(\bys),\bys)}{J^s_{V_2}T^{-n}(\bh_{12}(v_1(\bys),\bys))}
+ \ln \frac{J\bh_{V^n_1, V^n_2}(x_n)}{J\bh_{V^n_1, V^n_2}(y_n)} \\
&\qquad = \sum_{j=0}^{n-1} \ln J^s_{T^{-j}V_1}T^{-1}(x_j) - \ln J^s_{T^{-j}V_2}T^{-1}(\tx_j)
- \ln J^s_{T^{-j}V_1}T^{-1}(y_j) + \ln J^s_{T^{-j}V_2}T^{-1}(\ty_j) \\
& \qquad \qquad\quad+ \ln \frac{J\bh_{V^n_1, V^n_2}(x_n)}{J\bh_{V^n_1, V^n_2}(y_n)} \, ,
\end{split}
\end{equation}
where as in the proof of Lemma~\ref{lem:jac holonomy},
$x_j = T^{-j}(v_1(\bxs),\bxs)$, $\tx_j = T^{-j}(\bh_{12}(v_1(\bxs),\bxs))$, and similarly for $y_j$
and $\ty_j$.  

We begin by estimating the difference of stable Jacobians in \eqref{eq:holder split}.
The factors in each term involving the stable Jacobian are given by \eqref{eq:jac factors}, 
and we can bound the
differences by grouping together either the terms on the same stable curve 
(standard distortion bounds), or the terms on the same unstable curve (using 
Lemma~\ref{lem:jac holonomy}).  
Assuming without loss of generality that $d(x_0, \tx_0) \ge d(y_0, \ty_0)$,
this yields the following bound on the sum,
\begin{equation}
\label{eq:min sum}
 C \sum_{j=0}^{n-1} \min \{ d(x_{j+1}, y_{j+1}) k_{j+1}^2, d(x_{j+1}, \tx_{j+1}) k_{j+1}^2
+ d(x_j, \tx_j) + \phi(x_j, \tx_j) + \phi(x_{j+1}, \tx_{j+1}) \} \, ,
\end{equation}
where we have used the uniform bound on $J\bh_{12}$ given by Lemma~\ref{lem:jac holonomy} 
and \eqref{eq:uniform} to eliminate terms involving $d(\tx_j, \ty_j)$ and $d(y_j, \ty_j)$.
We estimate each term using one of two cases.

\smallskip
\noindent
{\em Case 1.}  $d(x_{j+1}, y_{j+1}) \le d(x_{j+1}, \tx_{j+1})$.  We write,
\begin{align}
\label{eq:split exp}
d(x_{j+1},y_{j+1}) &= d(x_{j+1}, y_{j+1})^{\varpi} d(x_{j+1}, y_{j+1})^{1-\varpi}\\
\nonumber
&\le C (J^s_{V_1}T^{-j-1}(x_0))^\varpi d(x_0, y_0)^\varpi d(x_{j+1}, \tx_{j+1})^{1-\varpi} \, .
\end{align}
On the one hand, we have 
\[
d(x_{j+1} , \tx_{j+1}) \le C d(x_0, \tx_0) (J^u_{T^{n-j}\ell}T^{j+1}(x_j))^{-1} \, ,
\]
where $\ell$ is the curve  so that $x_0 \in T^n(\ell) =: \bar\gamma$.
On the other hand,
by the uniform transversality of the curves $T^{-j-1}(V_1)$ and $T^{-j-1}(\bar\gamma)$,
we have 
\begin{equation}
\label{eq:translate factors}
\begin{split}
& J^s_{V_1}T^{-j-1}(x_0)J^u_{\bar\gamma}T^{-j-1}(x_0) = C^{\pm 1} JT^{-j-1}(x_0)
= C^{\pm 1} \tfrac{\cos \vf(x_0)}{\cos \vf(x_{j+1})} \\
\implies & J^s_{V_1}T^{-j-1}(x_0) = C^{\pm 1} k_W^{-2} k_{j+1}^{2} J^u_{T^{-j-1}(\bar\gamma)}T^{j+1}(x_{j+1}) \, ,
\end{split}
\end{equation}
where $JT^{-j-1}$ is the full Jacobian of the map $T^{-j-1}$.  This estimate together
with \eqref{eq:strip bound} implies
\begin{equation}
\label{eq:stable jac bound}
J^s_{V_1}T^{-j-1}(x_0) \le C k_W^{-10/3} \rho^{-2/3} 
(J^u_{T^{-j-1}\bar\gamma}T^{j+1}(x_{j+1}))^{5/3} \, .
\end{equation}
Now using this estimate together with
\eqref{eq:split exp} and again \eqref{eq:strip bound} yields,
\begin{equation}
\label{eq:case 1}
d(x_{j+1},y_{j+1})k_{j+1}^2 \le C d(x_0,y_0)^\varpi d(x_0, \tx_0)^{1-\varpi}
\frac{k_W^{-(4+10\varpi)/3}}{\rho^{2(1+\varpi)/3}} 
\frac{1}{(J^u_{T^{-j-1}\bar\gamma}T^{j+1}(x_{j+1}))^{(1-8\varpi)/3}}\, ,
\end{equation}
and this will be summable over $j$ as long as $\varpi < 1/8$.

\smallskip
\noindent
{\em Case 2.}  $d(x_{j+1}, \tx_{j+1}) \le d(x_{j+1}, y_{j+1})$.  
In order to control the terms involving $\phi(x_j, \tx_j)$, we 
use the expressions for the Jacobians given by \eqref{eq:jac factors} and  \eqref{eq:translate factors} 
together with \cite[eq. (5.27) and following]{chernov book}
 to write,
\begin{equation}
\begin{split}
\label{eq:angle bound}
\phi(x_j, \tx_j) &  \le  C\sum_{m=1}^j 
\frac{d(x_{j-m}, \tx_{j-m}) }{J^s_{T^{m-j-1}V_1}T^{-m+1}(x_{j-m+1}) J^u_{T^{-j}\bar\gamma}T^{m-1}(x_{j}) }
+ C \frac{\phi(x_0, \tx_0)}{J^s_{V_1}T^{-j}(x_0) J^u_{T^{-j}\bar\gamma}T^j(x_j) }  \\
& \le C d(x_0, \tx_0) \sum_{m=1}^j \frac{\Lambda_0^{-m+1}}{J^u_{T^{m-j}\bar\gamma}T^{j-m}(x_{j-m}) J^u_{T^{-j}\bar\gamma}T^{m-1}(x_{j})} + C \frac{\Lambda_0^{-j} d(x_0, \tx_0)}{J^u_{T^{-j}\bar\gamma}T^j(x_j)} \, , 
\end{split}
\end{equation}
where in the second line we have used the assumption that $\phi(x_0, \tx_0)$ is
proportional to $d(x_0, \tx_0)$ and that 
$d(x_{j-m}, \tx_{j-m}) \le C d(x_0, \tx_0)/J^u_{T^{m-j}\bar\gamma}T^{j-m}(x_{j-m})$.
Note that for terms with $m \ge 2$, there is a gap between the expansion factors of the 
unstable Jacobians
in the denominator of the sum:  The Jacobian $J^u_{T^{m-j-1}\bar\gamma}T(x_{j-m+1})$
is missing.  We use the fact that this is proportional to $k_{j-m}^2$ to estimate,
\[
\begin{split}
\frac{d(x_0, \tx_0)^{9\varpi}}{J^u_{T^{m-j}\bar\gamma}T^{j-m}(x_{j-m}) J^u_{T^{-j}\bar\gamma}T^{m-1}(x_{j})}
& \le \frac{C d(x_j, \tx_j)^{9\varpi} k_{j-m}^{18\varpi}}{(J^u_{T^{m-j}\bar\gamma}T^{j-m}(x_{j-m}) J^u_{T^{-j}\bar\gamma}T^{m-1}(x_{j}))^{1-9\varpi}} \\
& \le \frac{C d(x_j, \tx_j)^{9\varpi} k_W^{-12\varpi} \rho^{-6\varpi}}{(J^u_{T^{m-j}\bar\gamma}T^{j-m}(x_{j-m}))^{1-15\varpi}
(J^u_{T^{-j}\bar\gamma}T^{m-1}(x_{j}))^{1-9\varpi}} \\
& \le C d(x_j, \tx_j)^{9\varpi} k_W^{-12\varpi} \rho^{-6\varpi} \Lambda_0^{-(j-1)(1-15\varpi)}\,  ,
\end{split}
\]
where we have used \eqref{eq:strip bound} in the second line.  Notice that this estimate
holds for $m=1$ as well.
Putting this
estimate together with \eqref{eq:angle bound} yields,
\begin{equation}
\label{eq:angle final}
\phi(x_j , \tx_j) \le C d(x_0, \tx_0)^{1-9\varpi} d(x_j, \tx_j)^{9\varpi} 
k_W^{-12\varpi} \rho^{-6\varpi} \Lambda_0^{-(j-1)(1-15\varpi)}  \, .
\end{equation}
A similar estimate holds for $\phi(x_{j+1}, \tx_{j+1})$ with $d(x_{j+1}, \tx_{j+1})$ in place
of $d(x_j, \tx_j)$.

Now since
$d(x_j, \tx_j) = C^{\pm 1} d(x_{j+1}, \tx_{j+1}) J^u_{T^{-j-1}\bar\gamma}T(x_{j+1})$, we have
\[
d(x_j,\tx_j) \le C d(x_{j+1}, y_{j+1}) k_j^2
\]
by the assumption of this case as well.
Since 
$d(x_{j+1}, \tx_{j+1}) \le C d(x_j, \tx_j)$, combining similar terms in
\eqref{eq:min sum}, it remains to estimate the following expression in
this case,
\begin{equation}
\label{eq:case 2 first}
\begin{split}
& d(x_{j+1}, \tx_{j+1}) k_{j+1}^2 + d(x_j, \tx_j) + d(x_0, \tx_0)^{1-9\varpi} d(x_j, \tx_j)^{9\varpi} 
k_W^{-12\varpi} \rho^{-6\varpi}\Lambda_0^{-(j-1)(1-15\varpi)} \\
& \le d(x_{j+1}, \tx_{j+1}) k_{j+1}^2 + d(x_{j+1}, \tx_{j+1}) k_j^2 
+ \frac{d(x_0, \tx_0)^{1-9\varpi} d(x_{j+1}, \tx_{j+1})^{9\varpi} k_j^{18\varpi}} 
{k_W^{12\varpi} \rho^{6\varpi} \Lambda_0^{(j-1)(1-15\varpi)}}  \,  .
\end{split}
\end{equation}
For the first term, we use the assumption of this case to write,
\begin{align*}
d(x_{j+1}, \tx_{j+1}) &= d(x_{j+1}, \tx_{j+1})^\varpi d(x_{j+1}, \tx_{j+1})^{1-\varpi}\\
&\le C (J^s_{V_1}T^{-j-1}(x_0))^\varpi d(x_0, y_0)^\varpi
d(x_{j+1}, \tx_{j+1})^{1-\varpi} \, ,
\end{align*}
which is 
the same expression as in \eqref{eq:split exp}.  Thus the first term
in \eqref{eq:case 2 first} is bounded by \eqref{eq:case 1}.

For the second term in \eqref{eq:case 2 first}, we use \eqref{eq:strip bound} to bound
$k_j^2$ and again \eqref{eq:stable jac bound} to estimate,
\[
\begin{split}
d(x_{j+1}, \tx_{j+1}) k_j^2 
& \le C (J^s_{V_1}T^{-j-1}(x_0))^{\varpi} d(x_0, y_0)^\varpi 
\frac{d(x_0, \tx_0)^{1-\varpi} }{(J^u_{T^{-j-1}\bar\gamma}T^{j+1}(x_{j+1}))^{1-\varpi}} 
\frac{(J^u_{T^{-j}\bar\gamma}T^j(x_j))^{2/3}}{k_W^{4/3} \rho^{2/3}} \\
& \le C d(x_0, y_0)^\varpi d(x_0, \tx_0)^{1-\varpi} \frac{k_W^{-(4+ 10\varpi)/3}}{\rho^{2(1+\varpi)/3}}
\frac{1}{(J^u_{T^{-j-1}\bar\gamma}T^{j+1}(x_{j+1}))^{(1-8\varpi)/3}} \,  ,
\end{split}
\]
where in the second line we have used  the fact that
$J^u_{T^{-j}\bar\gamma}T^j(x_j) \le C J^u_{T^{-j-1}\bar\gamma}T^{j+1}(x_{j+1})$
as well as \eqref{eq:translate factors} to obtain a bound equivalent to
\eqref{eq:case 1}.

Finally, we estimate the third term in \eqref{eq:case 2 first} following a similar strategy, using again \eqref{eq:stable jac bound},
\[
\begin{split}
&d(x_{j+1}, \tx_{j+1})^{9\varpi} k_j^{18\varpi} \\
&\qquad \le C (J^s_{V_1}T^{-j-1}(x_0))^{\varpi} d(x_0, y_0)^\varpi 
\frac{d(x_0, \tx_0)^{8\varpi} }{(J^u_{T^{-j-1}\bar\gamma}T^{j+1}(x_{j+1}))^{8\varpi}} 
\frac{(J^u_{T^{-j}\bar\gamma}T^j(x_j))^{6\varpi}}{k_W^{12\varpi} \rho^{6\varpi}} \\
& \qquad\le C d(x_0, y_0)^\varpi d(x_0, \tx_0)^{8\varpi} \frac{k_W^{-46\varpi/3}}{\rho^{20\varpi/3}}
\frac{1}{(J^u_{T^{-j-1}\bar\gamma}T^{j+1}(x_{j+1}))^{\varpi/3}} \,  .
\end{split}
\]
Putting these three estimates together in \eqref{eq:case 2 first}
implies that the minimum factor from 
\eqref{eq:min sum} in this case is bounded by
\[
Cd(x_0,y_0)^\varpi d(x_0, \tx_0)^{1-\varpi} 
\Lambda_0^{-(j-1)(1-44\varpi/3)} \max \{ k_W^{-82\varpi/3} \rho^{-38\varpi/3},  k_W^{-4/3 - 10\varpi/3} \rho^{-2(1+\varpi)/3} \} \, ,
\]
which will be summable over $j$ as long as $\varpi \le 1/15$.

\smallskip
Finally, using Case 1 or Case 2 as appropriate for each term appearing in \eqref{eq:min sum}
and summing over $j$ completes the required estimate on the difference of stable Jacobians
appearing in \eqref{eq:holder split}.

It remains to estimate the term involving $J\bh_{V^n_1, V^n_2}$, which is the holonomy
of the seeding foliation $\{ \ell_z \}_{z \in \oW_i}$.  
Since $\{ \ell_z \}$ is uniformly $\cC^2$, we have on the one hand,

\[
\ln \frac{J\bh_{V^n_1, V^n_2}(x_n)}{J\bh_{V^n_1, V^n_2}(y_n)} \le
C(d(x_n, \tx_n) + \phi(x_n, \tx_n) + d(y_n, \ty_n) + \phi(y_n, \ty_n)) \, ,
\]
while on the other hand, using the fact that the curvatures of $V^n_1$ and
$V^n_2$ are uniformly bounded,
\[
\ln \frac{J\bh_{V^n_1, V^n_2}(x_n)}{J\bh_{V^n_1, V^n_2}(y_n)} \le
C ( d(x_n ,y_n) + d(\tx_n, \ty_n)) \, .
\]
Now 
$\ln \frac{J\bh_{V^n_1, V^n_2}(x_n)}{J\bh_{V^n_1, V^n_2}(y_n)}$ is bounded by the 
minimum of these two estimates, which we recognize as a simplified
version of the expression \eqref{eq:min sum} with $j=n$.  Thus this quantity satisfies the same bounds
as in Cases 1 and 2 above, completing the
proof of Lemma~\ref{lem:holder jac}.
\end{proof}

\begin{proof}[Proof of Lemma~\ref{lem:interpolate}]
Let $V_0$ be the subcurve of $P^+(W)$ corresponding to the gap; in the coordinates $(\bxu, \bxs)$,
$V_0$ is the vertical line segment identified by the interval $[a,b]$.  The boundary curves of the gap
from the surviving foliation on either side as described by 
$\{ (\barG(\bxu, a), a) : |\bxu| \le k_W^2 \rho \}$
and $\{ (\barG(\bxu, b), b) : |\bxu| \le k_W^2 \}$.  Let $x_a = (0, a)$ and
$y_b = (0,b)$.  Fix $\bxu$ with $|\bxu| \le k_W^2 \rho$ and define $\tx_a = (\bxu, \barG(\bxu, a))$
and $\ty_b = (\bxu, \barG(\bxu, b))$.  Let $V_1$ denote the stable curve (vertical in the $(\bxu, \bxs)$
coordinates) connecting $\tx_a$ and $\ty_b$.

Using similar notation to Subsection~ \ref{sec:step3}, let $j+1$ denote the least integer
$j' \ge 1$ such that $T^{j'}(\cS_0)$ intersects $V_0$.  This implies in
particular, that the surviving parts of the foliation immediately on either side of the gap 
lie in the same homogeneity
strip for the first $j$ interates of $T^{-1}$.
 Define $x_i = T^{-i}(x_a)$, 
$y_i = T^{-i}(y_b)$, $\tx_i = T^{-i}(\tx_a)$ and $\ty_i = T^{-i}(\ty_b)$, for $i = 0, 1, \ldots j$.

The following expressions for the quantities appearing in the slope will be useful,
\[
\barG(\bxu, b) - \barG(\bxu, a) = \int_{T^{-j}V_1} J^s_{T^{-j}V_1}T^j \, dm_W\, ,
\qquad b-a = \int_{T^{-j}V_0} J^s_{T^{-j}V_0} T^j \, dm_W \,  .
\]
And similarly,
\[
\partial_{\bxs} \barG(\bxu, a) = \frac{J^s_{T^{-j}V_1}T^j(y_j)}{J^s_{T^{-j}V_0}T^j(x_j)}
J \bh_{-j}(x_j) \, ,
\] 
where $\bh_{-j}$ is the holonomy map from $T^{-j}V_0$ to $T^{-j}V_1$ along the surviving foliation
on the sides of the gap.
Using these expressions, we estimate,
\begin{equation}
\label{eq:slope split}
\begin{split}
\ln \left| \frac{\barG(\bxu, b) - \barG(\bxu, a)}{(b-a) \partial_{\bxs} \barG(\bxu, a)} \right|
= & \ln \frac{\frac{1}{|T^{-j}V_1|} \int_{T^{-j}V_1} J^s_{T^{-j}V_1}T^j \, dm_W}{J^s_{T^{-j}V_1}T^j(y_j)}
- \ln J \bh_{-j}(x_j)  \\
& \; \; +  \ln \frac{J^s_{T^{-j}V_0}T^j(x_j)}{\frac{1}{|T^{-j}V_0|} \int_{T^{-j}V_0} J^s_{T^{-j}V_0}T^j \, dm_W} 
 + \ln \frac{|T^{-j}V_0|}{|T^{-j}V_1|}  \, .
\end{split}
\end{equation}
Since $J^s_{T^{-j}V_1}T^j$ is continuous on $T^{-j}V_1$, the average value of the function is
equal to the function evaluated at some point on $T^{-j}V_1$.  Thus by standard distortion bounds along
stable curves, the first term on the right-hand side of \eqref{eq:slope split} is bounded by 
$C_d |a-b|^{1/3}$.
A similar bound holds for the third term above.  

In order to estimate $J\bh_{-j}(x_j)$, we note that this is part of the surviving foliation
that originated at time $-n$.   Applying 
Lemma~\ref{lem:jac holonomy} from time $-j$ to time $-n$, there exists $C>0$ such that
\[
\ln J\bh_{-j}(x_j)
\le C (d(x_j, \tx_j) \rho^{-2/3} k_{W}^{-4/3} 
+ \phi(x_j, \tx_j))  \, ,
\]
where $\phi(x_j, \tx_j)$ represents the angle formed by the tangent vectors to $T^{-j}V_0$
and $T^{-j}V_1$ at $x_j$ and $\tx_j$, respectively.

For the first term above, 
$d(x_j, \tx_j )\le C d(x_0 , \tx_0)/ J^u_{T^{-j}\bgamma}T^j(x_j)$, where
$\bgamma$ denotes the element of the surviving foliation containing $x_a$.
For the second term, we use \eqref{eq:angle bound} to write,
\[
\phi(x_j, \tx_j) \le C d(x_0, \tx_0) 
\sum_{m=1}^{j} \frac{\Lambda_0^{-m+1}}{J^u_{T^{m-j}\bgamma}T^{j-m}(x_{j-m})
J^u_{T^{-j}\bgamma}T^{m-1}(x_j)} + C \frac{\Lambda_0^{-j} d(x_0, \tx_0)}{J^u_{T^{-j}\bgamma}T^j(x_j)} \, ,
\]
where we have used the fact that $\phi(x_0, \tx_0) \le C d(x_0, \tx_0)$ since both
$V_0$ and $V_1$ are vertical segments and the unstable curves in the surviving foliation
have uniformly bounded curvature.
As in the proof of Lemma~\ref{lem:holder jac}, there is a gap in the product of
unstable Jacobians; the factor $J^u_{T^{m-j}\bgamma} T(x_{j-m+1})$ is missing. 
This is proportional to $k^2_{j-m}$, but due to \eqref{eq:strip bound}, we do not ask for
a full power of this factor; rather, we estimate,
\[
\begin{split}
\frac{1}{J^u_{T^{m-j}\bgamma}T^{j-m}(x_{j-m})J^u_{T^{-j}\bgamma}T^{m-1}(x_j)}
&\le \frac{C k_{j-m}^{6/5}}{J^u_{T^{m-j}\bgamma}T^{j-m}(x_{j-m}) 
(J^u_{T^{m-j}\bgamma} T(x_{j-m+1}))^{3/5} J^u_{T^{-j}\bgamma}T^{m-1}(x_j)} \\
& \le C \frac{k_W^{4/5} \rho^{-2/5}}{(J^u_{T^{-j}\bgamma} T^j(x_j))^{3/5}} \, .
\end{split}
\]
This, together with the above estimates implies,
\begin{equation}
\label{eq:h bound} 
\ln J\bh_{-j}(x_j) \le C \left( \frac{d(x_0, \tx_0)}{\rho^{2/3} k_W^{4/3} J^u_{T^{-j} \bgamma}T^j(x_j)}
+ \frac{d(x_0, \tx_0) k_W^{4/5}}{\rho^{2/5} (J^u_{T^{-j}\bgamma}T^j(x_j))^{3/5}} \right)\,  .
\end{equation}

\begin{sublem}[Relating the gap size and the unstable Jacobian]
\label{sub:gap}
Let $a,b$ define the gap in $P^+(W)$ as in the statement of the lemma and let $j, x_j$ be as
defined above.  There exists $C>0$, depending only on $T$, and not on $n$, $j$ or $W$, such
that
\[
\frac{1}{J^u_{T^{-j}\bgamma}T^j(x_j) } \le C k_W^{-2} \rho^{-26/35} |a-b|^{9/35} \,  .
\]
\end{sublem}
\begin{proof}
We consider the size of the gap created in a neighborhood of $T^{-j-1}(x_a)$, depending on whether
this gap is created by an intersection with homogeneity strips of high index or not. 

\smallskip
\noindent
{\em Case 1.} The connected component of $T^{-j-1}(V_1)$ containing $T^{-j-1}(x_a)$
intersects $\cS_0$.  Then using \eqref{eq:new(2)} and 
\eqref{eq:gap length}, we have,
\[
|a-b| \ge \frac{C^{\pm 1} k_{j+1}^{-4}}{J^s_{V_1}T^{-j}(x_a)}
\ge \frac{C^{\pm 1} k_W^{8/3} \rho^{4/3}}{(J^u_{T^{-j-1}\bgamma}T^{j+1}(x_{j+1}))^{4/3} J^s_{V_1}T^{-j}(x_a)} \, ,
\]
where  we  used \eqref{eq:strip bound} in the second inequality.  Now 
$J^u_{T^{-j-1}\bgamma}T^{j+1}(x_{j+1}) = C^{\pm 1} k_j^2 J^u_{T^{-j}\bgamma}T^j(x_j)$. 
Using this together with \eqref{eq:translate factors} to convert between stable and unstable 
Jacobians yields,
\[
|a-b| \ge \frac{C^{\pm 1} k_W^{8/3} \rho^{4/3}}{(J^u_{T^{-j}\bgamma}T^j(x_j))^{7/3} 
k_j^{14/3} k_W^{-2}}
\ge  \frac{C^{\pm 1} k_W^{70/9} \rho^{26/9} }{(J^u_{T^{-j}\bgamma}T^j(x_j))^{35/9} } \, ,
\]
where in the second inequality we have used \eqref{eq:strip bound}.
This proves the sublemma in this case.

\smallskip
\noindent
{\em Case 2.} The component of $T^{-j-1}(V_1)$ containing $T^{-j-1}(x_a)$ does not intersect
$\cS_0$.  In this case, by the uniform transversality of $\cS_1$ with unstable curves,
the gap is bounded below by
\[
\begin{split}
|a-b| & \ge \frac{C^{\pm 1} \rho k_W^2}{J^u_{T^{-j-1}\bgamma}T^{j+1}(x_{j+1}) J^s_{V_1}T^{-j-1}(x_a)}
\ge \frac{C^{\pm 1} \rho k_W^4}{(J^u_{T^{-j-1}\bgamma}T^{j+1}(x_{j+1}))^2 k_{j+1}^2 } \\
& \ge \frac{C^{\pm 1} \rho k_W^4}{(J^u_{T^{-j}\bgamma}T^j(x_j))^2} \, ,
\end{split}
\]
where in the last step, we have used the fact that in Case 2, $k_{j+1}$ is of order 1.  This
is clearly a greater lower bound on $|a-b|$ than in Case 1, proving the sublemma.
\end{proof}

With the sublemma proved, we return to our estimate \eqref{eq:h bound},
\[
\begin{split}
\ln J\bh_{-j}(x_j) & \le C \left( \frac{d(x_0, \tx_0)}{\rho^{2/3} k_W^{4/3} J^u_{T^{-j} \bgamma}T^j(x_j)}
+ \frac{d(x_0, \tx_0) k_W^{4/5}}{\rho^{2/5} (J^u_{T^{-j}\bgamma}T^j(x_j))^{3/5}} \right) \\
& \le C  d(x_0, \tx_0)\left( k_W^{-10/3} \rho^{-148/105}  |a-b|^{9/35}
+  k_W^{-2/5} \rho^{-148/175} |a-b|^{27/175} \right) \\
& \le C \left( k_W^{-4/3} \rho^{-43/105}  |a-b|^{9/35} + k_W^{8/5} \rho^{27/175} |a-b|^{27/175} \right) \\
& \le C \left(  \rho^{-43/105} |a-b|^{9/35} + \rho^{-29/175} |a-b|^{27/175} \right)
\le C \rho^{-31/105} |a-b|^{1/7} \, ,
\end{split}
\]
where in the third line we have used the fact that $d(x_0, \tx_0) \le C \rho k_W^2$ and in the
last line that $k_W \le C \rho^{-1/5}$.  We have also opted to take simpler (slightly less than
optimal) exponents, taking the power $1/7$
rather than $27/175$ in the second term of the last line and converting $|a-b|^{4/35} \le C\rho^{4/35}$
in the first term.  Note that $|a-b| \le C\rho$ follows from Subsection ~ \ref{sec:step1}.

Finally, for the fourth term on the right-hand side of \eqref{eq:slope split}, we note that
the boundary curves for the gap containing $T^{-j}V_0$ and $T^{-j}V_1$ both lie in the unstable cone.
Since $T^{-j}V_0$ and $T^{-j}V_1$ have bounded curvature, we have
\[
\ln \frac{|T^{-j}V_0|}{|T^{-j}V_1|} \le C d(x_j, \tx_j) \le \frac{C d(x_0, \tx_0)}{J^u_{T^{-j} \bgamma}T^j(x_j)} \, .
\]
Using Sublemma~\ref{sub:gap}, this quantity is bounded by
\[
C d(x_0, \tx_0) k_W^{-2} \rho^{-26/35} |a-b|^{9/35} 
\le C \rho^{9/35} |a-b|^{9/35}\, ,
\]
where again we have used the fact that $d(x_0, \tx_0) \le k_W^2 \rho$.
Using the estimates for these four terms in \eqref{eq:slope split} ends the
proof of  Lemma~\ref{lem:interpolate} since
both the average slope and $\partial_{\bxs} \barG$ are uniformly bounded
away from 0 and infinity.
\end{proof}

\section{Estimates for the Dolgopyat bound (Lemmas ~ \ref{lem:hpr},  \ref{lem:cancel},
and \ref{lem:disco})}\label{sec:hoihoi}

This appendix contains several crucial, but technical,  estimates used in the Dolgopyat type cancellation argument developed in Section \ref{dodo}.

\begin{proof}[
Proof of Lemma \ref{lem:hpr}]
Recalling condition (ii) on the foliation from Section~\ref{Lipschitz}, the first identity in \eqref{eq:holo}
together with \eqref{trivialbound} gives  
\begin{equation}\label{eq:h-a}
\begin{split}
x^s= &G_{i,j,\up}(M_A(x^s),0)+\int_0^{\bh^s_A(x^s)}\!\!\!\!\!\!\!\!ds\;
\partial_{x^s} G_{i,j,\up}(M_A(x^s),s)\\
=&
\int_0^{M_A(x^s)}\!\!\!\!\!\!\!\! du\; \partial_{x^u}G_{i,j,\up}(u,0)+ \int_0^{\bh^s_A(x^s)}\!\!\!\!\!\!\!\! ds\left[1+\int_0^{M_A(x^s)}\!\!\!\!\!\!\!\! du\; \partial_{x^u}\partial_{x^s}G_{i,j,\up}(u,s)\right] \\
=&\bh^s_A(x^s)+\int_0^{M_A(x^s)}\!\!\!\!\!\!\!\!\!\! du\; \partial_{x^u}G_{i,j,\up}(u,0)+ \int_0^{\bh^s_A(x^s)}\!\!\!\!\!\!\!\! ds\int_0^{M_A(x^s)}\!\!\!\!\!\!\!\!\!\! du\; \partial_{x^u}\partial_{x^s}G_{i,j,\up}(u,s) \, ,
\end{split}
\end{equation}
where we have used the fact that $\partial_{x^s} G_{i,j,\up}(0,s) =1$ for each $s$
by property (ii) of the foliation.
Thus, combining the bound (vi) from Section ~\ref{Lipschitz}, that is 
$|\partial_{x^u}\partial_{x^s} G_{i,j,\up}|\le C\rho^{-4/5}=C\epsilonr^{-4\varsigma/5}$,
the bound  $|M'_A|\leq \Cs \epsilonr^{1-\frac 45\theta}$ from \eqref{condF_A}
(which implies $|M_A(x^s)|\le \Cs \epsilonr$),
and the condition  $|\bh^s_A(x^s)|\le \Cs \epsilonr^\theta$ from (\ref{trivialbound}),
we get $\bh^s_A(x^s)=x^s(1+\cO(\epsilonr^{1-4\varsigma/5}))+\cO(\epsilonr)$.

Next, differentiating the first identity in \eqref{eq:holo}, we find
\begin{equation}\label{eq:hprimo}
(\bh^s_A)'(x^s)=\frac{1-\partial_{x^u} G_{i,j,\up}(M_A(x^s),\bh^s_A(x^s))
M'_A(x^s)}{\partial_{x^s} G_{i,j,\up}(M_A(x^s),\bh^s_A(x^s))}\, .
\end{equation}
We have
\[
\begin{split}
\partial_{x^s} G_{i,j,\up}(M_A(x^s),\bh^s_A(x^s))&=\partial_{x^s} G_{i,j,\up}(0,\bh^s_A(x^s))+
\int_0^{M_A(x^s)}\!\!\!\!\!\!\! \partial_{x^u}\partial_{x^s} G_{i,j,\up}(u,\bh^s_A(x^s)) du\\
&=1+\cO(\epsilonr^{1-4\varsigma/5})\, ,
\end{split}
\]
while $|\partial_{x^u} G_{i,j,\up}|_\infty\leq \Cs$, hence we have the second inequality of 
Lemma ~ \ref{lem:hpr}.
\begin{equation}\label{eq:deriv-done}
|1- (\bh^s_A)'|\leq \Cs \epsilonr^{1-\frac 45\varsigma} \, ,
\end{equation}
while the bound on  $|(\bh^s_A)(x^s)-x^s|$ follows by integration.
Finally,  $\bh^s_A$ is invertible  and the claimed bound on
$|(\bh^s_A)^{-1}(x^s)-x^s|$ holds 
using $((\bh^s_A)^{-1}(x^s) - x^s)' = \frac{1}{(\bh^s_A)' \circ (\bh^s_A)^{-1}} - 1$
and integrating as before.
\end{proof}

\begin{proof}[Proof of Lemma~\ref{lem:cancel}]
To start with, note that formula \cite[(E.1)]{BaL} is obtained by a purely geometric argument and uses only that the strong manifolds are in the kernel of the contact form and the weak manifolds in the
kernel of its differential. Since the exact same situation holds here, we have, for all relevant manifolds 
$W_A$, the formula
\begin{equation}\label{eq:symplectic-geo}
\bomega_A(x^s)=\int_0^{x^s}ds\int_0^{M_A((\bh_A^s)^{-1}(s))}du\;\partial_{x^s}G_{i,j\up}(u,s)\, .
\end{equation}
Note that this implies
\begin{equation}\label{eq:omegaC0}
|\bomega_A|_{\infty}\leq \Cs \epsilonr^{1+\theta} \, .
\end{equation}
To obtain the formula we are interested in for each pair $A,B$, it suffices to differentiate. 
Remembering properties (ii) and (vi) of the foliation constructed in Section \ref{Lipschitz}, we have
\begin{equation}\label{eq:o-ab}
\begin{split}
\partial_{x^s} \bomega_B(x^s)-\partial_{x^s} \bomega_A(x^s)&=\int_{M_A((\bh_A^s)^{-1}(x^s))}^{M_B((\bh_B^s)^{-1}(x^s))}\!du\;\partial_{x^s} G_{i,j,\up}(u,x^s)\\
&=\int_{M_A((\bh_A^s)^{-1}(x^s))}^{M_B((\bh_B^s)^{-1}(x^s))}\!du\left[ 1+\int_0^u du_1\; \partial_{x^u} \partial_{x^s}G_{i,j,\up}(u_1,x^s)\right]\\
&=\left[M_B((\bh_B^s)^{-1}(x^s))-M_A((\bh_A^s)^{-1}(x^s))\right](1+\cO(\epsilonr^{1-\frac 45\varsigma}))\, .
\end{split}
\end{equation}
Next,  the distance between $W_A^0\cap \{x^s=0\}$ and $W_B^0\cap \{x^s=0\}$ is given by $|M_A(0)-M_B(0)|=:d(W_A,W_B)$. Then, by the argument developed in 
\eqref{eq:shortcutG} (proof of Lemma~ \ref{lem:smallbox}), we have
\begin{equation}\label{eq:onceclosealways}
|\,|M_A(0)-M_B(0)|-|M_A(x^s)-M_B(x^s)|\,|\leq \Cs d(W_A,W_B)\epsilonr^{\theta/5} \, .
\end{equation}
On the other hand, by \eqref{eq:h-a} and property (vi) of the foliation, we have 
\begin{equation}\label{eq:h-ab}
\begin{split}
\big|(\bh^s_B)^{-1}(x^s) -(\bh^s_A)^{-1}(x^s)\big| &=
\left|\int_{M_A\circ (\bh^s_A)^{-1}(x^s)}^{M_B\circ (\bh^s_B)^{-1}(x^s)} \!\!\!\!\!\!\! 
du \; \partial_{x^u} G_{i,j,\up}(u,0) \right| \\
& \qquad + \left|\int_{M_A\circ (\bh^s_A)^{-1}(x^s)}^{M_B\circ (\bh^s_B)^{-1}(x^s)}\!\!\!\!\!\!\!du\int_0^{x^s}\! ds\; \partial_{x^u}\partial_{x^s}G_{i,j,\up}(u,s)\right|  \\
&\leq \Cs (1+ \epsilonr^{\theta-\frac{4}{5} \varsigma}) \left|M_B\circ(\bh^s_B)^{-1}(x^s)-M_A\circ (\bh^s_A)^{-1}(x^s)\right|\, .
\end{split}
\end{equation}
To conclude recall that we are working in coordinates in which $|M'_{A}|\leq \Cs \epsilonr^{1-\frac 45\theta}$, cf. the proof of Lemma \ref{lem:smallbox}, hence
\[
\begin{split}
&\left|M_B\circ(\bh^s_B)^{-1}-M_A\circ (\bh^s_A)^{-1}\right|
\geq\left| M_B\circ (\bh^s_B)^{-1}-M_A\circ (\bh^s_B)^{-1}\right|-\Cs \epsilonr^{1-\frac 45\theta}|(\bh^s_B)^{-1}-(\bh^s_A)^{-1}|\\
&\qquad \geq\left| M_B\circ (\bh^s_B)^{-1}-M_A\circ (\bh^s_B)^{-1}\right|
-\Cs (\epsilonr^{1-\frac 45\theta} + \epsilonr^{1-\frac 45 \varsigma+ \frac 15 \theta})\left|M_B\circ(\bh^s_B)^{-1}-M_A\circ (\bh^s_A)^{-1}\right|\, ,
\end{split}
\]
which, together with \eqref{eq:o-ab} and \eqref{eq:onceclosealways}, proves \eqref{eq:below-o}. 

To prove the second statement, let us introduce the shorthand notation $\bomega_{A,B}(x^s)=\bomega_A(x^s)-\bomega_B(x^s)$ and $\Beta_A=(\bh_A^s)^{-1}$, $\Beta_B=(\bh_B^s)^{-1}$.
By \eqref{eq:o-ab} we have
\[
\begin{split}
&\left|\partial_{x^s}\bomega_{A,B}(x^s)-\partial_{x^s}\bomega_{A,B}(y^s)\right|=\left|\int_{M_A(\Beta_A(x^s))}^{M_B(\Beta_B(x^s))}\hskip-.6cm du\;\partial_{x^s} G_{i,j,\up}(u,x^s)-\int_{M_A(\Beta_A(y^s))}^{M_B(\Beta_B(y^s))}\hskip-.6cm du\;\partial_{x^s} G_{i,j,\up}(u,y^s)\right|\\
&\qquad \leq\left|\int_{M_A(\Beta_A(x^s))}^{M_B(\Beta_B(x^s))}\!du\;\partial_{x^s} G_{i,j,\up}(u,x^s)- \partial_{x^s}G_{i,j,\up}(u,y^s)\right|\\
&\qquad\qquad\quad+\left|\int_{M_B(\Beta_B(y^s))}^{M_B(\Beta_B(x^s))}\!du\;\partial_{x^s} G_{i,j,\up}(u,y^s)\right|+\left|\int_{M_A(\Beta_A(x^s))}^{M_A(\Beta_B(y^s))}\!du\;\partial_{x^s} G_{i,j,\up}(u,y^s)\right|
\, .
\end{split}
\]
Notice that  (ii) of the foliation implies
\[
\begin{split}
\partial_{x^s}G_{i,j,\up}(u,x^s)- \partial_{x^s}G_{i,j,\up}(u,y^s)=&\partial_{x^s}G_{i,j,\up}(u,x^s)-\partial_{x^s} G_{i,j,\up}(u,y^s)-\partial_{x^s} G_{i,j,\up}(0,x^s)\\
&\qquad+\partial_{x^s} G_{i,j,\up}(0,y^s) \, .
\end{split}
\]
Thus, the four-point property (vii) (applied to the points $(u,x^s)$ and $(0,y^s)$) implies
\begin{equation}\label{eq:four-point-at-work}
|\partial_{x^s}G_{i,j,\up}(u,x^s)- \partial_{x^s}G_{i,j,\up}(u,y^s)|\leq \Cs \epsilonr^{-(\frac 45+\frac{11\ho}{15})\varsigma+1-7\ho}|x^s-y^s|^\ho
 \, .
\end{equation}
On the other hand, equation \eqref{condF_A} and Lemma \ref{lem:hpr} imply
\[
|M_B(\Beta_B(x^s))-M_B(\Beta_B(y^s))|\leq \Cs \epsilonr^{1-\frac45\theta}|x^s-y^s|\, ,
\]
and the same for $A$. Remembering that $|x^s|, |y^s|\leq c\epsilonr^\theta$, property (v) of the foliation, and our conditions on $\ho, \theta,\varsigma$ from \eqref{eq:varvar}, the above facts yield
\[
\left|\partial_{x^s}\bomega_{A,B}(x^s)-\partial_{x^s}\bomega_{A,B}(y^s)\right|\leq  \Cs \epsilonr |x^s-y^s|^\ho+\epsilonr^{1-\frac45\theta}|x^s-y^s|\leq\Cs \epsilonr|x^s-y^s|^\ho
\]
which, together with \eqref{eq:omegaC0}, proves the second statement of  Lemma~\ref{lem:cancel}.

To continue, let 
\[
\BG_{A,B}(x^s, x^0)=\frac{[(m-1)!]^2}{(\ell\vu)^{2m-2}}
e^{-2a\ell \vu}
{\bf G}^*_{\ell,m,i,A}(x^s, x^0)\overline{{\bf G}^*_{\ell,m,i,B}(x^s, x^0)} \, .
\]
Next we introduce a sequence $\{w_j\}_{j=0}^M\subset \bR$ 
such that $w_0=-cr^\theta$ and $\partial_{x^s}\bomega_{A,B}(w_j)(w_{j+1}-w_j)=2\pi b^{-1}$ and let $M\in\bN$ be such that $w_M\leq c \epsilonr^\theta$ and $w_{M+1}> c\epsilonr^\theta$.\footnote{ Note that $M$ exists and is finite by the first part of the lemma, \eqref{eq:below-o}.} Also, we 
set $\bdelta_j=w_{j+1}-w_j$. 
By Lemma \ref{lem:hpr} it follows that, for each $x^s\in [w_j,w_{j+1}]$,
\[
|\bomega_{A,B}(x^s)-\bomega_{A,B}(w_j)-\partial_{x^s}\bomega_{A,B}(w_j)(x^s-w_j)|\leq \Cs \epsilonr\bdelta_j^{1+\ho}\, .
\]
In addition, the bounds in  \eqref{eq:hoG} imply
\[
|\BG_{A,B}(x^s,x^0)-\BG_{A,B}(w_j,x^0)|\leq \Cs \bdelta_j \epsilonr^{-3\theta}  \, .
\]
Then, using the first part of the Lemma,
\[
\begin{split}
&\left|\int_{w_j}^{w_{j+1}} e^{-ib\bomega_{A,B}(x^s)} \BG_{A,B}(x^s,x^0) d x^s\right|\\
&\quad=\left|
\int_{w_j}^{w_{j+1}}  e^{-ib[\partial_{x^s}\bomega_{A,B}(w_i)(x^s-w_j)+\cO(\epsilonr\bdelta_j^{1+\ho})]} [\BG_{A,B}(w_j, x^0)+\cO(\epsilonr^{-3\theta}\bdelta_j)] dx^s\right|\\
&\quad\leq\Cs\left(b\bdelta_j^{1+\ho}\epsilonr^{-2\theta+1}+\epsilonr^{-3\theta}\bdelta_j\right)\bdelta_j
\leq\Cs\left(\frac{\epsilonr^{-2\theta+1}}{d(W_A,W_B)^{1+\ho} b^\ho}+\frac{\epsilonr^{-3\theta}}{d(W_A,W_B) b}\right)\bdelta_j \, .
\end{split}
\]
We may be left with the integral over the interval $[w_M,cr^\theta]$ which is trivially bounded by $\Cs \epsilonr^{-2\theta}\bdelta_M\leq \Cs [\epsilonr^{2\theta}bd(W_A, W_B)]^{-1}$.
The statement follows since the manifolds we are considering have length at most $c\epsilonr^\theta$, hence $\sum_{j=0}^{M-1}\bdelta_j \leq c\epsilonr^\theta$.
\end{proof}

\begin{proof}[Proof of Lemma~\ref{lem:disco}]
We start by introducing a function $\bRtime:W\to\bN$ such that $\bRtime(\xi)$ is the first $t\in\bN$ at which $\Phi_{-t\vu}\xi$ belongs to a component of $\Phi_{-t\vu}W$ of size larger than $\kappa_* L_{0}$, $\kappa_*<1/3$, and distant more than $L_0$ from $\partial\Omega_0$.

\begin{remark}\label{rem:atleastL03} We choose $\kappa_*$ such that, if $\Phi_{-t\vu}W$ is a regular piece of size larger than $L_{0}/3$ but in an $L_0$ neighborhood of $\partial\Omega_0$, then either $\Phi_{-(t+k)\vu}W$ or $\Phi_{-(t-k)\vu}W$ will satisfy our requirement for some $\Cs>k\vu>L_0$.
\end{remark}

We define then $\Rtime(\xi)=\min\{\bRtime(\xi),\ell\}$. Let $\cP=\{J_{i}\}$ be the coarser partition of $W$ in intervals on which $\Rtime$ is constant. Note that, for each $W_{B,i}$, $\Phi_{\ell\vu}(W_{B,i})\subset J_{j}$ for some $J_{j}\in\cP$.

Let $\Sigma_{\ell,j}=\{(B,i)\;:\; i\in\bN,\, B\in E_{\ell,i},\, d(W_{B,i}, W^0_{A,i})\leq \rho_{*},\; \Phi_{\ell\vu}W_{B,i}\subset J_{j}\}$. Then, by Lemma~ \ref{lem:distortion}, for each $(B,i)\in \Sigma_{\ell,j}$
\begin{equation}\label{eq:zerozeroZ}
Z_{\ell,B,i} = \int_{W_{B,i}}J^s_{\ell \vu}\leq \Cs\frac{|J_{j}|}{|\overline W_{j}|}\int_{W_{B,i}} J^s_{(\ell-\Rtime_j) \vu}
 \leq \Cs\frac{|J_{j}|}{|\overline W_{j}|}|\Phi_{(\ell-R_{j})\vu}W_{B,i}|\,,
\end{equation}
where $\Rtime_{j}=\Rtime(J_{j})$ and $\overline W_{j}=\Phi_{-\Rtime_{j}\vu}J_{j}$. Note that, by construction, either $|\overline W_{j}|\geq \kappa_* L_{0}$ or $\Rtime_j=\ell$ and $\overline W_{j}=W_{B,i}$.  Next, consider the local weak stable surfaces
$W_{B,i}^{0} := \cup_{t \in [-cr^\theta, cr^\theta]} \Phi_t W_{B,i}$ and
$\overline W^{0}_j := \cup_{t \in [-cr^\theta, cr^\theta]} \Phi_t \overline W_{j}$.

Let us analyze first the case in which $\Rtime_j<\ell$.
Let $\rho\leq \Cs L_0^{\frac 53}$. Then, by assumption, $\overline W_{j}$ is a manifold with satisfies condition (a) at the beginning of Section~\ref{Lipschitz}. Indeed $\overline W_{j}$ is too long to belong to $\bH_k$ with $k\geq C\rho^{-\frac 15}$, so that $\overline W_j$ satisfies condition (b)
of that section as well. We can then use the construction in Section~\ref{Lipschitz} 
to define an approximate unstable foliation in a $\rho$ neighborhood of 
the surface $\overline W^{0}_j$.  
Let $\Gamma_{B,i}$ be the set of leaves that intersect 
$\Delta_{\ell-\Rtime_{j}}\cap \Phi_{(\ell-\Rtime_{j})\vu}W^{0}_{B,i}$. 
By the construction of the covering $B_{c\epsilonr^\theta}(x_i)$, the 
$\Gamma_{B,i}$ can have at most $\Cs$ overlaps and since $W_B$ completely crosses
$B_{c\epsilonr^\theta}(x_i)$, there can be no gaps between the curves in $\Gamma_{B,i}$.
Now using the uniform transversality between the stable, unstable and flow directions, 
\[
\sum_{(B,i)\in \Sigma_{\ell,j}}m(\Gamma_{B,i}) \geq \Cs\!\!\sum_{(B,i)\in \Sigma_{\ell,j}}|\Phi_{(\ell-\Rtime_{j})\vu}W_{B,i}|\rho \epsilonr^\theta .
\]

Accordingly, for each $j$ such that $\Rtime_j<\ell$, remembering \eqref{eq:zerozeroZ},
\begin{equation}\label{eq:common-1}
\sum_{(B,i)\in \Sigma_{\ell,j}}Z_{\ell,B,i}\leq \Cs|J_{j}| \rho^{-1} \epsilonr^{-\theta} \sum_{(B,i)\in \Sigma_{\ell,j}}m(\Phi_{-(\ell-\Rtime_{j})\vu}\Gamma_{B,i}) \,,
\end{equation}
where we have used the invariance of the volume.
Remembering that the $\Phi_{-(\ell-\Rtime_{j})\vu}\Gamma_{B,i}$ have a fixed maximal number of overlaps and since they must be all contained in a box containing $\widetilde O$
of length $\Cs \epsilonr^\theta$ in the flow direction, of length $\Cs \epsilonr^\theta$ in the stable direction and 
of length $\Cs(\rho_{*} + \Lambda^{-(\ell-\Rtime_j)\vu}\rho)$ in the unstable direction, 
we have,\footnote{ Remember that the $J_j$ are all disjoint and $\cup_j J_j =W$ and $|W|\leq L_0$.}
\begin{equation}\label{eq:common-2}
\begin{split}
\sum_j\sum_{(B,i)\in \Sigma_{\ell,j}}Z_{B,i}\leq &\sum_{\{j\;:\; \Rtime_{j}\le \frac \ell 2\}} \Cs |J_{j}| \, \rho^{-1} \epsilonr^\theta (\rho_{*}+\Lambda^{-\frac{\ell}2\vu} \rho )  
+\sum_{\{j\;:\; \Rtime_{j}> \frac\ell 2\}}\sum_{(B,i)\in \Sigma_{\ell,j}}Z_{B,i}\, .
\end{split}
\end{equation}
To estimate the sum on $\Rtime_{j}> \frac\ell 2$ we use the growth Lemma \ref{lem:growth}(a), with 
$\oldbetazero=0$, which, remembering Remark \ref{rem:atleastL03}, implies
\[
\begin{split}
\sum_{\{j\;:\; \Rtime_{j}>\frac  \ell 2\}}\sum_{(B,i)\in \Sigma_{\ell,j}}Z_{B,i}&\leq \sum_{\{j\;:\; \Rtime_{j}> \frac \ell 2\}} \sum_{(B,i)\in \Sigma_{\ell,j}}\Cs\int_{\Phi_{\ell\vu/2}W_{B,i}}J^s_{\Phi_{\ell\vu/2}W_{B,i}}\Phi_{\ell\vu/2}\\
&\leq\sum_{W_B\in \cI_{\ell\vu/2}(W)} \Cs \int_{W_{B}}J^s_{W_{B}}\Phi_{\ell\vu/2} \\
&\leq \sum_{W_B\in \cI_{\ell\vu/2}(W)} \Cs L_0|J^s_{W_{B}}\Phi_{\ell\vu/2}|_{\cC^0}\leq \Cs \lambda^{\ell \vu/2}\, .
\end{split}
\]
Since $\lambda > \Lambda^{-1}$, the lemma follows by choosing $\rho= \rho_{*}^{1/2}$.
\end{proof}


\end{document}